\documentclass[a4paper,10pt]{amsart}
\usepackage{amsmath,amsthm,amssymb,enumerate}
\usepackage{stmaryrd}
\usepackage{epsfig}
\usepackage{amssymb}
\usepackage{amsmath}
\usepackage{amssymb}
\usepackage{amsmath,amsthm}
\usepackage{eqparbox}
\usepackage[latin1]{inputenc}
\usepackage{bm}
\usepackage{bbm}
\usepackage{esint}

\usepackage{tikz-cd}
\usetikzlibrary{decorations.pathreplacing}
\usepgflibrary{shadings} 
\usepackage{amsfonts}
\usepackage{amsxtra}
\usepackage{euscript,mathrsfs}
\usepackage{color}
\usepackage{verbatim}
\usepackage[left=3.75cm,right=3.75cm,top=3cm,bottom=2.5cm]{geometry}
\usepackage[colorlinks=true, linktocpage=true, linkcolor=red!70!black, citecolor=green!50!black, urlcolor=blue!50!black]{hyperref}
\allowdisplaybreaks

\usepackage{tikz-cd}
\usetikzlibrary{decorations.pathreplacing}

\usepackage{enumitem}
\setenumerate{label={\rm (\alph{*})}}

\usepackage{amsfonts}
\usepackage{amsxtra}

\numberwithin{equation}{section}

\usepackage{xcolor}
\colorlet{ColorPink}{red!30}
\usepackage{graphicx}

\newcommand{\R}{\mathbb R}

\newcommand{\FF}{{\boldsymbol F}}
\newcommand{\Sbold}{{\boldsymbol S}}

\newcommand{\exterior}{\mathrm{ext}}
\newcommand{\interior}{\mathrm{int}}

\renewcommand\ae{{a.\@e.\@}}

\renewcommand\d{{\rm d}}

\newcommand\qq{\qquad}
\newcommand{\GG}{\mathbf{G}}

\newcommand{\tosc}{\mathrm{tOsc}}
\newcommand{\dista}{\operatorname{dist}}
\newcommand{\divm}{\mathrm{div}_{\tau}}
\newcommand{\dif}{\mathrm{d}}

\DeclareMathOperator{\dist}{dist}

\DeclareMathOperator{\diver}{div}

\renewcommand{\dif}{\operatorname{d}\!}
\newcommand{\lebe}{\operatorname{L}}
\newcommand{\sobo}{\operatorname{W}}

\newcommand{\locc}{\operatorname{loc}}

\newcommand{\hold}{\operatorname{C}}

\newcommand{\uu}{\mathbf{u}}

\newcommand{\bv}{\operatorname{BV}}

\newcommand{\ball}{\operatorname{B}}
\newcommand{\di}{\operatorname{div}}

\newcommand{\besov}{\operatorname{B}}

\newcommand{\mres}{\mathbin{\vrule height 1.6ex depth 0pt width
0.13ex\vrule height 0.13ex depth 0pt width 1.3ex}}

\newcommand{\dashint}{\fint}
\newcommand{\spt}{\operatorname{spt}}
\newcommand{\bmo}{\operatorname{BMO}}

\allowdisplaybreaks

\renewcommand\d{\mathrm{d}}

\renewcommand\qq\qquad

\newcommand{\bphi}{\bm{\varphi}}
\newcommand{\bmu}{\bm{\mu}}

\newcommand{\cm}{\mathscr{C}\!\mathscr{M}}
\newcommand{\curl}{\mathrm{curl\,}}
\renewcommand{\d}{\mathrm{d}}

\newcommand{\lla}{\langle\!\langle}
\newcommand{\rra}{\rangle\!\rangle}

\theoremstyle{plain}

\newtheorem*{theorem*}{Theorem}
\newtheorem{theorem}{Theorem}[section]
\newtheorem{lem}[theorem]{Lemma}

\newtheorem{prop}[theorem]{Proposition}
\newtheorem{corollary}[theorem]{Corollary}

\newtheorem{definition}[theorem]{Definition}

\theoremstyle{definition}

\newtheorem{rem}[theorem]{Remark}
\newtheorem{example}[theorem]{Example}

\begin{document}


\title[Curl-Measure Fields, the Stokes Theorem and Vorticity Fluxes]{Curl-Measure Fields, the  Generalized\\ Stokes Theorem and Vorticity Fluxes}
\author[G.-Q. Chen]{Gui-Qiang G. Chen}
\address{G.-Q.G.C.: Mathematical Institute, University of Oxford, Andrew Wiles Building,
Radcliffe Observatory Quarter, Woodstock Road, Oxford OX2 6GG
United Kingdom}
\email{gui-qiang.chen@maths.ox.ac.uk}
\author[F. Gmeineder]{Franz Gmeineder}
\address{F.G.: Department of Mathematics and Statistics, University of Konstanz, Universit\"{a}tsstra\ss e 10, 78464 Konstanz, Germany}
\email{franz.gmeineder@uni-konstanz.de}
\author[M. Torres]{Monica Torres}
\address{M.T.: Department of Mathematics, Purdue University, 150 N. University Street, West Lafayette, IN 47907, USA}
\email{torresm@purdue.edu}

\date{\today}
\maketitle

\begin{abstract}
We introduce and analyze the class $\mathscr{CM}^{p}$ of curl-measure fields that are
$p$-integrable vector fields whose distributional curl is a vector-valued finite Radon measure. 
These spaces provide a unifying framework for problems involving vorticity.
A central focus of this paper is the development of Stokes-type theorems in
{low-regularity regimes},
made possible by new trace theorems for curl-measure fields. 
To this end, we introduce Stokes functionals on so-called good manifolds, 
defined by the finiteness of manifold-adapted maximal operators. 
Using novel techniques that may be of independent interest, we establish results 
that are new even in classical settings, such as Sobolev spaces 
or their curl-variants $\mathrm{H}^{\curl}(\R^{3})$, 
which arise, for example, in the study of Maxwell's equations. 
The sharpness of our theorems is illustrated through several fundamental examples.

\end{abstract}

\setcounter{tocdepth}{1}
\tableofcontents

\section{Introduction}
The main purpose of this paper is to introduce and analyze a new class of weak vector fields, called $\mathscr{CM}^{p}$ curl-measure fields -- that is, $p$-integrable vector fields whose distributional curl is a vector-valued finite Radon measure. Building on new trace theorems for such fields, we develop Stokes-type theorems that remain valid in low-regularity regimes, well beyond the classical smooth setting. This framework establishes a unified analytic setting for problems involving singular vorticity, encompassing phenomena such as vortex sheets, concentrated swirls, and other irregular structures in fluid and continuum mechanics. 

A wealth of problems arising in fluid or continuum mechanics involves the description of vortices. For instance, the description of how gases or waters are swirled directly leads to the study of their vorticity, in turn being expressed as the curl of the flow velocity. On the other hand, the mathematical modeling of electromagnetism is based on the vorticity of the magnetic fields. This is manifested in Maxwell's equations, and a combination with fluid mechanics then gives rise to the equations of magnetohydrodynamics. For example, the latter are employed in the mathematical study of plasmas. 

The analytic treatment of such phenomena necessitates not only
suitable weak formulations of the underlying problems, but also suitable function spaces which capture the essential features one aims to describe. 
Depending on the specific physical aspects, this task requires different spaces 
that
reflect the anticipated behaviour. 
This might be point swirls or along lines 
or surfaces, 
for example,
in the description of whirlwinds in geophysical applications, 
or discontinuities of the tangential fluid velocities, 
as is the case for vortex sheets, to be addressed in detail below.

In this paper,
we introduce and analyze the class of \emph{curl-measure fields}
which gives a unifying framework for such vorticity-based problems. 

\begin{definition}[Curl-measure fields]\label{def:cmfields}
Let $1\leq p\leq \infty$. Given an open subset $\Omega\subset\R^{3}$,
the space of \emph{$p$-curl measure fields} over $\Omega$ is defined as the linear space 
\begin{align}
    \cm^{p}(\Omega) :=\{ \FF \in\lebe^{p}(\Omega;\R^{3})\,\colon\;\curl \FF \in \mathrm{RM}_{\mathrm{fin}}(\Omega;\R^{3})\}, 
\end{align}
where $\curl \FF \in\mathrm{RM}_{\mathrm{fin}}(\Omega;\R^{3})$ means that the distributional curl of $\FF$ can be represented by an $\R^{3}$-valued finite Radon measure. 
\end{definition}

Curl-measure fields in the sense of Definition \ref{def:cmfields} generalize the usual $\mathrm{H}^{\curl}$- or $\sobo^{\curl,p}$-spaces, which we recall to be defined via 
\begin{align}\label{eq:sobocurl}
\begin{split}
&\mathrm{H}^{\curl}(\Omega):=\{\FF\in\lebe^{2}(\Omega;\R^{3})\,\colon\;\curl \, \FF \in\lebe^{2}(\Omega;\R^{3})\},\\ 
&\sobo^{\curl,p}(\Omega):=\{\FF\in\lebe^{p}(\Omega;\R^{3})\,\colon\;\curl \, \FF \in\lebe^{p}(\Omega;\R^{3})\}.
\end{split}
\end{align}
For instance, such spaces arise naturally in the study of the Maxwell equations 
or vorticity-based problems in fluid mechanics. In view of boundary value problems 
on \emph{e.g.} Lipschitz domains and thus boundary traces, the situation for spaces \eqref{eq:sobocurl} is well-understood; see Alonso \& Valli \cite{Alonso}, Buffa et al. \cite{BuffaCiarlet1,BuffaCiarlet2,BuffaOverview,BuffaCostabelSheen}, Sheen \cite{Sheen}, 
and Tartar \cite{Tartar1997}: There exists a (surjective) tangential trace operator 
\begin{align}\label{eq:tracep}
\mathrm{tr}_{\tau}^{p}\colon \sobo^{\curl,p}(\Omega)\to \mathcal{X}_{\partial\Omega}^{p}:=\left\{T\in\sobo^{-1/p,p}(\partial\Omega;\R^{3})\,\colon\!\!\!\, \begin{array}{l}\,T\cdot\nu_{\partial\Omega}=0,\\ \;\mathrm{div}_{\tau}(T)\in\sobo^{-1/p,p}(\partial\Omega)\end{array}\!\!\!\right\}
\end{align}
with the distributional tangential divergence $\mathrm{div}_{\tau}$ and the inward unit normal $\nu_{\partial\Omega'}\colon\partial\Omega'\to\mathbb{S}^{2}$. We refer the reader to Appendix \hyperref[sec:AppendixC]{C}, where this mapping property is revisited. 

In various situations, however, one requires function spaces which describe concentration effects of the curls. A simple yet typical instance thereof is given by vortex sheets. Such sheets arise when two layers of fluids slip over another: Consider an open and bounded container $\Omega=\Omega_{1}\cup\Omega_{2}\subset\R^{3}$, and let $\mathbf{u}\colon \Omega\to\R^{3}$ be the velocity field of a fluid.
If the normal velocity of the fluid is continuous along $\Sigma:=\Omega\cap\partial\Omega_{1}$,
but the tangential velocities are discontinuous along $\Sigma$, 
then $\Sigma$ is referred to as a vortex sheet; see Figure \ref{fig:vortflux}. In this case,
the vorticity $\curl\mathbf{u}$ of the fluid has support on $\Sigma$. 
Even when the fluid velocities are bounded, so $\mathbf{u}\in\lebe^{\infty}(\Omega;\R^{3})$, 
it is 
where the more classical spaces \eqref{eq:sobocurl} do not suffice for a global description. Indeed, the tangential discontinuities cause the distributional expression $\curl\mathbf{u}\in\mathscr{D}'(\Omega;\R^{3})$ to be a (finite Radon) measure on $\Sigma$ which is \emph{not} absolutely continuous with respect to $\mathscr{L}^{3}$; hence, $\mathbf{u}\notin \sobo^{\curl,p}(\Omega)$ for any $1\leq p \leq \infty$.  However, one then has $\mathbf{u}\in\mathscr{CM}^{\infty}(\Omega)$ in the sense of Definition \ref{def:cmfields}. Without introducing the underlying spaces explicitly, such functions have been considered, \emph{e.g.} by Majda \cite{Majda}. 

From a mathematical perspective, $\mathscr{CM}^{p}$-fields thus provide a framework that includes both more classical spaces of weakly differentiable functions such as $\sobo^{\curl,p}$ but also allows for the incorporation of concentrations of the vorticities into various models.

\subsection{Vorticity fluxes}

In physics or engineering, one of the most fundamental problems is the description of \emph{vorticity fluxes} through a surface $\Sigma\Subset\Omega\subset\R^{3}$ oriented by $\nu\colon\Sigma\to\mathbb{S}^{2}$; see Figure \ref{fig:vortflux}. By this, we understand the expression  $\curl\mathbf{u}\cdot\nu|_{\Sigma}$ provided that the underlying sufficiently smooth velocity field is  $\mathbf{u}\colon\Omega\to\R^{3}$. In this context, we refer to 
\begin{align}\label{eq:totalvortflux}
\int_{\Sigma}\curl\mathbf{u}\cdot\nu\dif\mathscr{H}^{2}
\end{align}
as \emph{total vorticity flux} through $\Sigma$. This quantity plays a crucial role in fluid mechanics or electromagnetism, eventually leading to the Maxwell equations; 
see, \emph{e.g.}, Majda and Bertozzi \cite{Majdaetal} as well as Feynman \cite{Feynman} and MacKay \cite{Mackay}. Vorticity fluxes are also instrumental in both geophysical applications \cite{McWilliams,SmithMontgomery}, aerospace engineering, and the description of circulation or stream phenomena in the human body.
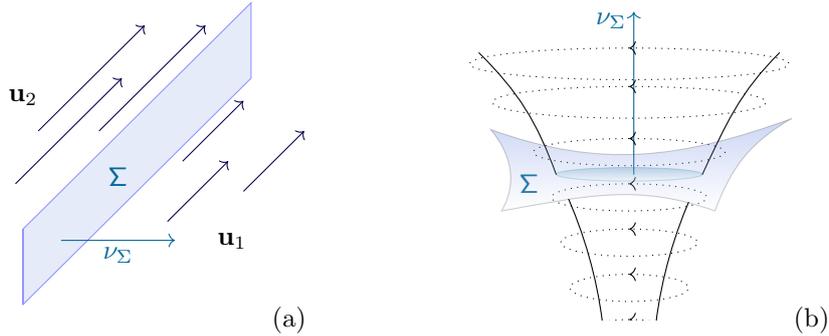
\begin{figure}
\begin{tikzpicture}
\draw[-,blue!40!white] (0,1) -- (0,0) -- (3,3) -- (3,4) -- (0,1) -- (0,0);
\draw[-,blue!60!white,fill=green!20!blue,opacity=.1] (0,1) -- (0,0) -- (3,3) -- (3,4) -- (0,1) -- (0,0);
\draw[->,blue!30!black] (1,2.3)--(2.4,3.7);
\draw[->,blue!30!black] (-0.1,1.6)--(1.3,3);
\draw[->,blue!30!black] (0.2,2.3)--(1.6,3.7);
\draw[->,blue!30!black] (2.9,1.5)--(3.7,2.3);
\draw[->,blue!30!black] (2.1,1.9)--(2.9,2.7);
\draw[->,blue!30!black] (1.9,1.1)--(2.7,1.9);
\node at (0,2.75) {$\mathbf{u}_{2}$};
\node at (2.75,0.85) {$\mathbf{u}_{1}$};
    \draw[->,green!40!blue] (0.5,0.85)--(2,0.85);
    \node[green!40!blue] at (1.25,0.65) {$\nu_{\Sigma}$};
    \node[green!40!blue] at (1.25,1.7) {$\mathsf{\Sigma}$};
    \node at (3.5,-0.2) {(a)};
\end{tikzpicture}
\hspace{1.5cm}
\begin{tikzpicture}[scale=1.2,rotate=90]
\node[green!40!blue] at (1.5,2.1) {$\mathsf{\Sigma}$};
    \draw[green!40!blue,->] (1.6,0.95) -- (3.4,0.95);
    \node[green!40!blue] at (3.3,1.2) {$\nu_{\Sigma}$};
\begin{scope}
    \clip(-0,-1) rectangle (3.2,3);
\begin{scope}
    \clip(-0,-1) rectangle (1.325,2);
\draw[-] (-4,2) [out=-20, in = 210] to (2,2);
\end{scope}
\begin{scope}
    \clip(-0,-1) rectangle (1.2975,2);
\draw[-] (-4,0) [out=20, in = -210] to (2,0);
\end{scope}
\draw[dotted,->] (0,1) arc
    [
        start angle=0,
        end angle=360,
        x radius=0.15cm,
        y radius =0.5cm
    ] ;
    \draw[dotted,->] (0.5,1) arc
    [
        start angle=0,
        end angle=360,
        x radius=0.15cm,
        y radius =0.65cm
    ] ;
    \draw[dotted,->] (1,1) arc
    [
        start angle=0,
        end angle=360,
        x radius=0.15cm,
        y radius =0.75cm
    ] ;
    \draw[dotted,->] (1.5,1) arc
    [
        start angle=0,
        end angle=360,
        x radius=0.15cm,
        y radius =0.85cm
    ] ;
    \draw[dotted,->] (2,1) arc
    [
        start angle=0,
        end angle=360,
        x radius=0.15cm,
        y radius =1.05cm
    ] ;
    \draw[dotted,->] (2.575,1) arc
    [
       start angle=0,
        end angle=360,
        x radius=0.175cm,
        y radius =1.5cm
    ] ;
    \draw[dotted,->] (3,1) arc
    [
        start angle=0,
        end angle=360,
        x radius=0.175cm,
        y radius =1.75cm
    ] ;
    \draw[-] (1.6,1.8) [out=20, in =225] to (2.95,2.65);
    \draw[-] (1.6,.2) [out=-25, in =-225] to (2.95,-0.65);
    \draw[-,top color=green!20!blue,opacity=.2,scale=1.2] (1,0.05) [out=70, in =280] to (1,2) [out=-30, in = 220] to (1.75,2.1) [out = 250, in =110] to (1.85,-0.65) [out=120,in=-20] to (1,0.05);
    \end{scope}
    \draw[green!40!blue,top color=green!40!blue,opacity=.2] (1.68,1) arc
    [
        start angle=0,
        end angle=360,
        x radius=0.075cm,
        y radius =0.80cm
    ] ;
    \node at (0,-1) {(b)};
\end{tikzpicture}
\caption{Vorticity fluxes through a surface $\Sigma$ oriented by $\nu_{\Sigma}\colon\Sigma\to\mathbb{S}^{2}$. (a) The situation of vortex sheets $\Sigma$, where the entire vorticity is concentrated on $\Sigma$. (b) The situation as encountered \emph{e.g.} in whirlwinds. }\label{fig:vortflux}
\end{figure}
For instance, they represent a crucial tool to describe the underlying mechanisms for the lift-off of aeroplanes; 
see, \emph{e.g.} Schmitz  and Chattot \cite{SchmitzChattot}. In such applications, a central tool is the \emph{Stokes theorem}, letting us express \eqref{eq:totalvortflux} in terms of circulation integrals; see \S \ref{sec:mainresults} 
for a more detailed discussion.

From a mathematical perspective, the rigorous treatment of such problems is often based on appropriate weak formulations. As a key point, this leads to spaces where one \emph{a priori} controls curl-type energy norms. 
The spaces $\mathscr{CM}^{p}(\Omega)$ from Definition \ref{def:cmfields} provide a unifying framework for the latter, even allowing for discontinuities in the tangential velocity fields and thereby being applicable to vortex sheets. Yet, even in the physically reasonable case of bounded velocities (and so $\mathbf{u}\in\mathscr{CM}^{\infty}(\Omega)$), the definition and properties of total vorticity fluxes \eqref{eq:totalvortflux} are far from obvious. This is the main theme of the present paper to be addressed in the next subsection, and is crucially intertwined with the pointwise properties of $\mathscr{CM}^{p}$-fields along surfaces $\Sigma\Subset\Omega$. 
\subsection{Main results}\label{sec:mainresults}
In view of our above discussion, it is firstly necessary to examine the behaviour of $\mathscr{CM}^{p}$-fields along surfaces or sufficiently regular boundaries $\partial\Omega'$ with  $\Omega'\Subset\Omega$ and inner unit normals $\nu_{\partial\Omega'}$. This requires suitable trace operators, and here it is clear that controlling the curl can only give us a definition of tangential traces. Inspired by  \cite{Alonso,BuffaCiarlet1,BuffaCiarlet2,BuffaOverview,BuffaCostabelSheen,Sheen,Tartar1997} in the case of $\sobo^{\curl,p}$-fields and \cite{ChenFrid1999,ChenFrid2003,ChenComiTorres,ChenTorres2005,ChenTorresZiemer,Silhavy} in the case of $\mathscr{DM}^{p}$-divergence measure fields, the definition of distributional tangential traces along Lipschitz boundaries $\partial\Omega'$ is based on the smooth integration-by-parts formula 
\begin{align}\label{eq:helpfulcurl}
  \int_{\partial\Omega'} \varphi\,\FF \times \nu_{\partial\Omega'}\dif\mathscr{H}^{2} =     \int_{\Omega'}(\varphi\,\curl  \FF - \FF \times\nabla\varphi)\dif x\
\end{align}
for any $\FF\in\hold^{1}(\Omega;\R^{3})$ and $\varphi\in\hold^{1}(\Omega)$.
Up to admitting measures, the right-hand side of \eqref{eq:helpfulcurl} also makes sense in the non-smooth context  and will serve as the definition of (distributional) tangential traces. 
Since $\curl\FF$ is a measure in our situation, this approach can be read in analogy with distributional normal traces for $\mathscr{DM}^{p}$-fields, see \cite{ChenFrid1999,ChenComiTorres}. However, unlike the latter (where the normal traces are scalar throughout), the abstractly defined  tangential traces are $\R^{3}$-valued. In this respect, it 
turns out crucial in \S \ref{sec:stokes}--\S \ref{sec:Stokesgeneral} to have suitable tangentiality results for such traces. Besides the typically different behavior of $\mathscr{CM}^{p}$-fields,  the trace theory thus is a first instance where the present setting methodologically differs from $\mathscr{DM}^{p}$-fields; see Examples \ref{ex:oscillations}, \ref{ex:limitationsp}--\ref{ex:tangentialjumps}, \ref{ex:nonex}, 
and \ref{ex:nonexctd}.
In this regard, we have
\begin{itemize}
\item \textbf{Theorems \ref{thm:tangtrace}, \ref{thm:tracemain1}, and \ref{cor:tangentiality}.} 
We establish that $\mathscr{CM}^{p}$-fields admit distributional tangential trace operators along boundaries of Lipschitz subsets. If $p=\infty$, they even belong to $\lebe^{\infty}$ and then take values in the tangent spaces to $\partial\Omega'$ $\mathscr{H}^{2}$-{a.e.}; see Theorem \ref{thm:tracemain1}. Theorem \ref{cor:tangentiality} yields the corresponding result in the case $1\leq p<\infty$, where the corresponding traces are merely distributions.
\end{itemize}
The last of these results is deferred to \S \ref{sec:Stokesgeneral}, since
we aim to address vorticity fluxes earlier in the paper. 
Most importantly, since bounded velocities are natural in many physical applications, 
we focus on $\mathscr{CM}^{\infty}$-fields first; for the latter, Theorem \ref{thm:tracemain1} suffices and even comes with convenient stronger properties.

The central tool for the description of vorticity fluxes is the \emph{Stokes theorem}, linking the vorticity flux through a surface and its circulation.  Classically, the Stokes theorem asserts that, whenever $\Sigma\subset\R^{3}$ is an oriented $\hold^{1}$-hypersurface with boundary $\Gamma_{\Sigma}$, we have\footnote{The minus sign on the right-hand side of \eqref{eq:StokesSmooth} is a matter of convention depending on the direction of $\nu$; see Figures  \ref{fig:orientationprelims} and \ref{fig:conventionsorientations}. This turns out convenient later when we work with inward unit normals.} 
\begin{align}\label{eq:StokesSmooth}
\int_{\Sigma}(\curl \FF)\cdot\nu\dif\mathscr{H}^{2} 
= - \int_{\Gamma_{\Sigma}}\FF\cdot\tau\dif\mathscr{H}^{1}\qquad\text{for all}\;\FF\in\hold^{1}(\R^{3};\R^{3}),
\end{align}
where $\nu\colon\mathrm{int}(\Sigma)\to\mathbb{S}^{2}$ is the normal to $\Sigma$ inducing its orientation, and $\tau$ denotes the tangential field to $\Gamma_{\Sigma}$. In the present situation of weakly differentiable functions, in particular curl-measure fields, not even the meaning of either side of \eqref{eq:StokesSmooth} is clear. To elaborate on this point, note that $\mathscr{CM}^{\infty}(\Omega)$ -- as $\sobo^{1,1}(\Omega;\R^{3})$ or $\bv(\Omega;\R^{3})$ -- is a space with differentiability order $s=1$ with first-order derivatives or combinations thereof being $1$-integrable or measures. By  standard principles in the theory of fine properties of functions, it is only possible to restrict or assign traces to such functions on at least $n-1=2$-dimensional, sufficiently regular subsets of $\R^{3}$; see, \emph{e.g.}, \cite{EvansGariepy,MalyZiemer}. In particular, the line integral on the right-hand side of \eqref{eq:StokesSmooth} requires to evaluate $\FF$ along a one-dimensional curve, revealing a crucial dimension gap of $1$. Yet, based on the distributional form of the classical div-curl-complex 
    \[\begin{tikzcd}
\hold^{\infty}(\R^{3}) \arrow{r}{\nabla}  & \hold^{\infty}(\R^{3};\R^{3})\arrow{r}{\curl} & \hold^{\infty}(\R^{3};\R^{3}) \arrow{r}{\mathrm{div}} & \hold^{\infty}(\R^{3}),
\end{tikzcd}
\]
it is possible to interpret an \emph{enlarged variant} of the left-hand side of \eqref{eq:StokesSmooth} as a trace functional with respect to the extended divergence-measure field $\curl\FF$ in the sense of \cite{ChenFrid2003,ChenIrvingTorres}. Hence the additional compatibility condition of $\curl\FF$ allows to assign traces on a suitable superset, but comes at the cost of having to localize suitably; see Section \ref{sec:stokes} for a detailed discussion. Firstly, working in the framework of $\mathscr{CM}^{\infty}$-fields and employing the maximal operator-based selection criteria for manifolds, we have 

	\begin{itemize}
		\item \textbf{Theorem \ref{thm:stokes}.} 
        We establish that, whenever a reference manifold $\widetilde{\Sigma}$ is fixed, the Stokes theorem is available on \emph{almost every} two-dimensional Lipschitz submanifold $\Sigma\Subset\widetilde{\Sigma}$. The identification of such good submanifolds is accomplished by the tangential maximal-type operator $\mathcal{M}^{\Psi}$ to be introduced in \S \ref{sec:maxfunctions}. As to the method of proof, we first introduce a Stokes-type functional on $\mathscr{D}(\Omega)$ (generalizing the left-hand side of \eqref{eq:StokesSmooth} to the non-smooth context) and then establish that it is a distribution of order zero on almost every Lipschitz submanifold. In essence, this leads to an expression of the form (where $\mu$ is a one-dimensional tangential trace measure along $\Gamma_{\Sigma}$)
        \begin{align}\label{eq:stokesmeasureintro}
        \int_{\Gamma_{\Sigma}}\dif\mu\;\;\;\;\text{as a substitute for the right-hand side of}\;\eqref{eq:StokesSmooth}.
        \end{align}
		\item \textbf{Theorem \ref{thm:stokes1st}.} In comparison with Theorem \ref{thm:stokes}, where a reference manifold $\widetilde{\Sigma}$ is fixed and $\Sigma\Subset\widetilde{\Sigma}$ might vary within or tangentially relative to  $\widetilde{\Sigma}$, Theorem \ref{thm:stokes1st} is concerned with variations of the reference manifold in transversal directions. The identification of good manifolds in transversal directions is accomplished by use of the normal maximal-type operator $\mathcal{M}^{\Phi}$; 
        see \S \ref{sec:maxfunctions}. In \S \ref{sec:divmeasfieldsmanif}, 
        we introduce divergence-measure fields on submanifolds of $\R^{n}$, 
        and then establish that the left-hand side of \eqref{eq:StokesSmooth} can be reduced to certain trace integrals \emph{with integrands being divergence-measure fields}. This allows us to reduce a generalized version of \eqref{eq:StokesSmooth} to the Gauss--Green theorem for divergence measure fields on manifolds.
		\end{itemize}
In view of general $\mathscr{CM}^{p}$-fields, $1\leq p \leq \infty$, we single out the case $p=\infty$ in Sections \ref{sec:stokes} and \ref{sec:divmeasfieldsmanif}. This is for two key reasons: First, in the case $p=\infty$, the Stokes theorems with respect to tangential and transversal variations are unconditional. Here, \emph{unconditional} means that the results apply to all $\mathscr{CM}^{\infty}$-fields and not to a subclass. Secondly, because of the better properties of the tangential trace operator in the case $p=\infty$, stronger results are available; specifically, this concerns the existence of the Stokes measure $\mu$ as in \eqref{eq:stokesmeasureintro} and its measure-theoretic properties; see Theorem \ref{thm:stokes} for more detail. 

Throughout, we distinguish between \emph{strong} Stokes theorems and \emph{weak} Stokes theorems (or Stokes theorems in the vorticity flux formulation). This distinction essentially relies on the set of admissible test maps; see \emph{e.g.} Theorem \ref{thm:stokes} in comparison with Corollary \ref{cor:stokesvortflux}. Whereas strong Stokes theorems can be established for $\mathscr{CM}^{\infty}$-fields, Section \ref{sec:Stokesgeneral} provides weak Stokes theorems in the case $1\leq p<\infty$ for a certain subclass $\mathscr{CM}_{1}^{p}\subset\mathscr{CM}^{p}$. The latter is the subspace of $\mathscr{CM}^{p}$-fields for which the tangential distributional divergence of the tangential trace along any manifold $\Sigma$, \emph{a priori} belonging to the dual $\mathrm{Lip}(\Sigma;T_{\Sigma})'$ by Theorem \ref{thm:tangtrace}, can be written as the tangential distributional divergence of an $\lebe^{1}$-vector field on $\Sigma$. Subject to the additional $\mathscr{CM}_{1}^{p}$-hypothesis, we have 
\begin{itemize}
    \item \textbf{Theorem \ref{thm:StokesDist}.} 
    We prove the Stokes theorem in the vorticity flux formulation for tangential variations.
    In essence, this is achieved by a localization procedure for certain tangential distributions to closed sets and makes use of some auxiliary results from \S \ref{sec:stokes} and \S \ref{sec:divmeasfieldsmanif}. 
\end{itemize}
In general, and apart from more obvious regularity scenarios to be recalled for the reader's convenience
in \S \ref{sec:StokesMostGeneral}, the validity of Stokes-type theorems seems to be novel 
even for the spaces $\sobo^{\curl,p}$. 
Solving certain elliptic equations on manifolds, it can be established that $\sobo^{p,\curl}\subset\mathscr{CM}_{1}^{p}$ for a certain range of $p$. This automatically yields the Stokes theorems 
in the vorticity flux formulation for the more classical $\sobo^{\curl,p}$-spaces 
(see Theorem \ref{thm:StokesWcurl}) and leads to overall findings as summarized in Table \ref{table:1}. 
\begin{table}[t]  
\centering 
\begin{tabular}{l c c c} 
\hline\hline   
 Space & Range of $p$ & Variations   & Result
\\ [0.5ex]  
\hline   
 & & tangential & Theorem  \ref{thm:stokes} \\[-1ex]  
\raisebox{1.5ex}{$\mathscr{CM}^{\infty}$} & \raisebox{1.5ex}{$p=\infty$}& transversal
& Theorem \ref{thm:stokes1st}  \\[2ex]  
{$\mathscr{CM}_{1}^{p}$} & {$1 \leq p < \infty$}& tangential  
& Theorem \ref{thm:StokesDist}  \\[2ex]  
 & & none  &  \\ [-1ex]
\raisebox{1.5ex}{$\sobo^{\curl,p}$} & \raisebox{1.5ex}{$3<p\leq\infty$}& (all manifolds)
& \raisebox{1.5ex}{Theorem \ref{thm:StokesWcurl}}   \\[2ex]
{$\sobo^{\curl,p}$} & {$\frac{3}{2}< p\leq 3$}& tangential  
& Theorem \ref{thm:StokesWcurl} 

\\[1ex]  
\hline 
\end{tabular}  
\vspace{0.2cm}
\caption{Stokes theorems for different classes of curl measure fields.} \label{table:1}
\end{table}  
\begin{rem}[Generalizations]
In presenting our results, we have not sought the highest possible generality. This concerns both more general operators than the curl and more irregular domains. Based on the techniques developed here, the follow-up papers \cite{ChenGmeinederStephanTorresYeh,ChenGmeinederStephanTorresYeh1} will bridge between traces in the overdetermined (elliptic) context \cite{BreitDieningGmeineder,DieningGmeineder} and  underdetermined scenarios as considered in this paper; for irregular sets, see Section \ref{sec:remsgen}.
\end{rem}

\subsection{Vorticity flux problems and interplay with  $\mathscr{DM}$-fields} Curl-measure fields are in some sense complementary to the well-established divergence measure fields, which have been intensively studied over the past three decades; see, \emph{e.g.},  
\cite{ChenFrid1999,ChenFrid2003,ChenComiTorres,ChenTorres2005,ChenTorresZiemer,ComiPayne,Silhavy2005,Silhavy}. On a formal level, this can be seen easiest via the Helmholtz decomposition 
\begin{align}\label{eq:helmholtz}
\FF(x) = \FF_{\mathrm{curl}}(x) + \FF_{\mathrm{div}}(x), \qquad x\in\R^{3}
\end{align}
for vector fields $\FF\in\hold_{\rm c}^{\infty}(\R^{3};\R^{3})$, 
where $\FF_{\curl}\in(\hold^{\infty}\cap\hold_{0})(\R^{3};\R^{3})$ is $\curl$-free and $\FF_{\mathrm{div}}\in(\hold^{\infty}\cap\hold_{0})(\R^{3};\R^{3})$ is $\di$-free. The decomposition \eqref{eq:helmholtz}, however, only serves as a motivation and does not have a straightforward $\mathscr{CM}^{p}$-$\mathscr{DM}^{p}$-analogue; indeed, the maps $
\FF\mapsto \FF_{\curl}$ and $\FF\mapsto\FF_{\mathrm{div}}$ are singular integrals of convolution type and do not map $\lebe^{1}\to\lebe^{1}$ or $\mathrm{RM}_{\mathrm{fin}}\to\mathrm{RM}_{\mathrm{fin}}$ in general. As alluded to in \S\ref{sec:mainresults}, the key difference between the single types of fields are  the components of the traces which are controlled, and it is in this sense that the $\mathscr{CM}^{p}$-fields complement the more established $\mathscr{DM}^{p}$-fields: 
\begin{example}[Gradient fields]\label{ex:oscillations}
Let $\Omega\subset\R^{3}$ be open and bounded. Whenever $v\in\lebe_{\locc}^{1}(\Omega)$ is weakly differentiable with weak gradient $\nabla v\in\lebe^{\infty}(\Omega;\R^{3})$, $\mathbf{u}:=\nabla v\in\cm^{\infty}(\Omega)$ since even $\mathrm{curl}\,\mathbf{u}=0$ in $\mathscr{D}'(\Omega;\R^{3})$. In the situation depicted in Figure \ref{fig:osc}, a function $v\in\hold(\overline{\Omega})$ with $\mathbf{u}=\nabla v\in\lebe^{\infty}(\Omega;\R^{3})$ is indicated; in particular, $A,B\in\R^{3\times 3}$ are such that $a:=A-B\neq0\in\R^{3\times 3}$, and the $\mathscr{H}^{2}$-measures of interfaces $S_{j}$ do not converge to zero. One then has 
\begin{align*}
|\mathrm{div}\,\mathbf{u}|(\Omega)\geq |a|\sum_{j=1}^{\infty}\mathscr{H}^{2}(S_{j})=\infty,
\end{align*}
whence $\mathbf{u}\notin\mathscr{D}\mathscr{M}^{\infty}(\Omega)$. 
Since $\mathbf{u}$ oscillates on arbitrarily fine scales as the boundary is approached in the normal direction, no normal traces can be assigned to $\mathbf{u}$. 
However, by construction, the tangential trace along $\partial\Omega$ is well-defined and equals zero for $\mathbf{u}$. 
\end{example}
\begin{figure}[t]
\begin{tikzpicture}[scale=0.7,rotate=90]
\draw[-] (0.5,-2.5) -- (0.5,2.5);
\draw[->] (0.5,-2.3) -- (1.5,-2.3);
\node[right] at (1,-2.3) {$\nu_{\partial\Omega}$};
\draw[-] (2,-2) -- (2,2); 
\draw[-,fill=black!20!white] (2,-2) -- (2,2) -- (1.5,1.5) -- (1.5,-1.5) -- (2,-2); 
\draw[dotted,<-] (2,-2) -- (2,-4);
\node[right] at (2,-4) {$S_{1}$};
\draw[dotted,<-] (1.5,-1.6) -- (1.5,-4);
\node[right] at (1.5,-4) {$S_{2}$};
\node[right] at (1,-4.1) {$\vdots$};
\draw[-,fill=black!10!white] (2,-2) -- (2,2) -- (2.5,1.5) -- (2.5,-1.5) -- (2,-2); 
\draw[-,fill=black!10!white] (1.5,-1.5) -- (1.5,1.5) -- (1.25,1.75) -- (1.25,-1.75)--(1.5,-1.5);
\node[black!80!white] at (1.75,3) {{\footnotesize $\nabla v = B$}};
\node[black!50!white] at (1.3,2.9) {{\footnotesize $\vdots$}};
\node[black!65!white] at (2.35,3) {{\footnotesize $\nabla v = A$}};
\draw[-,fill=black!20!white] (1.25,-1.75)--(1.25,1.75) -- (1,1.5) -- (1,-1.5) -- (1.25,-1.75);
\draw[-,fill=black!10!white] (1,-1.5) -- (1,1.5) -- (0.875,1.625) -- (0.875,-1.625)-- (1,-1.5);
\node[rotate=90] at (0.7,0) {{\tiny $...$}};
\end{tikzpicture}
\caption{Oscillations close to the boundary.}\label{fig:osc}
\end{figure}
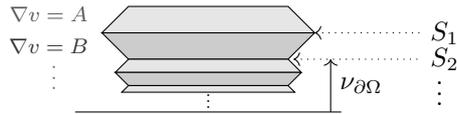
The present paper primarily focuses on properties and Stokes theorems for $\mathscr{CM}^{p}$-fields. However, these spaces are ultimately designed to yield a robust analytic framework for problems involving vortices and vorticity fluxes, potentially allowing for tangential discontinuities. 
While we defer a detailed study of models involving $\mathscr{CM}^{p}$-fields to future work, we outline several immediate consequences 
in \S \ref{sec:applications}.
Here, we first give an application of the Stokes theorem in the derivation of the Maxwell equations from electromagnetism. This relies on an interplay between divergence- and curl-based spaces in the setting of absolutely continuous measures. Secondly, to underline potential concentration effects, we foreshadow the usage of $\mathscr{CM}^{p}$-fields in the context of vortex sheets. This especially concerns a discussion of vorticity fluxes and the non-local three-dimensional Birkhoff-Rott equation on the evolution of parametrized vortex sheets 
in the framework of curl measure fields; see \eqref{eq:birkhoffrott} below and Caflisch et al. \cite{Caflisch1988,Caflisch1989,Caflisch1992} and \cite{Lopez,Meiron,Sakajo,Wu2002,Wu2006} for this system and underlying numerical studies.

\subsection{Organization of the paper}
In \S \ref{sec:prelims}, we fix notation and gather background material from measure theory, extended divergence-measure fields, 
and elementary constructions from differential geometry.  
The tangential traces for $\mathscr{CM}^{p}$-fields are introduced 
and analyzed in \S \ref{sec:tracetheorem}. 
Here, we provide several examples that will be continued in the subsequent sections, 
paying special attention to the case $p=\infty$, 
where stronger assertions are available.  
The intermediate section \S \ref{sec:maxfunctions} 
serves to introduce transversal and tangential maximal functions, 
which 
Help us identify the manifolds that support the Stokes theorem. Based on the latter,  
we then prove the Stokes theorem for $\mathscr{CM}^{\infty}$-fields for variations 
in the tangential direction in \S\ref{sec:stokes}. 
Divergence measure fields on manifolds are introduced in \S\ref{sec:divmeasfieldsmanif},
allowing us to give a proof of the Stokes theorem for $\mathscr{CM}^{\infty}$-fields 
with respect to variations in the transversal direction. 
The results of the previous sections are complemented in \S\ref{sec:Stokesgeneral} 
by discussing the case $1\leq p<\infty$. 
Here, only weaker assertions can be made, which is due to the non-integrability of the tangential traces. Finally, in \S \ref{sec:applications}, we showcase sample applications of $\mathscr{CM}^{p}$-fields and the Stokes theorems in view of problems involving vorticities. 
The appendix collects the proofs of background facts used in the main part, 
so \emph{e.g.} function spaces on manifolds and supplementary calculations in local coordinates.

\section{Preliminaries}\label{sec:prelims}
In this section, we fix notation and collect background material from measure theory, 
extended divergence-measure fields, and basic constructions in differential geometry.

\subsection{Notation}
We briefly comment on the notation. Throughout, we use bold letters, such as $\bphi$, to denote vectorial quantities, whereas non-bold ones, such as $\varphi$, are reserved for scalar quantities. We view the elements of $\R^{n}$ as column vectors, but often simply write $x=(x_{1} \cdots,x_{n})$; if it is important, we write $(x_{1},\cdots,x_{n})^{\top}$. 

\textbf{Open sets and balls.} Unless mentioned otherwise, 
$\Omega\subset\R^{n}$ always denotes an open and bounded set; typically, $n=3$. As an important convention, we denote for an open set $\Omega\subset\R^{n}$ with Lipschitz boundary by $\nu_{\partial\Omega}\colon\partial\Omega\to\mathbb{S}^{n-1}$ its \emph{inner} unit normal. The open ball of radius $r>0$ centered at $x_{0}\in\R^{n}$ is denoted by $\ball_{r}(x_{0})$, and we write 
\begin{align*}
\ball_{r}^{(n-1)}(x'_{0}):=\{(x_{1},\cdots,x_{n-1})\in\R^{n-1}\colon\;|x_{1}-x'_{0,1}|^{2} + \cdots + |x_{n-1}-x'_{0,n-1}|^{2}<r^{2}\}
\end{align*}
for the two-dimensional ball of radius $r>0$ centered at $x'_{0}=(x'_{0,1}, \cdots,x'_{0,n-1})\in\R^{n-1}$. In general, if $\ball$ is an open ball, we denote its radius by $r(\ball)$. Lastly, for $U\subset\R^{n}$, we write $\mathscr{B}(U)$ the Borel $\sigma$-algebra on $U$.

\textbf{Measures.} We use $\mathscr{L}^{n}$ and $\mathscr{H}^{s}$, $0\leq s\leq n$, to denote 
the $n$-dimensional Lebesgue measure or the $s$-dimensional Hausdorff measures, respectively. 
Moreover, given a finite-dimensional inner product space $V$, we denote by $\mathrm{RM}(\Omega;V)$ the $V$-valued Radon measures on $\Omega$, and by $\mathrm{RM}_{\mathrm{fin}}(\Omega;V)$ the finite, $V$-valued Radon measures on $\Omega$; if $V=\R$, we abbreviate $\mathrm{RM}_{(\mathrm{fin}}(\Omega):=\mathrm{RM}_{(\mathrm{fin})}(\Omega;\R)$. By finiteness we here understand that the total variation of $\bmu$, denoted by $|\bmu|(\Omega)$, is finite. Lastly, given $\bmu\in\mathrm{RM}(\Omega;V)$ and a $\bmu$-measurable set $U\subset\Omega$, we write $\bmu\mres U:=\bmu(\cdot\cap U)$ for the restriction of $\bmu$ to $U$. 

\textbf{Averages.} Given a positive measure $\mu$ on the Borel $\sigma$-algebra $\mathscr{B}(\Omega)$ and $U\in\mathscr{B}(\Omega)$ with $0<\mu(U)<\infty$, we denote the mean value of a locally $\mu$-integrable map $\FF$ by 
\begin{align*}
\dashint_{U}\FF\dif\mu := \frac{1}{\mu(U)}\int_{U}\FF\dif \mu. 
\end{align*}
As usual, if $\mu=\mathscr{L}^{3}$, we simply write $\dif x = \dif\mathscr{L}^{3}$. For $\FF\in\lebe_{\locc}^{1}(\Omega;\R^{3})$, we denote by $\mathcal{L}_{\FF}$ the set of its Lebesgue points, and write 
\begin{align}\label{eq:preciserepresentative}
\FF^{*}(x):=\begin{cases}
    \displaystyle \lim_{r\searrow 0}\dashint_{\ball_{r}(x)}\FF(y)\dif y&\;\text{if}\;x\in\mathcal{L}_{\FF}, \\ 
    0&\;\text{otherwise}, 
\end{cases}
\end{align}
for the precise representative of $\FF$. 

\textbf{Continuous functions.} We denote by $\hold_{b}(\Omega)$ the bounded continuous functions on $\Omega$, and by $\hold_{b}^{k}(\Omega)$ the $\hold^{k}$-functions on $\Omega$ with bounded derivatives up to order $k$. The Lipschitz continuous functions on $\Omega$ are denoted by $\mathrm{Lip}(\Omega)$; since $\Omega$ is assumed to be open and bounded, $\mathrm{Lip}(\Omega)=\mathrm{Lip}(\overline{\Omega})$, and we write $\mathrm{Lip}_{0}(\Omega)$ for the collection of all $f\in\mathrm{Lip}(\Omega)$ with $f|_{\partial\Omega}=0$. Moreover, we use subscripts in $\hold_{\rm c}(\Omega)$ or $\mathrm{Lip}_{\rm c}(\Omega)$ to denote the compactly supported elements of these spaces. Moreover, by a \emph{standard mollifier} on $\R^{n}$ we understand a radially symmetric function $\rho\in\hold_{\rm c}^{\infty}(\ball_{1}(0);\R_{\geq 0})$ such that $\int_{\R^{n}}\rho\dif x = 1$. For $\delta>0$, its $\delta$-rescaled version is given by $\rho_{\delta}(x):=\delta^{-n}\rho(\frac{x}{\delta})$.

\textbf{Constants.} Finally, $c,C>0$ denote generic constants that may change from one line to the other, 
and are only specific if their precise value is required in the sequel. 

\subsection{Differential geometry on $\R^{n}$}\label{sec:notionsdiffgeom}
In this subsection, we collect several basic notions and background results 
from differential geometry on $\R^{n}$. 
Most importantly, they be used from \S \ref{sec:stokes} onwards, so \emph{e.g.} in the definition of $\mathscr{DM}^{\infty}$-fields on manifolds.

Let $\Sigma\subset\R^{n}$ be a $(n-1)$-dimensional  $\hold^{1}$-submanifold of $\R^{n}$ and denote, for $x\in \Sigma$, by $T_{\Sigma}(x)$ its tangent space at $x$. We recall that $x_{0}\in \Sigma$ is called an \emph{interior point} if, for some $r>0$, $\ball_{r}(x_{0})\cap\Sigma$ is homeomorphic to an open subset of $\R^{n-1}$. The interior $\mathrm{int}(\Sigma)$ of $\Sigma$ is the collection of all interior points. By slight abuse of terminology, we then call $\Gamma_{\Sigma}:=\overline{\Sigma}\setminus\mathrm{int}(\Sigma)$ the \emph{boundary} of $\Sigma$. In the main part, we 
almost exclusively consider (orientable) manifolds $\Sigma$ with $\Sigma=\mathrm{int}(\Sigma)$. On the one hand, this comprises \emph{closed} manifolds (which, by definition, are compact manifolds $\Sigma$ without boundary, meaning that $\Gamma_{\Sigma}=\emptyset$) such as the sphere or the torus. On the other hand, this includes \emph{open} manifolds (which, by definition, are manifolds without boundary such that no component is compact) such as the boundary manifolds to be introduced in \S \ref{sec:bdrymanifolds}. 

Let $k\in\mathbb{N}$. If $\Sigma\subset\R^{n-1}$ is an $(n-1)$-dimensional  $\hold^{k}$-manifold with $\Sigma=\mathrm{int}(\Sigma)$, then the $\hold^{k}$-property of a function $\varphi\colon\Sigma\to\R$ or a map $X\colon \Sigma\to\R^{n}$ is defined as usual by local charts. 
We recall that a \emph{vector field} $X$ on $\Sigma$ is a map that assigns 
to every $x\in M$ a tangent vector $X(x)\in T_{\Sigma}(x)$. 
For $1\leq p\leq\infty$, we denote
\begin{align*}
& \hold(\Sigma;T_{\Sigma}):=\{X\;\text{is a continuous vector field on $\Sigma$}\}, \\ 
& \hold^{k}(\Sigma;T_{\Sigma}):=\{X\in\hold^{k}(\Sigma;\R^{n})\colon\;X(x)\in T_{\Sigma}(x)\;\text{for all}\;x\in\Sigma\}, \\ 
& \lebe_{(\locc)}^{p}(\Sigma;T_{\Sigma}) :=\{X\in\lebe_{(\locc)}^{p}(\Sigma;\R^{n})\colon\;X(x)\in T_{\Sigma}(x)\;\;\text{for $\mathscr{H}^{d}$-{a.e.}}\; x\in \Sigma\}.
\end{align*}
The spaces $\hold_{\rm c}^{k}(\Sigma;T_{\Sigma})$, $\mathrm{Lip}(\Sigma;T_{\Sigma})$ and $\mathrm{Lip}_{\rm c}(\Sigma;T_{\Sigma})$ are declared in the natural way. Subject to sufficient regularity, these notions inherit to $d$-dimensional Lipschitz submanifolds $\Sigma$ of $\R^{n}$, where we require $X(x)\in T_{\Sigma}(x)$ for $\mathscr{H}^{d}$-a.e. $x\in\Sigma$ throughout. 

For our later objectives, we require an integration by parts-formula on $(n-1)$-dimensional manifolds. Let $\Sigma\subset\R^{n}$ be an $(n-1)$-dimensional $\hold^{2}$-manifold oriented by $\nu\colon \Sigma\to\mathbb{S}^{n-1}$. For $T\in\hold^{1}(\R^{n};\R^{n})$, we define the \emph{tangential divergence} by 
\begin{align*}
\divm(T(x)):= \mathrm{div}(T(x))-((\nabla T(x))\nu(x))\cdot\nu(x),\qquad x\in\mathrm{int}(\Sigma).  
\end{align*}
\begin{theorem}[Smooth integration by parts-formula]\label{thm:IBPsmooth}
Let $\Sigma\subset\R^{n}$ be an oriented $(n-1)$-dimensional $\hold^{2}$-manifold oriented by $\nu\colon\Sigma\to\mathbb{S}^{n-1}$ such that  $\Sigma=\mathrm{int}(\Sigma)$ and $\overline{\Sigma}$ is compact. Moreover, suppose that, for each $x_{0}\in\Gamma_{\Sigma}$, there exists both 
an open and bounded set $U\subset\R^{n}$ and a $\hold^{2}$-diffeomorphism 
$f\colon (-1,1)^{n-1}\to U\cap\Sigma$ such that $f^{-1}(\Gamma_{\Sigma}\cap U)$ 
can be written as the graph of a Lipschitz function $g\colon(-1,1)^{n-2}\to (-1,1)$ 
together with 
\begin{align*}
f^{-1}(U\cap\Sigma)=\{(x'',x_{n-1})\in (-1,1)^{n-1}\,\colon\;g(x'')<x_{n-1}\}. 
\end{align*}
Then there exist both a normal vector field $\mathbf{H}_{\Sigma}\in\hold^{0}(\Sigma;\R^{n})$ to $\Sigma$ 
and a normal vector field $\nu_{\Gamma_{\Sigma}}\in\lebe^{\infty}(\Gamma_{\Sigma};\mathbb{S}^{n-1})$ 
such that 
\begin{align}\label{eq:IBPmanifolds}
\int_{\Sigma}\divm(T)\dif\mathscr{H}^{n-1} = \int_{\Sigma}T\cdot\mathbf{H}_{\Sigma}\dif\mathscr{H}^{n-1}-\int_{\Gamma_{\Sigma}}(T\cdot\nu_{\Gamma_{\Sigma}})\dif\mathscr{H}^{n-2}
\end{align}
holds for every $T\in\hold_{\rm c}^{1}(\R^{n};\R^{n})$. 
\end{theorem}
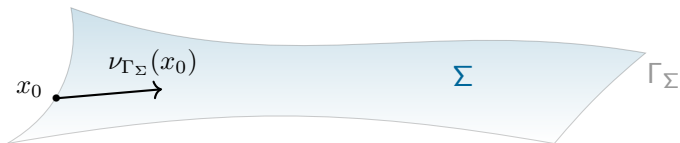
\begin{figure}
\begin{tikzpicture}[scale=1.2]
\begin{scope}
\draw[top color=green!40!blue,opacity=.2] (-3,0) [out=10, in =170] to (3,0) [out= 50, in = 220] to (4,1) [out= 170, in = -20] to (-2.3,1.5) [out= 280, in =30] to (-3,0);
\draw[->,thick] (-2.46,0.5) -- (-1.3,0.6);
\node at (-1.4,0.9) {$\nu_{\Gamma_{\Sigma}}(x_{0})$};
\node[left] at (-2.5,0.6) {$x_{0}$};
\node[-,green!40!blue] at (2,0.75) {\large $\mathsf{\Sigma}$};
\node[-,black!40!white] at (4.2,0.75) {\large $\mathsf{\Gamma}_{\mathsf{\Sigma}}$};
\node at (-2.46,0.5) {\tiny\textbullet};
\end{scope}
\end{tikzpicture}
\caption{Geometric set-up underlying \S \ref{sec:notionsdiffgeom}--\S \ref{sec:bdrymanifolds}.}\label{fig:orientationprelims}
\end{figure}
In the situation of Theorem \ref{thm:IBPsmooth},
we see that $\nu_{\Gamma_\Sigma}(x)\in T_{\Sigma}(x)$ for $\mathscr{H}^{n-2}$-a.e. $x\in\Gamma_{\Sigma}$, and the vector $\nu_{\Gamma_{\Sigma}}$ is chosen to point inside $\Sigma$; see Figure \ref{fig:orientationprelims}. 
This theorem can be found subject to higher regularity assumptions 
on $\Gamma_{\Sigma}$ in {\cite[Theorem 11.8 \emph{ff.}]{Maggi}}, 
and the version as given here follows from there in a routine way. 
The term $\mathbf{H}_{\Sigma}$ on the right-hand side of \eqref{eq:IBPmanifolds} 
is an oriented curvature term of form $\mathbf{H}_{\Sigma}=H_{\Sigma}\nu$ 
related to the  {Weingarten map}. 
If $U\subset\R^{n}$ is such that $U\cap \Sigma$ can be written as the graph of a function $g\colon \R^{n-1}\supset V \to\R$ with an open set $V\subset\R^{n-1}$, then 
\begin{align*}
H_{M}|_{U\cap \Sigma}= - \mathrm{div}'\Big(\frac{\nabla' g}{\sqrt{1+|\nabla'g|^{2}}}\Big). 
\end{align*}
Here, the dashes indicate that the gradient or the divergence of $g$ are taken with respect to the $(n-1)$ parametrizing variables. 

If, in the situation of Theorem \ref{thm:IBPsmooth}, $T\in\hold_{\rm c}^{1}(\R^{n};\R^{n})$ is such that $T(x)\in T_{\Sigma}(x)$ for $\mathscr{H}^{n-1}$-{a.e.} $x\in \Sigma$,  then $T\cdot\mathbf{H}_{\Sigma}=0$ $\mathscr{H}^{n-1}$-{a.e.} on $\Sigma$. Hence, for $\varphi\in\hold^{1}(\R^{n})$, \eqref{eq:IBPmanifolds} yields
\begin{align}\label{eq:IBPmanifolds00}
\begin{split}
\int_{\Gamma_{\Sigma}}  \varphi (T\cdot\nu_{\Gamma_{\Sigma}})\dif\mathscr{H}^{n-1} & = - \int_{\Sigma}\mathrm{div}_{\tau}(\varphi T)\dif\mathscr{H}^{n-1} \\ & = - \int_{\Sigma} \big(\varphi\,\mathrm{div}(T)- (\varphi (\nabla T)\nu)\cdot\nu\big)\dif\mathscr{H}^{n-1} \\ & \;\;\;\; - \int_{\Sigma} \big(\nabla\varphi\cdot T - ((T\otimes\nabla\varphi)\nu)\cdot\nu\big)\dif\mathscr{H}^{n-1} \\ 
& = -\int_{\Sigma}\varphi\,\divm(T)\dif\mathscr{H}^{n-1} - \int_{\Sigma}\nabla_{\tau}\varphi\cdot T\dif\mathscr{H}^{n-1}, 
\end{split}
\end{align}
since $(\nabla\varphi(x))\cdot T(x) = \nabla_{\tau}\varphi(x)\cdot T(x)$ and 
\begin{align*}
((T(x)\otimes\nabla\varphi(x))\nu(x))\cdot \nu(x)= (\nu(x)\cdot T(x))(\nabla\varphi(x)\cdot\nu(x))= 0
\end{align*}
hold for $\mathscr{H}^{n-1}$-{a.e.} $x\in \Sigma$; recall that $T(x)\in T_{\Sigma}(x)$ for $\mathscr{H}^{n-1}$-{a.e.} $x\in \Sigma$. 
If, in addition, $T(x)=0$ holds for $\mathscr{H}^{n-2}$-{a.e.} $x\in\Gamma_{\Sigma}$, then \eqref{eq:IBPmanifolds00} gives us 
\begin{align}\label{eq:IBPmanifolds0}
\begin{split}
 0  = \int_{\Sigma}\varphi\,\divm(T)\dif\mathscr{H}^{n-1} +  \int_{\Sigma}\nabla_{\tau}\varphi\cdot T\dif\mathscr{H}^{n-1}. 
\end{split}
\end{align}
\begin{rem}
By use of extensions to $\R^{n}$, Theorem \ref{thm:IBPsmooth} and \eqref{eq:IBPmanifolds00}--\eqref{eq:IBPmanifolds0} hold  true for maps $T\in\hold^{1}(\overline{\Sigma};T_{\Sigma})$. This particularly applies to the case where $\Omega\subset\R^{n}$ is open and bounded with boundary of class $\hold^{2}$,  inner unit normal $\nu_{\partial\Omega}\colon \partial\Omega\to\mathbb{S}^{n-1}$ and $\Sigma=\partial\Omega$.
\end{rem}
Next, we record a version of the coarea formula for Lipschitz maps on $\hold^{1}$-manifolds; see, \emph{e.g.}, \cite[Theorem 5.3.9]{KrantzParks} and the discussion beforehand.
\begin{lem}[Coarea formula on $\hold^{1}$-manifolds]\label{lem:coarea}
Let $N_{1},N_{2},M\in\mathbb{N}$ be such that $M\geq N_{2}$, and let $f\colon\R^{N_{1}}\to\R^{N_{2}}$ be a Lipschitz map. If $\Sigma\subset\R^{N_{1}}$ is an $M$-dimensional $\hold^{1}$-manifold, then 
\begin{align*}
\int_{\Sigma}g\,\mathbf{J}_{N_{2}}^{\Sigma}f\,\dif\mathscr{H}^{M} = \int_{\R^{N_{2}}}\int_{\Sigma\cap f^{-1}(\{y\})}g\,\dif\mathscr{H}^{M-N_{2}}\dif\mathscr{L}^{N_{2}}(y)
\end{align*}
holds for every $\mathscr{H}^{M}$-measurable function $g\colon \Sigma\to\R$. Here, $\mathbf{J}_{N_{2}}^{\Sigma}f$ denotes the Jacobian of $f$ on $\Sigma$. 
\end{lem}

\subsection{Boundary manifolds, collar theorems, and deformations}\label{sec:bdrymanifolds}
We now specify the manifolds with which we will work in the main part of the paper.

\begin{definition}[Boundary manifolds]\label{def:bdrymanifolds}
Let $n\geq 3$, and let $\Omega'\subset\R^{n}$ be open and bounded with $\hold^{k}$-boundary. 
We say that $\Sigma\subset\partial\Omega'$ is a \emph{$\hold^{k}$-regular Lipschitz boundary 
manifold relative to $\Omega'$} if either $\Sigma=\partial\Omega'$ or all of the the following hold: 
\begin{enumerate}
\item $\Sigma$ is relatively open in $\partial\Omega'$, 
\item There exist open and bounded sets $U_{1},\cdots,U_{N}\subset\R^{n}$ with the following property: 
\begin{itemize}
    \item[\emph{(i)}] $\Gamma_{\Sigma}\subset\bigcup_{\ell=1}^{N}(U_{\ell}\cap\partial\Omega')$, 
    \item[\emph{(ii)}] for every $\ell\in\{1, \cdots,N\}$, \,
    there exist a bi-Lipschitz $\hold^{k}$-diffeomorphism \\ $f_{\ell}\colon (-1,1)^{n-1}\to U_{\ell}\cap\partial\Omega'$, a number $0<\theta<1$, and a Lipschitz function $g_{\ell}\colon (-1,1)^{n-2}\to(-\theta,\theta)$ such that 
    \begin{align*}
    & f_{\ell}^{-1}(U_{\ell}\cap\Sigma) = \{(x'',x_{n-1})\in (-1,1)^{n-1}\,\colon\;\;g_{\ell}(x'')<x_{n-1} \}, \\ 
    & f_{\ell}^{-1}(U_{\ell}\cap \Gamma_{\Sigma})=\{(x'',g_{\ell}(x''))\,\colon\;|x''|<1\} (=\mathrm{graph}(g_{\ell})), \\ 
    & f_{\ell}^{-1}(U_{\ell}\cap(\partial\Omega'\setminus\overline{\Sigma}) = \{(x'',x_{n-1})\in (-1,1)^{n-1}\,\colon\;|x''|<1,\;g_{\ell}(x'')>x_{n-1} \}.
    \end{align*}
\end{itemize}
\end{enumerate}
We say that $\Sigma$ is a \emph{$\hold^{k}$-regular Lipschitz boundary manifold}  if there exists an open and bounded set $\Omega'\subset\R^{n}$ with $\hold^{k}$-boundary such that $\Sigma\subset\partial\Omega'$ is a $\hold^{k}$-regular Lipschitz boundary manifold relative to $\Omega'$. If the context is clear, we 
simply speak of \emph{boundary manifolds}.
\end{definition}

By their very definition, all $\hold^{1}$-regular Lipschitz boundary manifolds are orientable. In particular, non-orientable smooth manifolds with boundary such as the M\"{o}bius strip are not boundary manifolds. On the other hand, spiral-type manifolds such as 
\begin{align*}
\Sigma = \Big\{(\frac{1}{s+1}\cos(s),\frac{1}{s+1}\sin(s),t)\,\colon\;s>0,\;0<t<1\Big\} 
\end{align*}
are equally seen to not qualify as boundary manifolds. However, they can be written as countable union of boundary manifolds.

For future reference, mostly from Section \ref{sec:maxfunctions} onwards, we now explain what we understand by \emph{shrinking} and \emph{enlarging} boundary manifolds. In the setting considered here, this requires two sorts of collar theorems (see Lee \cite[Chapter 9]{Lee} for more on this terminology). The latter are linked to the Lipschitz deformability of Lipschitz domains (see Ball and Zarnescu \cite[\S 5]{BallZarnescu}, Chen et al. \cite[\S 8]{ChenComiTorres}, \cite{ChenFrid1999}, Hofmann et al. \cite{Hofmann}). The first of these results will be applied to $\hold^{1}$-regular Lipschitz boundary manifolds, and follows from \cite[Theorem 9.25]{Lee}. Later on in \S \ref{sec:divmeasfieldsmanif}, it helps us to formalize 
the following notion of \emph{$\mathscr{L}^{1}$-a.e. boundary manifold in  the transversal  direction}: 

\begin{lem}[Collar-type I]\label{lem:collar1}
Let $\Omega'\subset\R^{n}$ be open and bounded with Lipschitz  boundary $\partial\Omega'$ (or of class $\hold^{k}$, $k\in\mathbb{N}$).  
Then there exist both a neighborhood $\mathcal{O}$ of $\partial\Omega'$ and a bi-Lipschitz map  $\Phi_{\partial\Omega'}\colon (-1,1)\times\partial\Omega'\to\mathcal{O}$ with the following properties{\rm :} 
\begin{enumerate}
    \item\label{item:collA1} $\Phi_{\partial\Omega'}(0,\partial\Omega')=\partial\Omega'$, 
    \item\label{item:collB1} $\Phi_{\partial\Omega'}(t,\partial\Omega')\subset\Omega'$ for all $0<t<1$, 
    \item\label{item:collC1} 
    $\Phi_{\partial\Omega'}(t,\partial\Omega')\subset\R^{n}\setminus\overline{\Omega'}$ for all $-1<t<0$, and 
    \item\label{item:collD1} $\Phi_{\partial\Omega'}(t,\partial\Omega')$ is a Lipschitz manifold (or of class $\hold^{k}$) for each $-1<t<1$.
\end{enumerate}
\end{lem}
This
lemma can be deduced from \cite{BallZarnescu,ChenComiTorres,Hofmann}, and we briefly pause to discuss one possible  underlying construction. If the open and bounded set  $\Omega'$ has Lipschitz boundary $\partial\Omega'$ oriented by the inner unit normal $\nu_{\partial\Omega'}$, then \cite[Proposition 2.3 and Eq. (2.23)ff.]{Hofmann} yields the existence of a globally transversal field $X\in\hold^{\infty}(\R^{n};\R^{n})$ such that $|X|=1$ on $\partial\Omega$; this means that there exists $\kappa>0$ such that $\nu_{\partial\Omega}\cdot X\leq -\kappa$ $\mathscr{H}^{n-1}$-a.e. on $\partial\Omega$. Thus, by continuity, $0<\varepsilon\leq |X|\leq L<\infty$ holds in an open neighborhood $U$ of $\partial\Omega$. By normalization, we thus find $h\in\hold^{\infty}(\R^{n};\R^{n})$ such that $|h|=1$ in $U$ and $\nu_{\partial\Omega'}\cdot h\leq -\widetilde{\kappa}$ holds $\mathscr{H}^{n-1}$-{a.e.} 
on $\partial\Omega'$ for some $\widetilde{\kappa}>0$. Put $\Phi_{\partial\Omega'}(t,x):=x-th(x)$. In this situation, it is established in \cite[Eq. (4.66)\emph{ff.}]{Hofmann} that there exists $t_{0}\in (0,1)$ such that $\Phi_{\partial\Omega'}\colon (-t_{0},t_{0})\times\partial\Omega'\to\R^{n}$ is bi-Lipschitz onto its image. From here, the assertions \ref{item:collA1}--\ref{item:collC1} follow. 
Moreover, \cite[Proposition 4.19(i)]{Hofmann} establishes that, 
for all $-t_{0}<t<t_{0}$, $\Phi_{\partial\Omega'}(\{t\}\times\partial\Omega')$ arises as the boundary of a domain with Lipschitz boundary. 
This yields \ref{item:collD1} if $\partial\Omega'$ is Lipschitz. If $\partial\Omega'$ is of class $\hold^{k}$, then $\Phi_{\partial\Omega'}(\{t\}\times\partial\Omega')$ still arises as the boundary of a domain with Lipschitz boundary. Now, if $-t_{0}<t<t_{0}$ and $y_{0}\in\Phi_{\partial\Omega'}(\{t\}\times\partial\Omega')$, we put $x_{0}:=\Phi_{\partial\Omega'}(t,\cdot)^{-1}(\{y_{0}\})$. Then there exists an isometric coordinate system such that $x_{0}$ is the origin and, in this new system,  $\partial\Omega'$ coincides with the graph of a $\hold^{k}$-function $f\colon\R^{n-1}\to\R$  close to the origin. In the new coordinate system and expressing $h$ accordingly, one then defines 
\begin{align*}
F\colon (x',s)\mapsto (x',f(x'))-sh(x',f(x')),\qquad (x',s)\in\R^{n-1}\times(-t_{0},t_{0}).
\end{align*}
As established in \cite[Eq. (4.59)--(4.62)]{Hofmann}, a suitably small 
choice of $t_{0}$ entails that $\det(DF(0,0))\neq 0$; 
in fact, this is the easier case in \cite{Hofmann} which does not require mollification. By the inverse function theorem, $F$ is locally invertible at $x_{0}=(0,0)$. Noting that we may choose $h$ to be of class $\hold^{\infty}$, this implies that its local inverse at zero is of class $\hold^{k}$ too. Thus, if $|t|$ is sufficiently small, we find $0<r,\varepsilon<1$ such that $F$ maps $Z_{r,\varepsilon}:=\ball_{r}^{(n-1)}(0)\times (t-\varepsilon,t+\varepsilon)$ bijectively onto its image. But $W_{r,\varepsilon}(y_{0}):=F(Z_{r,\varepsilon})$ is open, contains $y_{0}=F(0,t)$ and satisfies $F^{-1}(W_{r,\varepsilon}(y_{0})\cap\Phi_{\partial\Omega'}(t,\partial\Omega'))=\ball_{r}^{(n-1)}(0)\times\{t\}$. Since $y_{0}$ was arbitrary, we conclude that $\Phi_{\partial\Omega'}(t,\partial\Omega')$ is a $\hold^{k}$-manifold, and so \ref{item:collD1} follows.

Based on Lemma \ref{lem:collar1},
it is convenient to introduce the following notion:
\begin{definition}[Transversal shifting]\label{def:shrinknormal} 
In the situation of {\rm Lemma \ref{lem:collar1}}, 
let  $\Sigma\subset\R^{n}$ be a $\hold^{1}$-regular Lipschitz boundary manifold relative to $\Omega'$. With  $\Phi_{\partial\Omega'}$ as in Lemma \ref{lem:collar1}, we define for $-1<t<1$ the \emph{{transversally shifted manifold}} $\Sigma_{\Omega'}^{\Phi,t}$ by 
\begin{align}\label{eq:normmod}
\Sigma_{\Omega'}^{\Phi,t}:=\Phi_{\partial\Omega'}(t,\Sigma).
\end{align}
\end{definition}
For $\partial\Omega'$ of class $\hold^{k}$ and suitable choices of $\Phi_{\partial\Omega'}$, $\Phi_{\partial\Omega'}(t,\partial\Omega')$ can be shown to be the boundary of a set with $\hold^{k}$-boundary, and the transversally shifted manifolds $\Sigma_{\Omega'}^{\Phi,t}$ inherit the property of being  $\hold^{k}$-regular Lipschitz boundary manifolds (relative to $\Phi_{\partial\Omega'}(t,\partial\Omega')$) from that of $\Sigma$ (relative to $\partial\Omega'$). 

For our future applications, it is sometimes necessary to employ a particular construction of the collar or deformation maps $\Phi_{\partial\Omega'}$. The following refinement, discussed in parts after Lemma \ref{lem:collar1}, is 
due to Hofmann et al. \cite[Proposition 4.19]{Hofmann}; 
see also Ball \& Zarnescu \cite[\S 5]{BallZarnescu}, Chen et al. \cite[\S 8]{ChenComiTorres}, 
Doktor \cite{Doktor} and Ne\v{c}as \cite{Necas} for related results.

\begin{rem}\label{rem:aeintro}
If $\Omega'\subset\R^{n}$ is open and bounded with Lipschitz boundary, 
then, by \cite[Proposition 2.3]{Hofmann} there exist both $\kappa>0$ and a $\hold^{1}$-vector field $h\colon\R^{n}\to\R^{n}$ such that $|h(x)|=1$ and $h(x)\cdot\nu_{\partial\Omega'}(x)\leq-\kappa$ 
for $\mathscr{H}^{n-1}$-a.e. $x\in\partial\Omega'$; 
recall that $\nu_{\partial\Omega'}$ is the inner unit normal to $\partial\Omega'$. 
Define
$\Omega'_{t}:=\{x-th(x)\colon\;x\in\Omega\}$. By \cite[Prop. 4.19]{Hofmann}, there exists $t_{0}>0$ with the following properties:
\begin{enumerate}
\item\label{item:collario1} With $\Phi_{\partial\Omega'}\colon (-t_{0},t_{0})\times\partial\Omega'\to \mathcal{O}\subset\R^{n}$ given by $\Phi_{\partial\Omega'}(t,x):=x-th(x)$, we have
\begin{align}\label{eq:Phiconvergence}
|\Phi_{\partial\Omega'}(t,x)-x|\to 0\qquad \text{as}\;t\searrow 0\;\text{uniformly in $x\in\partial\Omega'$}. 
\end{align}
\item\label{item:collario2} For each $0<t<t_{0}$, $\Lambda_{t}\colon\partial\Omega\ni x\mapsto x-th(x)\in\partial\Omega_{t}$ is a bi-Lipschitz map, with bi-Lipschitz constants being uniformly bounded in $t$. 
\item\label{item:collario3} \emph{Coordinate representations.} There exists a finite covering of $\partial\Omega'$ by coordinate cylinders $Z$ which, for every $0<t<t_{0}$, is also a covering of $\partial\Omega'_{t}$ by coordinate cylinders. For each such cylinder, let $\psi_{Z},\psi_{Z,t}\colon  U_{Z}\to \R$ denote the Lipschitz functions whose graphs parametrize (up to a rotation and a translation) $Z\cap\partial\Omega$ and $Z\cap\partial\Omega_{t}$, respectively, where $U_{Z}\subset\R^{n-1}$ is open. 
Then 
$\|\nabla\psi_{Z,t}\|_{\lebe^{\infty}(U_{Z})}\leq c\|\nabla\psi_{Z}\|_{\lebe^{\infty}(U_{Z})}$ for all $0<t<t_{0}$ with a constant $c>0$ independent of $0<t<t_{0}$. Moreover, we have
\begin{align*}
\nabla\psi_{Z,t}\to\nabla\psi_{Z}\qquad\,\,\mathscr{L}^{n-1}\text{-{a.e.} in $U_{Z}$}
\end{align*}
and, for every $1\leq q<\infty$, $\nabla\psi_{Z,t}\to\nabla\psi_{Z}$ strongly in $\lebe^{q}(U_{Z};\R^{n-1})$. 
\item\label{item:collario4} \emph{Approaching the boundary non-tangentially.} There exists $c>0$ such that 
\begin{align*}
\sup_{x\in\partial\Omega'}|x-\Lambda_{t}(x)|\leq ct\qquad\,\,\text{for all}\;0<t<t_{0}. 
\end{align*}
\item\label{item:collario5} \emph{Convergence of the interior unit normals.} Denoting by $\nu_{t}$ the interior unit normal to $\partial\Omega_{t}$, there exists $C>0$ such that 
\begin{align*}
\sup_{x\in\partial\Omega'}|\nu_{\partial\Omega'}(x)-\nu_{t}(\Lambda_{t}(x))|\leq Ct\qquad\,\,\text{for all}\;0<t<t_{0}. 
\end{align*}
\end{enumerate}
\end{rem}
Before we come to a collar-type theorem on manifolds, we record a result on the Lebesgue points. 
The proof is the same as in \cite[Lemma 2.1]{ChenFrid1999}.

\begin{lem}\label{lem:Lebesgue} Let $\Omega'\subset\R^{3}$ be open and bounded with Lipschitz boundary. 
Let $\FF\in\lebe_{\locc}^{1}(\R^{3};\R^{3})$. Then, for $\mathscr{L}^{1}$-a.e. $t\in (-1,1)$, $\mathscr{H}^{2}$-a.e. $x_{0}\in\Phi_{\partial\Omega'}(\{t\}\times\partial\Omega')$ is a Lebesgue point of $\FF$. 
\end{lem}
For future reference, we moreover record the following comparability result.  
\begin{lem}[Geodesic/intrinsic versus extrinsic distances]\label{lem:geodesic} Let $\Omega'\subset\R^{n}$ be an open and bounded set with $\hold^{1}$-boundary. Then the Euclidean distance and the geodesic distance on $\partial\Omega'$ are mutually comparable. More precisely, setting 
\begin{align*}
d_{g}(x,y):=\inf\Big\{\int_{0}^{1}|\dot{\gamma}(t)|\dif t\colon\;  \begin{array}{c}\gamma\colon\;[0,1]\to\partial\Omega'\;\text{is continuous and} \\ \text{piecewise $\hold^{1}$ with $\gamma(0)=x,\gamma(1)=y$}\end{array}\Big\},\qquad x,y\in\partial\Omega', 
\end{align*}
there exists a constant $c>1$ such that 
\begin{align*}
|x-y|\leq d_{g}(x,y)\leq c|x-y|\qquad\text{for all}\;x,y\in\partial\Omega'. 
\end{align*}
\end{lem}
Based on this lemma, we may in particular work with distances with respect to the Euclidean or equivalently the intrinsic metric, but think of one choice as fixed throughout. 
We now continue with the second type of collar theorems.
\begin{lem}[Collar-type II]\label{lem:collar2}
Let $\Omega'\subset\R^{n}$ be open and bounded with boundary $\partial\Omega'$ of class $\hold^{1}$, 
and let $\Sigma\subsetneq\partial\Omega'$ be a $\hold^{1}$-regular Lipschitz boundary manifold relative 
to $\Omega'$ and with boundary $\Gamma_{\Sigma}$. 
Then there exist both a relatively open neighborhood $\mathcal{O}$ 
of $\Gamma_{\Sigma}$ in $\partial\Omega'$ and 
a bi-Lipschitz \emph{collar map} $\Psi_{\Sigma}\colon (-1,1)\times \Gamma_{\Sigma}\to\mathcal{O}$ 
with the following properties: 
\begin{enumerate}
    \item\label{item:collA2} $\Psi_{\Sigma}(0,\Gamma_{\Sigma})=\Gamma_{\Sigma}$, 
    \item\label{item:collB2} $\Psi_{\Sigma}(t,\Gamma_{\Sigma})\subset \mathrm{int}(\Sigma)$ for all $0<t<1$, 
    \item\label{item:collC2} $\Psi_{\Sigma}(t,\Gamma_{\Sigma})\subset \partial\Omega'\setminus\overline{\Sigma}$ for all $-1<t<0$. 
    \item\label{item:collD2} $\Psi_{\Sigma}(t,\Gamma_{\Sigma})$ is an {$(n-2)$-dimensional Lipschitz submanifold} for each $-1<t<1$. 
    \item\label{item:collE2} If $0<s<t<1$, then $\Psi_{\Sigma}((0,s)\times\Gamma_{\Sigma})\subset\Psi_{\Sigma}((0,t)\times\Gamma_{\Sigma})$. 
    \item\label{item:collF2} There exists a constant $\theta\geq 2$ such that,
    for all $0<s<t<1$,
\begin{align}
    &\frac{t}{\theta}\leq \mathrm{dist}(\Gamma_{\Sigma},x)\leq \theta t &&\qquad\text{for all}\;x\in\Psi_{\Sigma}(t,\Gamma_{\Sigma}),\notag\\ 
    & \frac{t-s}{\theta}\leq \mathrm{dist}(\Psi_{\Sigma}(s,\Gamma_{\Sigma}),x)\leq \theta(t-s)&&\qquad\text{for all}\;x\in\Psi_{\Sigma}(t,\Gamma_{\Sigma}),\label{eq:PsiSigmaBounds}\\
    & \frac{s}{\theta}\leq\mathrm{dist}(\Gamma_{\Sigma},x)\leq\theta t&&\qquad\text{for all}\;x\in\Psi_{\Sigma}((s,t),\Gamma_{\Sigma}).\notag
    \end{align}
\end{enumerate}
\end{lem}
The proof of this lemma can be accomplished similarly as Lemma \ref{lem:collar1} by localization and flattening. When there is no confusion about the particular  manifold $\Sigma$ and the choice of $\Psi_{\Sigma}$, we simply write $\Psi=\Psi_{\Sigma}$. 
\begin{definition}[Tangential  shrinking and enlarging]\label{def:shrinktangential}
In the situation of {\rm Lemma \ref{lem:collar2}},  define 
\begin{align}\label{eq:tangmod}
 \Sigma_{\Omega',\Psi_{\Sigma}}^{\tau,t}:=
 \begin{cases} \Sigma\setminus \Psi_{\Sigma}((0,t]\times\Gamma_{\Sigma}))&\qquad\mbox{for $\,\,0<t<1$},\\[1mm] 
\Sigma\cup \Psi_{\Sigma}((t,0]\times\Gamma_{\Sigma})
&\qquad 
\mbox{for $\,\,-1<t<0$}, 
\end{cases}
\end{align}
and call $\Sigma_{\Omega',\Psi_{\Sigma}}^{\tau,t}$ the \emph{tangentially shrunk manifold} provided that $0<t<1$, and the \emph{{tangentially enlarged manifold}} provided that $-1<t<0$. If the map $\Psi_{\Sigma}$ is clear from the context, we write $\Sigma_{\Omega'}^{\tau,t}$ for the sets from \eqref{eq:tangmod}.
\end{definition}
In the setting of the preceding definition, the tangentially shrunk or enlarged manifolds are $\hold^{1}$-regular Lipschitz boundary manifolds relative to $\Omega'$ too. 
For notational brevity and thinking of $\Psi_{\Sigma}$ being fixed, we moreover put 
\begin{align}\label{eq:foliation}
\Gamma_{\Sigma}^{t} := \Psi_{\Sigma}(t,\Gamma_{\Sigma})\qquad\,\,\mbox{for $\,\, |t|<1$}.
\end{align}
 With this convention, we define for  $-\frac{1}{2}<t<\frac{1}{2}$ and $0<\varepsilon<\frac{1}{2}$ 
\begin{align}\label{eq:onesidedGammas}
(\Gamma_{\Sigma})_{t,\varepsilon} =\bigcup_{t-\varepsilon<s<t+\varepsilon}\Gamma_{\Sigma}^{s},\;\;\;(\Gamma_{\Sigma})_{t,\varepsilon}^{+} =\bigcup_{t<s<t+\varepsilon}\Gamma_{\Sigma}^{s},\;\;\;(\Gamma_{\Sigma})_{t,\varepsilon}^{-} =\bigcup_{t-\varepsilon<s<t}\Gamma_{\Sigma}^{s}.
\end{align}
Based on our choice of $t$ and $\varepsilon$, we always have $-1<t-\varepsilon<t+\varepsilon<1$,
so that
the sets from \eqref{eq:onesidedGammas} are well-defined indeed. 
\subsection{Function spaces}
We now collect the definitions and background results for function spaces as used in the main part of the paper. Here we focus on the spaces which appear most frequently; in particular, we  postpone the definition and properties of function spaces on manifolds to Appendix \hyperref[sec:AppendixA]{A}, as they only enter selected sections. 
\subsubsection{Lipschitz and $\hold^{1}$-spaces}
In the main part, we are frequently concerned with extending maps from boundaries of sets to the interior. To state the first lemma, we note that $\hold^{1}(\partial\Omega')$ for an open and bounded set $\Omega'$ with $\hold^{1}$-boundary is defined as usual by means of localization.
\begin{lem}\label{lem:GoodLip}
Let $\Omega'\subset\R^{n}$ be open and bounded with $\hold^{1}$-boundary. 
\begin{enumerate}
\item\label{item:Lipextend1} 
If $f\in\hold^{1}(\partial\Omega')$, then, for each $0<\delta<1$, 
there exists $f_{\delta}\in\hold^{1}(\overline{\Omega'})$ such that 
 \begin{itemize}
 \item[\rm (i)] $f_{\delta}|_{\partial\Omega'}=f$, 
  \item[\rm (ii)] $f_{\delta}(x)=0$ whenever $x\in\Omega'$ satisfies $\mathrm{dist}(x,\partial\Omega')>\delta$, 
  \item[\rm (iii)] there exist both a constant $c>0$ independent of $f$ and $0<\delta<1$ so that
\begin{align*}
\|\nabla f_{\delta}\|_{\lebe^{\infty}(\Omega')}\leq c\big(\|\nabla_{\tau}f\|_{\lebe^{\infty}(\partial\Omega')} + \frac{1}{\delta}\|f\|_{\lebe^{\infty}(\partial\Omega')}\big)\quad\text{in $\Omega'$}. 
\end{align*}
    \end{itemize}
    \item\label{item:Lipextend2} If $f\in\mathrm{Lip}(\partial\Omega')$, then there exists $g\in\mathrm{Lip}_{\rm c}(\R^{n})$ such that 
    \begin{itemize}
     \item[\emph{(i)}] $g|_{\partial\Omega'}=f$, 
     \item[\emph{(ii)}] $g|_{\R^{n}\setminus\overline{\Omega'}}\in\hold_{b}^{1}(\R^{n}\setminus\overline{\Omega'})$ and $g|_{\Omega'}\in\hold_{b}^{1}(\Omega')$, 
     \item[\emph{(iii)}] if $\rho_{\delta}$ is the $\delta$-rescaled version of a standard mollifier, then there exists a sequence $(\delta_{j})\subset (0,1)$ such that $\delta_{j}\searrow 0$ and 
     \begin{align*}
       \nabla_{\tau}(\rho_{\delta_{j}}*g)\to \nabla_{\tau}f
       \qquad\text{$\mathscr{H}^{n-1}$-a.e. on $\partial\Omega'$ as $j\to\infty$}. 
     \end{align*}
    \end{itemize}
\end{enumerate}
\end{lem}
We note that \ref{item:Lipextend1} cannot be improved to $g\in\hold_{\rm c}^{1}(\R^{n})$ in general. It is difficult to trace the precise statement of Lemma \ref{lem:GoodLip} back to a specific reference, which is why a self-contained proof is included in Appendix \hyperref[sec:AppendixB]{B} for the reader's convenience. 

\subsubsection{Extended  divergence-measure fields} Let $\Omega\subset\R^{3}$ be open. If $\FF\in\cm^{\infty}(\Omega)$, then $\mathrm{div}(\curl \FF)=0$ in $\mathscr{D}'(\Omega)$. 
Hence, $\curl \FF$ can be regarded as an \emph{extended divergence-measure field} in the terminology of \cite{ChenFrid2003}: 

\begin{definition}[Extended divergence-measure fields]\label{def:extdivmeas}
Let $\Omega\subset\R^{n}$ be open and bounded. An \emph{extended divergence-measure field} 
$\FF\in\mathscr{DM}^{\mathrm{ext}}(\Omega)$
is a vector field $\FF\in \mathrm{RM}_{\mathrm{fin}}(\Omega;\R^{n})$ such that $\di\FF\in\mathrm{RM}_{\mathrm{fin}}(\Omega)$. 
\end{definition}
For our further purpose, we now document the process of assigning distributional normal traces 
to extended divergence-measure fields and record some of their basic properties.
\begin{definition}[Normal traces]\label{def:normaltraces}
Let $\Omega\subset\R^{n}$ be open and bounded, $\FF\in\mathscr{DM}^{\mathrm{ext}}(\Omega)$, and let $E\Subset\Omega$ be a Borel set. We then define the \emph{normal trace} of $\FF$ on $\partial E$ by
\begin{align}\label{eq:normaltraceextdivmeas}
\langle \FF\cdot\nu,\varphi\rangle_{\partial E} := - \int_{E}\nabla\varphi\cdot\dif\FF - \int_{E}\varphi\dif\,(\di \FF),\qquad\varphi\in\hold_{\rm c}^{1}(\Omega). 
\end{align}
\end{definition}
The preceding definition appears \emph{e.g.} in \cite{ChenFrid2003,ChenIrvingTorres,Silhavy}. We will require a generalization of \eqref{eq:normaltraceextdivmeas} to Lipschitz maps $\varphi\in\mathrm{Lip}_{\rm c}(\Omega)$. This requires the measure $\overline{\nabla\varphi\cdot\FF}$ which, following \cite{Silhavy,ChenIrvingTorres}, can be introduced as the unique signed Radon measure on $\Omega$ which satisfies
\begin{align}\label{eq:LipschitzPairing}
\int_{\Omega}\psi\dif\overline{\nabla\varphi\cdot\FF} = \lim_{\delta\searrow 0}\int_{\Omega}\psi\nabla(\rho_{\delta}*\varphi)\cdot\dif\FF\qquad\text{for all}\;\psi\in\hold_{\rm c}(\Omega). 
\end{align}
This pairing is instrumental in generalizing
identity \eqref{eq:normaltraceextdivmeas} to Lipschitz competitors: 

\begin{lem}[{\cite[Corollary 2.9]{ChenIrvingTorres}}]\label{lem:normaltraceextend} 
In the situation of {\rm Definition \ref{def:extdivmeas}}, 
the normal trace functional \eqref{eq:normaltraceextdivmeas} extends to 
a bounded linear functional on $\mathrm{Lip}(\Omega)$ by setting
\begin{align}\label{eq:normaltraceextdivmeas1}
\langle \FF\cdot\nu,\varphi\rangle_{\partial E} := - \int_{E}\dif\overline{\nabla\varphi\cdot\FF} - \int_{E}\varphi\dif\,(\di(\FF))\qquad\,\,\mbox{for any $\varphi\in\mathrm{Lip}(\Omega)$}. 
\end{align}
\end{lem}
The aforementioned references directly deal with Borel sets $E$, 
whereas we largely deal with much more regular sets in the main part. 
We then only require an easier version: 

\begin{lem}\label{lem:easyprod}
Let $\Omega\subset\R^{n}$ be open and bounded,  and let $\Omega'\Subset\Omega$ be open. Moreover, let $\FF\in\mathscr{DM}^{\mathrm{ext}}(\Omega)$ 
and  $\varphi\in\mathrm{Lip}_{\rm c}(\Omega)$ be 
such that $\varphi|_{\Omega'}\in\hold_{b}^{1}(\Omega')$. Then $\overline{\nabla\varphi\cdot\FF}(\Omega')=(\nabla\varphi\cdot\FF)(\Omega')$ so that
\begin{align}\label{eq:easynormaltraceextdivmeas1}
\langle\FF\cdot\nu,\varphi\rangle_{\partial\Omega'}=-\int_{\Omega'}\nabla\varphi\cdot\dif\FF - \int_{\Omega'}\varphi\dif\,(\mathrm{div}\FF). 
\end{align}
\end{lem}
\begin{proof}
Let $(\psi_{j})\subset\hold_{\rm c}^{1}(\Omega';[0,1])$ be such that $\psi_{j}\to 1$ pointwise in $\Omega'$ as $j\to\infty$. Since $\overline{\nabla\varphi\cdot\FF}$ is a signed measure, we obtain by dominated convergence in the first step:
\begin{align*}
\overline{\nabla{\varphi}\cdot\FF}(\Omega') & = \lim_{j\to\infty}\int_{\Omega}\psi_{j}\dif\overline{\nabla{\varphi}\cdot\FF} \stackrel{\eqref{eq:LipschitzPairing}}{=} \lim_{j\to\infty}\lim_{i\to\infty} \int_{\Omega'}\psi_{j}\nabla(\rho_{\delta_{i}}*{\varphi})\cdot\dif\FF \\ 
& = \lim_{j\to\infty}\lim_{i\to\infty} \int_{\Omega'}\psi_{j}(\rho_{\delta_{i}}*\nabla{\varphi})\cdot\dif\FF \stackrel{(*)}{=} \lim_{j\to\infty} \int_{\Omega'}\psi_{j}\nabla{\varphi}\cdot\dif\FF \stackrel{(**)}{=} \int_{\Omega'}\nabla{\varphi}\cdot\dif\FF. 
\end{align*}
Here we used at $(*)$ that $\nabla{\varphi}|_{\Omega'}\in\hold_{b}(\Omega';\R^{n})$ by \ref{item:Lipchoose2}, whereby $\rho_{\delta_{i}}*\nabla{\varphi}\to\nabla{\varphi}$ everywhere in $\Omega'$ as $i\to\infty$. In particular, for every $j\in\mathbb{N}$, $|\psi_{j}(\rho_{\delta_{i}}*\nabla{\varphi})|\leq \|\nabla{\varphi}\|_{\lebe^{\infty}(\Omega)}$ and $\psi_{j}(\rho_{\delta_{i}}*\nabla{\varphi})\to\psi\nabla{\varphi}$ everywhere in $\Omega'$. Since $\psi_{j}\nabla{\varphi}\in\hold_{\rm c}(\Omega';\R^{3})$, the integral of $\psi_{j}\nabla{\varphi}$ with respect to the vectorial measure $\FF$ is well-defined, and so $(*)$ follows by dominated convergence. As to $(**)$, it suffices to realize  that $\psi_{j}\nabla{\varphi}\to\nabla{\varphi}$ everywhere in $\Omega'$ as $j\to\infty$, whereby $|\psi_{j}\nabla{\varphi}|\leq |\nabla{\varphi}|$ yields $(**)$ again by dominated convergence. Now \eqref{eq:easynormaltraceextdivmeas1} follows from \eqref{eq:normaltraceextdivmeas1}, and the proof is complete.
\end{proof}

Moreover, we require the following two results on the local nature of the normal traces:
\begin{lem}[Locality, {\cite[Theorem 2.15]{ChenIrvingTorres}}]\label{lem:LipschitzVanish} Let $\Omega\subset\R^{n}$ be open, and let $E\Subset\Omega$ be open. If $\varphi\in\mathrm{Lip}_{\rm c}(\Omega)$ vanishes on $\partial E$, then $\langle \FF\cdot\nu,\varphi\rangle_{\partial E}=0$. 
    \end{lem}
\begin{lem}[Localisability of the normal trace, {\cite[Theorem 5.3]{ChenIrvingTorres}}]\label{lem:localisabilitynormaltrace}
Let $\Omega\subset\R^{n}$ be open, $\FF \in\mathscr{DM}^{\mathrm{ext}}(\Omega)$, and let $U,V\Subset\Omega$. If $A\subset\R^{n}$ is open such that $
A\cap U = A\cap V$, 
then 
\begin{align*}
\langle \FF\cdot\nu,\varphi\rangle_{\partial U} = \langle \FF\cdot\nu,\varphi\rangle_{\partial V}\qquad\text{for all}\;\varphi\in\sobo_{\rm c}^{1,\infty}(A). 
\end{align*}
Moreover, abbreviating $d:=\mathrm{dist}(\cdot,\partial U)$ and $U^{\varepsilon}:=\{x\in U\colon\;d(x)>\varepsilon\}$, the condition 
\begin{align}\label{eq:measregcrit}
\liminf_{\varepsilon\searrow 0} \frac{1}{\varepsilon}|\nabla d\cdot\FF|(U\setminus\overline{U^{\varepsilon}})<\infty
\end{align}
implies that  $\langle\FF\cdot\nu,\cdot\rangle_{\partial U}$ can be represented by a Radon measure on $\partial U$ and so can be localized in the above sense. 
\end{lem}

\begin{definition}[Intrinsic normal traces]\label{def:intrinsicnormaltrace}
Let $\Omega\subset\R^{n}$ be open, and let $U\Subset\Omega$ be an open subset. For an extended divergence measure field $\FF\in\mathscr{DM}^{\mathrm{ext}}(\Omega)$, define 
\begin{align}\label{eq:normaltraceextdivmeasintrinsic}
\mathcal{N}_{\FF\cdot\nu,U}(\varphi):= \langle \FF\cdot\nu,\overline{\varphi}\rangle_{\partial U}
\qquad\,\,\mbox{for any $\varphi\in\mathrm{Lip}(\partial U)$}, 
\end{align}
where $\overline{\varphi}\in\mathrm{Lip}_{\rm c}(\Omega)$ satisfies $\overline{\varphi}|_{\partial U}=\varphi$.
\end{definition}
The functional \eqref{eq:normaltraceextdivmeasintrinsic} is well-defined, meaning that it does not depend on the specific Lipschitz extension $\overline{\varphi}$. Moreover, since $U\Subset\Omega$ is open in the situation of the previous definition, we conclude in view of \eqref{eq:normaltraceextdivmeas} that we might specify to $\overline{\varphi}\in\mathrm{Lip}(U)$ with $\overline{\varphi}|_{\partial U}=\varphi$; in \eqref{eq:normaltraceextdivmeas}, only the values of $\overline{\varphi}$ inside $U$ matter. This observation implies:

\begin{rem}[Open sets $U$ with $\hold^{1}$-boundaries] If, in the situation of Definition \ref{def:intrinsicnormaltrace}, $U\Subset\Omega$ is open with $\hold^{1}$-boundary $\partial U$, then there is no need to invoke the extended pairing \eqref{eq:LipschitzPairing}. Namely, in this case, we may apply Lemma \ref{lem:GoodLip}\ref{item:Lipextend2}(i)--(ii) to extend $\varphi\in\mathrm{Lip}(\partial\Omega)$ to a function $\overline{\varphi}\in(\mathrm{Lip}\cap\hold_{b}^{1})(U)$; note that we do not claim $\overline{\varphi}\in\hold^{1}(\overline{U})$ here. Since the right-hand side of \eqref{eq:normaltraceextdivmeasintrinsic} is independent of the specific choice of the particular Lipschitz extension, we may directly work with $\overline{\varphi}$. For this choice, we see that 
$\nabla\overline{\varphi}\in\hold_{b}(U;\R^{n})$ and 
\begin{align}\label{eq:cameron}
\mathcal{N}_{\FF\cdot\nu,U}(\varphi)= - \int_{U}\nabla\overline{\varphi}\cdot\dif\FF - \int_{U}\overline{\varphi}\dif\,(\mathrm{div}\FF).
\end{align}
In particular, note that the first integral on the right-hand side of \eqref{eq:cameron} 
is well-defined as the integral of a bounded continuous function with respect 
to the finite Radon measure $\FF$. While this will be the generic scenario 
from \S \ref{sec:stokes} onwards, this simple description only works due 
to the openness and regularity assumption on $U$ and $\partial U$, respectively.
\end{rem}

\subsection{Level set estimates} 
In the main part, we require estimates on the $(n-1)$-dimensional Hausdorff measure of the level sets of distance functions. The following result is a special case of a more general assertion for asymmetric norms due to Caraballo \cite{Caraballo}; see also Kraft \cite{Kraft} for related estimates.
\begin{lem}[$\mathscr{H}^{n-1}$-measure of level sets, {\cite[Theorem 6]{Caraballo}}]\label{lem:distancelevelsets} For every $n\geq 2$, there exists a constant $c=c(n)>0$ with the following property:  Let $C\subset\R^{n}$ be non-empty and compact with $\mathscr{H}^{n-1}(C)<\infty$, and put $d(x):=\mathrm{dist}(x,C)$. Moreover, suppose that there exist $\delta,\theta>0$ such that 
\begin{align}\label{eq:measuredensityHn-1}
    \frac{\mathscr{H}^{n-1}(C\cap\ball_{\gamma}(x))}{\gamma^{n-1}}\geq\theta\qquad\,\text{for all}\;0<\gamma\leq\delta\;\text{and all}\;x\in C.
\end{align}
Then 
\begin{align}\label{eq:levelsetbd}
\mathscr{H}^{n-1}(d^{-1}(\{r\}))\leq \frac{c}{\theta}\,\sup\big\{1,\big(\frac{r}{\delta}\big)^{n-1}\big\}\,\mathscr{H}^{n-1}(C)\quad\mbox{for $\mathscr{L}^{1}$-a.e. $0<r<\infty$}.
\end{align}
\end{lem}
We note that, if $\Omega\subset\R^{n}$ is open and bounded with Lipschitz boundary, then \eqref{eq:measuredensityHn-1} is automatically fulfilled for $C=\partial\Omega$ and  
suitable choices of $\delta,\theta>0$.

\section{The Trace Theorem for $\cm^{p}$-Fields}\label{sec:tracetheorem}
In this section, we introduce (distributional) tangential traces of $\mathscr{CM}^{p}$-fields. 
In \S \ref{subsec:trace}, we are concerned with elementary properties and various examples in the full exponent regime $1\leq p\leq\infty$, which will be continued throughout the paper. 
In \S \ref{sec:Cminftytrace}, we are especially concerned with the case $p=\infty$, 
where stronger results are available 
and which will be
the basis for the results in \S \ref{sec:stokes}--\S\ref{sec:divmeasfieldsmanif}.



\subsection{The tangential trace functional}\label{subsec:trace} 
We directly give the definition of the tangential trace functional
by duality.

\begin{definition}[Tangential trace]\label{def:tangentialtrace}
Let $\Omega\subset\R^{3}$ be open,
let $\FF \in \cm^{p} (\Omega)$ for $1\leq p\leq \infty$,  
and let $U \Subset \Omega$ be open and bounded. 
Define the \emph{interior} and \emph{exterior} tangential traces 
of $\FF$ on $\partial U$ as $\R^{3}$-valued distributions $(\FF\times\nu)_{U}^{\mathrm{int}} \in\mathscr{D}'(\Omega;\R^{3})$ and 
$(\FF\times \nu)_{U}^{\mathrm{ext}} \in\mathscr{D}'(\Omega;\R^{3})$  
in $\Omega$ by
\begin{align}\label{eq:tangentialtrace}
  \begin{split}
&\langle (\FF \times \nu)_{U}^{\mathrm{int}}, \varphi \rangle_{\partial U} := \int_{U} \varphi \, \d ( \curl \FF) - \int_{U}   \FF\times\nabla \varphi  \,\dif x, \\ 
& \langle (\FF \times \nu)_{U}^{\mathrm{ext}}, \varphi \rangle_{\partial U} := - \int_{\Omega\setminus \overline{U}} \varphi \, \d ( \curl \FF) 
+ \int_{\Omega\setminus \overline{U}}   \FF\times \nabla \varphi\,\dif x 
\end{split}
  \end{align}
for any $\varphi\in\hold_{\rm c}^{\infty}(\Omega)$.
\end{definition} 

For a $\hold^{1}$-map $\FF$, the identities in \eqref{eq:tangentialtrace} 
hold per definition, provided that we interpret the pairings on the left-hand side of \eqref{eq:tangentialtrace} as integrals of $(\FF\times\nu_{\partial U})\varphi$ 
with respect to $\mathscr{H}^{2}$ on $\partial U$; 
see formula \eqref{eq:calculus2a} in Appendix \hyperref[sec:AppendixD]{D} 
for more detail. 
Next, based on the convention of Definition \ref{def:tangentialtrace}, we see 
that 
\begin{align}\label{eq:attention}
\langle(\FF\times\nu)_{U}^{\mathrm{ext}},\varphi\rangle_{\partial U} = - \langle(\FF\times\nu)_{\Omega\setminus \overline{U}}^{\mathrm{int}},\varphi\rangle_{\partial(\Omega\setminus\overline{U})}\qquad\text{for all}\;\varphi\in\hold_{\rm c}^{\infty}(\Omega).
 \end{align}
Now let $1\leq p \leq \infty$. In the situation of Definition \ref{def:tangentialtrace}, 
the estimate
\begin{align}\label{eq:pope}
|\langle (\FF\times \nu)_{U}^{\mathrm{int}},\varphi\rangle_{\partial U}| 
\leq \|\varphi\|_{\lebe^{\infty}(U)}|\curl \FF|(U) 
+ \|\nabla\varphi\|_{\lebe^{p'}(U)}\|\FF\|_{\lebe^{p}(U)}
\end{align}
for $\varphi\in\hold_{\rm c}^{\infty}(\Omega)$ 
entails that $\langle(\FF\times \nu)_{U}^{\mathrm{int}},\,\cdot\,\rangle_{\partial U}$  is a well-defined $\R^{3}$-valued distribution
\begin{align}\label{eq:components}
\langle(\FF\times\nu)_{U}^{\mathrm{int}},\,\cdot\,\rangle_{\partial U}=\langle((\FF\times\nu)_{U,i}^{\mathrm{int}})_{1\leq i\leq 3},\,\cdot\,\rangle_{\partial U} \in \mathrm{Lip}_{\rm c}(\Omega;\R^{3})'.
\end{align}
In a routine manner, we may therefore introduce for $\bphi=(\varphi_{1},\varphi_{2},\varphi_{3})\in\hold_{\rm c}^{\infty}(\Omega;\R^{3})$
\begin{align}\label{eq:vectorialtrace}
\begin{split}
\langle (\FF\times \nu)_{U}^{\mathrm{int}}, \bphi \rangle_{\partial U}  
& := \sum_{1\leq i \leq 3} \langle ( \FF\times \nu)_{U,i}^{\mathrm{int}},\varphi_{i}\rangle_{\partial U}\\ 
& = \int_{U}\bphi\cdot\dif\,(\curl \FF) -\int_{U}\FF\cdot\curl \bphi\dif x, 
\end{split} 
\end{align} 
and the equivalence of \eqref{eq:tangentialtrace} and \eqref{eq:vectorialtrace} 
is directly seen by use of \eqref{eq:components}; 
see also formula \eqref{eq:calculus7} in Appendix \hyperref[sec:AppendixD]{D}. 
Since we consistently denote vector fields with bold letters, there will be no confusion about the meaning of $\langle(\FF\times\nu)_{U}^{\mathrm{int}},\varphi\rangle_{\partial U}$ and  $\langle(\FF\times\nu)_{U}^{\mathrm{int}},\bphi\rangle_{\partial U}$. Lastly, \eqref{eq:pope}--\eqref{eq:vectorialtrace} hold true for $\langle(\FF\times\nu)_{U}^{\mathrm{ext}},\,\cdot\,\rangle_{\partial U}$ with the obvious modifications. 
We now have
\begin{theorem}[Distributional Tangential Traces]\label{thm:tangtrace}
In the situation of {\rm Definition \ref{def:tangentialtrace}}, 
let $U\Subset\Omega$ be open and bounded. Then the following hold:
\begin{enumerate}
    \item\label{item:tantrace1} $ \hold_{\rm c}^{\infty}(\Omega)\ni \varphi\mapsto \langle(\FF\times\nu)_{U}^{\mathrm{int}},\varphi\rangle_{\partial U}$ is an \emph{$\R^{3}$-valued distribution supported on $\partial U$ of order at most one}.
    \item\label{item:tantrace2} $\hold_{\rm c}^{\infty}(\Omega;\R^{3})\ni\bphi \mapsto \langle (\FF\times\nu)_{U}^{\mathrm{int}},\bphi\rangle_{\partial\Omega'}$ is a distribution supported on $\partial U$ of order at most one.

\end{enumerate}
With the obvious modifications, the analogous assertions hold true for $\langle(\FF\times\nu)_{U}^{\mathrm{ext}},\cdot\rangle_{\partial U}$. 
\end{theorem}
For the proof of Theorem \ref{thm:tangtrace} and subsequent results, we require a smooth approximation result. It can be established along the same lines as \cite[\S 5.2.2, Thm. 2]{EvansGariepy}, or directly retrieved from \cite[Thm. 2.8]{BreitDieningGmeineder} by setting $\mathbb{A}=\curl$.
\begin{lem}\label{lem:smoothapprox}
Let $\Omega\subset\R^{3}$ be open and bounded, and let $\FF\in\mathscr{CM}^{p}(\Omega)$ for 
$1\leq p\leq \infty$. 
Then there exist sequences $({\FF}_{j})\subset\hold_{\rm c}^{\infty}(\Omega;\R^{3})$ and 
$(\widetilde{\FF}_{j})\subset\mathscr{CM}^{p}(\Omega)\cap\hold^{\infty}(\Omega;\R^{3})$ such that the following hold: 
\begin{enumerate}
    \item\label{item:SMAP1} If $1\leq p<\infty$, then ${\FF}_{j}\to\FF$ strongly in $\lebe^{p}(\Omega;\R^{3})$ and $(\curl{\FF}_{j})\mathscr{L}^{3}\stackrel{*}{\rightharpoonup}\curl\FF$ in $\mathrm{RM}_{\mathrm{fin}}(\Omega;\R^{3})$. Moreover,  $\widetilde{\FF}_{j}\to \FF$  $\mathscr{CM}^{p}$\emph{-strictly}, meaning that 
    \begin{align*}
    d_{s}^{p}(\widetilde{\FF}_{j},\FF):=\|\widetilde{\FF}_{j}-\FF\|_{\lebe^{p}(\Omega)}
    +\big||\curl \widetilde{\FF}_{j}|(\Omega)-|\curl \FF|(\Omega)\big|\stackrel{j\to\infty}
    {\longrightarrow} 0.
    \end{align*}
    \item\label{item:SMAP2} If $p=\infty$, then ${\FF}_{j}\stackrel{*}{\rightharpoonup}\FF$  \emph{in the weak*-sense} on $\mathscr{CM}^{\infty}(\Omega)$, meaning that 
    \begin{align*}
    {\FF}_{j}\stackrel{*}{\rightharpoonup}\FF\;\,\,
    \text{in $\,\lebe^{\infty}(\Omega;\R^{3})$},
    \qquad (\curl {\FF}_{j})\mathscr{L}^{3}\stackrel{*}{\rightharpoonup}\curl \FF \;\,\,
    \text{in $\,\mathrm{RM}_{\mathrm{fin}}(\Omega;\R^{3})$}.
    \end{align*}
    Moreover, ${\FF}_{j}\to \FF$ strongly in $\lebe_{\locc}^{q}(\Omega;\R^{3})$ for any $1\leq q<\infty$.
\end{enumerate}
In both cases, $\FF_{j}(x),\widetilde{\FF}_{j}(x)\to\FF^{*}(x)$ for every Lebesgue point $x$ of $\FF$. 
\end{lem}
\begin{proof}[Proof of Theorem \ref{thm:tangtrace}] We divide the proof into two steps correspondingly:

\smallskip
For \ref{item:tantrace1}, let $W \Subset \Omega$ be open with $W \cap \partial U=\emptyset$, and let $\varphi\in\mathscr{D}(\Omega)$ be such that $\spt(\varphi)\subset W$. Since $\mathrm{dist}(\spt(\varphi),\partial W)>0$, there exists an open set $W'$ with smooth boundary 
such that $\spt(\varphi)\Subset W'\Subset\ W$,
$|\curl \FF|(\partial(U\cap W'))=0$, and $\partial (U\cap W')$ is smooth.

By Lemma \ref{lem:smoothapprox}, we may choose a sequence $(\FF_{j})\subset\mathscr{D}(\Omega;\R^{3})$ such that $\FF_{j}\to\FF$ strongly in $\lebe_{\locc}^{1}(\Omega;\R^{3})$ and $(\curl \FF_{j})\mathscr{L}^{3}\stackrel{*}{\rightharpoonup}\curl \FF$ locally in $\mathrm{RM}(\Omega;\R^{3})$. By the properties of  $W'$, we see that 
$(\curl \FF_{j})\mathscr{L}^{3}\mres(U\cap W')\stackrel{*}{\rightharpoonup}\curl\FF\mres(U\cap W')$ 
in $\mathrm{RM}_{\mathrm{fin}}(U\cap W')$. 
Since
\begin{align}\label{eq:govan}
\int_{U}\curl(\varphi\GG)\dif x = \int_{U\cap W'}\curl(\varphi\GG)\dif x = \int_{\partial(U\cap W')}\varphi\GG\times\nu_{\partial(U\cap W')}\dif\mathscr{H}^{2}= 0
\end{align}
for all $\GG\in\hold^{1}(\Omega;\R^{3})$
(see \eqref{eq:calculus2} in Appendix \hyperref[sec:AppendixD]{D}), 
it follows that 
\begin{align*}
0  & \stackrel{\eqref{eq:govan}}{=} \lim_{j\to\infty}\int_{U}\curl(\varphi\FF_{j})\dif x\\
&\,\, = \lim_{j\to\infty}\Big(\int_{U\cap W'}\varphi\,\curl \FF_{j}\dif x - \int_{U\cap W'}\FF_j \times \nabla\varphi \dif x\Big) \\ 
& \;\, = \int_{U\cap W'}\varphi \dif\,(\curl\FF) -\int_{U\cap W'}\FF \times \nabla\varphi\dif x = \langle (\FF \times \nu)_{U}^{\mathrm{int}}, \varphi \rangle_{\partial U}. 
\end{align*}
By arbitrariness of the open set $W\Subset\Omega$ with $W\cap\partial U=\emptyset$ and $\varphi\in\mathscr{D}(\Omega)$ with $\spt(\varphi)\subset W$, this yields the assertion on the support. Since the statement on the order of $\langle(\FF\times\nu)_{U}^{\mathrm{int}},\cdot\rangle_{\partial U}$ being at most one is a direct consequence of \eqref{eq:pope}, this implies \ref{item:tantrace1}. 
\smallskip
For \ref{item:tantrace2}, this assertion directly follows from \ref{item:tantrace1} as in the smooth case; see also \eqref{eq:calculus7} in Appendix \hyperref[sec:AppendixD]{D}. 
This completes the proof.
\end{proof}

\begin{rem}\label{rem:ondefs}
In view of the methods employed for the closely related divergence-measure fields (see \cite{ChenComiTorres,ChenFrid2003,ChenTorres2005,Silhavy2005}), 
one can also establish product rule-type theorems for $\varphi \FF$, 
\emph{e.g.} where $\FF\in\mathscr{CM}^{\infty}(\Omega)$ and $\varphi\in (\bv\cap\lebe^{\infty})(\Omega)$. 
To keep our exposition at a reasonable length and since we do not need this in the sequel, 
we will defer such results to future work. 
Moreover, since we are mostly concerned with subsets $U\subset\Omega$ with Lipschitz, 
$\hold^{1}$- or $\hold^{2}$-boundary, this also applies to potentially more irregular sets $U$.
\end{rem}

We next illustrate the preceding theorem with several examples. 
Specifically, Examples \ref{ex:limitationsp}--\ref{ex:1leqpleqinfty} 
deal with the cases for $1\leq p<\infty$, 
and display different phenomena in various exponent regimes. As a key point, if $1\leq p <\infty$, the tangential traces from Definition \ref{def:tangentialtrace} can merely be distributions 
that cannot be represented by $\lebe_{\locc}^{1}$-fields. This differs from $p=\infty$, which is the subject of  Example \ref{ex:tangentialjumps} and the following Subsection \ref{sec:Cminftytrace}.

\begin{example}[Exponents $1\leq p<\frac{3}{2}$]\label{ex:limitationsp}
Let $\Omega=\ball_{2}(0)$ and $U:=\ball_{1}^{+}(0):=\{(x',x_{3})\in\ball_{1}(0)\colon\;\;x_{3}>0\}$, where we abbreviate $x'=(x_{1},x_{2})$ as usual. We consider the vector field
\begin{align}\label{eq:FSLapl}
\FF(x)=-\frac{1}{4\pi}\frac{x}{|x|^{3}}\qquad\,\,\mbox{for $x\in\Omega$}, 
\end{align}
so that $\FF\in\lebe^{p}(\Omega;\R^{3})$ for any $1\leq p<\frac{3}{2}$.
Since $\FF=\nabla\Phi$ with the fundamental solution $\Phi$ of the negative Laplacian $(-\Delta)$, it follows that $\curl \FF=0$ in $\Omega$. Hence, $\FF \in\mathscr{CM}^{p}(\Omega)$ for all $1\leq p<\frac{3}{2}$. We now claim that, for all $\varphi\in\mathscr{D}(\Omega)$, there holds 
\begin{align}\label{eq:maxlaim}
\begin{split}
 \langle (\FF\times\nu)_{U}^{\mathrm{int}},\varphi\rangle_{\partial U} & = - \frac{1}{4\pi}\mathrm{P.V.}\int_{\ball_{1}^{(2)}(0)\times\{0\}}\frac{(\mathbf{Q}x',0)}{|x'|^{3}}\varphi(x',0)\dif\mathscr{H}^{2}(x'), \\ 
 & := - \lim_{\varepsilon\searrow 0}  \frac{1}{4\pi}\int_{(\ball_{1}^{(2)}(0)\setminus\ball_{\varepsilon}^{(2)}(0))\times\{0\}}\frac{(\mathbf{Q}x',0)}{|x'|^{3}}\varphi(x',0)\dif\mathscr{H}^{2}(x'),
 \end{split}
\end{align}
where $\mathbf{Q}x':=\mathbf{Q}(x_{1},x_{2})=(x_{2},-x_{1})$. 
To see \eqref{eq:maxlaim}, let $\varphi\in\mathscr{D}(\Omega)$ be arbitrary. 
For $0<\varepsilon<\frac{1}{2}$, because of $\curl\FF=0$, we have
\begin{align*}
\langle(\FF\times\nu)_{U}^{\mathrm{int}},\varphi\rangle_{\partial U} & =  -\frac{1}{4\pi} \int_{U\setminus\ball_{\varepsilon}(0)}\nabla\varphi(x)\times \frac{x}{|x|^{3}}\dif x - \underbrace{\frac{1}{4\pi} \int_{U\cap\ball_{\varepsilon}(0)}\nabla\varphi(x)\times \frac{x}{|x|^{3}}\dif x}_{\to0\;\text{as}\;\varepsilon\searrow 0} \\ 
& \!\!\stackrel{\varepsilon\searrow 0}{\longrightarrow} -\frac{1}{4\pi}\lim_{\varepsilon\searrow 0}\int_{(\ball_{1}^{(2)}(0)\times\{0\})\setminus\ball_{\varepsilon}(0)}\varphi(x',0)\frac{(\mathbf{Q}x',0)}{|x'|^{3}}\dif\mathscr{H}^{2}(x'), 
\end{align*}
where the last line is a consequence of the integration-by-parts formula for the curl
(see formula \eqref{eq:calculus2a} in Appendix \hyperref[sec:AppendixD]{D}), and 
\begin{align*}
\int_{\partial U\cap\{x_{3}>0\}}\varphi(x)\frac{x\times x}{|x|^{4}}\dif\mathscr{H}^{2}(x)=\int_{\partial\!\ball_{\varepsilon}(0)\cap U}\varphi(x)\frac{x\times x}{|x|^{4}}\dif\mathscr{H}^{2}(x) = 0.
\end{align*}
Thus, \eqref{eq:maxlaim} follows. 
Moreover, $x'\mapsto \mathbf{Q}x'/|x'|^{3}$ does not belong 
to $\lebe_{\locc}^{1}(\partial U\cap\{x_{3}=0\};\R^{3})$ and so cannot represent 
the density of a Radon measure which is absolutely continuous 
with respect to $\mathscr{H}^{2}$ on $\partial U$. Note that 
\begin{align*}
|\FF(x)\times \nu(x)| = \Big\vert \frac{1}{4\pi}\frac{(-x_{2},x_{1},0)}{|x|^{3}}\Big\vert = \frac{1}{4\pi}\frac{1}{|\tilde{x}|^{2}},\qquad \tilde{x}:=(x_{1},x_{2})\neq (0,0),   
\end{align*}
and the singularity is non-integrable close to zero in two dimensions. In particular, we deduce that $\langle(\FF\times\nu)_{U}^{\mathrm{int}},\cdot\rangle_{\partial U}$ does not give rise to a distribution of order zero. 
\end{example}
\begin{example}[Exponents $1\leq p <2$]\label{ex:singCMP}
Define a vector field $\FF\colon\R^{3}\setminus\{0\}\to\R^{3}$ by 
\begin{align}\label{eq:curldef}
\FF(x_{1},x_{2},x_{3}):=\frac{1}{2\pi}\Big(-\frac{x_{2}}{x_{1}^{2}+x_{2}^{2}},\frac{x_{1}}{x_{1}^{2}+x_{2}^{2}},x_{3} \Big),\qquad x=(x_{1},x_{2},x_{3})\neq 0.
\end{align}
We claim that 
\begin{align}\label{eq:distclaim}
\curl\FF=e_{3}\mathscr{H}^{1}\mres \R e_{3}\in\mathrm{RM}(\R^{3};\R^{3})\qquad\text{in}\;\mathscr{D}'(\R^{3};\R^{3}). 
\end{align}
To see \eqref{eq:distclaim}, let $\bphi=(\varphi_{1},\varphi_{2},\varphi_{3})\in\mathscr{D}'(\R^{3};\R^{3})$. We record that 
\begin{align*}
\int_{\R^{3}}\FF\cdot\curl\bphi\dif x & = \frac{1}{2\pi}\int_{-\infty}^{\infty} \int_{\R^{2}} \Big(\frac{-x_{2}}{x_{1}^{2}+x_{2}^{2}} \Big)\big(\partial_{2}\varphi_{3}(x)-\partial_{3}\varphi_{2}(x)\big) \dif\,(x_{1},x_{2})\dif x_{3} \\ 
& + \frac{1}{2\pi}\int_{-\infty}^{\infty} \int_{\R^{2}} \Big(\frac{x_{1}}{x_{1}^{2}+x_{2}^{2}} \Big)\big(\partial_{3}\varphi_{1}(x)-\partial_{1}\varphi_{3}(x)\big) \dif\,(x_{1},x_{2})\dif x_{3} \\ 
& = -\frac{1}{2\pi}\int_{-\infty}^{\infty} \int_{\R^{2}} \Big(\frac{x_{2}}{x_{1}^{2}+x_{2}^{2}} \partial_{2}\varphi_{3}(x) + \frac{x_{1}}{x_{1}^{2}+x_{2}^{2}}\partial_{1}\varphi_{3}(x)\Big)
\dif\,(x_{1},x_{2})\dif x_{3} \\ 
& =: -\frac{1}{2\pi}\int_{-\infty}^{\infty}\mathrm{I}(x_{3})\dif x_{3}, 
\end{align*}
where we used that several integrals vanish due to an integration by parts. We abbreviate $\ball'_{\varepsilon}:=\ball_{\varepsilon}^{(2)}(0)$ and compute for $\varepsilon>0$ and any  $x_{3}\in\R$:
\begin{align*}
\mathrm{I}(x_{3}) & = \Big(\int_{\ball'_{\varepsilon}}  + \int_{\R^{2}\setminus\ball'_{\varepsilon}}\Big) \frac{x_{1}}{x_{1}^{2}+x_{2}^{2}}\partial_{1}\varphi_{3}(x) + \frac{x_{2}}{x_{1}^{2}+x_{2}^{2}} \partial_{2}\varphi_{3}(x)  \dif\,(x_{1},x_{2}) =: \mathrm{II}_{\varepsilon}^{(1)} + \mathrm{II}_{\varepsilon}^{(2)}.
\end{align*}
Writing $x'=(x_{1},x_{2})$, we then have 
\begin{align}\label{eq:distbound1}
|\mathrm{II}_{\varepsilon}^{(1)}| & \leq c\|\nabla\varphi\|_{\lebe^{\infty}(\R^{3})}\int_{\ball'_{\varepsilon}}\frac{\dif  x'}{|x'|} \leq c\|\nabla\varphi\|_{\lebe^{\infty}(\R^{3})}\,\varepsilon \to 0\qquad \mbox{as $\varepsilon\searrow 0$}. 
\end{align}
Integrating by parts yields 
\begin{align}\label{eq:distbound2}
\begin{split}
\mathrm{II}_{\varepsilon}^{(2)} & = -\int_{\partial\!\ball'_{\varepsilon}}\frac{\varphi_{3}(x',x)}{|x'|}\dif\mathscr{H}^{1}(x') = -2\pi \dashint_{\partial\!\ball'_{\varepsilon}}\varphi_{3}(x',x_{3})\dif x' \to -2\pi \varphi_{3}(0,x_{3})
\end{split}
\end{align}
as $\varepsilon\searrow 0$. By dominated convergence, 
a combination of \eqref{eq:distbound1}--\eqref{eq:distbound2} yields
\begin{align*}
\int_{\R^{3}}(\curl\FF)\cdot\bphi\dif x \stackrel{{\hyperref[eq:calculus4]{(\mathrm{D.3})}}}{=} \int_{\R^{3}}\FF\cdot\curl\bphi\dif x = \int_{-\infty}^{\infty}\varphi_{3}(0,x_{3})\dif x_{3} =  \langle  e_{3}\mathscr{H}^{1}\mres \R e_{3},\bphi\rangle, 
\end{align*}
and this completes the proof of \eqref{eq:distclaim}. 
Considering $\FF$ given by \eqref{eq:curldef} on $\Omega :=\ball_{2}^{(2)}(0)\times (-2,2)$ 
and setting $\Omega':=\ball_{1}^{(2)}(0)\times (0,1)$, we obtain that,
for any $\bphi\in\hold_{\rm c}^{1}(\Omega;\R^{3})$, 
\begin{align*}
&\langle(\FF\times\nu)_{\Omega'}^{\mathrm{int}},\bphi\rangle_{\partial\Omega'}  = \frac{1}{2\pi}\int_{\partial\Omega'}\mathbf{G}(x)\cdot\bphi(x)\dif\mathscr{H}^{2}(x),
\end{align*}
where
$$
\mathbf{G}(x)= \mathbbm{1}_{\ball_{1}^{(2)}(0)\times\{0\}}(x)\frac{(x',0)}{|(x',0)|^{2}} 
-\mathbbm{1}_{\ball_{1}^{(2)}(0)\times\{1\}}(x)\frac{(x',0)}{|(x',0)|^{2}} + \mathbf{H}(x), 
$$
and $\mathbf{H}\colon \partial\!\ball_{1}^{(2)}(0)\times(0,1)\to\R^{3}$ is smooth. 
Hence, the tangential trace can be represented by an $\lebe^{q}(\partial\Omega';\R^{3})$-map 
for $1\leq q<2$. 
If we let $\Omega'=\{(x_{1},x_{2},x_{3})\colon\;(x_{1},x_{3})\in\ball_{1}^{(2)}(0),\;0<x_{2}<1\}$ instead, an argument as in Example \ref{ex:limitationsp} shows that the distributional tangential trace cannot be represented by an  $\lebe_{\locc}^{1}$-field. Yet, $\FF\in\cm^{p}(\Omega)$ for all $1\leq p<2$.
\end{example}

\begin{example}[$1<p<\infty$]\label{ex:1leqpleqinfty}
Let $\Omega\subset\R^{3}$ be open and bounded, and let $\Omega'\Subset\Omega$ be open and bounded with Lipschitz boundary. The space $\sobo^{\mathrm{curl},p}(\Omega')$ as defined in $\eqref{eq:sobocurl}_{2}$ admits a bounded, linear, and surjective tangential trace operator 
$\mathrm{tr}_{\tau}^{p}\colon \sobo^{\curl,p}(\Omega')\to\mathcal{X}_{\partial\Omega'}^{p}$; 
see also Appendix \hyperref[sec:AppendixC]{C}. 
For $\FF\in\sobo^{\curl,p}(\Omega)$, we then have 
\begin{align*}
\langle (\FF\times\nu)_{\Omega'}^{\mathrm{int}},\bphi\rangle_{\partial\Omega'} = \langle \mathrm{tr}_{\tau}^{p}(\FF|_{\Omega'}),\bphi\rangle\qquad\,\mbox{for any $\bphi\in\hold_{\rm c}^{1}(\Omega;\R^{3})$}. 
\end{align*}
Since $\mathcal{X}_{\partial\Omega'}^{p}\not\subset\lebe_{\locc}^{1}(\partial\Omega';\R^{3})$, 
the results from Examples \ref{ex:limitationsp}--\ref{ex:singCMP} cannot be improved 
to  tangential traces in $\lebe_{\locc}^{1}$ even if $2<p<\infty$.
\end{example}

\begin{example}[$p=\infty$]\label{ex:tangentialjumps}
 
Assume that vector 
fields $\mathbf{u},\mathbf{v}\in(\bv\cap\lebe^{\infty})({\ball_{2}(0)};\R^{3})$.
For $x=(x_{1},x_{2},x_{3})\in\ball_{2}(0)$, consider 
\begin{align*}
\FF(x):=\begin{cases} 
\mathbf{u}(x) &\;\text{if}\;x\in\ball_{2}^{+}(0):=\{x\in\ball_{2}(0)\colon\;x_{3}>0\}, \\ 
\mathbf{v}(x) &\;\text{if}\;x\in\ball_{2}^{-}(0):=\{x\in\ball_{2}(0)\colon\;x_{3}<0\}.
\end{cases} 
\end{align*}
Denote by $\mathrm{tr}^{+},\mathrm{tr}^{-}$ the classical $\bv$-trace operators 
on $\partial\!\ball_{2}^{+}(0)$ or $\partial\!\ball_{2}^{-}(0)$. 
Based on formula \eqref{eq:calculus2a} in Appendix \hyperref[sec:AppendixD]{D}, we have 
\begin{align}\label{eq:BVcurlrepresenter}
\begin{split}
\curl\FF  = &\, \curl\mathbf{u}\mres\ball_{2}^{+}(0) + \curl\mathbf{v}\mres\ball_{2}^{-}(0) \\
& +  (\mathrm{tr}^{-}(v_{2})-\mathrm{tr}^{+}(u_{2}),\mathrm{tr}^{+}(u_{1})-\mathrm{tr}^{-}(v_{1}),0)\mathscr{H}^{2}\mres (\ball_{2}^{(2)}(0)\times\{0\})
\end{split}
\end{align}
{as an identity in $\mathscr{D}'(\ball_{2}(0);\R^{3})$}. 
Let $\Omega':=\ball_{1}^{+}(0)$. Definition \ref{def:tangentialtrace} implies that $\langle(\FF\times\nu)_{\Omega'}^{\mathrm{int}},\cdot\rangle_{\partial\Omega'}$ can be represented by the $\lebe^{1}$-function $\mathrm{tr}^{+}(\mathbf{u})\times\nu_{\partial\!\ball_{2}^{+}(0)}$. On the other hand, with $\widetilde{\mathrm{tr}}\colon \bv(\ball_{2}^{+}(0)\setminus\overline{\ball}{_{1}^{+}}(0);\R^{3})\to\lebe^{1}(\partial({\ball}{_{2}^{+}}(0)\setminus\overline{\ball}{_{1}^{+}}(0));\R^{3})$ denoting the usual boundary trace operator on $\bv$, $\langle(\FF\times\nu)_{\Omega'}^{\mathrm{ext}},\cdot\rangle_{\partial\Omega'}$ can be represented as 
\begin{align*}
(\mathrm{tr}^{-}(\mathbf{v})\times e_{3})|_{\partial\!\ball_{1}^{+}(0)\cap\{x_{3}=0\}} -\left.   \widetilde{\mathrm{tr}}(\mathbf{u})\times\frac{x}{|x|}\right\vert_{(\partial\!\ball_{1}^{+}(0))\cap\{x_{3}>0\}}\qquad\text{along}\;\partial\!\ball_{2}^{+}(0).
\end{align*}
These results persist for $\mathbf{u},\mathbf{v}\in\bv(\ball_{2}(0);\R^{3})$, but this is only due to the fact that the \emph{full}  traces of $\bv$-maps belong to $\lebe^{1}$ along the respective boundaries in this case.
\end{example}

\subsection{The trace theorem for $\mathscr{CM}^{\infty}$-fields}\label{sec:Cminftytrace}
Throughout this subsection, let $\Omega\subset\R^{3}$ be open and bounded, and let $\Omega'\Subset\Omega$ be open with Lipschitz boundary. In \S 3.1, 
we have seen that the distributional tangential traces cannot, in general, be represented by measures supported on $\partial\Omega'$ provided that $1\leq p <\infty$; see Examples \ref{ex:limitationsp}--\ref{ex:1leqpleqinfty}. We now establish that, if $p=\infty$, such a singular behaviour of the tangential traces cannot occur. In particular, the tangential traces of $\mathscr{CM}^{\infty}$-fields along $\partial\Omega'$ can always be represented by $\lebe^{\infty}(\partial\Omega';T_{\partial\Omega'})$-fields: 

\begin{theorem}[Tangential Traces for $\mathscr{CM}^{\infty}$-Fields]\label{thm:tracemain1}
Let $\Omega\subset\R^{3}$ be open and bounded, and let $\Omega'\Subset\Omega$ be open with Lipschitz boundary. 
If $\FF \in \cm^{\infty}(\Omega)$,   
then the following hold: 
\begin{enumerate}
\item\label{item:trace1} \emph{Representation as $\lebe^{\infty}$-functions:} The interior and exterior tangential traces from {\rm Definition \ref{def:tangentialtrace}} can be represented as $\R^{3}$-valued $\lebe^{\infty}$-functions on $\partial\Omega'$. That is, there exist $ (\FF\times\nu_{\partial\Omega'})_{\partial\Omega'}^{\interior},(\FF\times\nu_{\partial\Omega'})_{\partial\Omega'}^{\exterior} \in \lebe^{\infty} (\partial\Omega'; \mathbb{R}^3)$ such that
\begin{equation}\label{eq:trace1}
\begin{split}
&\int_{\partial\Omega'} \varphi (\FF\times \nu_{\partial\Omega'})_{\partial\Omega'}^{\interior} \dif\mathscr{H}^{2 }  = \langle (\FF\times\nu)_{\Omega'}^{\mathrm{int}},\,\varphi\rangle_{\partial\Omega'},\\
&\int_{\partial\Omega'} \varphi (\FF\times \nu_{\partial\Omega'})_{\partial\Omega'}^{\exterior} \dif\mathscr{H}^{2 } = \langle (\FF\times\nu)_{\Omega'}^{\mathrm{ext}},\,\varphi\rangle_{\partial\Omega'}
\end{split}
\end{equation}
hold for every $\varphi \in \hold_{\rm c}^{1}(\Omega)$.
\item\label{item:trace2} \emph{Tangential character:} For $\mathscr{H}^{2}$-a.e. $x\in \partial\Omega'$, 
then
\begin{align*}
(\FF\times\nu_{\partial\Omega'})_{\partial\Omega'}^{\mathrm{int}}(x)\in T_{\partial\Omega'}(x)\;\;\;\text{and}\;\;\;(\FF\times\nu_{\partial\Omega'})_{\partial\Omega'}^{\mathrm{ext}}(x)\in T_{\partial\Omega'}(x). 
\end{align*}
\item\label{item:trace3}\emph{Vectorial version of \eqref{eq:trace1}:} If $\bphi\in\hold_{\rm c}^{1}(\Omega;\R^{3})$, then
\begin{align}\label{eq:scaltovec}
\begin{split}
&\int_{\partial\Omega'}(\FF\times\nu_{\partial\Omega'})_{\partial\Omega'}^{\mathrm{int}}\cdot\bphi\dif\mathscr{H}^{2} = \int_{\Omega'}\bphi\cdot\dif\,(\curl\FF) - \int_{\Omega'}(\curl \bphi)\cdot\FF\dif x  \\ 
& \int_{\partial\Omega'}(\FF\times\nu_{\partial\Omega'})_{\partial\Omega'}^{\mathrm{ext}}\cdot\bphi\dif\mathscr{H}^{2} = -\int_{\Omega\setminus\overline{\Omega'}}\bphi\cdot\dif\,(\curl\FF) + \int_{\Omega\setminus\overline{\Omega'}} (\curl\bphi)\cdot\FF\dif x.
\end{split}
\end{align}
\end{enumerate}
\end{theorem}
We wish to point out that, compared with analogous results for divergence measure fields (\emph{e.g.}  \cite{ChenFrid1999}), the previous theorem departs at \ref{item:trace2}. This is, in fact, a crucial property which we will frequently use in \S \ref{sec:stokes}--\S \ref{sec:divmeasfieldsmanif}, and which does not have an analogue in the theory of $\mathscr{DM}^{\infty}$-fields; for the latter, the normal traces are scalar-valued throughout. 
    
\begin{proof}[Proof of Theorem \ref{thm:tracemain1}]  We focus our proof on the case of interior traces, 
since it is analogous to the case of exterior traces. 
    
For \ref{item:trace1}, our reasoning is inspired by that in
Chen-Frid \cite{ChenFrid1999}.
Since we require the detailed proof for \ref{item:trace2} and \ref{item:trace3}, we carry out the full argument. 
Let $\Phi_{\partial\Omega'}\colon [0,1)\times\partial\Omega'\to\mathcal{O}\Subset\Omega$ be the collar map from Lemma \ref{lem:collar1}, where we may assume that it satisfies the properties listed in Remark \ref{rem:aeintro}. We put 
$$
(\Omega')^{\Phi,t}:=\Omega'\setminus\Phi_{\partial\Omega'}([0,t]\times\partial\Omega')
\qquad\mbox{for $0<t<1$}. 
$$
Given $\FF\in\mathscr{CM}^{\infty}(\Omega)$, we choose a sequence $(\FF_{j})\subset\mathscr{D}(\Omega;\R^{3})$ as in Lemma \ref{lem:smoothapprox}\ref{item:SMAP2}. 
By formula \eqref{eq:calculus2a} in Appendix \hyperref[sec:AppendixD]{D},  
\begin{align}\label{eq:classicalGGcurl}
\int_{\partial(\Omega')^{\Phi,t}}\varphi (\FF_{j}\times\nu_{\partial(\Omega')^{\Phi,t}})\dif\mathscr{H}^{2} =  \int_{(\Omega')^{\Phi,t}}\varphi\,\curl\FF_{j}\dif x - \int_{(\Omega')^{\Phi,t}}\FF_{j}\times\nabla\varphi\dif x
\end{align}
holds for all $j\in\mathbb{N}$, all $0<t<1$, and any $\varphi\in\hold_{\rm c}^{\infty}(\Omega)$. 
By Lemma \ref{lem:Lebesgue} in conjunction with Lemma \ref{lem:smoothapprox}, there exists a set $I_{1}\subset(0,1)$ with $\mathscr{L}^{1}((0,1)\setminus I_{1})=0$ such that, for every $t\in I_{1}$, $\FF_{j}\to\FF^{*}$ pointwise  $\mathscr{H}^{2}$-a.e. on $\partial(\Omega')^{\Phi,t}$ together with $\|\FF^{*}\|_{\lebe^{\infty}(\partial(\Omega')^{\Phi,t})}\leq \|\FF\|_{\lebe^{\infty}(\Omega)}$ for all $t\in I_{1}$. This yields the convergence of the first term in \eqref{eq:classicalGGcurl}. On the other hand, there exists a set $I_{2}\subset (0,1)$ with $\mathscr{L}^{1}((0,1)\setminus I_{2})=0$ such that $|\curl\FF|(\partial(\Omega')^{\Phi,t})=0$ for all $t\in I_{2}$. Hence, for each $t\in I_{2}$, the second term in \eqref{eq:classicalGGcurl} converges to the integral of $\varphi$ over $(\Omega')^{\Phi,t}$ with respect to the $\R^{3}$-valued measure $\curl\FF$. Since $\FF_{j}\to\FF$ strongly in $\lebe_{\locc}^{1}(\Omega;\R^{3})$, we may send $j\to\infty$ in \eqref{eq:classicalGGcurl} to 
obtain 
\begin{align}\label{eq:classicalGGcurl1}
\int_{\partial(\Omega')^{\Phi,t}}\varphi (\FF^{*}\times\nu_{\partial(\Omega')^{\Phi,t}})\dif\mathscr{H}^{2} =  \int_{(\Omega')^{\Phi,t}}\varphi\,\dif\,(\curl\FF) -  \int_{(\Omega')^{\Phi,t}}\FF\times \nabla\varphi\dif x
\end{align}
whenever $t\in I:=I_{1}\cap I_{2}$. Since $\curl\FF$ is a finite Radon measure 
and $\mathbbm{1}_{(\Omega')^{\Phi,t}}\varphi\to \mathbbm{1}_{\Omega'}\varphi$ pointwise everywhere in $\Omega'$ as $t\searrow 0$, Lebesgue's theorem on dominated convergence applied to the single components yields that the right-hand side of \eqref{eq:classicalGGcurl1} converges to the right-hand side of  $\eqref{eq:trace1}_{1}$  as $t\searrow0$ with $t\in I$. Now, for $t\in I$, we define  the measure $\bmu_{t}$ on $\mathscr{B}(\partial\Omega')$ via push-forward by  
\begin{align}\label{eq:whale}
\bmu_{t}:= {\Phi_{\partial\Omega'}(t,\cdot)^{-1}}_{\#} \big((\FF^{*}\times\nu_{\partial(\Omega')^{\Phi,t}})\mathscr{H}^{2}\mres \partial(\Omega')^{\Phi,t}\big). 
\end{align}
By Lemma \ref{lem:collar1} and because of $\FF\in\lebe^{\infty}(\Omega;\R^{3})$, we see that 
$\sup_{0<t<\frac{1}{2}}|\bmu_{t}|(\partial\Omega')<\infty$. 
By the Banach-Alaoglu-Bourbaki theorem, for any sequence $(t_{i})\subset I$ with $t_{i}\searrow 0$, there exist both a subsequence $\mathfrak{t}=(t_{i_{k}})\subset(t_{i})$ 
and $\bmu^{\mathfrak{t}}\in\mathrm{RM}_{\mathrm{fin}}(\partial\Omega';\R^{3})$ 
such that $\bmu_{t_{i_{k}}}\stackrel{*}{\rightharpoonup}\bmu^{\mathfrak{t}}$ 
in $\mathrm{RM}_{\mathrm{fin}}(\partial\Omega';\R^{3})$ as $k\to\infty$. 
By the convergences established after \eqref{eq:classicalGGcurl1}, we have
\begin{align}\label{eq:zipper}
&\langle(\FF\times\nu)_{\Omega'}^{\mathrm{int}},\varphi\rangle_{\partial\Omega'}\\
& = \lim_{k\to\infty} \int_{\partial(\Omega')^{\Phi,t_{i_{k}}}}(\varphi-\varphi\circ\Phi_{\partial\Omega'}(t_{i_{k}},\cdot)^{-1})(\FF^{*}\times\nu_{\partial(\Omega')^{\Phi,t_{i_{k}}}})\dif \mathscr{H}^{2} 
+ \lim_{k\to\infty} \int_{\partial\Omega'}\varphi\dif\bmu_{t_{i_{k}}}\notag\\
&=: \mathrm{I}_{1} + \mathrm{I}_{2} \stackrel{(*)}{=}  \mathrm{I}_{2} = \int_{\partial\Omega'}\varphi\dif\bmu^{\mathfrak{t}}. \notag
\end{align}
For $(*)$, note that $\mathrm{I}_{1}=0$, which can be seen by 
$|\FF^{*}\times\nu_{\partial(\Omega')^{\Phi,t}}(x)|\leq \|\FF\|_{\lebe^{\infty}(\Omega)}$ 
for $\mathscr{H}^{2}$-a.e. $x\in\partial(\Omega')^{\Phi,t}$, 
\eqref{eq:Phiconvergence}, and Lebesgue's theorem on dominated convergence. 
Since the restrictions of $\hold_{\rm c}^{\infty}(\Omega)$-functions to $\partial\Omega'$ are dense in $(\hold(\partial\Omega'),\|\cdot\|_{\hold(\partial\Omega')})$, \eqref{eq:zipper} implies that $\bmu:=\bmu^{\mathfrak{t}}$ is independent of $\mathfrak{t}$ and  $\bmu_{t}\stackrel{*}{\rightharpoonup}\bmu$ in $\mathrm{RM}(\partial\Omega';\R^{3})$ as $t\searrow0$ with $t\in I$.

Now let $A\subset\partial\Omega'$ be compact with $\mathscr{H}^{2}(A)=0$. For $\varepsilon>0$, we  find a finite covering $(\ball_{i})_{i=1}^{i_{0}}$ by non-degenerate open balls  such that $A\subset \bigcup_{i=1}^{i_{0}}\ball_{i}$ and $\sum_{i=1}^{i_{0}}r(\ball_{i})^{2}<\varepsilon$. 
Let $\varphi\in\hold(\partial\Omega';[-1,1])$. 
Recalling that $|\FF^{*}|\leq \|\FF\|_{\lebe^{\infty}(\Omega)}$ 
on $\partial(\Omega')^{\Phi,t}$ for $t\in I$, we have
\begin{align*}
\left\vert\int_{A}\varphi\dif\bmu\right\vert
&\stackrel{\eqref{eq:zipper}}{\leq}\lim_{\substack{t\searrow 0\\ t \in I}}
\int_{\Phi_{\partial\Omega'}(t,\partial\Omega'\cap\bigcup_{i=1}^{i_{0}}\ball_{i})}
|\FF^{*}\times\nu_{\partial(\Omega')^{\Phi,t}}|\dif\mathscr{H}^{2} \\ 
&\,\, \leq \|\FF^{*}\|_{\lebe^{\infty}(\Omega)}\sup_{0<t<\frac{1}{2}}
\mathscr{H}^{2}
  (\Phi_{\partial\Omega'}(t,\partial\Omega'\cap\bigcup_{i=1}^{i_{0}}\ball_{i}))\\
&\,\,\leq c\|\FF^{*}\|_{\lebe^{\infty}(\Omega)}
\big(\sup_{0<t<\frac{1}{2}}\mathrm{Lip}(\Phi_{\partial\Omega'}(t,\cdot))\big)\varepsilon.
\end{align*}
We then pass to the supremum over all such $\varphi$. As a consequence of Lemma \ref{lem:collar1}, the Lipschitz constants of $\Phi_{\partial\Omega'}(t,\cdot)$ are bounded in $0<t<\frac{1}{2}$. Writing $\bmu=(\mu_{1},\mu_{2},\mu_{3})$, we may send $\varepsilon\searrow 0$ in the resulting inequality to see that $|\mu_{i}|(A)=0$ for $i\in\{1,2,3\}$, and so $|\bmu|(A)=0$. By the Radon property of $\bmu$, it follows that $\bmu$ is absolutely continuous for $\mathscr{H}^{2}\mres\partial\Omega'$: $\bmu\ll\mathscr{H}^{2}\mres\partial\Omega'$. Because $\partial\Omega'$ is $\sigma$-finite with respect to $\mathscr{H}^{2}$, the Radon-Nikod\'{y}m theorem implies the existence of a density 
\begin{align*}
(\FF\times\nu_{\partial{\Omega'}})_{\partial\Omega'}^{\mathrm{int}}:=\frac{\dif\bmu}{\dif\mathscr{H}^{2}},
\end{align*}
whereby 
$\bmu = (\FF\times\nu_{\partial\Omega'})_{\partial\Omega'}^{\mathrm{int}}\mathscr{H}^{2}\mres\partial\Omega'$.
We finally prove that $(\FF\times \nu_{\partial\Omega'})_{\partial\Omega'}^{\mathrm{int}}\in\lebe^{\infty}(\partial\Omega';\R^{3})$. To this end, let $x_{0}\in\partial\Omega'$ be a $\mathscr{H}^{2}$-Lebesgue point of $(\FF\times\nu_{\partial\Omega'})_{\partial\Omega'}^{\mathrm{int}}$. We then obtain that, with $a_{x_{0},r}:=\mathscr{H}^{2}(\ball_{r}(x_{0})\cap\partial\Omega')$,
\begin{align*}
\left\vert ((\FF\times\nu_{\partial\Omega'})_{\partial\Omega'}^{\mathrm{int}})^{*}(x_{0})\right\vert & = \lim_{r\searrow 0} \left\vert \dashint_{\ball_{r}(x_{0})\cap\partial\Omega'}(\FF\times\nu_{\partial\Omega'})_{\partial\Omega'}^{\mathrm{int}}\dif\mathscr{H}^{2}\right\vert \\ & \!\!\!\!\stackrel{\eqref{eq:zipper}}{\leq} \lim_{r\searrow 0}\lim_{\substack{t\searrow 0 \\ t \in I}} \frac{1}{a_{x_{0},r}} \int_{\Phi_{\partial\Omega'}(t,\ball_{r}(x_{0})\cap\partial\Omega')}|\FF^{*}\times\nu_{\partial(\Omega')^{\Phi,t}}|\dif\mathscr{H}^{2} \\ 
& \leq \|\FF^{*}\|_{\lebe^{\infty}(\Omega)}\lim_{r\searrow 0}\lim_{\substack{t\searrow 0 \\ t \in I}} \frac{\mathscr{H}^{2}(\Phi_{\partial\Omega'}(t,\ball_{r}(x_{0})\cap\partial\Omega'))}{a_{x_{0},r}}. 
\end{align*}
Since the Lipschitz constants of $\Phi_{\partial\Omega'}(t,\cdot)$ are bounded in $0<t<\frac{1}{2}$,  \ref{item:trace1} follows.

\smallskip
For \ref{item:trace2}, we choose a covering of $\partial\Omega'$ by the coordinate cylinders $Z$ and associated coordinate maps $\psi_{Z}$ as in Remark \ref{rem:aeintro}\ref{item:collario3}; we may assume $Z$ to be open. Up to a rotation and translation, we may assume that $Z\cap \partial\Omega'$ can be written as $\mathrm{graph}(\psi_{Z})\cap Z$, where $\psi_{Z}\colon U\to \R$ is a Lipschitz function and $U\subset\R^{2}$ is open. By Remark \ref{rem:aeintro}\ref{item:collario3}, we may write $Z\cap\partial(\Omega')^{\Phi,t}=Z\cap\mathrm{graph}(\psi_{Z,t})$ for all $0<t<t'\ll 1$, where $\psi_{Z,t}\colon U\to\R$ is a Lipschitz function such that $\nabla\psi_{Z,t}\to\nabla\psi_{Z}$ $\mathscr{L}^{2}$-a.e. in $U$ and strongly in $\lebe^{q}(U;\R^{2})$ for every $1\leq q<\infty$ as $t\searrow 0$. Now let $x'_{0}\in U$ be a differentiability point of $\psi_{Z}$ and Lebesgue point of $\nabla\psi_{Z}$; by the Rademacher and the Lebesgue differentiation theorems, $\mathscr{L}^{2}$-a.e. $x'_{0}\in U$ will do. We consider $x_{0}:=(x'_{0},\psi_{Z}(x'_{0}))$. If $r>0$ is sufficiently small, then $\ball_{r}(x_{0})\cap\partial\Omega'\subset Z$. By Remark \ref{rem:aeintro}\ref{item:collario4} and diminishing $t'$ if necessary, we may assume that  $\Phi_{\partial\Omega'}(\{t\}\times(\ball_{r}(x_{0})\cap\partial\Omega'))\subset Z$ for all $0<t<t'$. 

With $I$ as in the proof of \ref{item:trace1}, let $(t_{k})\subset I\cap(0,t')$ be such that $t_{k}\searrow 0$. We then see that, 
with $a_{x_{0},r}:=\mathscr{H}^{2}(\ball_{r}(x_{0})\cap\partial\Omega')$ and $
U_{k}^{r}:=\{x'\in U\colon\;(x',\psi_{Z,t_{k}}(x'))\in \Phi_{\partial\Omega'}(\{t_{k}\}\times(\ball_{r}(x_{0})\cap\partial\Omega'))\}$,
\begin{align*}
&\left\vert \frac{\dif\bmu}{\dif\mathscr{H}^{2}}(x_{0})\cdot\nu_{\partial\Omega'}(x_{0})\right\vert\\
&\,\,\, = \lim_{r\searrow 0}\left\vert\dashint_{\ball_{r}(x_{0})\cap\partial\Omega'}(\FF\times\nu_{\partial\Omega'})_{\partial\Omega'}^{\mathrm{int}}(x)\dif\mathscr{H}^{2}(x)\cdot\nu_{\partial\Omega'}(x_{0}) \right\vert  \\ 
&\stackrel{\eqref{eq:whale}\text{\emph{ff.}}}{=} \lim_{r\searrow 0}\lim_{k\to\infty}\left\vert\frac{1}{a_{x_{0},r}}\int_{\Phi_{\partial\Omega'}(t_{k},\ball_{r}(x_{0})\cap\partial\Omega')}\FF^{*}(x)\times\nu_{\partial(\Omega')^{\Phi,t_{k}}}(x)\dif\mathscr{H}^{2}(x)\cdot\nu_{\partial\Omega'}(x_{0}) \right\vert  \\ 
& 
\,\,\,=\lim_{r\searrow 0}\lim_{k\to\infty}\left\vert\frac{1}{a_{x_{0},r}}\int_{\Phi_{\partial\Omega'}(t_{k},\ball_{r}(x_{0})\cap\partial\Omega')}\FF^{*}(x)\cdot(\nu_{\partial(\Omega')^{\Phi,t_{k}}}(x)\times\nu_{\partial\Omega'}(x_{0}))\dif\mathscr{H}^{2}(x)\right\vert \\ 
&
\,\,\,= \lim_{r\searrow 0}\lim_{k\to\infty}\left\vert\frac{1}{a_{x_{0},r}}\int_{U_{k}^{r}}\FF^{*}(x',\psi_{Z,t_{k}}(x'))\cdot
(\nu_{\partial(\Omega')^{\Phi,t_{k}}}(x',\psi_{Z,t_{k}}(x'))\times\nu_{\partial\Omega'}(x_{0}))\times\right.\\
&\qquad\qquad\qquad\qquad\qquad \,\,\,\,\left.\times\sqrt{1+|\nabla\psi_{Z,t_{k}}(x')|^{2}}\dif\mathscr{L}^{2}(x')\right\vert =: \mathrm{I}. 
\end{align*}
Here we have used that, in the second line,  the weak*-limit $\bmu$ of $\bmu_{t_{k}}$ is independent of the specific choice of $(t_{k})\subset I$ with $t_{k}\searrow 0$ (as established in the first part of the proof) and, in the third line, the rule $(\mathbf{a}\times\mathbf{b})\cdot\mathbf{c}=\mathbf{a}\cdot(\mathbf{b}\times\mathbf{c})$ for $\mathbf{a},\mathbf{b},\mathbf{c}\in\R^{3}$. 

In order to estimate $\mathrm{I}$, we note that, for all $a,b,a',b'\in\R$,
\begin{align}\label{eq:easy1}
\begin{split}
|(a,b,-1)\times(a',b',-1)| & = |(b'-b,a-a',ab'-a'b)| \\ & \leq c(1+|a|+|b|)(|a-a'|+|b-b'|) \\ 
& \leq c \sqrt{1+|a|^{2}+|b|^{2}}(|a-a'|+|b-b'|). 
\end{split}
\end{align}
We define 
\begin{align*}
\mathscr{U}':=\bigcup_{k=1}^{\infty}\big\{x'\in U\colon\;x'\;\text{is not a differentiability point of}\;\psi_{Z,t_{k}}\big\}, 
\end{align*}
whereby $\mathscr{L}^{2}(U\setminus\mathscr{U}')=0$ by Rademacher's theorem. In particular, every $x'\in U\setminus \mathscr{U}'$ is a differentiability point of every $\psi_{Z,t_{k}}$. For elements $x=(x',\psi_{Z,t_{k}}(x'))$ with $x'\in U_{k}^{r}\setminus\mathscr{U}'$, we  compute  
\begin{align*}
|\nu_{\partial(\Omega')^{\Phi,t_{k}}}(x)\times\nu_{\partial\Omega'}(x_{0})| & \stackrel{\eqref{eq:easy1}}{\leq} c \frac{|\nabla\psi_{Z,t_{k}}(x')-\nabla\psi_{Z}(x'_{0})|}{\sqrt{1+|\nabla\psi_{Z}(x'_{0})|^{2}}} \\ 
& \leq  c \frac{|\nabla\psi_{Z,t_{k}}(x')-\nabla\psi_{Z}(x')|}{\sqrt{1+|\nabla\psi_{Z}(x'_{0})|^{2}}} + c \frac{|\nabla\psi_{Z}(x')-\nabla\psi_{Z}(x'_{0})|}{\sqrt{1+|\nabla\psi_{Z}(x'_{0})|^{2}}}\\
&=: \mathrm{J}_{1}+\mathrm{J}_{2}. 
\end{align*}
We deal with the single terms $\mathrm{J}_{1}$ and $\mathrm{J}_{2}$ in $\mathrm{I}$, separately. 

\emph{On the $\mathrm{J}_{1}$-integral.} By Remark \ref{rem:aeintro}\ref{item:collario3}, 
we see that $\|\nabla\psi_{Z,t_{k}}\|_{\lebe^{\infty}(U)}\leq c\|\nabla\psi_{Z}\|_{\lebe^{\infty}(U)}\leq C$ for all $k\in\mathbb{N}$, and $\nabla\psi_{Z,t_{k}}\to\nabla\psi_{Z}$ in $\lebe^{1}(U;\R^{2})$ as $k\to\infty$. We conclude by use of $\FF\in\lebe^{\infty}(\Omega;\R^{3})$ and \eqref{eq:easy1} that 
\begin{align*}
\lim_{r\searrow 0}\lim_{k\to\infty}&\left\vert \frac{1}{a_{x_{0},r}}\int_{U_{k}^{r}}\FF^{*}(x',\psi_{Z,t_{k}}(x'))\cdot \mathrm{J}_{1}(x'){\sqrt{1+|\nabla\psi_{Z,t_{k}}(x')|^{2}}}\dif\mathscr{L}^{2}(x')\right\vert \\ &   \!\!\!\!\! \!\!\!\!\!\!\!\!\!\!\!\!\!\!\!\!\!\leq c\big(1+\|\nabla\psi_{Z}\|_{\lebe^{\infty}(U)}\big) 
\|\FF\|_{\lebe^{\infty}(\Omega)}\lim_{r\searrow 0}\frac{1}{a_{x_{0},r}}\underbrace{\lim_{k\to\infty} \int_{U}|\nabla\psi_{Z,t_{k}}(x')-\nabla\psi_{Z}(x')|\dif\mathscr{L}^{2}(x')}_{= 0}=0. 
\end{align*}

\emph{On the $\mathrm{J}_{2}$-integral.} Again, we recall that $\|\nabla\psi_{Z,t_{k}}\|_{\lebe^{\infty}(U)}\leq c\|\nabla\psi_{Z}\|_{\lebe^{\infty}(U)}\leq C$ for all $k\in\mathbb{N}$. 
Denote $U_{\infty}^{r}:=\{x'\in U\colon\;(x',\psi_{Z}(x'))\in \ball_{r}(x_{0})\cap\partial\Omega'\}$, so that $\mathbbm{1}_{U_{k}^{r}}\to\mathbbm{1}_{U_{\infty}^{r}}$ $\mathscr{L}^{2}$-a.e. in $U$. For all sufficiently small $r>0$, $\mathbf{G}\colon U_{\infty}^{r}\ni x'\mapsto (x',\psi_{Z}(x'))\in\ball_{r}(x_{0})\cap\partial\Omega'$ has bi-Lipschitz constant independent of $r>0$, whereby there exists $\mathtt{c}>1$ with 
\begin{align*}
\frac{1}{\mathtt{c}}\mathscr{L}^{2}(U_{\infty}^{r})\leq a_{x_{0},r} \leq {\mathtt{c}}\mathscr{L}^{2}(U_{\infty}^{r})\qquad\text{for all sufficiently small $0<r<1$}. 
\end{align*}
Hence, Lebesgue's theorem on dominated convergence and \eqref{eq:easy1} imply that 
\begin{align*}
\lim_{r\searrow 0}\lim_{k\to\infty}&\left\vert \frac{1}{a_{x_{0},r}}\int_{U_{k}^{r}}\FF^{*}(x',\psi_{Z,t_{k}}(x'))\cdot \mathrm{J}_{2}(x'){\sqrt{1+|\nabla\psi_{Z,t_{k}}(x')|^{2}}}\dif\mathscr{L}^{2}(x')\right\vert \\ &  \!\!\!\!\!\!\!\!\!\!\!\!\leq c \big(1+\|\nabla\psi_{Z}\|_{\lebe^{\infty}(U)}\big)\|\FF\|_{\lebe^{\infty}(\Omega)}\lim_{r\searrow 0}\underbrace{\frac{\mathscr{L}^{2}(U_{\infty}^{r})}{a_{x_{0},r}}}_{\leq\mathtt{c}} \dashint_{U_{\infty}^{r}}|\nabla\psi_{Z}(x'_{0})-\nabla\psi_{Z}(x')|\dif\mathscr{L}^{2}(x') \\ 
& \!\!\!\!\!\!\!\!\!\!\!\!\leq c (1+\|\nabla\psi_{Z}\|_{\lebe^{\infty}(U)})\|\FF\|_{\lebe^{\infty}(\Omega)}\underbrace{\lim_{r\searrow 0}\dashint_{U_{\infty}^{r}}|\nabla\psi_{Z}(x'_{0})-\nabla\psi_{Z}(x')|\dif\mathscr{L}^{2}(x')}_{ (*)} = 0. 
\end{align*}
To see that $(*)=0$, note that the Lipschitz property of $\mathbf{G}$ implies that there exists a constant $0<\theta<1$ such that $\ball_{\theta r}(x'_{0})\subset U_{\infty}^{r}\subset\ball_{r}(x'_{0})$ holds for all sufficiently small $r>0$;  \emph{e.g.}, we may take $\theta=(1+\mathrm{Lip}(\psi_{Z})^{2})^{-\frac{1}{2}}$. Therefore, we have
\begin{align*}
(*) & \leq \theta^{n}\lim_{r\searrow 0}\frac{1}{\mathscr{L}^{2}(U_{\infty}^{r})}\dashint_{\ball_{r}(x'_{0})}|\nabla\psi_{Z}(x'_{0})-\nabla\psi_{Z}(x')|\dif\mathscr{L}^{2}(x') = 0, 
\end{align*}
since $x'_{0}$ is both a differentiability point of $\psi_{Z}$ and a Lebesgue point of $\nabla\psi_{Z}$.
Then $(*)=0$. In total, $\mathrm{I}=0$, and so \ref{item:trace2} follows.

Finally, \ref{item:trace3} is a direct consequence of \ref{item:trace1} and \ref{item:trace2}; 
see also \eqref{eq:calculus7} in Appendix \hyperref[sec:AppendixD]{D} in the smooth case. 
This completes the proof.
\end{proof}

We point out that the functions $(\FF\times\nu_{\partial\Omega'})_{\partial\Omega'}^{\mathrm{int}}$
and $(\FF\times\nu_{\partial\Omega'})_{\partial\Omega'}^{\mathrm{ext}}$ only make sense as a whole, 
since the \emph{full} trace of $\FF\in\mathscr{CM}^{\infty}(\Omega)$ might not exist. Moreover, based on our conventions, we have \emph{e.g.} for $\FF\in\hold_{b}^{1}(\Omega;\R^{3})$ (in which case the traces along $\partial\Omega'$ are unambiguously defined), 
the relation $(\FF\times\nu_{\partial\Omega'})_{\partial\Omega'}^{\mathrm{int}}=(\FF\times\nu_{\partial\Omega'})_{\partial\Omega'}^{\mathrm{ext}}=\FF\times\nu_{\partial\Omega'}$ along $\partial\Omega'$. 
For our applications in Section \ref{sec:stokes}, we record a consequence  of Theorem \ref{thm:tracemain1}\ref{item:trace3} which will be accessible by Lemma \ref{lem:GoodLip} later on. 

\begin{corollary}\label{cor:LipAdmit}
In the situation of {\rm Theorem \ref{thm:tracemain1}}, 
let $\Omega'$ have $\hold^{1}$-boundary and let $\varphi\in\mathrm{Lip}_{\rm c}(\Omega')$ be such that 
\begin{enumerate}
\item\label{item:Lipchoose2} ${\varphi}|_{\Omega'}\in\hold_{b}^{1}(\Omega')$, ${\varphi}|_{\Omega\setminus\overline{\Omega'}}\in\hold_{b}^{1}(\Omega\setminus\overline{\Omega'})$, 
\item\label{item:Lipchoose3} there exists a sequence $(\delta_{i})\subset (0,1)$ with $\delta_{i}\searrow 0$ such that, for $\mathscr{H}^{2}$-a.e. $x\in\partial\Omega'$, 
\begin{align*}
\nabla_{\tau}(\rho_{\delta_{i}}*{\varphi})(x)\to \nabla_{\tau}\varphi(x)\qquad \mbox{as $i\to\infty$}, 
\end{align*}
where $\rho_{\delta}\in\hold_{\rm c}^{\infty}(\ball_{1}(0);[0,1])$ is the $\delta$-rescaled version of an arbitrary but fixed standard mollifier. 
\end{enumerate}
Then the following identity holds{\rm :}
\begin{align}\label{eq:criticalId}
\int_{\partial\Omega'}(\FF\times\nu_{\partial\Omega'})_{\partial\Omega'}^{\mathrm{int}}\cdot\nabla_{\tau}\varphi\dif\mathscr{H}^{2} = (\nabla{\varphi}\cdot\curl\FF)(\Omega')
=\int_{\Omega'}\nabla\varphi\cdot\dif\,(\curl\FF),
\end{align}
and an analogous formula holds for $(\FF\times\nu_{\partial\Omega'})_{\partial\Omega'}^{\mathrm{ext}}$ with the direct modifications.
\end{corollary}
\begin{proof}
By dominated convergence and since $\varphi|_{\Omega'}\in\hold_{b}^{1}(\Omega')$, we have  
\begin{align}\label{eq:combino2}
\int_{\Omega'}\nabla{\varphi}\cdot\dif\,(\curl\FF) & = \lim_{i\to\infty}\int_{\Omega'}\nabla(\rho_{\delta_{i}}*{\varphi})\cdot\dif\,(\curl\FF),
\end{align}
and  $\rho_{\delta_{i}}*{\varphi}$ belongs to $\hold_{\rm c}^{\infty}(\Omega)$ for all sufficiently large $i\in\mathbb{N}$. Hence, recalling that $\curl(\nabla(\rho_{\delta_{i}}*\varphi))=0$, we obtain 
\begin{align}\label{eq:combino3}
\begin{split}
\lim_{i\to\infty}\int_{\Omega'}\nabla(\rho_{\delta_{i}}*{\varphi})\cdot\dif\,(\curl\FF) & \stackrel{\eqref{eq:scaltovec}_{1}}{=}  \lim_{i\to\infty}\int_{\partial\Omega'}(\FF\times\nu_{\partial\Omega'})_{\partial\Omega'}^{\mathrm{int}}\cdot\nabla(\rho_{\delta_{i}}*{\varphi})\dif\mathscr{H}^{2} \\ 
& \!\!\!\!\!\!\!\!\!\!\!\! \stackrel{\text{Thm. \ref{thm:tracemain1}\ref{item:trace2}}}{=} \lim_{i\to\infty} \int_{\partial\Omega'}(\FF\times\nu_{\partial\Omega'})_{\partial\Omega'}^{\mathrm{int}}\cdot\nabla_{\tau}(\rho_{\delta_{i}}*{\varphi})\dif\mathscr{H}^{2} \\ 
& \!\! \stackrel{\ref{item:Lipchoose3}}{=} \int_{\partial\Omega'}(\FF\times\nu_{\partial\Omega'})_{\partial\Omega'}^{\mathrm{int}}\cdot\nabla_{\tau}\varphi\dif\mathscr{H}^{2}. 
\end{split}
\end{align}
Again, we note that $|\nabla_{\tau}(\rho_{\delta_{i}}*{\varphi})|\leq \|\nabla{\varphi}\|_{\lebe^{\infty}(\Omega)}$, so that the ultimate equality follows from \ref{item:Lipchoose3} and dominated convergence; recall that $(\FF\times\nu_{\partial\Omega'})_{\partial\Omega'}^{\mathrm{int}}\in\lebe^{\infty}(\partial\Omega';T_{\partial\Omega'})$. Combining \eqref{eq:combino2} with \eqref{eq:combino3} yields \eqref{eq:criticalId}. This completes the proof.
\end{proof}
    
    As a routine consequence of Theorem \ref{thm:tracemain1} (see \emph{e.g.}  \cite[Chpt. 5.4]{EvansGariepy} in the full gradient case), we collect a gluing result. 
\begin{corollary}[Gluing]\label{cor:gluing} Let $\Omega\subset\R^{3}$ be open and bounded, and let $\Omega'\Subset\Omega$ be open with Lipschitz boundary. If $\FF_{1}\in\mathscr{CM}^{\infty}(\Omega')$ and $\FF_{2}\in\mathscr{CM}^{\infty}(\Omega\setminus\overline{\Omega'})$, then the glued map 
\begin{align*}
\FF(x):=\begin{cases} 
\FF_{1}(x)&\;\text{if}\;x\in\Omega',\\ 
\FF_{2}(x)&\;\text{if}\;x\in\Omega\setminus\overline{\Omega'},
\end{cases}
\end{align*}
belongs to $\mathscr{CM}^{\infty}(\Omega)$ together with 
\begin{align*}
|\curl\FF|(\Omega) 
=& |\curl\FF_{1}|(\Omega') + |\curl\FF_{2}|(\Omega\setminus\overline{\Omega'}) \\ 
&+ \int_{\partial\Omega'}|(\FF\times\nu_{\partial\Omega})_{\partial\Omega'}^{\mathrm{int}}-(\FF\times\nu_{\partial\Omega})_{\partial\Omega'}^{\mathrm{ext}}|\dif\mathscr{H}^{2}.
\end{align*}
\end{corollary}

Finally, we record a consequence of the above results that might be of independent interest. Namely, we may use Theorem \ref{thm:tracemain1} to assign interior tangential traces in $\lebe^{\infty}$ to the \emph{gradients} of Lipschitz functions; note that this is impossible for general $\lebe^{\infty}$-fields. Previous results \cite{Alonso,BuffaCiarlet1,BuffaCiarlet2,Sheen,Tartar1997} in $\sobo^{\curl,p}$ for $1<p<\infty$, as recalled for the reader's convenience in Appendix \hyperref[sec:AppendixC]{C}, only allow for such an assignment in negative Sobolev spaces.

\begin{corollary}[Gradients of Lipschitz Maps]\label{cor:traceLipschitz} 
Let $\Omega\subset\R^{3}$ be open and bounded, and let $\Omega'\Subset\Omega$ be open with Lipschitz boundary. Then there exists a bounded linear operator
\begin{align*}
\mathrm{tr}_{\tau,\partial\Omega'}\colon \nabla\mathrm{Lip}(\Omega):=\{\nabla u\colon\;u\in\mathrm{Lip}(\Omega)\}\to\lebe^{\infty}(\partial\Omega';T_{\partial\Omega'})
\end{align*}
such that $\mathrm{tr}_{\tau,\partial\Omega'}(\nabla u)(x)=\nabla u(x)\times\nu_{\partial\Omega'}(x)$ holds for $\mathscr{H}^{2}$-a.e. $x\in\partial\Omega'$ and every $u\in\hold_{b}^{1}(\Omega)$, where 
$\nabla\mathrm{Lip}(\Omega)$ is endowed with the $\lebe^{\infty}$-norm.
\end{corollary}

\begin{proof}
We note that, if $u\in\mathrm{Lip}(\Omega)$, then $\nabla u\in\lebe^{\infty}(\Omega;\R^{3})$ and $\curl(\nabla u)=0$ in $\mathscr{D}'(\Omega;\R^{3})$. In particular, $\|\nabla u\|_{\mathscr{CM}^{\infty}(\Omega)} = \|\nabla u\|_{\lebe^{\infty}(\Omega)}$ holds for all $u\in\mathrm{Lip}(\Omega)$, where $\|\FF\|_{\mathscr{CM}^{\infty}(\Omega)}:=\|\FF\|_{\lebe^{\infty}(\Omega)}+|\curl\FF|(\Omega)$. The assignment  $\mathscr{CM}^{\infty}(\Omega)\ni\FF\mapsto (\FF\times\nu_{\partial\Omega'})_{\partial\Omega'}^{\mathrm{int}}\in\lebe^{\infty}(\partial\Omega')$ as in Theorem \ref{thm:tracemain1} is linear, 
and 
$(\FF\times\nu_{\partial\Omega'})_{\partial\Omega'}^{\mathrm{int}}=(\FF\times\nu_{\partial\Omega'})_{\partial\Omega'}^{\mathrm{ext}}$ $\mathscr{H}^{2}$-a.e. on $\partial\Omega'$ provided that $|\curl\FF|(\partial\Omega')=0$. 
For $u\in\mathrm{Lip}(\Omega)$, we define 
$$
\mathrm{tr}_{\tau,\partial\Omega'}(\nabla u):=(\nabla u\times\nu_{\partial\Omega'})_{\partial\Omega'}^{\mathrm{int}}(=(\nabla u\times\nu_{\partial\Omega'})_{\partial\Omega'}^{\mathrm{ext}}),
$$
Then $\mathrm{tr}_{\tau,\partial\Omega'}$ is clearly well-defined, linear, and bounded, 
and has the requisite property on $\hold_{b}^{1}(\Omega)$. This completes the proof. 
\end{proof}

\section{Transversal and Tangential Maximal Functions}\label{sec:maxfunctions}
In this intermediate section, we single out two weak-(1,1)-type estimates for certain transversal and tangential maximal operators. Most importantly, they will allow us to identify the manifolds where the Stokes theorem for 
$\mathscr{CM}^{p}$-fields is available.  
In particular, in \S\ref{sec:stokes}--\S\ref{sec:divmeasfieldsmanif}, 
the use of tangential or transversal maximal functions will lead to different sorts of Stokes theorems; Figure \ref{fig:geometricmain} below displays the effects which are captured by the corresponding maximal functions.

\smallskip
In what follows, let $\Omega\subset\R^{n}$ be open and bounded, 
let $\Omega'\Subset\Omega$ be open with $\hold^{1}$-boundary,
and let $\Sigma\subset\partial\Omega'$ be a $\hold^{1}$-regular Lipschitz boundary 
manifold relative to $\Omega'$. Potentially redefining $\Phi_{\partial\Omega'}$ from Lemma \ref{lem:collar1}, it is no loss of generality to assume that $\Phi_{\partial\Omega'}((-1,1)\times\partial\Omega')\Subset\Omega$. For $|t|<\frac{1}{2}$ and $0<\varepsilon<\frac{1}{4}$, we then define 
\begin{align}\label{eq:SigmaShiftNormal}
\begin{split}
&(\Sigma_{\Omega'}^{\Phi,t})_{\varepsilon} := \Phi_{\partial\Omega'}((t-\varepsilon,t+\varepsilon)\times\Sigma)),\\ & (\Sigma_{\Omega'}^{\Phi,t})_{\varepsilon}^{+}:= \Phi_{\partial\Omega'}((t,t+\varepsilon)\times\Sigma),\\ & (\Sigma_{\Omega'}^{\Phi,t})_{\varepsilon}^{-}:=\Phi_{\partial\Omega'}((t-\varepsilon,t)\times\Sigma).  
\end{split}
\end{align}

\begin{definition}[Transversal  Maximal Functions]\label{def:normalmax} In the above situation, let $\mu\in\mathrm{RM}(\Omega)$ or $\mu\in\mathrm{RM}(\Omega;\R^{3})$. We define the \emph{maximal function in  transversal direction} by 
\begin{align}\label{eq:normalmaxfunct}
\mathcal{M}_{\Sigma,\partial\Omega'}^{\Phi}\mu(t) := \sup_{0<\varepsilon<\frac{1}{4}}\frac{1}{\varepsilon}\int_{(\Sigma_{\Omega'}^{\Phi,t})_{\varepsilon}}\dif|\mu| \qquad\mbox{for $|t|<\frac{1}{2}$}, 
\end{align}
and the \emph{lower and upper maximal functions in  transversal direction} by 
\begin{align}\label{eq:upperlowermaxfunct}
\mathcal{M}_{\Sigma,\partial\Omega'}^{\Phi,\pm}\mu(t) := \sup_{0<\varepsilon<\frac{1}{4}}\frac{1}{\varepsilon}\int_{(\Sigma_{\Omega'}^{\Phi,t})_{\varepsilon}^{\pm}}\dif|\mu| \qquad 
\mbox{for $|t|<\frac{1}{2}$}. 
\end{align}
\end{definition}
From this definition, it is immediate that, for all $\mu\in\mathrm{RM}(\Omega)$ or $\mu\in\mathrm{RM}(\Omega;\R^{3})$,
\begin{align}\label{eq:maximalfunctionscomparison}
\mathcal{M}_{\Sigma,\partial\Omega'}^{\Phi,\pm}\mu(t) \leq \mathcal{M}_{\Sigma,\partial\Omega'}^{\Phi}\mu(t)\qquad\text{for all}\;|t|<\frac{1}{2}. 
\end{align}
For our later purposes, we also need tangential maximal functions. To this end, we let $\Psi_{\Sigma}\colon (-1,1)\times\Gamma_{\Sigma}\to\mathcal{O}(\subset\partial\Omega')$ be the bi-Lipschitz map provided by Lemma \ref{lem:collar2} and denote, for $-1<t<1$, $\Gamma_{\Sigma}^{t}:=\Psi(\{t\}\times\Gamma_{\Sigma})$. Moreover, we put for $|t|<\frac{1}{2}$ and $0<\varepsilon<\frac{1}{4}$
\begin{align}\label{eq:SigmaShiftTangential}
\begin{split}
& (\Gamma_{\Sigma}^{t})_{\varepsilon}:=\Psi_{\Sigma}((t-\varepsilon,t+\varepsilon)\times\Gamma_{\Sigma}), \\ 
& (\Gamma_{\Sigma}^{t})_{\varepsilon}^{+} := \Psi_{\Sigma}((t,t+\varepsilon)\times\Gamma_{\Sigma}), \\ 
& (\Gamma_{\Sigma}^{t})_{\varepsilon}^{-} := \Psi_{\Sigma}((t-\varepsilon,t)\times\Gamma_{\Sigma}). 
\end{split}
\end{align}

\begin{definition}[Tangential Maximal Functions]\label{def:tangentialmax} In the above situation, let $\mu\in\mathrm{RM}(\Omega)$ or $\mu\in\mathrm{RM}_{\mathrm{fin}}(\Omega;\R^{3})$. 
We then define the \emph{maximal function in tangential direction} by 
\begin{align}\label{eq:normalmaxfunct}
\mathcal{M}_{\Sigma,\partial\Omega'}^{\Psi}\mu(t) := \sup_{0<\varepsilon<\frac{1}{4}}\frac{1}{\varepsilon}\int_{(\Gamma_{\Sigma}^{t})_{\varepsilon}}\dif|\mu| \qquad \mbox{for $|t|<\frac{1}{2}$}, 
\end{align}
and the \emph{lower and upper maximal functions in tangential direction} by 
\begin{align}\label{eq:upperlowermaxfunct}
\mathcal{M}_{\Sigma,\partial\Omega'}^{\Psi,\pm}\mu(t) := \sup_{0<\varepsilon<\frac{1}{4}}\frac{1}{\varepsilon}\int_{(\Gamma_{\Sigma}^{t})_{\varepsilon}^{\pm}}\dif|\mu| \qquad \mbox{for $|t|<\frac{1}{2}$}. 
\end{align}
\end{definition}
The next proposition gives the key auxiliary tool for the following sections:

\begin{prop}[Hardy-Littlewood-Type Inequalities]\label{prop:HLWnormal}
Let $\mu\in\mathrm{RM}(\Omega)$. 
\begin{enumerate}
\item\label{item:HLW1} In the setting of {\rm Definition \ref{def:normalmax}}, 
for all $\lambda>0$
\begin{align}\label{eq:HLW1}
\mathscr{L}^{1}(\{t\in (-\tfrac{1}{2},\tfrac{1}{2})\colon\;\mathcal{M}_{\Sigma,\partial\Omega'}^{\Phi(,\pm)}\mu(t)>\lambda\})\leq \frac{10}{\lambda}|\mu|(\Phi_{\partial\Omega'}((-1,1)\times\partial\Omega'))
\end{align}
holds for all $\lambda>0$. 
In particular, the following holds: 
\begin{align}\label{eq:a.e.normal}
\mathcal{M}_{\Sigma,\partial\Omega'}^{\Phi(,\pm)}\mu(t)<\infty\qquad\text{for $\mathscr{L}^{1}$-a.e. $|t|<\frac{1}{2}$.}
\end{align} 
    \item\label{item:HLW2} In the setting of {\rm Definition \ref{def:tangentialmax}}, 
    for all $\lambda>0$,
    \begin{align}
  \mathscr{L}^{1}(\{t\in(-\tfrac{1}{2},\tfrac{1}{2})\colon\;\mathcal{M}_{\Sigma,\partial\Omega'}^{\Psi(,\pm)}\mu(t)>\lambda\})\leq \frac{10}{\lambda}|\mu|(\Psi_{\Sigma}((-1,1)\times\Gamma_{\Sigma}))
    \end{align}
 holds for all $\lambda>0$.    In particular, the following holds: 
\begin{align}\label{eq:a.e.normal1}
\mathcal{M}_{\Sigma,\partial\Omega'}^{\Psi(,\pm)}\mu(t)<\infty\qquad\text{for $\mathscr{L}^{1}$-a.e. $|t|<\frac{1}{2}$.}
\end{align}
\end{enumerate}
\end{prop} 
\begin{proof}
We focus on assertion \ref{item:HLW1} for $\mathcal{M}_{\Sigma,\partial\Omega'}^{\Phi}$, since the remaining assertions can be established by analogous means. The proof is a variation of the classical weak-$(1,1)$-property of the Hardy-Littlewood maximal operator: Let $\mu\in\mathrm{RM}(\Omega)$ or $\mu\in\mathrm{RM}(\Omega;\R^{3})$, and denote the set on the left-hand side of \eqref{eq:HLW1} by $\mathcal{O}_{\lambda}$. We then choose  $0<\varepsilon(t)<\frac{1}{4}$ 
for each $t\in\mathcal{O}_{\lambda}$ such that $|\mu|((\Sigma_{\Omega'}^{\Phi,t})_{\varepsilon(t)})>\lambda\varepsilon(t)$. By the Vitali covering lemma  (see \emph{e.g.} \cite[Thm. 1.24]{EvansGariepy}), there exists a sequence $(t_{j})\subset (-\frac{1}{2},\frac{1}{2})$ such that the intervals $I_{j}:=(t_{j}-\varepsilon(t_{j}),t_{j}+\varepsilon(t_{j}))$ are pairwise disjoint and satisfy 
$\mathcal{O}_{\lambda}\subset\bigcup_{j\in\mathbb{N}}5I_{j}$, where $5I_{j}:=(t_{j}-5\varepsilon(t_{j}),t_{j}+5\varepsilon(t_{j}))$. By Lemma \ref{lem:collar1}, the map $\Phi_{\partial\Omega'}$ is injective on $(-1,1)\times\partial\Omega'$, whereby the sets $(\Sigma_{\Omega'}^{\Phi,t_{j}})_{\varepsilon(t_{j})}(\subset\Phi_{\partial\Omega'}((-1,1)\times\partial\Omega'))$ are also pairwise disjoint. In consequence, 
\begin{align*}
\mathscr{L}^{1}(\mathcal{O}_{\lambda})\leq 10\sum_{j=1}^{\infty}\varepsilon(t_{j}) \leq \frac{10}{\lambda}\sum_{j=1}^{\infty}|\mu|((\Sigma_{\Omega'}^{\Phi,t_{j}})_{\varepsilon(t_{j})})\leq \frac{10}{\lambda}|\mu|(\Phi_{\partial\Omega'}((-1,1)\times\partial\Omega')). 
\end{align*}
 Hence, \eqref{eq:HLW1} follows, and \eqref{eq:a.e.normal} then is a direct consequence thereof. 
\end{proof}
The following simple observation will be used in \S \ref{sec:divmeasfieldsmanif}.

\begin{lem}\label{lem:curlvanishmax}
In the setting of {\rm Definition \ref{def:normalmax}}, 
suppose that $t\in(-\frac{1}{2},\frac{1}{2})$ is such that $\mathcal{M}_{\partial\Omega',\partial\Omega'}^{\Phi}\mu(t)<\infty$. Then 
$|\mu|(\Phi_{\partial\Omega'}(\{t\}\times\partial\Omega'))=0$.
\end{lem}

\begin{proof}
Suppose that $\mu\in\mathrm{RM}(\Omega)$ or $\mu\in\mathrm{RM}(\Omega;\R^{3})$ is such that $|\mu|(\Phi_{\partial\Omega'}(\{t\}\times\partial\Omega'))>0$. Since $\Phi_{\partial\Omega'}(\{t\}\times\partial\Omega')\subset\Phi_{\partial\Omega'}((t-\varepsilon,t+\varepsilon))\times\partial\Omega')$ for all $\varepsilon>0$, we conclude  
\begin{align*}
\mathcal{M}_{\partial\Omega',\partial\Omega'}^{\Phi}\mu(t) \geq \liminf_{\varepsilon\searrow 0}\frac{1}{\varepsilon}|\mu|(\Phi_{\partial\Omega'}(\{t\}\times\partial\Omega'))=\infty. 
\end{align*}
This completes the proof. 
\end{proof}

\section{The Stokes Theorem for $\mathscr{CM}^{\infty}$-Fields and Tangential Variations}\label{sec:stokes}
The classical Stokes theorem asserts that, if $\Sigma\subset\R^{3}$ is a smooth $2$-dimensional manifold oriented by $\nu_{\Sigma}\colon\mathrm{int}(\Sigma)\to\mathbb{S}^{2}$ and with boundary $\Gamma_{\Sigma}$, we have 
\begin{align}\label{eq:Stokes}
    \int_{\Sigma}\mathrm{curl}(\FF)\cdot\nu_{\Sigma}\dif\mathscr{H}^{2} = -\int_{\Gamma_{\Sigma}} \FF\cdot\tau_{\Gamma_{\Sigma}}\dif\mathscr{H}^{1}  
    \qquad\mbox{for all $\FF\in\hold^{1}(\R^{3};\R^{3})$}. 
\end{align}
Here, $\tau_{\Sigma}$ is the tangential field to $\Gamma_{\Sigma}$, in turn being determined by $\Sigma$ and $\nu_{\Sigma}$; the minus sign in \eqref{eq:Stokes} is due to our choice of orientation. As one of the main results of the present paper, the aim of this section is to generalize the Stokes formula \eqref{eq:Stokes} to $\mathscr{CM}^{\infty}$-fields; see Theorem \ref{thm:stokes} below. To this end, we fix the general setting in
\S \ref{sec:smoothsetCMinfty} and then proceed to the Stokes theorem with respect to tangential variations in \S \ref{sec:stokestanvarfirst}. 
In \S \ref{sec:consistency}, we are concerned with the consistency 
of Theorem \ref{thm:stokes} in view of \eqref{eq:Stokes} for smooth maps 
and, in \S \ref{sec:examplesstokes1st}, we discuss several examples and establish 
the optimality of the Stokes theorem, Theorem \ref{thm:stokes}.

The notion of \emph{tangential variations}  refers to varying a given manifold $\Sigma$ inside a larger boundary manifold, and corresponds to studying the light green manifolds displayed in Figure \ref{fig:geometricmain}. In contrast to this, 
\emph{transversal} variations correspond to the manifolds indicated by light blue in Figure \ref{fig:geometricmain}, and will be the main focus of \S \ref{sec:divmeasfieldsmanif}. 
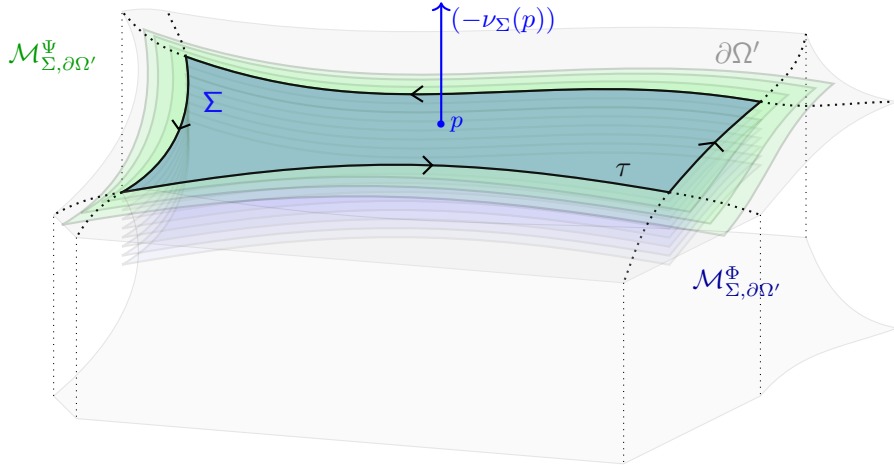
\begin{figure}
\begin{tikzpicture}[scale=1.2]
\draw[-,black,thick,fill=blue!50!white, opacity=0.9] (-3,0) [out=10, in =170] to (3,0) [out= 50, in = 220] to (4,1) [out= 170, in = -20] to (-2.3,1.5) [out= 280, in =30] to (-3,0);
\draw[-,black,thick,fill=blue!40!white, opacity=0.1] (-3,-0.1) [out=10, in =170] to (3,-0.1) [out= 50, in = 220] to (4,0.9) [out= 170, in = -20] to (-2.3,1.4) [out= 280, in =30] to (-3,-0.1);
\draw[-,black,thick,fill=blue!40!white, opacity=0.1] (-3,-0.2) [out=10, in =170] to (3,-0.2) [out= 50, in = 220] to (4,0.8) [out= 170, in = -20] to (-2.3,1.3) [out= 280, in =30] to (-3,-0.2);
\draw[-,black,thick,fill=blue!40!white, opacity=0.1] (-3,-0.3) [out=10, in =170] to (3,-0.3) [out= 50, in = 220] to (4,0.8) [out= 170, in = -20] to (-2.3,1.3) [out= 280, in =30] to (-3,-0.3);
\draw[-,black,thick,fill=blue!40!white, opacity=0.1] (-3,-0.4) [out=10, in =170] to (3,-0.4) [out= 50, in = 220] to (4,0.7) [out= 170, in = -20] to (-2.3,1.2) [out= 280, in =30] to (-3,-0.4);
\draw[-,black,thick,fill=blue!40!white, opacity=0.1] (-3,-0.5) [out=10, in =170] to (3,-0.5) [out= 50, in = 220] to (4,0.6) [out= 170, in = -20] to (-2.3,1.1) [out= 280, in =30] to (-3,-0.5);
\draw[-,black,thick,fill=blue!40!white, opacity=0.1] (-3,-0.6) [out=10, in =170] to (3,-0.6) [out= 50, in = 220] to (4,0.5) [out= 170, in = -20] to (-2.3,1) [out= 280, in =30] to (-3,-0.6);
\draw[-,black,thick,fill=blue!40!white, opacity=0.1] (-3,-0.7) [out=10, in =170] to (3,-0.7) [out= 50, in = 220] to (4,0.4) [out= 170, in = -20] to (-2.3,0.9) [out= 280, in =30] to (-3,-0.7);
\draw[-,black,thick,fill=blue!40!white, opacity=0.1] (-3,-0.8) [out=10, in =170] to (3,-0.8) [out= 50, in = 220] to (4,0.3) [out= 170, in = -20] to (-2.3,0.8) [out= 280, in =30] to (-3,-0.8);
\draw[-,green!20!black,thick,fill=green!40!white, opacity=0.2,scale=1.075] (-3,-0.1) [out=10, in =170] to (3,-0.2) [out= 50, in = 220] to (4,1) [out= 170, in = -20] to (-2.3,1.5) [out= 280, in =30] to (-3,-0.1);
\draw[-,green!20!black,thick,fill=green!40!white, opacity=0.2,scale=1.15] (-3,-0.2) [out=10, in =170] to (3,-0.3) [out= 50, in = 220] to (4,1) [out= 170, in = -20] to (-2.3,1.5) [out= 280, in =30] to (-3,-0.2);
\draw[-,green!20!black,thick,fill=green!40!white, opacity=0.2,scale=1.2] (-3.05,-0.3) [out=10, in =170] to (3,-0.4) [out= 50, in = 220] to (4,1) [out= 170, in = -20] to (-2.3,1.5) [out= 280, in =30] to (-3.05,-0.3);
\draw[-,black,thick] (-3,0) [out=10, in =170] to (3,0) [out= 50, in = 220] to (4,1) [out= 170, in = -20] to (-2.3,1.5) [out= 280, in =30] to (-3,0);
\draw[-,thick,dotted] (3,0) [out=-10, in =160,thick] to (4,-0.25);
\draw[-,thick,dotted] (4,1) [out= 45, in =260] to (4.5,1.75);
\draw[-,thick,dotted] (4,1) [out= -10, in =180] to (5.5,1.0);
\draw[-,thick,dotted] (3,0) [out= -130, in =70] to (2.5,-1.0);
\draw[-,thick,dotted] (-2.3,1.5) [out= 160, in =-50] to (-3,2.0);
\draw[-,thick,dotted] (-2.3,1.5) [out= 100, in =290] to (-2.5,2.0);
\draw[-,thick,dotted] (-3,0) [out= 210, in =60] to (-3.5,-0.5);
\draw[-,thick,dotted] (-3,0) [out= 190, in =30] to (-3.75,-0.25);
\node[blue] at (-2,1) {\large $\mathsf{\Sigma}$};
\node[black] at (2.5,0.25) {\large $\tau$};
\draw[-,fill=black!30!white,opacity=0.1] (-3.5,-0.5) -- (2.5,-1) -- (4,-0.25) [out=50, in =200] to (5.5,1.0) [out= 170, in = 290] to (4.5,1.75) [out= 175, in = -10] to (-2.5,2.0) [out=170, in = 10] to (-3,2.0) [out = -60, in =40] to (-3.75,-0.25) -- (-3.5,-0.5);
\draw[-,gray,fill=gray!20!white, opacity=0.2] (-3.5,-2.5) -- (2.5,-3) -- (4,-2.25) [out=50, in =200] to (5.5,-1.5) [out= 170, in = 290] to (4.5,-0.5) [out= 175, in = -10] to (-2.5,0) [out=170, in = 10] to (-3,0) [out = -60, in =40] to (-3.75,-2.25) -- (-3.5,-2.5);
\node[green!60!black] at (-3.75,1.5) {$\mathcal{M}_{\Sigma,\partial\Omega'}^{\Psi}$};
\node[blue!60!black] at (3.75,-1) {$\mathcal{M}_{\Sigma,\partial\Omega'}^{\Phi}$};
\draw[-,dotted] (-3.5,-0.5) -- (-3.5,-2.5);
\draw[-,dotted] (2.5,-1) -- (2.5,-3);
\draw[-,dotted] (4,-0.25) -- (4,-2.25);
\draw[-,dotted] ( 5.5,-1.5) -- (5.5,1);
\draw[-,dotted] ( 4.5,-0.5) -- (4.5,1.75);
\draw[-,dotted] ( -3,0) -- (-3,2);
\draw[-,dotted] (-3.75,-2.25) -- (-3.75,-0.25);
\draw[->,blue,thick] (0.5,0.75) -- (0.5,2.1);
\node[blue] at (0.5,0.75) {\tiny\textbullet};
\node[blue,right] at (0.5,0.75) {\small $p$};
\node[blue,right] at (0.5,1.9) {$(-\nu_{\Sigma}(p))$};
\draw[-,thick] (-2.45,0.8) -- (-2.38,0.7) -- (-2.25,0.75);
\draw[-,thick] (0.3,0.4) -- (0.4,0.3) -- (0.3,0.2);
\draw[-,thick] (3.32,0.525) -- (3.5,0.55) -- (3.595,0.45);
\draw[-,thick] (0.3,1.175) -- (0.175,1.08) -- (0.3,1.0);
\node[black!40!white] at (3.75,1.55) {\large $\mathsf{\partial\Omega'}$};
\end{tikzpicture}
\caption{The geometric situation in \S \ref{sec:stokes}--\S \ref{sec:divmeasfieldsmanif}.  The boundary manifolds in tangential direction as indicated in light green correspond to the tangentially varied manifolds as studied in \S \ref{sec:stokes}; see Theorem \ref{thm:stokes}. 
The boundary manifolds in transversal direction as indicated in light blue correspond to the setting of \S \ref{sec:divmeasfieldsmanif}; see Theorem  \ref{thm:stokes1st}. 
These variations are governed by $\mathcal{M}_{\Sigma,\partial\Omega'}^{\Psi}$ (green) 
and $\mathcal{M}_{\Sigma,\partial\Omega'}^{\Phi}$ (blue).}\label{fig:geometricmain}
\end{figure}

\subsection{Setting and Basic Facts}\label{sec:smoothsetCMinfty}
Throughout, let $\Omega\subset\R^{3}$ be open and bounded. 
Moreover, let $\Sigma$ be a $\hold^{1}$-regular Lipschitz boundary manifold relative 
to some $\Omega'\Subset\Omega$ with $\hold^{1}$-boundary, the latter being oriented by the inner unit normal $\nu_{\partial\Omega'}\colon \partial\Omega'\to\mathbb{S}^{2}$; 
we define $\nu_{\Sigma}:=\nu_{\partial\Omega'}|_{\Sigma}$.  

For the following, we fix a collar map $\Psi_{\Sigma}\colon (-1,1)\times\Gamma_{\Sigma}\to \mathcal{O}$ as in Lemma \ref{lem:collar2} and adopt the notation introduced afterwards. For $0\leq t<\frac{1}{2}$ and $0<\delta<\frac{1}{4}$, we define \emph{interior height functions} $\psi_{t,\delta,\Sigma}\colon\partial\Omega'\to[0,1]$ by
\begin{align}\label{eq:psiddefmain}
\psi_{t,\delta,\Sigma}(x):=\begin{cases}
0 & \;\text{if}\;x\in\partial\Omega'\setminus \Sigma^{\tau,t}, \\ 
\frac{1}{\delta}(s-t) &\;\text{if}\;x\in\Gamma_{\Sigma}^{s}(:=\Psi_{\Sigma}(\{s\}\times\Gamma_{\Sigma}))\;\text{for}\;t<s<t+\delta,\\ 
1 &\;\text{if}\;x\in \Sigma\setminus\Psi_{\Sigma}((t,t+\delta)\times\Gamma_{\Sigma}),
\end{cases}
\end{align}
where we have set $\Sigma^{\tau,t}:=\Sigma_{\Omega'}^{\tau,t}$ for brevity. 
This definition primarily concerns the case where $\Sigma\subsetneq\partial\Omega'$. 
If $\Sigma=\partial\Omega'$, 
then $\Gamma_{\Sigma}=\emptyset$, and $\psi_{t,\delta,\Sigma}\equiv 1$ globally on $\partial\Omega'$ in this case.

The interior height functions will serve as localizers 
in the Stokes functionals below. For our future applications, 
we require the uniform Lipschitz bounds on the $\psi_{t,\delta,\Sigma}$'s as follows: 

\begin{lem}[Lipschitz bounds]\label{lem:Lipbounds}
For each $0\leq t <\frac{1}{2}$ and each $0<\delta<\frac{1}{4}$, 
$\psi_{t,\delta,\Sigma}\in\mathrm{Lip}_{\rm c}(\partial\Omega';[0,1])$. 
More precisely,  there exists a constant $c=c(\Sigma,\Psi_{\Sigma})>0$ such that 
\begin{align}\label{eq:Lipbound1}
\|\nabla_{\tau}\psi_{t,\delta,\Sigma}\|_{\lebe^{\infty}(\partial\Omega')}\leq \frac{c}{\delta}\qquad \text{for all} \;\;0\leq t <\tfrac{1}{2}\;
\text{and} \;\;0<\delta<\tfrac{1}{4}. 
\end{align}
\end{lem}

\begin{proof}
Throughout, it is useful to keep in mind that, in the present geometric setting, it does not matter whether we derive the Lipschitz bounds for the geodesic or the ambient Euclidean metric; see Lemma \ref{lem:geodesic}. We record from Lemma \ref{lem:collar2}\ref{item:collF2} that there exists a constant $\theta=\theta(\Sigma,\Psi_{\Sigma})\geq 2$ with 
\begin{align}\label{eq:recallcollar}
 \frac{s'-s}{\theta}\leq \mathrm{dist}(\Gamma_{\Sigma}^{s'},\Gamma_{\Sigma}^{s})\leq \theta(s'-s)\qquad\text{for all}\;t\leq s< s' \leq t+\delta.
 \end{align}
There are now four non-trivial cases:  
\begin{itemize}
\item[(i)] Suppose that $\psi_{t,\delta,\Sigma}(x)=0$ and that $y\in\Gamma_{\Sigma}^{s}$ for some $t<s<t+\delta$.  For every $z\in\Gamma_{\Sigma}^{t}$, it follows that
\begin{align*}
\mathrm{dist}(y,\Gamma_{\Sigma}^{t})\leq |y-z|\leq |y-x|+|x-z|, 
\end{align*}
and infimising the previous inequality over $z\in\Gamma_{\Sigma}^{t}$ yields 
\begin{align}\label{eq:officermorgan}
\mathrm{dist}(y,\Gamma_{\Sigma}^{t})\leq  |x-y| + \mathrm{dist}(x,\Gamma_{\Sigma}^{t}) = |x-y| + c'\,\mathrm{dist}(x,\Sigma^{\tau,t}) \leq c|x-y|. 
\end{align}
Different from the geometrically flat case, it might happen that the constant from \eqref{eq:officermorgan} satisfies $c'>1$; see Figure \ref{fig:constantchoose}. We then conclude that
\begin{align}\label{eq:officermorganA}
|\psi_{t,\delta,\Sigma}(x)-\psi_{t,\delta,\Sigma}(y)|=\frac{s-t}{\delta} \stackrel{\eqref{eq:recallcollar}}{\leq} \frac{\theta}{\delta}\mathrm{dist}(\Gamma_{\Sigma}^{t},\Gamma_{\Sigma}^{s})\leq \frac{\theta}{\delta}\mathrm{dist}(y,\Gamma_{\Sigma}^{t})\stackrel{\eqref{eq:officermorgan}}{\leq} \frac{c\,\theta}{\delta}|x-y|.
\end{align}
\item[(ii)] If $t<s<s'<t+\delta$, $x\in\Gamma_{\Sigma}^{s}$ and $y\in\Gamma_{\Sigma}^{s'}$, then we estimate
\begin{align}\label{eq:officermorganB}
|\psi_{t,\delta,\Sigma}(x)-\psi_{t,\delta,\Sigma}(y)| = \frac{s'-s}{\delta} & \stackrel{\eqref{eq:recallcollar}}{\leq} \frac{\theta}{\delta}\mathrm{dist}(\Gamma_{\Sigma}^{s},\Gamma_{\Sigma}^{s'}) \leq \frac{\theta}{\delta}|x-y|. 
\end{align}
 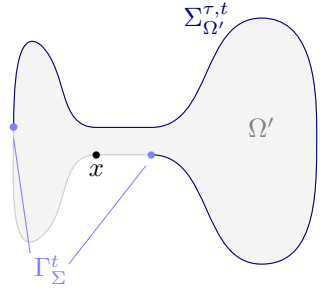
\begin{figure}
\begin{tikzpicture}[scale=1.45,rotate=90]
    \draw[-, fill=black!20!white,opacity=0.2] (0.25,1) -- (0.25,0.5) [out=270, in = 90] to (-0.75,-0.5) [out = 270,in=180] to (0.5,-1) [out=0,in = 270] to (1.5,-0.5) [out=90,in=270] to (0.5,0.5) -- (0.5,1) [out = 90, in = 225] to (1.25,1.5) [out=50,in=0] to (0.5,1.75) [out=180, in =130] to (-0.5,1.5) [out=310, in = 90] to (0.25,1);
    \node[black!50!white] at (0.5,-0.5) {$\Omega'$};
    \node[blue!50!black] at (1.5,0) {$\Sigma_{\Omega'}^{\tau,t}$};
   \draw[-,black!50!blue] (0.25,0.5) [out=270, in = 90] to (-0.75,-0.5) [out = 270,in=180] to (0.5,-1) [out=0,in = 270] to (1.5,-0.5) [out=90,in=270] to (0.5,0.5) -- (0.5,1) [out = 90, in = 225] to (1.25,1.5) [out=50,in=0] to (0.5,1.75) [out=180, in =130];
   \node[white!50!blue,left] at (-0.8,1.175) {$\Gamma_{\Sigma}^{t}$};
   \node[white!50!blue] at (0.25,0.5) {\tiny\textbullet};
   \node[white!50!blue] at (0.5,1.75) {\tiny\textbullet};
   \node[black] at (0.25,1) {\tiny\textbullet};
   \node[black,below] at (0.25,1) {$x$};
   \node[white!50!blue] at (0.5,1.75) {\tiny\textbullet};
   \draw[-,white!50!blue] (0.4,1.75) -- (-0.65,1.6);
   \draw[-,white!50!blue] (0.15,0.55) -- (-0.75,1.25);
\end{tikzpicture}
\caption{Lipschitz bounds on $\psi_{t,\delta,\Sigma}$. Different from the geometrically flat case, the constant $c'>0$ in \eqref{eq:officermorgan} cannot be chosen to equal $1$.}\label{fig:constantchoose}
 \end{figure}
\item[(iii)] If $\psi_{t,\delta,\Sigma}(x)=0$ and $\psi_{t,\delta,\Sigma}(y)=1$, then we have 
\begin{align}\label{eq:groupstep}
\delta \stackrel{\eqref{eq:recallcollar}}{\leq}\theta\,\mathrm{dist}(\Gamma_{\Sigma}^{t},\Gamma_{\Sigma}^{t+\delta}) \leq c'\,\theta\, \mathrm{dist}(\Sigma^{\tau,t+\delta},\partial\Omega'\setminus\Sigma^{\tau,t})\leq c\,\theta\,|x-y|, 
\end{align}
and therefore 
\begin{align}\label{eq:officermorganC}
|\psi_{t,\delta,\Sigma}(x)-\psi_{t,\delta,\Sigma}(y)|=1 \stackrel{\eqref{eq:groupstep}}{\leq} \frac{c\,\theta}{\delta}|x-y|.
\end{align}
\item[(iv)] If $t<s<t+\delta$,  $x\in\Gamma_{\Sigma}^{s}$ and $\psi_{t,\delta,\Sigma}(y)=1$, then 
\begin{align}\label{eq:officermorganD}
\begin{split}
|\psi_{t,\delta,\Sigma}(x)-\psi_{t,\delta,\Sigma}(y)| & = \frac{t+\delta- s}{\delta}\stackrel{\eqref{eq:recallcollar}}{\leq} \frac{\theta}{\delta}\,\mathrm{dist}(\Gamma_{\Sigma}^{t+\delta},\Gamma_{\Sigma}^{s}) \\ & \leq \frac{c'\,\theta}{\delta}\mathrm{dist}(\Sigma^{\tau,t+\delta},x)\leq \frac{c\,\theta}{\delta}|x-y|. 
\end{split}
\end{align}
\end{itemize}
Summarizing, \eqref{eq:officermorganA}--\eqref{eq:officermorganD} imply that $\psi_{t,\delta,\Sigma}$ is Lipschitz with Lipschitz constant at most $L=\frac{c\,\theta}{\delta}$, and then  \eqref{eq:Lipbound1} follows from  Lemma \ref{lem:LipCharManifold} 
in Appendix \hyperref[sec:AppendixA1]{A.1}. This completes the proof.
\end{proof}

\subsection{The Stokes theorem for the tangential variations}\label{sec:stokestanvarfirst}
We are now ready to approach the Stokes theorem with respect to the tangential variations. 
Without further mention, we adopt the geometric setting displayed in the previous subsection. 

Our strategy is as follows: In Definition \ref{def:stokes}, we introduce interior and exterior Stokes functionals; their definition is motivated by the left-hand side of \eqref{eq:Stokes} when being applied to $\varphi\FF$ with $\varphi\in\mathscr{D}(\Omega)$ and $\Sigma$ being replaced by $\Sigma^{\tau,t}$. This, in turn, is reduced to the (weak) interior and exterior  traces of the extended divergence measure field $\curl\FF\in\mathscr{DM}^{\mathrm{ext}}(\Omega)$ \emph{along} $\partial\Omega'$, and consequently must be localized to $\Sigma^{\tau,t}$. It is here where the  interior height functions $\psi_{t,\delta,\Sigma}$ from \eqref{eq:psiddefmain} enter. The localization procedure and, in particular, getting a definition independent of $\delta$, forces us to subsequently send $\delta\searrow 0$. As we show in Lemma \ref{lem:bdryintegralcont} and
Proposition \ref{prop:diff}, this is feasible for $\mathscr{L}^{1}$-a.e. $t\in[0,\frac{1}{2})$, in turn determining the \emph{good} manifolds $\Sigma^{\tau,t}$ where the Stokes functionals are well-defined indeed. Since, by construction, the manifolds $\Sigma^{\tau,t}$ are obtained by varying $\Sigma$ tangentially in $\partial\Omega'$, we speak of \emph{tangential variations}. 

The (strong) Stokes theorem (Theorem \ref{thm:stokes}) with respect to 
the tangential variations then asserts that, 
for $\mathscr{L}^{1}$-a.e. $\Sigma^{\tau,t}$ with $t\in[0,\frac{1}{2})$, 
the Stokes functionals are distributions of order zero, and so can be represented by Radon measures. The latter will be shown to be supported on $\Gamma_{\Sigma}^{t}$, 
and this will be the analogue of the right-hand side of \eqref{eq:Stokes}. Specifying to test functions $\varphi\in\mathscr{D}(\Omega)$ with $\varphi\equiv 1$ in a neighborhood of $\Sigma$ will eventually give us the requisite generalization of \eqref{eq:Stokes}, which means the weak or vorticity flux formulation of the Stokes theorem. In view of this program, we begin with the following definition.
\begin{definition}[Stokes Functionals]\label{def:stokes}
In the above situation, let $\FF\in\cm^{\infty}(\Omega)$. Given $0\leq t<\frac{1}{2}$ and $0<\delta<\frac{1}{4}$, let the interior height function $\psi_{t,\delta,\Sigma}$ 
be defined as in \eqref{eq:psiddefmain}. 
\begin{enumerate}
    \item Based on the normal trace for extended divergence measure fields 
    from {\rm Definition \ref{def:normaltraces}}, 
    we define
\begin{align}\label{eq:curllocaliser}
\begin{split}
&\left\langle\overline{\overline{(\curl\FF)\cdot\nu_{\Sigma^{\tau,t}}}},\varphi\right\rangle_{\Sigma^{\tau,t},\Omega'} := \lim_{\delta\searrow 0} \langle (\curl\FF)\cdot\nu,\overline{\psi}_{t,\delta,\Sigma}\varphi\rangle_{\partial\Omega'}
\\
&\left\langle\overline{\overline{(\curl\FF)\cdot\nu_{\Sigma^{\tau,t}}}},\varphi\right\rangle_{\Sigma^{\tau,t},\Omega\setminus\overline{\Omega'}} := \lim_{\delta\searrow 0} \langle (\curl\FF)\cdot\nu,\overline{\psi}_{t,\delta,\Sigma}\varphi\rangle_{\partial(\Omega\setminus\overline{\Omega'})},
\end{split}
\end{align}
for $\varphi\in\hold_{\rm c}^{1}(\Omega)$, provided that these limits exist,
where $\overline{\psi}_{t,\delta,\Sigma}\in\mathrm{Lip}_{\rm c}(\Omega)$ is an arbitrary Lipschitz extension of $\psi_{t,\delta,\Sigma}\colon\partial\Omega'\to[0,1]$. 
\item If the limits from \eqref{eq:curllocaliser} exist, we define the \emph{interior} and \emph{exterior Stokes functionals} along $\Sigma^{\tau,t}$ by
\begin{align}\label{eq:Stokesfunctionaldef}
\begin{split}
&\mathfrak{S}_{\Sigma^{\tau,t}}^{\mathrm{int}}(\varphi):= \left\langle\overline{\overline{(\curl\FF)\cdot\nu_{\Sigma^{\tau,t}}}},\varphi\right\rangle_{\Sigma^{\tau,t},\Omega'} +  \int_{\Sigma^{\tau,t}}(\FF\times\nu_{\partial\Omega'})_{\partial\Omega'}^{\interior}\cdot\nabla_{\tau}\varphi\dif\mathscr{H}^{2}\\
& \mathfrak{S}_{\Sigma^{\tau,t}}^{\mathrm{ext}}(\varphi):= \left\langle\overline{\overline{(\curl\FF)\cdot\nu_{\Sigma^{\tau,t}}}},\varphi\right\rangle_{\Sigma^{\tau,t},\Omega\setminus\overline{\Omega'}} +  \int_{\Sigma}\nabla_{\tau}\varphi\cdot(\FF\times\nu_{\partial(\Omega\setminus\overline{\Omega'})})_{T,\Omega'}^{\exterior}\dif\mathscr{H}^{2}.
\end{split}
\end{align}
\end{enumerate}
\end{definition}
We see later in \S\ref{sec:consistency} that this definition is natural 
and that it is consistent with
the classical Stokes formula \eqref{eq:Stokes} for sufficiently smooth maps $\FF$. Before we proceed, we note that $\psi_{t,\delta,\Sigma}\colon\partial\Omega'\to[0,1]$ is Lipschitz by Lemma \ref{lem:Lipbounds}, and therefore $\overline{\psi}_{t,\delta,\Sigma}$ exists by the classical McShane or Kirszbraun extension theorems. In the situation considered here, however, it will be convenient to work with the lifting provided by Lemma \ref{lem:GoodLip}\ref{item:Lipextend2}. For completeness, we note that the expressions as declared in \eqref{eq:curllocaliser} 
are well-defined, provided that they exist.

\begin{rem}[Well-definedness]\label{rem:welldefnessstokes}
In the situation of the previous definition, $\curl\FF\in\mathscr{DM}^{\mathrm{ext}}(\Omega)$, and so the pairings occuring on the right-hand side of \eqref{eq:curllocaliser} have a clear meaning. Moreover, if $\overline{\psi},\overline{\psi'}\in\mathrm{Lip}_{\rm c}(\Omega)$ are such that $\overline{\psi}=\overline{\psi'}$ on $\partial\Omega'$, then 
$\langle(\curl\FF)\cdot\nu,\overline{\psi}\rangle_{\partial\Omega'}=\langle(\curl\FF)\cdot\nu,\overline{\psi'}\rangle_{\partial\Omega'}$ by Lemma \ref{lem:LipschitzVanish} (recall that $\Omega'$ is open). 
\end{rem}

In the following, we give criteria for the limits  in \eqref{eq:curllocaliser} and \eqref{eq:Stokesfunctionaldef} to exist, respectively. To this end, we begin with an easier case:

\begin{lem}\label{lem:goodrepmeas}
In the situation of {\rm Definition \ref{def:stokes}}, 
suppose that $\langle(\curl\FF)\cdot\nu,\cdot\rangle_{\partial\Omega'}$ can be represented by a finite Radon measure $\mu_{(\curl\FF)\cdot\nu}\in\mathrm{RM}_{\mathrm{fin}}(\partial\Omega')$. Then  
\begin{align}\label{eq:measrep0}
\left\langle\overline{\overline{(\curl\FF)\cdot\nu_{\Sigma^{\tau,t}}}},\varphi\right\rangle_{\Sigma^{\tau,t},\Omega'} = \int_{\Sigma^{\tau,t}} \varphi \dif\mu_{(\curl\FF)\cdot\nu}
\end{align}
holds for all $\varphi\in\hold_{\rm c}^{1}(\Omega)$ and all $0\leq t <\frac{1}{2}$.
\end{lem}

\begin{proof}
Since $\mu_{(\curl\FF)\cdot\nu}\in\mathrm{RM}(\partial\Omega')$, $\spt(\psi_{t,\delta,\Sigma})\subset\overline{\Sigma}$, and $\psi_{t,\delta,\Sigma}|_{\Gamma_{\Sigma}^{t}}=0$, we have  
\begin{align}\label{eq:thenight}
\big\langle(\curl\FF)\cdot\nu,\psi_{t,\delta,\Sigma}\varphi\big\rangle_{\partial\Omega'} 
= \int_{\partial\Omega'}\psi_{t,\delta,\Sigma}\varphi\dif\mu_{(\curl\FF)\cdot\nu} = \int_{\Sigma^{\tau,t}}\psi_{t,\delta,\Sigma}\varphi\dif\mu_{(\curl\FF)\cdot\nu}
\end{align}
for all $\varphi\in\hold_{\rm c}^{1}(\Omega)$. Moreover, $|\psi_{t,\delta,\Sigma}|\leq 1$ for every $0<\delta<\frac{1}{4}$ and $\psi_{t,\delta,\Sigma}\to 1$ everywhere in the relatively open set $\Sigma^{\tau,t}$ as $\delta\searrow 0$. By Lebesgue's theorem on dominated convergence, 
we pass to the limit $\delta\searrow 0$ in \eqref{eq:thenight} and obtain \eqref{eq:measrep0}.  This completes the proof. 
\end{proof}

In general, we cannot expect $(\curl\FF)\cdot\nu$ to be represented by a measure on $\Sigma$. To deal with this general case (see  Proposition \ref{prop:diff} below), 
we record a preparatory lemma:
\begin{lem}\label{lem:bdryintegralcont}
    Let $\Omega'\Subset\Omega$ be open and bounded with $\hold^{1}$-boundary, and let $\Sigma$ be a $\hold^{1}$-regular Lipschitz boundary manifold relative to $\Omega'$.
Let $\mathbf{v}\in\hold^{1}(\partial\Omega';T_{\partial\Omega'})$. Then the limit 
\begin{align}\label{eq:lim0a}
\ell_{t}^{\mathbf{v}}(\varphi):=\lim_{\delta\searrow 0}\int_{\Sigma^{\tau,t}}\varphi \mathbf{v}\cdot\nabla_{\tau}\psi_{t,\delta,\Sigma}\dif\mathscr{H}^{2}\qquad\text{exists and is finite} 
\end{align}
for every $0\leq t<\frac{1}{2}$ and every $\varphi\in\hold^{1}(\partial\Omega')$. Moreover, for every $\varepsilon>0$, there exists $0<\delta'<\frac{1}{4}$ such that $0<\delta<\delta'$ implies that 
\begin{align}\label{eq:lim00a}
\left\vert \ell_{t}^{\mathbf{v}}(\varphi) - \int_{\Sigma^{\tau,t}}\varphi \mathbf{v}\cdot\nabla_{\tau}\psi_{t,\delta,\Sigma}\dif\mathscr{H}^{2}\right\vert < \varepsilon\|\varphi\|_{\hold^{1}(\partial\Omega')}\|\mathbf{v}\|_{\hold^{1}(\partial\Omega')}
\end{align}
holds for all $0\leq t<\frac{1}{2}$ and all $\varphi\in\hold^{1}(\partial\Omega')$. 
\end{lem}
\begin{proof}
First, let $0<\delta<\frac{1}{4}$ be arbitrary. 
By the smooth integration-by-parts formula from \eqref{eq:IBPmanifolds00}, we have 
\begin{align}\label{eq:ibpaux}
\int_{\Gamma_{{\Sigma}}^{s}}\varphi \mathbf{v}\cdot\nu_{\Gamma_{{\Sigma}}^{s}}\dif\mathscr{H}^{1}  = -\int_{{\Sigma}^{\tau,s}} \varphi\,\mathrm{div}_{\tau}(\mathbf{v})\dif\mathscr{H}^{2} - \int_{{\Sigma}^{\tau,s}}\nabla_{\tau}\varphi\cdot \mathbf{v} \dif\mathscr{H}^{2}
\end{align}
for all $t\leq s<t+\delta$. We apply the coarea formula from Lemma \ref{lem:coarea} to $g=\varphi\mathbf{v}\cdot\frac{\nabla_{\tau}\psi_{t.\delta,\Sigma}}{|\nabla_{\tau}\psi_{t,\delta,\Sigma}|}$ and $f=\psi_{t,\delta,\Sigma}$, whereby $\mathbf{J}_{1}^{{\Sigma}}f=|\nabla_{\tau}\psi_{\delta,t,\Sigma}|$. By the very definition of $\psi_{t,\delta,\Sigma}$, this gives us 
\begin{align}\label{eq:icetruckk}
\begin{split}
&\int_{\Psi_{\Sigma}((t,t+\delta)\times\Gamma_{{\Sigma}})} \varphi\mathbf{v}\cdot\nabla_{\tau}\psi_{t,\delta,\Sigma}\dif\mathscr{H}^{2} = \int_{{\Psi_{\Sigma}}((t,t+\delta)\times\Gamma_{{\Sigma}})} g\,\mathbf{J}_{1}^{{\Sigma}}f\dif\mathscr{H}^{2} \\ 
&\,\,=  \int_{0}^{1}\int_{{\Sigma}\cap f^{-1}(\{\vartheta\})}g\dif\mathscr{H}^{1}\dif\vartheta  = \int_{0}^{1}\int_{\Gamma_{\Sigma}^{t+\vartheta\delta}}\varphi\mathbf{v}\cdot \nu_{\Gamma_{\Sigma}^{t+\vartheta\delta}}\dif\mathscr{H}^{1}\dif \vartheta \\ 
&\,\,\stackrel{\eqref{eq:ibpaux}}{=} -\int_{0}^{1} \int_{{\Sigma}^{\tau,t+\vartheta\delta}} \varphi\,\mathrm{div}_{\tau}(\mathbf{v})\dif\mathscr{H}^{2}\dif\vartheta - \int_{0}^{1}\int_{{\Sigma}^{\tau,t+\vartheta\delta}}\nabla_{\tau}\varphi\cdot \mathbf{v} \dif\mathscr{H}^{2}\dif\vartheta.
\end{split}
\end{align}
In view of \eqref{eq:ibpaux}, which also holds for $s=t$, we thus conclude 
\begin{align}\label{eq:hotsummer}
\begin{split}
& \left\vert \int_{{\Psi_{\Sigma}}((t,t+\delta)\times\Gamma_{{\Sigma}})}\varphi \mathbf{v}\cdot\nabla_{\tau}\psi_{t,\delta,\Sigma}\dif\mathscr{H}^{2} - \int_{\Gamma_{{\Sigma}}^{t}}\varphi \mathbf{v}\cdot\nu_{\Gamma_{{\Sigma}}^{t}}\dif\mathscr{H}^{1}\right\vert  \\  
&\,\,\leq \left\vert \int_{0}^{1}\Big(\int_{{\Sigma}^{\tau,t+\vartheta\delta}} \varphi\,\mathrm{div}_{\tau}(\mathbf{v})\dif\mathscr{H}^{2} - \int_{{\Sigma}^{\tau,t}}\varphi\,\mathrm{div}_{\tau}(\mathbf{v})\dif\mathscr{H}^{2} \Big)\dif\vartheta  \right\vert \\ 
&\quad\,\,+ \left\vert \int_{0}^{1}\Big(\int_{{\Sigma}^{\tau,t+\vartheta\delta}}\nabla_{\tau}\varphi\cdot \mathbf{v} \dif\mathscr{H}^{2}
- \int_{{\Sigma}^{\tau,t}}\nabla_{\tau}\varphi\cdot \mathbf{v}\dif\mathscr{H}^{2}\Big)\dif \vartheta\right\vert  \\ 
&\,\,\leq 2\,\mathscr{H}^{2}({\Sigma}^{\tau,t}\setminus{\Sigma}^{\tau,t+\delta})\,\|\varphi\|_{\hold^{1}(\partial\Omega')}\|\mathbf{v}\|_{\hold^{1}(\partial\Omega')}.
\end{split}
\end{align}
To complete the proof, we recall that $\Psi_{\Sigma}\colon [-\frac{3}{4},\frac{3}{4}]\times\Gamma_{\Sigma}\to \Psi([-\frac{3}{4},\frac{3}{4}]\times\Gamma_{\Sigma})$ is bi-Lipschitz,  and $\mathfrak{L}:=\max\{\mathrm{Lip}(\Psi_{\Sigma}),\mathrm{Lip}(\Psi_{\Sigma}^{-1})\}<\infty$. Since ${\Sigma}^{\tau,t}\setminus{\Sigma}^{\tau,t+\delta}=\Psi((t,t+\delta]\times\Gamma_{\Sigma})$, we have
\begin{align}\label{eq:steinalgen}
\mathscr{H}^{2}({\Sigma}^{\tau,t}\setminus{\Sigma}^{\tau,t+\delta})\leq \mathfrak{L}\mathscr{H}^{2}((t,t+\delta]\times\Gamma_{\Sigma})=\mathfrak{L}\delta\mathscr{H}^{1}(\Gamma_{\Sigma}). 
\end{align}
Given $\varepsilon>0$, we may choose $\delta'=\min\{\varepsilon,1\}/(2\mathfrak{L}\mathscr{H}^{1}(\Gamma_{\Sigma})+4)$, whence \eqref{eq:hotsummer} and \eqref{eq:steinalgen} combine to \eqref{eq:lim0a} and \eqref{eq:lim00a}. The proof is complete.  
\end{proof}

\begin{rem}\label{rem:onthebdryproof}
It is important to note that the previous proof does not carry over to general $\lebe^{1}$-functions $\mathbf{v}$. However, it does extend to the case where $\mathbf{v}\in\lebe^{\infty}(\partial\Omega';T_{\partial\Omega'})$ is such that $\mathrm{div}_{\tau}(\mathbf{v})$ can be represented by a finite Radon measure. This will be a by-product of the proof of Theorem \ref{thm:stokes1st} in \S \ref{sec:divmeasfieldsmanif}.  
\end{rem} 
Even though the generic case of this section is for $\lebe^{\infty}$-fields, 
we directly state the next proposition for $\lebe^{1}$-fields,  since this  will play an important role in Section \ref{sec:Stokesgeneral}. In view of our applications below, it is important to note that we may  choose the set $\mathscr{I}$ as it appears in Proposition \ref{prop:diff} independently of the particular function $\varphi\in\hold^{1}(\partial\Omega')$.

\begin{prop}[Tangential limits on manifolds]\label{prop:diff}
Let $\Omega'\Subset\Omega$ be open and bounded with $\hold^{1}$-boundary, and let $\Sigma$ be a $\hold^{1}$-regular Lipschitz boundary manifold relative to $\Omega'$. Moreover, let $\mathbf{v}\in\lebe^{1}(\partial\Omega';T_{\partial\Omega'})$. Then there exists a set $\mathscr{I}\subset I:= [0,\frac{1}{2})$ such that $\mathscr{L}^{1}(I\setminus\mathscr{I})=0$, $\mathscr{I}\subset\{t\in I\colon\;\mathcal{M}_{\Sigma,\partial\Omega'}^{\Psi,+}\mathbf{v}(t)<\infty\}$, and, for every $t\in \mathscr{I}$, the limit 
\begin{align}\label{eq:lim1}
\ell_{t}^{\mathbf{v}}(\varphi):=\lim_{\delta\searrow 0}\int_{\Sigma^{\tau,t}}\varphi \mathbf{v}\cdot\nabla_{\tau}\psi_{t,\delta,\Sigma}\dif\mathscr{H}^{2}\qquad\text{exists and is finite} 
\end{align}
for every $\varphi\in\hold^{1}(\partial\Omega')$  
together with the estimate 
\begin{align}\label{eq:lim2}
|\ell_{t}^{\mathbf{v}}(\varphi)|\leq c\,(\mathcal{M}_{\Sigma,\partial\Omega'}^{\Psi,+}\mathbf{v}(t))\|\varphi\|_{\hold(\overline{\Sigma})}, 
\end{align}
where $c>0$ is independent of $\varphi$, $\mathbf{v}$, and $t$. 
\end{prop}

\begin{proof}
We begin with some preliminary remarks; for this, let $\varphi\in\hold_{\rm c}^{1}(\partial\Omega')$ be arbitrary with $0<\|\varphi\|_{\hold^{1}(\partial\Omega')}\leq 1$. We record from Lemma \ref{lem:Lipbounds}  that there exists a constant $\mathtt{c}_{1}=\mathtt{c}_{1}(\Sigma,\Psi_{\Sigma})>0$ such that
\begin{align}\label{eq:psideltagradient}
\|\nabla_{\tau}\psi_{t,\delta,\Sigma}\|_{\lebe^{\infty}(\Sigma)}\leq \frac{\mathtt{c}_{1}}{\delta}\qquad\text{for all $t\in I$ and all $0<\delta<\tfrac{1}{4}$}. 
\end{align}
This implies that, with the tangential maximal operator $\mathcal{M}_{\Sigma,\partial\Omega'}^{\Psi,+}$ from Definition \ref{def:tangentialmax}, 
\begin{align}\label{eq:maxo}
 \int_{\Sigma^{\tau,t}}|\varphi\,\mathbf{u}\cdot\nabla_{\tau}\psi_{t,\delta,\Sigma}|\dif\mathscr{H}^{2} \leq \mathtt{c}_{1}\|\varphi\|_{\hold(\partial\Omega')}\mathcal{M}_{\Sigma,\partial\Omega'}^{\Psi,+}\mathbf{u}(t) \leq \mathtt{c}_{1}\mathcal{M}_{\Sigma,\partial\Omega'}^{\Psi,+}\mathbf{u}(t)
\end{align}
holds for all $\mathbf{u}\in\lebe^{1}(\partial\Omega';T_{\partial\Omega'})$, all $t\in I$, 
and all $0<\delta<\frac{1}{4}$. 
Finally, by the weak-$(1,1)$-inequality from  Proposition \ref{prop:HLWnormal}\ref{item:HLW2}, there exists $\mathtt{c}_{2}>0$ such that, for all $\mathbf{u}\in\lebe^{1}(\partial\Omega';T_{\partial\Omega'})$ and all $\lambda>0$, 
\begin{align}\label{eq:highandlow}
\mathscr{L}^{1}(\{t\in I\colon\;\mathcal{M}_{\Sigma,\partial\Omega'}^{\Psi,+}\mathbf{u}(t)>\lambda\})\leq \frac{\mathtt{c}_{2}}{\lambda}\|\mathbf{u}\|_{\lebe^{1}(\partial\Omega')}. 
\end{align}
We now embark on the main part of the proof, and first  establish \eqref{eq:lim1}. To this end, we define 
\begin{align*}
\mathfrak{v}_{\delta}(t):=\int_{\Sigma^{\tau,t}}\varphi\, \mathbf{v}\cdot\nabla_{\tau}\psi_{t,\delta,\Sigma}\dif\mathscr{H}^{2}\qquad \mbox{for $t\in I$ and $0<\delta<\frac{1}{4}$},
\end{align*}
whereby $\mathfrak{v}_{\delta}\colon I\to\R$ is continuous; recall that $\mathbf{v}\in\lebe^{1}(\partial\Omega';T_{\partial\Omega'})$ and, by \eqref{eq:psideltagradient}, $\nabla_{\tau}\psi_{\delta,t,\Sigma}\in\lebe^{\infty}(\partial\Omega';T_{\partial\Omega'})$. 

We now prove that 
\begin{align}\label{eq:lim3}
\mathscr{L}^{1}(\mathscr{N}):=\mathscr{L}^{1}(\{t\in I\colon\;\lim_{\delta\searrow 0}\mathfrak{v}_{\delta}(t)\;\text{does not exist}\})=0.
\end{align}
To this end, we define  the \emph{tangential oscillation} 
for some given $0<\delta'<\frac{1}{4}$ by
\begin{align*}
\tosc_{\delta'}\mathbf{v}(t):=\sup_{0<\delta<\delta'}\mathfrak{v}_{\delta}(t)-\inf_{0<\delta<\delta'}\mathfrak{v}_{\delta}(t) \qquad \mbox{for $t \in I$}, 
\end{align*} 
so that $\tosc_{\delta'}\colon I\to [0,\infty]$ is $\mathscr{L}^{1}$-measurable. Let $\lambda>0$, and put $\varepsilon:=\frac{\lambda}{2}$. We then choose {$\mathbf{w}_{\lambda}\in\hold_{\rm c}^{1}(\partial\Omega';T_{\partial\Omega'})$} such that 
\begin{align}\label{eq:verfuegungvorschrift}
\|\mathbf{v}-\mathbf{w}_{\lambda}\|_{\lebe^{1}(\partial\Omega')} < \frac{\lambda\varepsilon}{4\mathtt{c}_{1}\mathtt{c}_{2}}. 
\end{align}
Since $\partial\Omega'$ is of class $\hold^{1}$, 
the existence of such a field $\mathbf{w}_{\lambda}$ 
can be established by flattening $\partial\Omega'$, 
routine mollification, and patching the mollified pieces together by use of a partition of unity.  
Since {$\mathbf{w}_{\lambda}\in\hold_{\rm c}^{1}(\partial\Omega';T_{\partial\Omega'})$}, Lemma \ref{lem:bdryintegralcont} implies that the limit: 
\begin{align}\label{eq:eighteen}
\mathfrak{l}_{t}:=\lim_{\delta\searrow 0}\int_{\Sigma^{\tau,t}}\varphi\,\mathbf{w}_{\lambda}\cdot\nabla_{\tau}\psi_{t,\delta,\Sigma}\dif\mathscr{H}^{2}
\end{align}
exists uniformly in $t\in I$. In consequence, there exists $0<\delta'<\frac{1}{4}$ such that 
\begin{align}\label{eq:threshold1}
\left\vert\mathfrak{l}_{t} - \int_{\Sigma^{\tau,t}}\varphi\,\mathbf{w}_{\lambda}\cdot\nabla_{\tau}\psi_{t,\delta,\Sigma}\dif\mathscr{H}^{2}\right\vert<\frac{\varepsilon}{4}\qquad\text{for all}\;0<\delta<\delta'\;\text{and all}\;t\in I. 
\end{align}
Hence, if $0<\delta_{1},\delta_{2}<\delta'$, then 
\begin{align}\label{eq:threshold2}
\left\vert \int_{\Sigma^{\tau,t}}\varphi\,\mathbf{w}_{\lambda}\cdot\nabla_{\tau}\psi_{t,\delta_{1},\Sigma}\dif\mathscr{H}^{2} -\int_{\Sigma^{\tau,t}}\varphi\,\mathbf{w}_{\lambda}\cdot\nabla_{\tau}\psi_{t,\delta_{2},\Sigma}\dif\mathscr{H}^{2}\right\vert<\frac{\varepsilon}{2}, 
\end{align}
and so in particular $
\tosc_{\delta'}\mathbf{w}_{\lambda}(t)<\varepsilon$ for all $t\in I$. 
Therefore, 
\begin{align*}
\mathrm{tOsc}_{\delta'}\mathbf{v}(t) & \leq \mathrm{tOsc}_{\delta'}(\mathbf{v}-\mathbf{w}_{\lambda})(t) + \mathrm{tOsc}_{\delta'}\mathbf{w}_{\lambda}(t) \leq \mathrm{tOsc}_{\delta'}(\mathbf{v}-\mathbf{w}_{\lambda})(t) + \varepsilon. 
\end{align*}
Hence, if $\mathrm{tOsc}_{\delta'}\mathbf{v}(t)>\lambda$, we recall $\varepsilon=\frac{\lambda}{2}$ to conclude that 
\begin{align*}
\frac{\lambda}{2} < \mathrm{tOsc}_{\delta'}(\mathbf{v}-\mathbf{w}_{\lambda})(t) & \leq 2\sup_{0<\delta<\delta'}\int_{\Sigma^{\tau,t}}|\varphi(\mathbf{v}-\mathbf{w}_{\lambda})\cdot\nabla_{\tau}\psi_{t,\delta,\Sigma}|\dif\mathscr{H}^{2} \\ & \!\!\!\!\stackrel{\eqref{eq:maxo}}{\leq} 2\mathtt{c}_{1}\,\mathcal{M}_{\Sigma,\partial\Omega'}^{\Psi,+}(\mathbf{v}-\mathbf{w}_{\lambda})(t).
\end{align*}
In particular, we have
\begin{align}\label{eq:janeklang}
\Big\{t\in I\colon\;\mathrm{tOsc}_{\delta'}\mathbf{v}(t)>{\lambda}\Big\}\subset \Big\{t\in I\colon\;\mathcal{M}_{\Sigma,\partial\Omega'}^{\Psi,+}(\mathbf{v}-\mathbf{w}_{\lambda})(t)>\frac{\lambda}{4\mathtt{c_{1}}}\Big\}=: J_{\lambda},
\end{align}
and 
$J_{\lambda}$ is independent of the specific choice of $\varphi$. 
We then see that 
\begin{align}\label{eq:dexter}
\begin{split}
\mathscr{L}^{1}(J_{\lambda}) \stackrel{\eqref{eq:highandlow}}{\leq}  \frac{4\mathtt{c}_{1}\mathtt{c}_{2}}{\lambda} \|\mathbf{v}-\mathbf{w}_{\lambda}\|_{\lebe^{1}(\partial\Omega')} \stackrel{\eqref{eq:verfuegungvorschrift}}{<}\varepsilon=\frac{\lambda}{2}. 
\end{split}
\end{align}
Next, we define
\begin{align*}
\mathscr{J}_{1}:=\bigcap_{N=1}^{\infty}\bigcup_{j=N}^{\infty}J_{2^{-j}},\quad
\mathscr{J}_{2}:=\{t\in I\colon\;\mathcal{M}_{\Sigma,\partial\Omega'}^{\Psi,+}\mathbf{v}(t)=\infty\},\quad
\mathscr{J}:=\mathscr{J}_{1}\cup\mathscr{J}_{2},
\end{align*}
whereby $\mathscr{J}_{1},\mathscr{J}_{2}$, and $\mathscr{J}$ are  obviously independent of the specific choice of $\varphi$. We have 
\begin{align*}
\mathscr{L}^{1}(\mathscr{J}_{1})\leq \lim_{N\to\infty}\mathscr{L}^{1}(\bigcup_{j=N}^{\infty}J_{2^{-j}}) \stackrel{\eqref{eq:dexter}}{\leq} 
\lim_{N\to\infty}\sum_{j=N}^{\infty}2^{-j-1}=0, 
\end{align*}
and Proposition \ref{prop:HLWnormal} yields that $\mathscr{L}^{1}(\mathscr{J}_{2})=0$. Hence, $\mathscr{L}^{1}(\mathscr{J})=0$, and we set $\mathscr{I}:=I\setminus\mathscr{J}$. To conclude the proof, we apply \eqref{eq:dexter} to $\lambda=2^{-j}$ for $j\in\mathbb{N}$, obtaining a sequence $(\delta_{j})\subset (0,\tfrac{1}{4})$ such that 
\begin{align}\label{eq:bruschetto}
\mathscr{L}^{1}(\{t\in I\colon\;\tosc_{\delta_{j}}\mathbf{v}(t)>2^{-j}\})<2^{-j-1}\qquad\text{for all}\;j\in\mathbb{N}.
\end{align}
By definition of  $\mathscr{N}$ (see  \eqref{eq:lim3}), we then arrive at
\begin{align*}
\mathscr{N} & \subset 
\big\{t\in I\colon\;\limsup_{\delta\searrow 0}\mathfrak{v}_{\delta}(t)-\liminf_{\delta\searrow 0}\mathfrak{v}_{\delta}(t)>0\big\} \\ 
& \subset \bigcup_{j=N}^{\infty}\big\{t\in I\colon\;\limsup_{\delta\searrow 0}\mathfrak{v}_{\delta}(t)-\liminf_{\delta\searrow 0}\mathfrak{v}_{\delta}(t)>2^{-j}\big\}\\ 
& \subset \bigcup_{j=N}^{\infty}\big\{t\in I\colon\;\tosc_{\delta_{j}}\mathbf{v}(t)>2^{-j}\big\}\stackrel{\eqref{eq:janeklang}}{\subset}\bigcup_{j=N}^{\infty}J_{2^{-j}}\qquad\text{for all}\;N\in\mathbb{N}. 
\end{align*}
In consequence, $\mathscr{N}\subset \mathscr{J}_{1}$, so that $\mathscr{L}^{1}(\mathscr{N})=0$, and \eqref{eq:lim3} follows. In particular, if $t\in \mathscr{I}=(I\setminus\mathscr{J}_{1})\cap(I\setminus\mathscr{J}_{2})$, then we have \eqref{eq:lim1}. For such $t$, \eqref{eq:lim2} is a direct consequence of \eqref{eq:maxo}. This completes the proof provided that $0<\|\varphi\|_{\hold^{1}(\partial\Omega)}\leq 1$. For general non-trivial $\varphi\in\hold^{1}(\partial\Omega)$, it suffices to observe that, if $t\in \mathscr{I}$, then the first part of the proof yields that 
\begin{align*}
\ell_{t}^{\mathbf{v}}(\varphi)=\|\varphi\|_{\hold^{1}(\partial\Omega)}\lim_{\delta\searrow 0} \int_{\Sigma^{\tau,t}}\Big(\frac{\varphi}{\|\varphi\|_{\hold^{1}(\partial\Omega)}}\Big)\mathbf{v}\cdot\nabla_{\tau}\psi_{t,\delta,\Sigma}\dif\mathscr{H}^{2}\;\;\;\text{exists and is finite}, 
\end{align*}
and the bound \eqref{eq:lim2} follows as in \eqref{eq:maxo}. This completes the proof. 
\end{proof}

\begin{rem}[On the existence of the limit \eqref{eq:lim1}]Below, we will apply Proposition \ref{prop:diff} to the particular choices $\mathbf{v}=(\FF\times\nu_{\partial\Omega'})_{\partial\Omega'}^{\mathrm{int}}$ and $\mathbf{v}=(\FF\times\nu_{\partial\Omega'})_{\partial\Omega'}^{\mathrm{ext}}$. Connecting with Remark \ref{rem:onthebdryproof}, we will prove in \S \ref{sec:divmeasfieldsmanif} that \eqref{eq:lim1} exists for every $t\in I$ if the tangential distributional divergence satisfies $\mathrm{div}_{\tau}(\mathbf{v})\in\mathrm{RM}_{\mathrm{fin}}(\partial\Omega')$. In the terminology of Section \ref{sec:divmeasfieldsmanif} below, this corresponds to $\mathbf{v}$ being a tangential divergence measure field. This can only be established for a.e. manifold in transversal direction, see Theorem \ref{thm:stokes1st}; in the context of Proposition \ref{prop:diff} and Theorem \ref{thm:stokes}, the underlying boundary $\partial\Omega'$ is
\emph{a priori} fixed and $\Sigma$ is only allowed to vary \emph{inside} $\partial\Omega$, so in tangential but not in transversal  direction.
\end{rem}

For the proof of the Stokes theorem (Theorem \ref{thm:stokes}), we isolate an auxiliary lemma as follows:
\begin{lem}\label{lem:geometrichelp}
Let $0\leq t<\frac{1}{2}$. Then there exists a constant $c=c(\Sigma,t)>0$ and a threshold number $r'>0$ such that 
\begin{align}\label{eq:technique}
\limsup_{\delta\searrow 0}\frac{\mathscr{H}^{2}(\Psi_{\Sigma}((t,t+\delta)\times\Gamma_{\Sigma})\cap\ball_{r}(x_{0}))}{\mathscr{H}^{2}(\Psi_{\Sigma}((t,t+\delta)\times\Gamma_{\Sigma}))}\leq c\,r
\end{align}
for all $x_{0}\in\Gamma_{\Sigma}^{t}$ and all $0<r<r'$. 
\end{lem}
\begin{proof}
In the present situation, there exist $N\in\mathbb{N}$ and  finitely many open sets $U^{1},\cdots,U^{N}\subset\R^{3}$ and $\widetilde{U}^{1}, \cdots,\widetilde{U}^{N}\subset\R^{3}$ 
such that all of the following hold:
\begin{itemize}
\item[(i)] For each  $k\in\{1, \cdots,N\}$, $U^{k}\subset\widetilde{U}^{k}$. 
\item[(ii)] $\Gamma_{\Sigma}^{t}\subset\bigcup_{k=1}^{N}U^{k}\subset\bigcup_{k=1}^{N}\widetilde{U}^{k}$. 
    \item[(iii)] For each $k\in\{1,\cdots,N\}$, there exist $0<\varepsilon_{k}<1$ and 
    a bi-Lipschitz $\hold^{1}$-diffeomorphism 
    $f_{k}\colon (-1-\varepsilon_{k},1+\varepsilon_{k})^{2}\to \widetilde{V}_{k}:=\partial\Omega'\cap \widetilde{U}^{k}$, 
    whose restriction $f_{k}\colon(-1,1)^{2}\to V_{k}:=\partial\Omega'\cap U^{k}$ 
    is equally a bi-Lipschitz $\hold^{1}$-diffeomorphism. In particular, 
 $\mathfrak{L}_{k}:=\max\{\mathrm{Lip}(f_{k}),\mathrm{Lip}(f_{k}^{-1})\}<\infty$, 
 and denote $\mathfrak{L}:=\max\{\mathfrak{L}_{1}, \cdots,\mathfrak{L}_{N}\}$.
    \item[(iv)] For each $k\in\{1,\cdots,N\}$, 
    there exists a Lipschitz function $g_{k}\colon (-1-\varepsilon_{k},1+\varepsilon_{k})\to (-\theta,\theta)$ for some $\theta=\theta_{k}\in (0,1)$ such that 
    \begin{align*}
    \mathrm{graph}(g_{k})=f_{k}^{-1}(\Gamma_{\Sigma}^{t}\cap \widetilde{U}^{k}), \qquad\mathrm{graph}(g_{k}|_{(-1,1)})=f_{k}^{-1}(\Gamma_{\Sigma}^{t}\cap U^{k})
    \end{align*}
    together with 
    \begin{align*}
    & \{(s,s')\colon\;|s|<1,\;g_{k}(s)<s'<1\} = f_{k}^{-1}(\Sigma\cap U^{k}), \\ 
    & \{(s,s')\colon\;|s|<1+\varepsilon_{k},\;g_{k}(s)<s'<1\} = f_{k}^{-1}(\Sigma\cap \widetilde{U}^{k}), \\ 
    & \{(s,s')\colon\;|s|<1,\;-1<s'<g_{k}(s)\} = f_{k}^{-1}((\partial\Omega'\setminus\overline{\Sigma})\cap U^{k}), \\ 
&\{(s,s')\colon\;|s|<1+\varepsilon_{k},\;-1<s'<g_{k}(s)\}=f_{k}^{-1}((\partial\Omega'\setminus \overline{\Sigma})\cap\widetilde{U}^{k}).
    \end{align*}
\end{itemize}
These properties are a direct consequence of Definition \ref{def:bdrymanifolds}.

We note that there exist $\delta_{0}>0$ and $r_{0}>0$ such that both $\Psi_{\Sigma}([t,t+\delta_{0}]\times\Gamma_{\Sigma})$ and $\Gamma_{\Sigma}^{t}+\overline{\ball}_{r_{0}}(0)$ are contained in $\bigcup_{k=1}^{N}U^{k}$. The set $\mathfrak{A}:=\Psi_{\Sigma}([t,t+\delta_{0}]\times\Gamma_{\Sigma})\cup ((\Gamma_{\Sigma}^{t}+\overline{\ball}_{r_{0}}(0))\cap\partial\Omega')$ is compact, whereby its Lebesgue number $d>0$ with respect to the covering $U^{1},\cdots,U^{N}$ is well-defined; in particular, whenever $A\subset\mathfrak{A}$ satisfies $\mathrm{diam}(A)<d$, then there exists $k=k(A)\in\{1,\cdots,N\}$ such that $A\subset U^{k}$. Diminishing $d$ if necessary, it is no loss of generality to assume that $d<\frac{r_{0}}{10}$. We put 
\begin{align}\label{eq:rchoose}
r':=\frac{1}{1000}\min\big\{\delta_{0},d,\frac{\varepsilon_{1}}{\mathfrak{L}}, \cdots,\frac{\varepsilon_{N}}{\mathfrak{L}}\big\}
\end{align}
and let $0<r<r'$. 

Now let $x_{0}\in\Gamma_{\Sigma}^{t}$ be arbitrary. The set $\ball_{d/3}(x_{0})\cap\partial\Omega'$ has diameter at most $\frac{2}{3}d$ and, because of $d<\frac{r_{0}}{10}$, is contained in $\Gamma_{\Sigma}^{t}+\overline{\ball}_{r_{0}}(0)$; in particular, $\ball_{d/3}(x_{0})\cap\partial\Omega'\subset V_{k}$ for some $k\in\{1,\cdots,N\}$ and $\ball_{r}(x_{0})\cap\partial\Omega'$ is relatively compact in $V_{k}$. We set  $(y_{0,1},y_{0,2}):=y_{0}:=f_{k}^{-1}(x_{0})$ and observe that $f_{k}^{-1}(\ball_{r}(x_{0})\cap\partial\Omega')$ is relatively compact in $(-1,1)^{2}$. On the other hand, we have 
\begin{align*}
\mathrm{diam}(f_{k}^{-1}(\ball_{r}(x_{0})\cap\partial\Omega')) \leq \mathfrak{L}r \stackrel{\eqref{eq:rchoose}}{<}\frac{\varepsilon_{k}}{1000},
\end{align*}
and so we may locate the open set $f_{k}^{-1}(\ball_{r}(x_{0})\cap\partial\Omega)$ 
quantitatively as 
\begin{align}\label{eq:justakissaway}
\begin{split}
f_{k}^{-1}(\ball_{r}(x_{0}) \cap\partial\Omega) & \subset (y_{0,1}-\mathfrak{L}r,y_{0,1}+\mathfrak{L}r)\times(-1,1) \\ & =: Z \subset (-1-\varepsilon_{k},1+\varepsilon_{k})\times (-1,1).
\end{split}
\end{align}
By Lemma \ref{lem:collar2}\ref{item:collF2}, there exists $M_{1}>0$ such that 
\begin{align}\label{eq:boiler}
\Psi_{\Sigma}([t,t+\delta)\times\Gamma_{\Sigma})
\subset \big\{x\in\overline{\Sigma^{\tau,t}}\colon\;\mathrm{dist}(x,\Gamma_{\Sigma}^{t})<M_{1}\delta\big\} 
\end{align}
for all sufficiently small $\delta>0$. Thus, for all such $\delta>0$, we have 
\begin{align*}
\Psi_{\Sigma}([t,t+\delta)\times\Gamma_{\Sigma})\cap\ball_{r}(x_{0})& \subset \{x\in\overline{\Sigma^{\tau,t}}\colon\;\mathrm{dist}(x,\Gamma_{\Sigma}^{t})<M_{1}\delta\} \cap\ball_{r}(x_{0}) \\ & \subset \ball_{d/3}(x_{0})\cap\partial\Omega'\subset V_{k}, 
\end{align*}
so that
\begin{align}\label{eq:justashotaway}
\begin{split}
&f_{k}^{-1}(\Psi_{\Sigma}([t,t+\delta)\times\Gamma_{\Sigma})\cap\ball_{r}(x_{0})) \\
& \subset f_{k}^{-1}(\{x\in\overline{\Sigma^{\tau,t}}\cap V_{k}\colon\;\mathrm{dist}(x,\Gamma_{\Sigma}^{t})<M_{1}\delta\} ) 
\cap f_{k}^{-1}(\ball_{r}(x_{0})\cap\partial\Omega') =: A_{k}^{1}\cap A_{k}^{2}.
\end{split}
\end{align}
Since $f_{k}$ is bi-Lipschitz with bi-Lipschitz 
constant $\mathfrak{L}_{k}\leq\mathfrak{L}<\infty$, then (iii), \eqref{eq:justakissaway}, and \eqref{eq:justashotaway} imply that there exists $M_{2}>0$ 
depending only on $\mathfrak{L}$ and $M_{1}$ such that 
\begin{align*}
A_{k}^{1}\cap A_{k}^{2} & \subset 
\big\{y=(s,s')\colon\;|s-y_{0,1}|<\mathfrak{L}r,\;s'\geq g_{k}(s),\;\mathrm{dist}(y,\mathrm{graph}(g_{k}))<M_{2}\delta\big\} \\ 
& =:A_{k} \subset (-1-\varepsilon_{k},1+\varepsilon_{k})^{2}
\end{align*}
for all sufficiently small $\delta>0$; see Figure \ref{fig:dinergetthebriefcase} for the geometric set-up. 

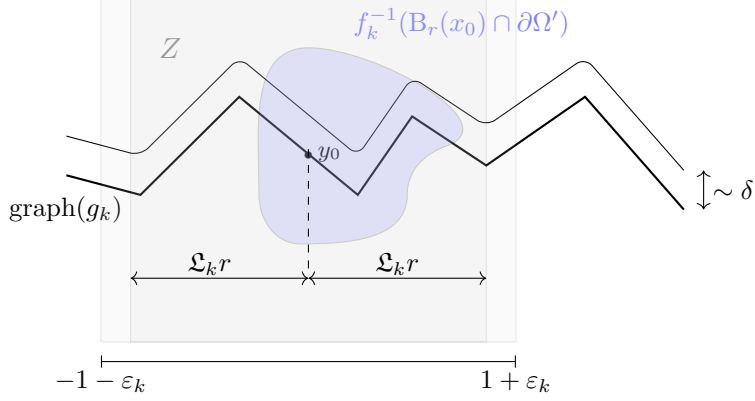
\begin{figure}[t]
\begin{tikzpicture}[scale=1.3]
\node at (0.7,0.4) {{\tiny \textbullet}};
\node at (-1.8,-0.15) {{ $\mathrm{graph}(g_{k})$}};
\node[right] at (0.7,0.42) {{\small $y_{0}$}};
\draw[-,thick] (-1.75,0.2) -- (-1,0) -- (0,1) -- (1.2,0) -- (1.75,0.8) -- (2.5,0.3) -- (3.5,1) -- (4.5,-0.15);
\draw[-,rounded corners] (-1.75,0.6) -- (-1,0.4) -- (0,1.4) -- (1.2,0.4) -- (1.75,1.2) -- (2.5,0.7) -- (3.5,1.4) -- (4.5,0.25);
\draw[<->] (4.7,0.25) -- (4.7,-0.15);
\node at (5,0.05) {$\sim\delta$};
\draw[-,fill=blue!50!white,opacity=0.2] (0.62,-0.5) [out = 0, in =270] to (1.7,0) [out =90, in =200] to (2.1,0.5) [out=20, in =0] to (0.7,1.5) [out=180, in =90] to (0.2,0) [out=270, in =180] to (0.62,-0.5);
\node[blue!50!white] at (2.25,1.75) {$f_{k}^{-1}(\ball_{r}(x_{0})\cap\partial\Omega')$};
\draw[fill=black!20!white,opacity=0.1] (-1.1,-1.5)-- (2.5,-1.5) -- (2.5,2) -- (-1.1,2) -- (-1.1,-1.5);
\draw[fill=black!20!white,opacity=0.1] (-1.4,-1.5)-- (2.8,-1.5) -- (2.8,2) -- (-1.4,2) -- (-1.4,-1.5);
\draw[<->] (-1.1,-0.85)-- (0.7,-0.85); 
\draw[<->] (0.7,-0.85)
-- (2.5,-0.85);
\node at (1.6,-0.7) {$\mathfrak{L}_{k}r$};
\node at (-0.3,-0.7) {$\mathfrak{L}_{k}r$};
\draw[dashed] (0.7,-0.75) -- (0.7,0.47);
\node[black!40!white] at (-0.7,1.5) {\large $Z$};
\draw[|-|] (-1.4,-1.7)-- (2.8,-1.7);
\node[below] at (-1.4,-1.7) {$-1-\varepsilon_{k}$};
\node[below] at (2.8,-1.7) {$1+\varepsilon_{k}$};
\end{tikzpicture}
\caption{On the proof of Lemma \ref{lem:geometrichelp}.}\label{fig:dinergetthebriefcase} 
\end{figure}

To conclude the proof, we employ the coarea formula. The distance function $y\mapsto \mathrm{dist}(y,\mathrm{graph}(g_{k}))$ is Lipschitz, and so is classically differentiable at $\mathscr{L}^{2}$-{a.e.} $y\in A_{k}$. At each such point, its approximate gradient satisfies $|\nabla\mathrm{dist}(y,\mathrm{graph}(g_{k}))|=1$. Then
\begin{align}\label{eq:anightaway}
\begin{split}
\mathscr{L}^{2}(A_{k}) & = \int_{0}^{M_{2}\delta}\mathscr{H}^{1}(A_{k}\cap\mathrm{dist}(\cdot,\mathrm{graph}(g_{k}))^{-1}(\{\theta\}))\dif\theta \\ & \leq c\,\delta \mathscr{H}^{1}(\{(s,g_{k}(s))\colon\;|s-y_{0,1}|<\mathfrak{L}r\}) \leq c\delta r 
\end{split}
\end{align}
by Lemma \ref{lem:distancelevelsets}. In particular, in the ultimate inequality, the constant $c>0$ is independent of $\delta$ and $r$. Since $f_{k}$ is bi-Lipschitz with bi-Lipschitz constants $\mathfrak{L}_{k}\leq\mathfrak{L}<\infty$, we deduce from \eqref{eq:justashotaway}--\eqref{eq:anightaway} that, with $c>0$ still being independent of $\delta$ and $r$, 
\begin{align}\label{eq:heretostay1}
\mathscr{H}^{2}(\Psi_{\Sigma}((t,t+\delta)\times\Gamma_{\Sigma})\cap\ball_{r}(x_{0}))\leq c\delta r. 
\end{align}
On the other hand,  there exists $c'>1$ such that 
\begin{align}\label{eq:heretostay2}
\frac{1}{c'}\delta\mathscr{H}^{1}(\Gamma_{\Sigma})\leq \mathscr{H}^{2}(\Psi_{\Sigma}((t,t+\delta)\times\Gamma_{\Sigma}))\leq c'\delta\mathscr{H}^{1}(\Gamma_{\Sigma})  
\qquad\mbox{for all sufficiently $\delta>0$}. 
\end{align}
Combining \eqref{eq:heretostay1} and \eqref{eq:heretostay2} then yields \eqref{eq:technique}, and the proof is complete. 
\end{proof}

We are now ready to state and prove the Stokes theorem for tangential variations. In particular, we establish in which sense the previously introduced Stokes functionals are well-defined on almost every manifold with respect to tangential variations. This cannot be improved to \emph{every} manifold, see Example \ref{ex:nonex} below.

\begin{theorem}[Stokes Theorem for the Tangential Variations]\label{thm:stokes} 
Let $\FF\in\mathscr{CM}^{\infty}(\Omega)$. In the situation described in \S \ref{sec:smoothsetCMinfty}, 
for $t\in I:=[0,\frac{1}{2})$, consider the Stokes functionals 
$\mathfrak{S}_{\Sigma^{\tau,t}}^{\mathrm{int}}$ and 
$\mathfrak{S}_{\Sigma^{\tau,t}}^{\mathrm{ext}}$ as introduced in  \eqref{eq:Stokesfunctionaldef}. Then there exist sets $\mathscr{I}_{1},\mathscr{I}_{2}\subset I$ with $\mathscr{L}^{1}(I\setminus\mathscr{I}_{1})=\mathscr{L}^{1}(I\setminus\mathscr{I}_{2})=0$ such that the following hold{\rm :} 
\begin{enumerate}
\item\label{item:elementary1} \emph{Measure representation:} For every $t\in\mathscr{I}_{1}$, $\mathfrak{S}_{\Sigma^{\tau,t}}^{\mathrm{int}}$ is a \emph{well-defined distribution of order zero} on $\Omega$ with 
\begin{align}\label{eq:stokessupport}
\spt(\mathfrak{S}_{\Sigma^{\tau,t}}^{\mathrm{int}})\subset\Gamma_{\Sigma}^{t},  
\end{align}
and for every $t\in\mathscr{I}_{2}$, $\mathfrak{S}_{\Sigma^{\tau,t}}^{\mathrm{ext}}$ is a well-defined distribution of order zero on $\Omega$ with 
\begin{align*}
\spt(\mathfrak{S}_{\Sigma^{\tau,t}}^{\mathrm{ext}})\subset\Gamma_{\Sigma}^{t}.
\end{align*}
In consequence, for such $t$, there exist {finite} and  {signed} 
Radon measures $\mu_{\Sigma,t}^{\tau,\mathrm{int}}$ and $\mu_{\Sigma,t}^{\tau,\mathrm{ext}}$ with 
\begin{align*}
\spt(\mu_{\Sigma,t}^{\tau,\mathrm{int}}),\spt(\mu_{\Sigma,t}^{\tau,\mathrm{ext}})\subset\Gamma_{\Sigma}^{t}
\end{align*} 
such that,
for all $\varphi\in\hold_{\rm c}^{1}(\Omega)$,
\begin{align}\label{eq:representasmeasure}
\mathfrak{S}_{\Sigma^{\tau,t}}^{\mathrm{int}}(\varphi) = \int_{\Omega}\varphi \dif \mu_{\Sigma,t}^{\tau,\mathrm{int}}, \qquad \mathfrak{S}_{\Sigma^{\tau,t}}^{\mathrm{ext}}(\varphi) = \int_{\Omega}\varphi \dif \mu_{\Sigma,t}^{\tau,\mathrm{ext}}.
\end{align}
 
\item\label{item:StokesAbsCon} \emph{$\mathscr{H}^{1}$-absolute continuity:} For all $t\in\mathscr{I}_{1}$ or $t\in\mathscr{I}_{2}$ as in  \ref{item:elementary1}, 
\begin{align*}
\mu_{\Sigma,t}^{\tau,\mathrm{int}}\ll\mathscr{H}^{1}\mres\Gamma_{\Sigma}^{t},
\qquad \mu_{\Sigma,t}^{\tau,\mathrm{ext}}\ll\mathscr{H}^{1}\mres\Gamma_{\Sigma}^{t}. 
\end{align*}
In particular, for $t\in\mathscr{I}_{1}$, there exists a density 
\begin{align}\label{eq:densityGammaSigma}
\frac{\dif\mu_{\Sigma,t}^{\tau,\interior}}{\dif\mathscr{H}^{1}}\in\lebe^{1}(\Gamma_{\Sigma}^{t})\;\;\;\;\text{such that}\;\;\;\;\mu_{\Sigma,t}^{\tau,\mathrm{int}} = \frac{\dif\mu_{\Sigma,t}^{\tau,\mathrm{int}}}{\dif\mathscr{H}^{1}}\mathscr{H}^{1}\mres\Gamma_{\Sigma}^{t}. 
\end{align}
The same holds true with the obvious modifications for $t\in\mathscr{I}_{2}$ and $\mu_{\Sigma,t}^{\tau,\mathrm{ext}}$. 
\item\label{item:StokesTangential} \emph{Tangential character of $\mu_{\Sigma,t}^{\tau,\mathrm{int}}$ 
and $\mu_{\Sigma,t}^{\tau,\mathrm{ext}}$:} 
For $t\in\mathscr{I}_{1}$, suppose that $x_{0}\in\Gamma_{\Sigma}^{t}$ is a continuity point of $(\FF\times\nu_{\partial\Omega'})_{\partial\Omega'}^{\mathrm{int}}$ in the sense that there exists $z_{0}\in \R^{3}$ such that a suitable $\lebe^{1}(\Sigma)$-representative $\mathbf{G}\colon\Sigma\to\R^{3}$ of $(\FF\times\nu_{\partial\Omega})_{\partial\Omega}^{\mathrm{int}}$ satisfies 
\begin{align}\label{eq:Lebesgue1}
\lim_{\substack{x\to x_{0} \\ x\in\Sigma}}|\mathbf{G}(x)-z_{0}|\to 0. 
\end{align}
In addition, if $x_{0}$ is a weak $\mathscr{H}^{1}$-Lebesgue point 
of $\nu_{\Gamma_{\Sigma}^{t}}$ in the sense that 
\begin{align}\label{eq:LebesgueNormal}
\nu_{\Gamma_{\Sigma}^{t}}(x_{0})=\lim_{r\searrow 0}\dashint_{\ball_{r}(x_{0})\cap\Gamma_{\Sigma}^{t}}\nu_{\Gamma_{\Sigma}^{t}}\dif\mathscr{H}^{1}, 
\end{align}
then 
\begin{align}\label{eq:densityrep1}
\frac{\dif\mu_{\Sigma,t}^{\tau,\mathrm{int}}}{\dif\mathscr{H}^{1}}(x_{0})= - z_{0}\cdot\nu_{\Gamma_{\Sigma}^{t}}(x_{0}). 
\end{align}
The same holds true with the obvious modifications for $\mu_{\Sigma,t}^{\tau,\exterior}$ with $t\in\mathscr{I}_{2}$. 
\end{enumerate}
\end{theorem}

\begin{proof}
Throughout, we focus on the interior Stokes functionals $\eqref{eq:Stokesfunctionaldef}_{1}$, 
the results for the exterior Stokes functionals follow by analogous means. 

We start by rewriting $\eqref{eq:Stokesfunctionaldef}_{1}$, and fix $t$ and $\delta$ for the time being. To this end, we recall from Remark \ref{rem:welldefnessstokes} that the limits in \eqref{eq:curllocaliser} are independent of the specific Lipschitz extension $\overline{\psi}=\overline{\psi}_{t,\delta,\Sigma}$ of $\psi_{t,\delta,\Sigma}$. In particular, we may specify to the Lipschitz extension provided by Lemma \ref{lem:GoodLip}\ref{item:Lipextend2}. The latter implies the existence of some $\overline{\psi}\in\mathrm{Lip}_{\rm c}(\Omega)$ such that $\overline{\psi}|_{\partial\Omega'}=\psi_{t,\delta,\Sigma}$ on $\partial\Omega'$, $\overline{\psi}|_{\Omega'}\in\hold_{b}^{1}(\Omega')$, $\overline{\psi}|_{\Omega\setminus\overline{\Omega'}}\in\hold_{b}^{1}(\Omega\setminus\overline{\Omega'})$, 
and $\nabla_{\tau}(\rho_{\varepsilon}*\overline{\psi})\to\nabla_{\tau}\psi_{t,\delta,\Sigma}$ 
$\mathscr{H}^{2}$-a.e. on $\partial\Omega'$, where $\rho_{\varepsilon}$ is the $\varepsilon$-rescaled version of a standard mollifier on $\R^{3}$. Given $\varphi\in\hold_{\rm c}^{1}(\Omega)$, 
these properties clearly hold with the obvious modifications for $\zeta:=\overline{\psi}\varphi$. 

Since $\Omega'\Subset\Omega$ is open and bounded with $\hold^{1}$-boundary and $\curl\FF\in\mathscr{DM}^{\mathrm{ext}}(\Omega)$, all  assumptions of Lemma \ref{lem:easyprod} and Corollary \ref{cor:LipAdmit} are fulfilled. Recalling that $\mathrm{div}(\curl\FF)=0$ in $\mathscr{D}'(\Omega)$ and using the product rule for the tangential gradients, we arrive at 
\begin{align}\label{eq:goslingdrive1}
\begin{split}
\langle(\curl\FF)\cdot\nu, \zeta\rangle_{\partial\Omega'} & \stackrel{\eqref{eq:easynormaltraceextdivmeas1}}{=} - \int_{\Omega'}\nabla\zeta\cdot\dif\,(\curl\FF)\\ &  \stackrel{\eqref{eq:criticalId}}{=} - \int_{\partial\Omega'}(\FF\times\nu_{\partial\Omega'})_{\partial\Omega'}^{\mathrm{int}}\cdot\nabla_{\tau}\zeta\dif\mathscr{H}^{2} \\ & \;\;\, = -\int_{\partial\Omega'}(\FF\times\nu_{\partial\Omega'})_{\partial\Omega'}^{\mathrm{int}}\cdot \psi_{t,\delta,\Sigma}\nabla_{\tau}\varphi \dif \mathscr{H}^{2} \\ & \;\;\;\;\;\;\, -\int_{\partial\Omega'}(\FF\times\nu_{\partial\Omega'})_{\partial\Omega'}^{\mathrm{int}}\cdot \varphi\nabla_{\tau}\psi_{t,\delta,\Sigma}\dif \mathscr{H}^{2}. 
\end{split}
\end{align}
We recall from Theorem \ref{thm:tracemain1} that $(\FF\times\nu_{\partial\Omega'})_{\partial\Omega'}^{\mathrm{int}}\in\lebe^{\infty}(\partial\Omega';T_{\partial\Omega'})$. By relative openness of $\Sigma^{\tau,t}$ in $\partial\Omega'$, we see that $\psi_{t,\delta,\Sigma}\to 1$ everywhere in $\Sigma^{\tau,t}$ as $\delta\searrow 0$. We thus find that 
\begin{align}\label{eq:goslingdrive2}
-\int_{\partial\Omega'}(\FF\times\nu_{\partial\Omega'})_{\partial\Omega'}^{\mathrm{int}}\cdot \psi_{t,\delta,\Sigma}\nabla_{\tau}\varphi \dif \mathscr{H}^{2} \to -\int_{\partial\Omega'}(\FF\times\nu_{\partial\Omega'})_{\partial\Omega'}^{\mathrm{int}}\cdot \nabla_{\tau}\varphi \dif \mathscr{H}^{2}
\end{align}
by dominated convergence. In view of definition $\eqref{eq:Stokesfunctionaldef}_{1}$ 
of the interior Stokes functionals, we therefore obtain the following alternative representation: 
\begin{align}\label{eq:alternativemain}
\mathfrak{S}_{\Sigma^{\tau,t}}^{\mathrm{int}}(\varphi) = - \lim_{\delta\searrow 0}\int_{\partial\Omega'}(\FF\times\nu_{\partial\Omega'})_{\partial\Omega'}^{\mathrm{int}}\cdot \varphi\nabla_{\tau}\psi_{t,\delta,\Sigma}\dif \mathscr{H}^{2}.
\end{align}
Based on \eqref{eq:alternativemain}, we now embark on the actual proof. 

\smallskip
For \ref{item:elementary1}, by  \eqref{eq:alternativemain}, we conclude that the well-definedness of the interior Stokes functional is equivalent to the existence and finiteness of the limit on the right-hand side of \eqref{eq:alternativemain}. By Theorem \ref{thm:tracemain1}, $\mathbf{v}:=(\FF\times\nu_{\partial\Omega'})_{\partial\Omega'}^{\mathrm{int}}\in\lebe^{\infty}(\partial\Omega',T_{\partial\Omega'})$. With this choice of $\mathbf{v}$, Proposition \ref{prop:diff} yields a set $\mathscr{I}\subset I=[0,\frac{1}{2})$ with $\mathscr{L}^{1}(I\setminus\mathscr{I})=0$  such that the limit on the right-hand side of \eqref{eq:alternativemain} exists for all $t\in\mathscr{I}$ \emph{and all} $\varphi\in\hold_{\rm c}^{1}(\Omega)$. In  consequence, for $t\in\mathscr{I}$, 
\begin{align*}
\mathfrak{S}_{\Sigma^{\tau,t}}^{\mathrm{int}} \stackrel{\eqref{eq:alternativemain}}{=} -\ell_{t}^{\mathbf{v}}\colon\hold_{\rm c}^{1}(\Omega)\ni \varphi \mapsto -\ell_{t}^{\mathbf{v}}(\varphi)\in\R\qquad\text{with $\ell_{t}^{\mathbf{v}}(\varphi)$ as in \eqref{eq:lim1}}\end{align*}
is well-defined and linear. Moreover, \eqref{eq:lim2} implies that there exists a constant $c>0$ independent of $\FF$ and $c(t)>0$ such that 
\begin{align}\label{eq:goose}
|\mathfrak{S}_{\Sigma^{\tau,t}}^{\mathrm{int}}(\varphi)| \leq c\,(\mathcal{M}_{\Sigma,\partial\Omega'}^{\Psi,+}\mathbf{v}(t))\|\varphi\|_{\hold(\overline{\Sigma})}\leq c(t)\,\|\varphi\|_{\hold(\overline{\Sigma})}\qquad\text{for all}\;\varphi\in\hold_{\rm c}^{1}(\Omega). 
\end{align}
Since $\hold_{\rm c}^{1}(\Omega)$ is dense in $\hold_{0}(\Omega)$ with respect to $\|\cdot\|_{\hold(\Omega)}$, 
estimate \eqref{eq:goose} entails that $\mathfrak{S}_{\Sigma^{\tau,t}}^{\mathrm{int}}$ 
extends to a (non-relabeled) bounded linear functional ${{\mathfrak{S}}}{_{\Sigma^{\tau,t}}^{\mathrm{int}}}\colon\hold_{0}(\Omega)\to\R$, which inherits bound \eqref{eq:goose}. 
Hence, the Riesz representation theorem provides us with a measure 
$\mu_{\Sigma,t}^{\tau,\mathrm{int}}\in\mathrm{RM}_{\mathrm{fin}}(\Omega)$ such that 
\begin{align}\label{eq:Stokesrepext}
{\mathfrak{S}}{_{\Sigma^{\tau,t}}^{\mathrm{int}}}(\varphi)=\int_{\Omega}\varphi\dif\mu_{\Sigma,t}^{\tau,\mathrm{int}}\qquad\,\,\text{for all $\;\varphi\in\hold_{0}(\Omega)$}, 
\end{align}
which is \eqref{eq:representasmeasure}. 

Next, we observe that, if $\varphi\in\hold_{\rm c}^{1}(\Omega)$ is such that $\spt(\varphi)\cap \Gamma_{\Sigma}^{t}=\emptyset$, the closedness of $\Gamma_{\Sigma}^{t}$ implies that $d:=\mathrm{dist}(\spt(\varphi),\Gamma_{\Sigma}^{t})>0$. By the support properties of the interior height functions $\psi_{t,\delta,\Sigma}$ (see \eqref{eq:psiddefmain}), 
it is clear that there exists $\delta_{0}=\delta_{0}(\varphi)>0$ such that $0<\delta<\delta_{0}$ implies that the integral on the right-hand side of \eqref{eq:alternativemain} vanishes. Hence, necessarily, $\spt(\mu_{\Sigma,t}^{\tau,\mathrm{int}})\subset \Gamma_{\Sigma}^{t}$. Summarising, we have established \ref{item:elementary1}. 

\smallskip
For \ref{item:StokesAbsCon}, let $t\in\mathscr{I}$, and let $\varepsilon>0$ be arbitrary. Moreover, let $A\subset\Gamma_{\Sigma}^{t}$ be a $\mathscr{H}^{1}$-measurable set with $\mathscr{H}^{1}(A)<\varepsilon$. 
Then there exists $\delta_{0}>0$ such that, for every $0<\delta'<\delta_{0}$, 
there exists a countable collection $(\ball_{j})=(\ball_{r_{j}}(x_{j}))$ of open balls such that 
\begin{align}\label{eq:radiichoose}
2r_{j}<\delta'\;\;\text{for all}\;j\in\mathbb{N},\qquad A\subset\bigcup_{j=1}^{\infty}\ball_{j},
\qquad\sum_{j=1}^{\infty}r_{j}<2\varepsilon.
\end{align}
Moreover, without loss of generality, we assume that each $\ball_{j}$ is centered at an element of $A$ and that, diminishing $\delta_{0}>0$ if necessary, $\ball_{j}\Subset\Omega$ holds for all $j\in\mathbb{N}$ and $\delta_{0}<2r'$ with the threshold number $r'$ from Lemma \ref{lem:geometrichelp}. In particular, $0<r_{j}<r'$ for all $j\in\mathbb{N}$.

Because $\mu_{\Sigma,t}^{\tau,\mathrm{int}}\in\mathrm{RM}_{\mathrm{fin}}(\Omega)$ and each $\ball_{j}$ is open, for each $j\in\mathbb{N}$, 
there exists $\varphi_{j}\in\hold_{\rm c}^{1}(\ball_{j})$ such that $\|\varphi_{j}\|_{\lebe^{\infty}(\ball_{j})}\leq 1$ and  
\begin{align}\label{eq:StokesCloseness}
|\mu_{\Sigma,t}^{\tau,\mathrm{int}}|(\ball_{j}) 
\leq \bigg\vert \int_{\ball_{j}}\varphi_{j}\dif\mu_{\Sigma,t}^{\tau,\mathrm{int}}\bigg\vert + \frac{\varepsilon}{2^{j}}. 
\end{align}
Now, since each $\ball_{j}$ is open, $\spt(\mu_{\Sigma,t}^{\tau,\mathrm{int}})\subset\Gamma_{\Sigma}^{t}$, 
and $\spt(\varphi_{j})\subset\ball_{j}$ for every $j\in\mathbb{N}$, we have 
\begin{align}\label{eq:AbsConEst}
\begin{split}
|\mu_{\Sigma,t}^{\tau,\mathrm{int}}|(A) & \stackrel{\eqref{eq:radiichoose}}{\leq}  \sum_{j=1}^{\infty}|\mu_{\Sigma,t}^{\tau,\mathrm{int}}|(\ball_{j}) \\ 
& \stackrel{\eqref{eq:StokesCloseness}}{\leq} \sum_{j=1}^{\infty}\bigg(\bigg\vert \int_{\ball_{j}}\varphi_{j}\dif\mu_{\Sigma,t}^{\tau,\mathrm{int}}\bigg\vert + \frac{\varepsilon}{2^{j}}\bigg)  
= \sum_{j=1}^{\infty} \bigg(\bigg\vert \int_{\Gamma_{\Sigma}^{t}}\varphi_{j}\dif\mu_{\Sigma,t}^{\tau,\mathrm{int}}\bigg\vert + \frac{\varepsilon}{2^{j}}\bigg).
\end{split}
\end{align}
Let $j\in\mathbb{N}$ be arbitrary. We recall that, by \eqref{eq:Lipbound1}, there exists $c=c(\Sigma,\Psi_{\Sigma})>0$ such that $|\nabla_{\tau}\psi_{t,\delta,\Sigma}|\leq \frac{c}{\delta}$ holds for all sufficiently small $\delta>0$. Moreover, since $\Psi_{\Sigma}\colon (-1,1)\times\Gamma_{\Sigma}\to\mathcal{O}\subset\partial\Omega'$ is bi-Lipschitz, there exists a constant $c=c(\Sigma,\Psi_{\Sigma})>1$ such that 
\begin{align}\label{eq:kegelclub}
\frac{\delta}{c} \leq \mathscr{H}^{2}(\Psi_{\Sigma}((t,t+\delta)\times\Gamma_{\Sigma}))\leq c\delta\qquad\text{for all sufficiently small $\delta>0$}. 
\end{align}
By definition of the measure $\mu_{\Sigma,t}^{\tau,\mathrm{int}}$, we may employ Lemma \ref{lem:geometrichelp} to obtain
\begin{align*}
\bigg\vert \int_{\Gamma_{\Sigma}^{t}}\varphi_{j}\dif\mu_{\Sigma,t}^{\tau,\mathrm{int}}\bigg\vert & \stackrel{\eqref{eq:alternativemain}}{=} \lim_{\delta\searrow 0} \left\vert \int_{\Sigma^{\tau,t}}\varphi_{j}(\FF\times\nu_{\partial\Omega'})_{\partial\Omega'}^{\mathrm{int}}\cdot\nabla_{\tau}\psi_{t,\delta,\Sigma}\dif\mathscr{H}^{2}  \right\vert \\ 
& \;\;\leq c\,\|(\FF\times\nu_{\partial\Omega'})_{\partial\Omega'}^{\mathrm{int}}\|_{\lebe^{\infty}(\Sigma)}\limsup_{\delta\searrow 0}\Big(\frac{1}{\delta}\int_{\Psi_{\Sigma}((t,t+\delta)\times\Gamma_{\Sigma})}|\varphi_{j}|\dif\mathscr{H}^{2}\Big) \\ 
& \!\!\!\!\!\!\!\!\!\!\!\!\!\!\stackrel{\eqref{eq:kegelclub},\,\spt(\varphi_{j})\subset\ball_{j}}{\leq}  c\,\|(\FF\times\nu_{\partial\Omega'})_{\partial\Omega'}^{\mathrm{int}}\|_{\lebe^{\infty}(\Sigma)}\limsup_{\delta\searrow 0}\frac{\mathscr{H}^{2}(\Psi_{\Sigma}((t,t+\delta)\times\Gamma_{\Sigma})\cap\ball_{j})}{\mathscr{H}^{2}(\Psi_{\Sigma}((t,t+\delta)\times\Gamma_{\Sigma}))} \\ 
& \!\!\!\!\!\!\!\!\!\!\!\!\stackrel{\text{Lem. \ref{lem:geometrichelp},\,$r_{j}<r'$}}{\leq} c\,\|(\FF\times\nu_{\partial\Omega'})_{\partial\Omega'}^{\mathrm{int}}\|_{\lebe^{\infty}(\Sigma)}\,r_{j} =: C\,r_{j}, 
\end{align*}
where $C>0$ clearly is independent of $j\in\mathbb{N}$. In conjunction with \eqref{eq:AbsConEst}, we thus obtain 
\begin{align}\label{eq:finalAbsCon}
|\mu_{\Sigma,t}^{\tau,\mathrm{int}}|(A) & \leq \sum_{j=1}^{\infty}\Big(Cr_{j}+\frac{\varepsilon}{2^{j}}\Big) \stackrel{\eqref{eq:radiichoose}}{\leq} c\,\varepsilon. 
\end{align}
Sending $\varepsilon\searrow 0$, we obtain that $|\mu_{\Sigma,t}^{\tau,\mathrm{int}}|(A)=0$ whenever $\mathscr{H}^{1}(A)=0$. Both $\mu_{\Sigma,t}^{\tau,\mathrm{int}}$ and $\mathscr{H}^{1}\mres\Gamma_{\Sigma}^{t}$ are finite measures on $\Gamma_{\Sigma}^{t}$, 
so that $\mu_{\Sigma,t}^{\tau,\mathrm{int}}\ll\mathscr{H}^{1}\mres\Gamma_{\Sigma}^{t}$ yields  \eqref{eq:densityGammaSigma} by use of the Radon-Nikod\'{y}m theorem. Hence,  \ref{item:StokesAbsCon} follows. 

\smallskip
For \ref{item:StokesTangential}, by what we have established above, 
density \eqref{eq:densityGammaSigma} exists at $\mathscr{H}^{1}$-a.e. $x_{0}\in\Gamma_{\Sigma}^{t}$. 
For an arbitrary $\varepsilon>0$, \eqref{eq:Lebesgue1} implies that there exists $r_{0}>0$ such that 
\begin{align}\label{eq:makessense}
|\mathbf{G}(x)-z_{0}|<\varepsilon\qquad\text{for all}\;x\in\Sigma\;\text{with}\;|x-x_{0}|<r_{0}. 
\end{align}
We proceed to compute the density \eqref{eq:densityGammaSigma}. To this end, let $r>0$ be sufficiently small. Then we choose a sequence $(\widetilde{\varphi}_{j})\subset\hold_{\rm c}^{1}(\ball_{r}(x_{0});[0,1])$ such that $|\widetilde{\varphi}_{j}|\leq 1$ and $\widetilde{\varphi}_{j}\to 1$ everywhere in $\ball_{r}(x_{0})$. 
In particular, because $\spt(\mu_{\Sigma,t}^{\tau,\interior})\subset\Gamma_{\Sigma}^{t}$,  
there exists $j(r)\in\mathbb{N}$ such that 
\begin{align}\label{eq:sowhat?}
\begin{split}
& \bigg\vert\mu_{\Sigma,t}^{\tau,\mathrm{int}}(\ball_{r}(x_{0})\cap\Gamma_{\Sigma}^{t}) - \int_{\ball_{r}(x_{0})}\widetilde{\varphi}_{j(r)}\dif\mu_{\Sigma,t}^{\tau,\interior}\bigg\vert
< r^{2}, \\[1mm]
&  
\bigg\vert \int_{\Gamma_{\Sigma}^{t}\cap\ball_{r}(x_{0})}\nu_{\Gamma_{\Sigma}^{t}}\dif\mathscr{H}^{1} - \int_{\Gamma_{\Sigma}^{t}}\widetilde{\varphi}_{j(r)}\nu_{\Gamma_{\Sigma}^{t}}\dif\mathscr{H}^{1}\bigg\vert < r^{2}.
\end{split}
\end{align}
This is possible by the dominated convergence theorem; note that $\nu_{\Gamma_{\Sigma}^{t}}\in\lebe^{\infty}(\Gamma_{\Sigma}^{t};\R^{3})$. We then define $\varphi_{r}:=\widetilde{\varphi}_{j(r)}$. By \eqref{eq:sowhat?}, for each sufficiently small $r>0$, there exist $\theta_{r}\in\R$ and $\Theta_{r}\in\R^{3}$ such that 
\begin{align}\label{eq:bricko}
\begin{split}
& \mu_{\Sigma,t}^{\tau,\mathrm{int}}(\Gamma_{\Sigma}^{t}\cap\ball_{r}(x_{0})) = \int_{\ball_{r}(x_{0})}\varphi_{r}\dif\mu_{\Sigma,t}^{\tau,\interior} + \theta_{r}, \\ 
& \int_{\Gamma_{\Sigma}^{t}\cap\ball_{r}(x_{0})}\nu_{\Gamma_{\Sigma}^{t}}\dif\mathscr{H}^{1} = \int_{\Gamma_{\Sigma}^{t}}\varphi_{r}\nu_{\Gamma_{\Sigma}^{t}}\dif\mathscr{H}^{1} + \Theta_{r},\\ & \limsup_{r\searrow 0}\frac{|\theta_{r}|}{r^{2}}\leq 1, \quad\,\,\;\limsup_{r\searrow 0}\frac{|\Theta_{r}|}{r^{2}}\leq 1. 
\end{split}
\end{align}
We compute 
\begin{align*}
\frac{\dif\mu_{\Sigma,t}^{\tau,\mathrm{int}}}{\dif\mathscr{H}^{1}}(x_{0}) & = \lim_{r\searrow 0}\frac{\mu_{\Sigma,t}^{\tau,\mathrm{int}}(\ball_{r}(x_{0})\cap\Gamma_{\Sigma}^{t})}{\mathscr{H}^{1}(\ball_{r}(x_{0})\cap\Gamma_{\Sigma}^{t})} \\ 
& \!\!\!\!\!\!\!\!\!\!\!\!\!\!\!\! \stackrel{\eqref{eq:bricko}_{1}}{=} \lim_{r\searrow 0} \frac{1}{\mathscr{H}^{1}(\ball_{r}(x_{0})\cap\Gamma_{\Sigma}^{t})}\int_{\ball_{r}(x_{0})}\varphi_{r}\dif\mu_{\Sigma,t}^{\tau,\interior} + \lim_{r\searrow 0}\frac{\theta_{r}}{\mathscr{H}^{1}(\ball_{r}(x_{0})\cap\Gamma_{\Sigma}^{t})} \\ 
& \!\!\!\!\!\!\!\!\!\!\!\!\!\!\!\!\!\!\!\!\!\!\stackrel{\eqref{eq:alternativemain},\,\eqref{eq:Stokesrepext}}{=} -\lim_{r\searrow 0}\frac{1}{\mathscr{H}^{1}(\ball_{r}(x_{0})\cap\Gamma_{\Sigma}^{t})}\Big(\lim_{\delta\searrow 0}\int_{\Sigma^{\tau,t}}\varphi_{r}\big((\FF\times\nu_{\partial\Omega'})_{\partial\Omega'}^{\mathrm{int}}-z_{0}\big)\cdot\nabla_{\tau}\psi_{t,\delta,\Sigma}\dif\mathscr{H}^{2}\Big)\\ 
& \!\!\!\!\!\!\!\!\!\!\! - \lim_{r\searrow 0}\frac{1}{\mathscr{H}^{1}(\ball_{r}(x_{0})\cap\Gamma_{\Sigma}^{t})}\Big(\lim_{\delta\searrow 0}\int_{\Sigma^{\tau,t}}\varphi_{r}\nabla_{\tau}\psi_{t,\delta,\Sigma}\dif\mathscr{H}^{2}\Big)\cdot z_{0} \\ 
& \!\!\!\!\!\!\!\!\!\!\! +  \lim_{r\searrow 0}\frac{\theta_{r}}{\mathscr{H}^{1}(\ball_{r}(x_{0})\cap\Gamma_{\Sigma}^{t})} =: \mathrm{I}+\mathrm{II} + \mathrm{III}. 
\end{align*}
To estimate these terms separately, we note that $\Gamma_{\Sigma}^{t}$ is also 
the boundary of a Lipschitz boundary manifold relative to $\Omega'$. 
Hence, there exists a constant $c>1$ such that 
\begin{align}\label{eq:tkkg1}
\frac{r}{c} \leq \mathscr{H}^{1}(\ball_{r}(x_{0})\cap\Gamma_{\Sigma}^{t})\leq cr\qquad\text{for all sufficiently small}\;r>0. 
\end{align}

\smallskip
For $\mathrm{I}$, let $\varepsilon>0$ be arbitrary. Because of the Lipschitz bound \eqref{eq:Lipbound1},  \eqref{eq:makessense}, and \eqref{eq:tkkg1}, 
\begin{align}\label{eq:nightdrive}
\begin{split}
|\mathrm{I}|  &\stackrel{\eqref{eq:Lipbound1},\eqref{eq:tkkg1}}{\leq}  \limsup_{r\searrow 0}\Big(\limsup_{\delta\searrow 0}
\frac{c}{r\delta}\int_{\ball_{r}(x_{0})\cap\Psi_{\Sigma}((t,t+\delta)\times\Gamma_{\Sigma})}|\mathbf{G}-z_{0}|\dif\mathscr{H}^{2}\Big) \\ & \;\;\; \stackrel{\eqref{eq:makessense}}{<} \varepsilon\limsup_{r\searrow 0}\Big(\limsup_{\delta\searrow 0}\frac{c}{r\delta}\mathscr{H}^{2}(\ball_{r}(x_{0})\cap\Psi_{\Sigma}((t,t+\delta)\times\Gamma_{\Sigma}))\Big) \\ 
& \;\;\; \stackrel{\eqref{eq:kegelclub}}{\leq} c\,\varepsilon\limsup_{r\searrow 0}\frac{1}{r}\Big(\limsup_{\delta\searrow 0} \frac{\mathscr{H}^{2}(\ball_{r}(x_{0})\cap\Psi_{\Sigma}((t,t+\delta)\times\Gamma_{\Sigma}))}{\mathscr{H}^{2}(\Psi_{\Sigma}((t,t+\delta)\times\Gamma_{\Sigma}))}\Big) \\ 
& \;\;\; \stackrel{\eqref{eq:technique}}{\leq} c\,\varepsilon, 
\end{split}
\end{align}
where $c>0$ is clearly independent of $\varepsilon$. We may therefore send $\varepsilon\searrow 0$ to conclude that $\mathrm{I}=0$. 

\smallskip
For $\mathrm{II}$, an analogous argument as in the proof of Lemma \ref{lem:bdryintegralcont}, based on the coarea formula on manifolds, yields that
\begin{align*}
\lim_{\delta\searrow 0}\int_{\Sigma^{\tau,t}}\varphi_{r}\nabla_{\tau}\psi_{t,\delta,\Sigma}\dif\mathscr{H}^{2} = \int_{\Gamma_{\Sigma}^{t}}\varphi_{r}\nu_{\Gamma_{\Sigma}^{t}}\dif\mathscr{H}^{1}.
\end{align*}
Hence, we arrive at 
\begin{align}\label{eq:pressure}
\begin{split}
\mathrm{II} & =  - \Big(\lim_{r\searrow 0} \dashint_{\Gamma_{\Sigma}^{t}\cap\ball_{r}(x_{0})}\nu_{\Gamma_{\Sigma}^{t}}\dif\mathscr{H}^{1}\Big)\cdot z_{0}  +\Big(\underbrace{\lim_{r\searrow 0}\frac{\Theta_{r}}{\mathscr{H}^{1}(\ball_{r}(x_{0})\cap\Gamma_{\Sigma}^{t})}}_{(*)}\Big)\cdot z_{0} \\ & \!\!\!\!\stackrel{\eqref{eq:LebesgueNormal}}{=}  - \nu_{\Gamma_{\Sigma}^{t}}(x_{0})\cdot z_{0}.
\end{split}
\end{align}
Here, $(*)=0$ can be seen as a consequence of \eqref{eq:tkkg1} and the second part of $\eqref{eq:bricko}_{3}$.

\smallskip
For $\mathrm{III}$, the bounds from \eqref{eq:tkkg1} in conjunction with the first part of  $\eqref{eq:bricko}_{3}$ directly yield that $\mathrm{III}=0$. 

Gathering the above identities, it follows that  
\begin{align*}
\frac{\dif\mu_{\Sigma,t}^{\tau,\mathrm{int}}}{\dif\mathscr{H}^{1}}(x_{0}) = \mathrm{II} \stackrel{\eqref{eq:pressure}}{=}  - \nu_{\Gamma_{\Sigma}^{t}}(x_{0})\cdot z_{0}. 
\end{align*}
This is \eqref{eq:densityrep1}, and the proof is complete. 
\end{proof}

\begin{rem}[On Theorem \ref{thm:stokes}]\label{rem:stokesdiscus}
The existence (or well-definedness) of the Stokes functionals 
$\mathfrak{S}_{\Sigma^{\tau,t}}^{\mathrm{int}}$ and $\mathfrak{S}_{\Sigma^{\tau,t}}^{\mathrm{int}}$ 
depends strongly on $\FF$, and hereafter $t\in I$. Moreover, Theorem \ref{thm:stokes}\ref{item:StokesTangential} underlines in which (weak) sense the density of $\mu_{\Sigma,t}^{\tau,\interior}$ can be understood to be tangential to $\Gamma_{\Sigma}^{t}$ inside $\Sigma$. To explain this, assume that we may write $(\FF\times\nu_{\partial\Omega'})_{\partial\Omega'}^{\mathrm{int}}=\mathbf{H}\times\nu_{\partial\Omega'}$ with some $\mathbf{H}\in\lebe^{\infty}(\partial\Omega';\R^{3})$ such that, for simplicity, $x_{0}\in\Gamma_{\Sigma}^{t}$ is a continuity point for $\mathbf{H}$. In the situation of Theorem \ref{thm:stokes}\ref{item:StokesTangential}, \eqref{eq:densityrep1} gives us 
\begin{align}\label{eq:densitycompute1}
\begin{split}
\frac{\dif\mu_{\Sigma,t}^{\tau,\mathrm{int}}}{\dif\mathscr{H}^{1}}(x_{0}) & = - (\mathbf{H}(x_{0})\times \nu_{\partial\Omega'}(x_{0}))\cdot \nu_{\Gamma_{\Sigma}^{t}}(x_{0}) \\ &  =  - \mathbf{H}(x_{0})\cdot(\nu_{\partial\Omega'}(x_{0})\times\nu_{\Gamma_{\Sigma}^{t}}(x_{0})) = - \mathbf{H}(x_{0})\cdot\tau_{\Gamma_{\Sigma}^{t}}(x_{0}), 
\end{split}
\end{align}
where the tangential field $\tau_{\Gamma_{\Sigma}^{t}}\colon\Gamma_{\Sigma}^{t}\to\mathbb{S}^{2}$ carries the orientation as depicted in Figure \ref{fig:conventionsorientations}. 

We 
point out that it is not clear to us how to prove \eqref{eq:densityrep1} 
if we only assume $x_{0}$ to be a strong Lebesgue point of $(\FF\times\nu_{\partial\Omega'})_{\partial\Omega'}^{\mathrm{int}}$ with respect to $\mathscr{H}^{2}$, meaning that 
\begin{align}\label{eq:strongconj}
\lim_{r\searrow 0} \dashint_{\ball_{r}(x_{0})\cap\partial\Omega'}|(\FF\times\nu_{\partial\Omega'})_{\partial\Omega'}^{\mathrm{int}}-z_{0}|\dif\mathscr{H}^{2}=0. 
\end{align}
The key obstruction in proving \eqref{eq:densityrep1} is the order of limits in \eqref{eq:nightdrive}, and the low regularity of $(\FF\times\nu_{\partial\Omega'})_{\partial\Omega'}^{\mathrm{int}}$ 
does not seem to give us access to \eqref{eq:strongconj}, 
\emph{e.g.} by use of the coarea formula. 
\end{rem}

Based on Theorem \ref{thm:stokes}, we now give a version of the Stokes in the classical vorticity flux-circulation version. 

\begin{corollary}[Stokes in the Vorticity Flux Formulation]\label{cor:stokesvortflux} 
In the situation of {\rm Theorem \ref{thm:stokes}}, let \emph{(i)} $t\in\mathscr{I}_{1}$ or \emph{(ii)} $t\in\mathscr{I}_{2}$. We define the  \emph{total interior and exterior vorticity fluxes} through $\Sigma^{\tau,t}$ by 
\begin{align}\label{eq:vortflux}
\begin{split}
&\Big[\overline{\overline{(\curl\FF)\cdot\nu_{\Sigma^{\tau,t}}}}\Big]{_{\Sigma^{\tau,t}}^{\mathrm{int}}} := \left\langle\overline{\overline{(\curl\FF)\cdot\nu_{\Sigma^{\tau,t}}}},\varphi\right\rangle_{\Sigma^{\tau,t},\Omega'}, \\
&\Big[\overline{\overline{(\curl\FF)\cdot\nu_{\Sigma^{\tau,t}}}}\Big]{_{\Sigma^{\tau,t}}^{\mathrm{ext}}} := \left\langle\overline{\overline{(\curl\FF)\cdot\nu_{\Sigma^{\tau,t}}}},\varphi\right\rangle_{\Sigma^{\tau,t},\Omega\setminus\overline{\Omega'}}, 
\end{split}
\end{align}
where, in each case, $\varphi\in\hold_{\rm c}^{1}(\Omega)$ is such that $\varphi=1$ in an open neighborhood of $\Sigma^{\tau,t}$. These definitions are independent of the specific choice of $\varphi$ with this property,
and we have the \emph{Stokes theorem in the vorticity flux formulation}:  
\begin{align*}
&\Big[\overline{\overline{(\curl\FF)\cdot\nu_{\Sigma^{\tau,t}}}}\Big]{_{\Sigma^{\tau,t}}^{\mathrm{int}}} = \mu_{\Sigma,t}^{\tau,\mathrm{int}}(\Gamma_{\Sigma}^{t}) = \int_{\Gamma_{\Sigma}^{t}}\frac{\dif\mu_{\Sigma,t}^{\tau,\mathrm{int}}}{\dif\mathscr{H}^{1}}\dif\mathscr{H}^{1}\qquad\;\text{if}\;t\in\mathscr{I}_{1}, \\ & \Big[\overline{\overline{(\curl\FF)\cdot\nu_{\Sigma^{\tau,t}}}}\Big]{_{\Sigma^{\tau,t}}^{\mathrm{ext}}} = \mu_{\Sigma,t}^{\tau,\exterior}(\Gamma_{\Sigma}^{t}) = \int_{\Gamma_{\Sigma}^{t}}\frac{\dif\mu_{\Sigma,t}^{\tau,\mathrm{ext}}}{\dif\mathscr{H}^{1}}\dif\mathscr{H}^{1}\qquad\;\text{if}\;t\in\mathscr{I}_{2}. 
\end{align*}
\end{corollary}
\begin{proof}
Based on \eqref{eq:alternativemain}, it is clear that the quantities from \eqref{eq:vortflux} are independent of the specific function $\varphi\in\hold_{\rm c}^{1}(\Omega)$ as long as $\varphi=1$ in an open neighborhood $U$ of $\Sigma^{\tau,t}$. Since the supports of $\mu_{\Sigma,t}^{\tau,\interior},\mu_{\Sigma,t}^{\tau,\interior}$ are contained in $\Gamma_{\Sigma}^{t}\subset U$ in each of the cases, the remaining assertions then are immediate from Theorem \ref{thm:stokes}; see \eqref{eq:representasmeasure}. The proof is complete.
\end{proof}
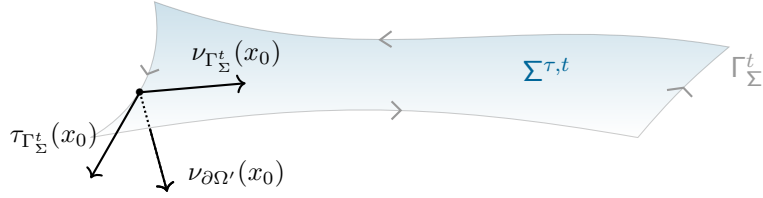
\begin{figure}
\begin{tikzpicture}[scale=1.2]
\draw[top color=green!40!blue,opacity=.2] (-3,0) [out=10, in =170] to (3,0) [out= 50, in = 220] to (4,1) [out= 170, in = -20] to (-2.3,1.5) [out= 280, in =30] to (-3,0);
\draw[->,thick] (-2.46,0.5) -- (-3,-0.45);
\draw[->,thick] (-2.46,0.5) -- (-1.3,0.6);
\draw[->,thick,densely dotted] (-2.46,0.5) -- (-2.16,-0.6);
\draw[->,thick] (-2.35,0.1) -- (-2.16,-0.6);
\node at (-3.4,0) {$\tau_{\Gamma_{\Sigma}^{t}}(x_{0})$};
\node at (-1.4,-0.4) {$\nu_{\partial\Omega'}(x_{0})$};
\node at (-1.4,0.9) {$\nu_{\Gamma_{\Sigma}^{t}}(x_{0})$};
\node[-,green!40!blue] at (2,0.75) {\large $\mathsf{\Sigma}^{\tau,t}$};
\node[-,black!40!white] at (4.2,0.75) {\large $\mathsf{\Gamma}_{\mathsf{\Sigma}}^{t}$};
\draw[-,thick,black!40!white] (-2.45,0.8) -- (-2.38,0.7) -- (-2.25,0.75);
\draw[-,thick,black!40!white] (0.3,0.4) -- (0.4,0.3) -- (0.3,0.2);
\draw[-,thick,black!40!white] (3.32,0.525) -- (3.5,0.55) -- (3.595,0.45);
\draw[-,thick,black!40!white] (0.3,1.175) -- (0.175,1.08) -- (0.3,1.0);
\node at (-2.46,0.5) {\tiny\textbullet};
\end{tikzpicture}
\caption{Figure \ref{fig:orientationprelims}, continued: On orientations in the Stokes theorem, continuing with the setting displayed in Figure \ref{fig:geometricmain}. }\label{fig:conventionsorientations}
\end{figure}

\subsection{Consistency and remarks on Definition \ref{def:stokes}}\label{sec:consistency}
In this intermediate subsection, we establish that Theorem \ref{thm:stokes} lets us retrieve the classical Stokes formula \eqref{eq:Stokes}. As a key point, this will also explain why Definition \ref{def:stokes} is natural. We begin with 
\begin{prop}[Consistency]\label{prop:consistency}
Retaining the geometric setting displayed in \S\ref{sec:smoothsetCMinfty}, 
let $\FF\in\hold^{1}(\Omega;\R^{3})$. Then Theorem \ref{thm:stokes} implies that 
\begin{align}\label{eq:stokes1A}
\int_{\Sigma}\curl(\varphi\FF)|_{\Sigma}\cdot\nu_{\partial\Omega'}\dif\mathscr{H}^{2} 
= - \int_{\Gamma_{\Sigma}}\varphi\FF\cdot\tau_{\Gamma_{\Sigma}}\dif\mathscr{H}^{1}
\qquad\text{for all}\;\varphi\in\hold_{\rm c}^{1}(\Omega). 
\end{align}
\end{prop}

\begin{proof} 
In the present situation, let  $\varphi\in\hold_{\rm c}^{1}(\Omega)$, and let $(\psi_{\delta})\subset\hold_{\rm c}^{1}(\partial\Omega;[0,1])$ 
be a sequence such that, for each $\delta$,  
$\psi_{\delta}=0$ on $\partial\Omega'\setminus\Sigma$ 
and $\psi_{\delta}=1$ on $\Sigma^{\tau,\delta}$. 
Since $\psi_{0,\delta,\Sigma}\in\sobo_{0}^{1,\infty}(\Sigma)$, 
we have $\psi_{\delta}-\psi_{0,\delta,\Sigma}\in\sobo_{0}^{1,\infty}(\Sigma\setminus\overline{\Sigma^{\tau,\delta}})$. Since $\FF(x)\times\nu_{\partial\Omega'}(x)\in T_{\partial\Omega'}(x)$ for any $x\in\partial\Omega'$, 
integrating by parts gives us 
\begin{align}\label{eq:helpotscherp}
\begin{split}
&\left\vert\int_{\Sigma\setminus\overline{\Sigma^{\tau,\delta}}}\varphi\nabla_{\tau} (\psi_{\delta}-\psi_{0,\delta,\Sigma})\cdot(\FF\times\nu_{\partial\Omega'})\dif\mathscr{H}^{2}\right\vert \\ &  = \left\vert \int_{\Sigma\setminus\overline{\Sigma^{\tau,\delta}}}\underbrace{(\psi_{\delta}-\psi_{0,\delta,\Sigma})}_{|\cdot|\leq 2}\underbrace{\mathrm{div}_{\tau}(\varphi(\FF\times\nu_{\partial\Omega'}))}_{|\cdot|\leq C'<\infty}\dif\mathscr{H}^{2}\right\vert \leq C\,\mathscr{H}^{2}(\Sigma\setminus\overline{\Sigma^{\tau,\delta}})\stackrel{\delta\searrow 0}{\longrightarrow} 0, 
\end{split}
\end{align}
where we have used that $\varphi\in\hold_{\rm c}^{1}(\Omega)$ and $\FF\in\hold^{1}(\Omega;\R^{3})$. 
The dominated convergence theorem and the classical product rule for   $\curl$ then yield
\begin{align*}
&\int_{\Sigma}\curl(\varphi\FF)|_{\Sigma} \cdot\nu_{\partial\Omega'}\dif\mathscr{H}^{2} \\ 
&= \lim_{\delta\searrow 0}
\int_{\partial\Omega'}\psi_{\delta}\curl(\varphi \FF)|_{\Sigma}\cdot\nu_{\partial\Omega'}\dif\mathscr{H}^{2} \\ 
&= \lim_{\delta\searrow 0} \Big(\underbrace{\int_{\partial\Omega'}\curl(\psi_{\delta}(\varphi\FF))|_{\Sigma}\cdot\nu_{\partial\Omega'}\dif\mathscr{H}^{2}}_{=0\;\text{by the divergence theorem}} -\int_{\partial\Omega'}(\nabla\psi_{\delta}\times(\varphi\FF))\cdot\nu_{\partial\Omega'}\dif\mathscr{H}^{2}\Big) \\ 
&= - \lim_{\delta\searrow 0}\int_{\partial\Omega'}\varphi\nabla\psi_{\delta}\cdot(\FF\times\nu_{\partial\Omega'})\dif\mathscr{H}^{2}\\
&= - \lim_{\delta\searrow 0}\int_{\partial\Omega'}\varphi\nabla_{\tau}\psi_{\delta}\cdot(\FF\times\nu_{\partial\Omega'})\dif\mathscr{H}^{2} \\ 
&\stackrel{\eqref{eq:helpotscherp}}{=} - \lim_{\delta\searrow 0}\int_{\partial\Omega'}\varphi\nabla_{\tau}\psi_{0,\delta,\Sigma}\cdot(\FF\times\nu_{\partial\Omega'})_{\partial\Omega'}^{\mathrm{int}}\dif\mathscr{H}^{2} \stackrel{\eqref{eq:alternativemain}}{=} \mathfrak{S}_{\Sigma^{\tau,0}}^{\mathrm{int}}(\varphi). 
\end{align*}
The ultimate limit exists as a consequence of Lemma \ref{lem:bdryintegralcont} and our regularity assumptions on $\FF$ and $\varphi$; in particular, in the present $\hold^{1}$-context, no tangential variations of $\Sigma$ are required. By Theorem \ref{thm:stokes} and Remark \ref{rem:stokesdiscus}(see \eqref{eq:densitycompute1}), we moreover see that 
\begin{align*}
\mathfrak{S}_{\Sigma^{\tau,0}}^{\mathrm{int}}(\varphi) & = \int_{\Gamma_{\Sigma}^{0}}\varphi\dif\mu_{t,\Sigma}^{\tau,\interior} = - \int_{\Gamma_{\Sigma}^{0}}\varphi\FF\cdot\tau_{\Gamma_{\Sigma}^{0}}\dif\mathscr{H}^{1} = - \int_{\Gamma_{\Sigma}}\varphi\FF\cdot\tau_{\Gamma_{\Sigma}}\dif\mathscr{H}^{1}. 
\end{align*}
This implies \eqref{eq:stokes1A}, and the proof is complete. 
\end{proof}

Based on the preceding proof, we briefly comment on the specific choice of Definition \ref{def:stokes}. 
In order to obtain access to the lower smoothness context compared with \eqref{eq:Stokes} 
for a $\hold^{1}$-field $\FF$, one applies \eqref{eq:Stokes} to $\varphi\FF$, 
where $\varphi$ and $\FF$ can be assumed smooth for the time being. This leads to 
\begin{align}\label{eq:daylight}
-\int_{\Gamma_{\Sigma}}\varphi\FF\cdot\tau_{\Gamma_{\Sigma}}\dif\mathscr{H}^{1} =  \int_{\Sigma} (\nabla\varphi\times\FF)\cdot\nu_{\partial\Omega'}\dif\mathscr{H}^{2} + \int_{\Sigma}\varphi\,(\curl\FF)\cdot\nu_{\partial\Omega'}\dif\mathscr{H}^{2}, 
\end{align}
which we may compactly write as $-\mathrm{T}_{1}=\mathrm{T}_{2}+\mathrm{T}_{3}$. Ultimately, it is $\mathrm{T}_{1}$ what we wish to define in the non-smooth context considered here. 
Rewriting $\mathrm{T}_{2}$ by use of the cross product rule $(\mathbf{a}\times\mathbf{b})\cdot\mathbf{c}=\mathbf{a}\cdot(\mathbf{b}\times\mathbf{c})$,
the trace theorem (Theorem \ref{thm:tracemain1}) implies that $\mathrm{T}_{2}$ can be meaningfully 
defined by employing $(\FF\times\nu_{\partial\Omega'})_{\partial\Omega'}^{\mathrm{int}}$; as an $\lebe^{\infty}(\partial\Omega;T_{\partial\Omega'})$-field, it is straightforward to restrict the latter to $\Sigma$.  

This is not so for term $\mathrm{T}_{3}$. Indeed, in order to make sense of the expression $(\curl\FF)\cdot\nu_{\partial\Omega'}$, we need to involve the normal trace of $\curl\FF$ along $\partial\Omega'$. Since $\curl\FF\in\mathscr{DM}^{\mathrm{ext}}(\Omega)$, this requires the normal traces of $\mathscr{DM}^{\mathrm{ext}}$-fields from Definition \ref{def:normaltraces}. The latter, in turn, are given by \eqref{eq:normaltraceextdivmeas} and thus are objects which are defined globally on $\partial\Omega'$. In view of term $\mathrm{T}_{3}$, which only involves a boundary integral over $\Sigma$ in the smooth case, they must be localized to $\Sigma$. This is the key reason for the appearance of the localizing functions $\psi_{t,\delta,\Sigma}$. In essence, Proposition \ref{prop:diff}  asserts that such a localization procedure can be meaningfully executed for $\mathscr{L}^{1}$-a.e. $t$. In this regard, in view of the subsequent sections, an elementary remark is in order.

\begin{rem}
For $\FF\in\mathscr{CM}^{\infty}(\Omega)$, $\curl\FF\in\mathrm{RM}_{\mathrm{fin}}(\Omega;\R^{3})$ can be restricted to the Borel set $\partial\Omega'$. Regarding $\mathrm{T}_{3}$, it is strictly wrong to interpret $\mathrm{T}_{3}$ as 
\begin{align}\label{eq:rushofbloodtothehead}
\int_{\Sigma} \varphi\,\nu_{\partial\Omega'}\cdot\dif\,(\curl\FF)\mres\partial\Omega'), 
\end{align}
which would allow for an ad-hoc generalization of $\mathrm{T}_{3}$ to the non-smooth context of $\mathscr{CM}^{\infty}$-fields. For instance, even if $\FF\in\hold_{b}^{1}(\Omega;\R^{3})$, then the \emph{measure} $\curl\FF$ satisfies $\curl\FF\ll\mathscr{L}^{3}\mres\Omega$. Since $\mathscr{L}^{3}(\Sigma)=0$, the expression from \eqref{eq:rushofbloodtothehead} will vanish throughout. Instead, the correct interpretation of the right-hand side of \eqref{eq:walter} requires the normal trace of measure $\curl\FF$ along $\Sigma$, and this is incorporated in Definition \ref{def:stokes}. 
\end{rem}

These considerations lead to the specific form of Definition \ref{def:stokes}. Finally, we note that  \eqref{eq:alternativemain} and Proposition \ref{prop:consistency} suggest directly defining 
the Stokes functionals by \eqref{eq:alternativemain}. 
This is also possible, but it then requires the link on the right-hand side of \eqref{eq:alternativemain} 
with the flux of $\curl$ through $\Sigma$. 
Our strategy proceeds the other way around, namely, starting from the flux of curl 
as in Definition \ref{def:stokes}, we establish representation \eqref{eq:alternativemain}. 

We conclude this subsection by discussing the standing of the localizers $\psi_{t,\delta,\Sigma}$. 
First, once $t$ is fixed, it seems natural to employ intrinsic localizers that are  
based directly on the collar map $\Psi_{\Sigma^{\tau,t}}$, instead of $\Psi_{\Sigma}$. 
By this, we mean to replace $\psi_{t,\delta,\Sigma}$ in \eqref{eq:curllocaliser} by 
\begin{align}\label{eq:permanentstate}
\widetilde{\psi}_{t,\delta,\Sigma}(x):=\begin{cases}
0 & \;\text{if}\;x\in\partial\Omega'\setminus \Sigma^{\tau,t}, \\ 
\frac{s}{\delta} &\;\text{if}\;x\in\Gamma_{\Sigma^{\tau,t}}^{s}(:=\Psi_{\Sigma^{\tau,t}}(\{s\}\times\Gamma_{\Sigma^{\tau,t}}))\;\text{for}\;0<s<\delta,\\ 
1 &\;\text{if}\;x\in \Sigma\setminus\Psi_{\Sigma^{\tau,t}}((0,\delta)\times\Gamma_{\Sigma^{\tau,t}}). 
\end{cases}
\end{align}
While this seems feasible at a first glance, it is unclear how this can lead to a useful version of Theorem \ref{thm:stokes}. Throughout the above proof, one is bound to employ the comparability estimates 
of layers $\mathscr{H}^{2}(\Psi_{\Sigma}((t,t+\delta)\times\Gamma_{\Sigma}))$ in terms of $\delta$. 
These estimates depend implicitly on the Lipschitz character of $\Gamma_{\Sigma}$ in $\partial\Omega'$ and, when working with \eqref{eq:psiddefmain}, can be assumed to be uniform in the collar parameter $t$. 

To the contrary, if we work with \eqref{eq:permanentstate}, such uniform estimates are not possible. 
Namely, \emph{a priori}, the collar maps $\Psi_{\Sigma^{\tau,t}}$ do not have to be linked with $\Psi_{\Sigma}$ in any way. In essence, for each $t$, $\Psi_{\Sigma^{\tau,t}}$ could follow a completely different mapping rule. In particular, the bi-Lipschitz constants of $\Psi_{\Sigma^{\tau,t}}$ need not be comparable at all, rendering uniform comparability estimates impossible. This can be circumvented by linking $\Psi_{\Sigma^{\tau,t}}$'s with each other, 
thereby forcing them to follow a uniform mapping rule. 
In this case, however, it is equally natural to work directly with the localizer 
family \eqref{eq:psiddefmain}, which satisfies the requisite uniformity by definition. 

Finally, we address different localization families and single out the following remark.

\begin{rem}[Smooth $\Sigma$ and smooth localizers $\psi_{\delta}$]
Suppose that $\partial\Omega'$ is sufficiently smooth and $\Sigma\subset\partial\Omega'$ is a $\hold^{k}$-boundary manifold (in the obvious sense) relative to $\Omega'$ for a potentially large $k\in\mathbb{N}$. 
Then an adaptation of the arguments of Gilbarg-Trudinger \cite{GilbargTrudinger} 
for $k\geq 2$ (see also Krantz-Parks \cite{KrantzParks0} for $k=1$ and $\Gamma_{\Sigma}$ being of \emph{positive reach} in the sense of Federer \cite{Federer0}) yields that the distance function $x\mapsto \dista(x,\Gamma_{\Sigma})$ is of class $\hold^{k}$ in $U\cap\Sigma$, where $U$ is an open neighborhood of $\Gamma_{\Sigma}$. For sufficiently small $\delta>0$, 
\begin{align*}
\widetilde{\psi}_{t,\delta,\Sigma}(x) := \begin{cases} 0 & \;\text{if}\;x\in\partial\Omega'\setminus\Sigma^{\tau,t}, \\ \frac{1}{\delta}\dista(x,\Psi(\{t\}\times\Gamma_{\Sigma}))&\;\text{if}\;0\leq \dista(x,\Psi(\{t\}\times\Gamma_{\Sigma}))\leq\delta, \\ 
1&\;\text{if}\;\dista(x,\Psi(\{t\}\times\Gamma_{\Sigma}))>\delta, 
\end{cases}
\end{align*}
can be employed as a localizing function and is of class $\hold^{k}$ in a suitable intersection $U\cap\Sigma^{\tau,t}$, where $U$ is open and independent of $\delta$. This additional differentiability offers some simplifications in the above proofs, especially as to the $\hold^{1}$-extensions to $\Omega'$. However, in the case of $\hold^{1}$-regular Lipschitz boundary manifolds, the use of $\widetilde{\psi}_{t,\delta,\Sigma}$ does  not come with substantial simplifications compared with $\psi_{t,\delta,\Sigma}$ given by \eqref{eq:psiddefmain}. 
\end{rem}

\subsection{Examples}\label{sec:examplesstokes1st}
We conclude the overall section by discussing two examples which illustrate Theorem \ref{thm:stokes}. First, we establish that the key assertion of the Stokes theorem (Theorem \ref{thm:stokes}) only holds for $\mathscr{L}^{1}$-a.e. $t\in[0,\frac{1}{2})$, and thus requires a selection of \emph{good} tangential variations indeed.
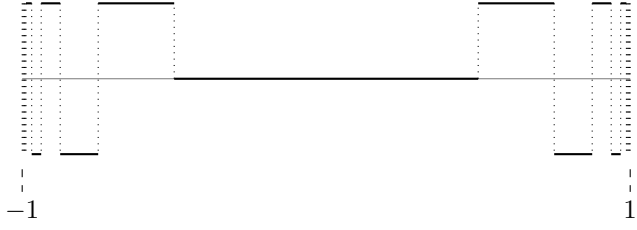
\begin{figure}[t]
\begin{tikzpicture}
\draw[-,black!50!white] (-4,0) -- (4,0);
\draw[-,thick] (-2,0) -- (2,0); 
\draw[dotted] (-2,0) -- (-2,1);
\draw[dotted] (2,0) -- (2,1);
\draw[-,thick] (-3,1) -- (-2,1);
\draw[dotted] (-3,1) -- (-3,-1);
\draw[-,thick] (-3.5,-1) -- (-3,-1);
\draw[dotted] (-3.5,1) -- (-3.5,-1);
\draw[-,thick] (-3.75,1) -- (-3.5,1);
\draw[dotted] (-3.75,1) -- (-3.75,-1);
\draw[-,thick] (-3.75,-1) -- (-3.875,-1);
\draw[-,thick] (-3.95,1) -- (-3.875,1);
\draw[dotted] (-3.875,1) -- (-3.875,-1);
\draw[dotted] (-3.95,1) -- (-3.95,-1);
\draw[dotted] (-3.96,1) -- (-3.96,-1);
\draw[dotted] (-3.97,1) -- (-3.97,-1);
\draw[dotted] (-3.98,1) -- (-3.98,-1);
\draw[dotted] (-3.99,1) -- (-3.99,-1);
\draw[dotted] (-4,1) -- (-4,-1);
\draw[-,thick] (3,1) -- (2,1);
\draw[dotted] (3,1) -- (3,-1);
\draw[-,thick] (3.5,-1) -- (3,-1);
\draw[dotted] (3.5,1) -- (3.5,-1);
\draw[-,thick] (3.75,1) -- (3.5,1);
\draw[dotted] (3.75,1) -- (3.75,-1);
\draw[-,thick] (3.75,-1) -- (3.875,-1);
\draw[-,thick] (3.95,1) -- (3.875,1);
\draw[dotted] (3.875,1) -- (3.875,-1);
\draw[dotted] (3.95,1) -- (3.95,-1);
\draw[dotted] (3.96,1) -- (3.96,-1);
\draw[dotted] (3.97,1) -- (3.97,-1);
\draw[dotted] (3.98,1) -- (3.98,-1);
\draw[dotted] (3.99,1) -- (3.99,-1);
\draw[dotted] (4,1) -- (4,-1);
\draw[dashed] (4,-1.2) -- (4,-1.5);
\node[below] at (4,-1.5) {$1$};
\draw[dashed] (-4,-1.2) -- (-4,-1.5);
\node[below] at (-4,-1.5) {$-1$};
\end{tikzpicture}
\caption{Profile of the function $\eta$ from Example \ref{ex:nonex}.}
\end{figure}
\begin{example}[Non-existence of limit \eqref{eq:curllocaliser}]\label{ex:nonex}  
We now give an example of a $\mathscr{CM}^{\infty}$-field for which limit \eqref{eq:curllocaliser} 
does not exist for $t=0$. 
To this end, define  annuli $A_{k}\subset\R^{2}$ for $k\in\mathbb{N}$ by 
\begin{align*}
A_{k}:=\ball_{1-2^{-k-1}}^{(2)}(0)\setminus\overline{\ball}_{1-2^{-k}}^{(2)}(0)
\end{align*}
and set 
\begin{align*}
\eta(x_{1},x_{2}):=\sum_{k=1}^{\infty}(-1)^{k+1}\mathbbm{1}_{A_{k}}(x_{1},x_{2})
\qquad\mbox{for $(x_{1},x_{2})\in\R^{3}$}. 
\end{align*}
We set $\Omega:=\ball_{2}^{(2)}(0)\times(-2,2)$ and let $\Omega'\Subset \Omega\cap\{x_{3}>0\}$ have $\hold^{1}$-boundary such that $\Sigma:=\ball_{1}^{(2)}(0)\times\{0\}\subset\partial\Omega'$. We define  
\begin{align}\label{eq:exGdef}
\mathbf{g}(x_{1},x_{2},0):=\frac{\eta(x_{1},x_{2})}{\sqrt{x_{1}^{2}+x_{2}^{2}}}(x_{2},-x_{1},0)\qquad\,
\mbox{for $(x_{1},x_{2},0)\in \Sigma$}, 
\end{align}
whereby $\mathbf{g}\in\lebe^{\infty}(\Sigma;\R^{3})$. 

Now note that the classical (interior) Sobolev trace operator $\mathrm{tr}_{\Sigma}\colon (\sobo^{1,1}\cap\lebe^{\infty})(\Omega;\R^{3})\to \lebe^{\infty}(\Sigma;\R^{3})$ is surjective. To see this, we first recall that $\mathrm{tr}_{\Sigma}\colon\sobo^{1,1}(\Omega;\R^{3})\to\lebe^{1}(\Sigma;\R^{3})$ is well-known to be surjective; see \cite{Gagliardo} or \cite{Mironescu}. Moreover, by \cite[Chapter 5.3, Theorem 2]{EvansGariepy}, we see that,
for all $\mathbf{h}\in\sobo^{1,1}(\Omega;\R^{3})$, 
\begin{align*}
\mathrm{tr}_{\Sigma}(\mathbf{h})(x_{0})=\lim_{r\searrow 0}\dashint_{\ball_{r}(x_{0})}\mathbf{h}\dif x \qquad\text{for $\mathscr{H}^{2}$-a.e. $x_{0}\in\Sigma$},  
\end{align*}
and this immediately implies that $\mathrm{tr}_{\Sigma}\colon(\sobo^{1,1}\cap\lebe^{\infty})(\Omega;\R^{3})\to\lebe^{\infty}(\Sigma;\R^{3})$ is a bounded linear operator. On the other hand, the constructions given in \cite{Gagliardo,Mironescu} directly yield that, for any $\mathbf{g}\in\lebe^{\infty}(\Sigma;\R^{3})$, there exists $\mathbf{h}\in(\sobo^{1,1}\cap\lebe^{\infty})(\Omega;\R^{3})$ such that $\mathrm{tr}_{\Sigma}(\mathbf{h})=\mathbf{g}$ $\mathscr{H}^{2}$-a.e. on $\Sigma$. Hence, the claimed surjectivity follows.

In particular, with $\mathbf{g}$ as in \eqref{eq:exGdef}, there exists $\FF\in(\sobo^{1,1}\cap\lebe^{\infty})(\Omega;\R^{3})$ such that 
$\mathrm{tr}_{\Sigma}(\FF)=g$ holds $\mathscr{H}^{2}$-a.e. on $\Sigma$. As a consequence of Theorem \ref{thm:tracemain1}, we record that  
\begin{align}\label{eq:allthesame}
(\FF\times\nu_{\partial\Omega'})_{\partial\Omega'}^{\mathrm{int}}=(\FF\times\nu_{\partial\Omega'})_{\partial\Omega'}^{\mathrm{ext}} = \mathrm{tr}_{\Sigma}(\FF)\times\nu_{\partial\Omega'}=\mathbf{g}\times\nu_{\partial\Omega'}\qquad\text{$\mathscr{H}^{2}$-a.e. on $\Sigma$}. 
\end{align}
We fix the particular collar map 
\begin{align*}
\Psi_{\Sigma}(t,x):=(1-t)x\qquad\,\,\,\mbox{for $\,0\leq t<\frac{1}{2},\;x=(x',0)$, and $|x'|=1$}, 
\end{align*}
which implies that $\psi_{0,\delta,\Sigma}$ is given by 
\begin{align*}
\psi_{0,\delta,\Sigma}(x',0)=\frac{1}{\delta}\big(1-|x'|\big)\qquad\text{if $\;0\leq \dist(x',\Gamma_{\Sigma})\leq 1-\delta$}. 
\end{align*}
In all other regions, its tangential gradient will be zero. Since $\nu_{\partial\Omega'}(x)=(0,0,1)$ on $\Sigma$, we may compute for $\mathscr{H}^{2}$-a.e. $x=(x_{1},x_{2},0)\in\Sigma$:
\begin{align}\label{eq:gdef}
\begin{split}
&(\FF\times\nu_{\partial\Omega'})_{\partial\Omega'}^{\mathrm{int}}(x)=-\frac{\eta(x_{1},x_{2})}{\sqrt{x_{1}^{2}+x_{2}^{2}}}(x_{1},x_{2},0), \\ 
&\nabla_{\tau}\psi_{0,\delta,\Sigma}(x) = -\frac{1}{\delta}\frac{1}{\sqrt{x_{1}^{2}+x_{2}^{2}}}(x_{1},x_{2},0).
\end{split}
\end{align}
Now let $\varphi\in\hold_{\rm c}^{1}(\Omega)$ be such that $\varphi\equiv 1$ in an open neighborhood of $\partial\!\ball_{1}^{(2)}(0)\times\{0\}$. Considering $\delta=2^{-j}$ for sufficiently large $j\in\mathbb{N}$, we have 
\begin{align}\label{eq:limitnon}
\begin{split}
&\lim_{j\to\infty} \int_{(\ball_{1}^{(2)}(0)\setminus\overline{\ball}_{1-2^{-j}}^{(2)}(0))\times\{0\}}\varphi\nabla_{\tau}\psi_{0,2^{-j},\Sigma}\cdot(\FF\times\nu_{\partial\Omega'})_{\partial\Omega'}^{\mathrm{int}}\dif\mathscr{H}^{2} \\ 
&\,\,\stackrel{\eqref{eq:gdef}}{=} \lim_{j\to\infty} 2^{j} \int_{(\ball_{1}^{(2)}(0)\setminus\overline{\ball}_{1-2^{-j}}^{(2)}(0))\times\{0\}}\eta(x_{1},x_{2})\dif\mathscr{H}^{2} \\ 
&\,\, = \lim_{j\to\infty} 2^{j}\sum_{k=j}^{\infty} (-1)^{k+1} \mathscr{L}^{2}(A_{k}) =:\mathrm{I}.
\end{split}
\end{align}
The ultimate limit does not exist. To see this, we compute for $j\in\mathbb{N}$:
\begin{align}\label{eq:explicitcompute}
\begin{split}
2^{j} \sum_{k=j}^{\infty}(-1)^{k+1}\mathscr{L}^{2}(A_{k}) & = -2^{j}\pi \Big(\sum_{k=j}^{\infty}(-2)^{-k} - \frac{3}{4}\sum_{k=j}^{\infty}(-4)^{-k}\Big) \\ 
& =  \pi\Big(\frac{2(-1)^{-j+1}}{3} -\frac{3(-1)^{j+1}2^{-j}}{5} \Big)
\end{split}
\end{align}
by use of the geometric series. Based on \eqref{eq:explicitcompute} and considering even and odd indices $j$ separately, it is clear that the limit $\mathrm{I}$ from \eqref{eq:limitnon} as $j\to\infty$ does not exist. By the proof of Theorem \ref{thm:stokes} (see \eqref{eq:alternativemain}), it is however precisely the first limit in \eqref{eq:limitnon} that is decisive for $\mathfrak{S}_{\Sigma^{\tau,0}}^{\mathrm{int}}(\varphi)$ to be well-defined. In conclusion, $\mathfrak{S}_{\Sigma^{\tau,0}}^{\mathrm{int}}(\varphi)$ is ill-defined in the present situation. 
Yet, the tangential oscillation of $\eta$ is bounded in sufficiently small neighborhoods 
of $\partial\!\ball_{r}^{(2)}(0)\times\{0\}$ for any $0<r<1$, so that
$\mathfrak{S}_{\Sigma^{\tau,t}}^{\mathrm{int}}(\varphi)$ exists and is well-defined for all $0<t<\frac{1}{2}$. 
\end{example}

Anticipating the terminology from \S\ref{sec:divmeasfieldsmanif}, the function:  $(\FF\times\nu_{\partial\Omega'})_{\partial\Omega'}^{\mathrm{int}}$ from \eqref{eq:gdef} is not a divergence measure field, whereby Example \ref{ex:nonex} is also in line with the Stokes theorem in the transversal direction, Theorem \ref{thm:stokes1st} below; see Example \ref{ex:nonexctd} for more detail. 

\begin{example}[Tangential jumps in higher smoothness scenarios]\label{ex:tanjumps1A}
We adopt the setting described in Section \ref{sec:smoothsetCMinfty}. Let $\FF_{1}\in(\sobo^{2,1}\cap\lebe^{\infty})(\Omega';\R^{3})$ and $\FF_{2}\in(\sobo^{2,1}\cap\lebe^{\infty})(\Omega\setminus\overline{\Omega'})$, and define the glued function $\FF$ by 
\begin{align*}
\FF:=\begin{cases} 
\FF_{1}&\;\text{in}\;\Omega',\\ 
\FF_{2}&\;\text{in}\;\Omega\setminus\overline{\Omega'}. 
\end{cases}
\end{align*}
Then $\FF\in(\bv\cap\lebe^{\infty})(\Omega;\R^{3})$, whereby $\FF\in\mathscr{CM}^{\infty}(\Omega)$. We note that there exist bounded linear trace operators $\mathrm{tr}_{\partial\Omega'}^{\mathrm{in}}\colon (\sobo^{2,1}\cap\lebe^{\infty})(\Omega';\R^{3})\to(\sobo^{1,1}\cap\lebe^{\infty})(\partial\Omega';\R^{3})$ and $\mathrm{tr}_{\partial\Omega'}^{\mathrm{ex}}\colon (\sobo^{2,1}\cap\lebe^{\infty})(\Omega\setminus\overline{\Omega'};\R^{3})\to(\sobo^{1,1}\cap\lebe^{\infty})(\partial\Omega';\R^{3})$. On the manifold $\partial\Omega'$, in turn, there is a bounded linear trace operator $\mathrm{tr}_{\Gamma_{\Sigma}}\colon(\sobo^{1,1}\cap\lebe^{\infty})(\partial\Omega';\R^{3})\to\lebe^{\infty}(\Gamma_{\Sigma};\R^{3})$. Now, if $\mathbf{v}\in\sobo^{1,1}(\partial\Omega';\R^{3})$, we have 
\begin{align*}
\int_{\Gamma_{\Sigma}}\mathrm{tr}_{\Gamma_{\Sigma}}(\mathbf{v})\cdot\nu_{\Gamma_{\Sigma}} \dif\mathscr{H}^{1}=\lim_{\delta\searrow 0}\int_{\Sigma}\varphi\nabla_{\tau}\psi_{t,\delta,\Sigma}\cdot \mathbf{v}\dif\mathscr{H}^{2}. 
\end{align*}
This identity can be established analogously to Lemma \ref{lem:bdryintegralcont}, 
now taking into account the properties of the classical trace operator on $\sobo^{1,1}$ (see, \emph{e.g.}, 
\cite[Chapter 5.3]{EvansGariepy}) and adapting it to the spaces on manifolds. Then
\begin{align*}
\mathfrak{S}_{\Sigma^{\tau,0}}^{\mathrm{int}}(\varphi)=-\int_{\Gamma_{\Sigma}} \varphi\,\mathrm{tr}_{\Gamma_{\Sigma}}(\mathrm{tr}_{\partial\Omega'}(\FF_{1})\times\nu_{\Sigma})\cdot\nu_{\Gamma_{\Sigma}}\dif\mathscr{H}^{1}, 
\end{align*}
and analogously for $\mathfrak{S}_{\Sigma^{\tau,0}}^{\mathrm{ext}}$ and $\FF_{2}$. 
\end{example}

\section{$\mathscr{DM}^{\infty}$-Fields on Manifolds, Stokes Theorem, 
and Transversal Variations}\label{sec:divmeasfieldsmanif} 
In \S 5,
The manifold $\Sigma$ and the set $\Omega'\Subset\Omega$ with $\Sigma\Subset\partial\Omega'$ were thought of as fixed; 
the Stokes theorem, Theorem \ref{thm:stokes}, consequently made use of the tangential variations, 
meaning that we vary $\Sigma$ inside $\partial\Omega'$. 
Suppose we allow $\Sigma$ to vary transversally. 
In that case, the key outcome of the present section is that the Stokes theorem is available on $\mathscr{L}^{1}$-a.e. 
such transversally varied manifolds \emph{without} having to additionally vary tangentially. 

This, in turn, is achieved by proving that, 
for $\FF\in\mathscr{CM}^{\infty}(\Omega)$, the tangential traces provided by Theorem \ref{thm:tracemain1} are tangential divergence measure fields on $\mathscr{L}^{1}$-a.e. transversally varied manifold. On such manifolds, the Stokes theorem turns out to be equivalent to the Gauss--Green theorem for $\mathscr{DM}^{\infty}$-fields on manifolds; 
see Theorem \ref{thm:stokes1st}. 
This connection is easily visible from the following simple example, which we take as a starting point 
for this section: 

\begin{example}\label{ex:tangentialex}
Let $\Omega=(-2,2)^{2}\times (-1,1)$ and let $\Sigma:=(-1,1)^{2}\times\{0\}$, which we view as a subset of $\partial\Omega'$, where $\Omega'\subset\Omega\cap \{x_{3}<0\}$ has $\hold^{1}$-boundary. We put $\nu_{\Sigma}:=\nu_{\partial\Omega'}|_{\Sigma}=(0,0,-1)^{\top}$. If $\FF\in\hold^{1}({\Omega};\R^{3})$, then $\FF\times \nu_{\Sigma}\in\hold^{1}(\Sigma;\R^{3})$ and satisfies 
\begin{align}\label{eq:nucrosscompute}
\FF\times\nu_{\Sigma} = (-\FF_{2},\FF_{1},0)^{\top}, \quad\,\,
(\curl\FF)\cdot\nu_{\Sigma}=\partial_{2}\FF_{1}-\partial_{1}\FF_{2}\qquad\;\text{along $\;\Sigma$}. 
\end{align}
As usual, we denote by $\tau_{\Gamma_{\Sigma}}\colon \Gamma_{\Sigma}\to\mathbb{S}^{2}$ the unit tangential field to $\Gamma_{\Sigma}$, where $\Sigma$ inherits the orientation from $\partial\Omega'$; 
see Figure \ref{fig:conventionsorientations}. 
In particular, for $x\in\Gamma_{\Sigma}$, the vectors $\nu_{\Sigma}(x),\nu_{\Gamma_{\Sigma}}(x)$,
and $\tau_{\Gamma_{\Sigma}}(x)$ satisfy, in this order, the right-hand rule and are mutually orthogonal. Letting  
\begin{align*}
    \mathbf{Q}=\left(\begin{matrix} 0 & -1 & 0 \\ 1 & 0 & 0 \\ 0 & 0 & 0 \end{matrix}\right),
\end{align*}
we have in the present geometrically flat case that $\mathbf{Q}\tau(x)=\nu_{\Gamma_{\Sigma}}(x)$ for every $x\in \Gamma_{\Sigma}$ except at the edges of square $\Sigma$. At every such $x$, we have 
\begin{align}\label{eq:useful0}
(\FF(x)\times\nu_{\Sigma}(x))\cdot \nu_{\Gamma_{\Sigma}}(x) =  \FF(x)\cdot (\nu_{\Sigma}(x)\times\nu_{\Gamma_{\Sigma}}(x)) =\FF(x)\cdot\tau_{\Gamma_{\Sigma}}(x). 
\end{align}
Thus, based on the classical Gauss--Green theorem in two dimensions, we conclude that 
\begin{align*}
\int_{\Sigma}(\curl\FF)\cdot\nu_{\Sigma}\dif\mathscr{H}^{2}  & \stackrel{\eqref{eq:nucrosscompute}}{=}  \int_{\Sigma} \mathrm{div}_{\tau}((-\FF_{2},\FF_{1},0)^{\top})\dif\mathscr{H}^{2} = -\int_{\Gamma_{\Sigma}}(-\FF_{2},\FF_{1},0)^{\top}\cdot \nu_{\Gamma_{\Sigma}}\dif\mathscr{H}^{1} \\ 
& \stackrel{\eqref{eq:nucrosscompute}}{=} - \int_{\Gamma_{\Sigma}}(\FF\times\nu_{\Sigma})\cdot\nu_{\Gamma_{\Sigma}}\dif\mathscr{H}^{1} \stackrel{\eqref{eq:useful0}}{=} - \int_{\Gamma_{\Sigma}}\FF\cdot\tau_{\Gamma_{\Sigma}}\dif\mathscr{H}^{1}, 
\end{align*}
which is exactly the Stokes formula based on our convention on orientations.  
As a consequence of this representation, the {tangential trace} of $\FF|_{\Sigma}$ along $\Gamma_{\Sigma}$ as required in the Stokes formula is the {normal trace} of 
$(\FF\times\nu_{\Sigma})|_{\Sigma}$ along $\Gamma_{\Sigma}$. 
\end{example}

\subsection{$\mathscr{DM}^{\infty}$-fields on manifolds}
In view of our above discussion, we pause and introduce $\mathscr{DM}^{\infty}$-fields on oriented $(n-1)$-dimensional $\hold^{2}$-manifolds $\Sigma\subset\R^{n}$ with $\Sigma=\mathrm{int}(\Sigma)$ first. To motivate their definition, we adopt the notation of Theorem \ref{thm:IBPsmooth}. If $\mathbf{v}\in\hold_{\rm c}^{1}(\Sigma;T_{\Sigma})$, then formula \eqref{eq:IBPmanifolds0} yields
\begin{align}\label{eq:IBPmanifolds1}
\int_{\Sigma}\varphi\,\divm(\mathbf{v})\dif\mathscr{H}^{n-1} = -\int_{\Sigma}\nabla_{\tau}\varphi\cdot\mathbf{v}\dif\mathscr{H}^{n-1}\qquad\text{for all}\;\varphi\in\hold_{\rm c}^{1}(\Sigma), 
\end{align}
and gives rise to the following definition: 
\begin{definition}[$\mathscr{DM}^{\infty}$-fields on submanifolds of $\R^{n}$]\label{def:DMinftymanifolds} Let $\Sigma\subset\R^{n}$ be an oriented, $(n-1)$-dimensional $\hold^{2}$-manifold with $\Sigma=\mathrm{int}(\Sigma)$. We say that $\mathbf{v}\in\lebe^{\infty}(\Sigma;T_{\Sigma})$ is a \emph{{bounded divergence measure field on $\Sigma$}}, in formulas $\mathbf{v}\in\mathscr{DM}^{\infty}(\Sigma)$, if there exists $\mu\in\mathrm{RM}_{\mathrm{fin}}(\Sigma)$ such that 
\begin{align}\label{eq:DMINFTYfieldsMANIFOLD}
\int_{\Sigma}\varphi\dif\mu = - \int_{\Sigma}\nabla_{\tau}\varphi\cdot \mathbf{v}\dif\mathscr{H}^{n-1}\qquad\text{for all}\;\varphi\in\hold_{\rm c}^{1}(\Sigma). 
\end{align}
In this case, we write $\mathrm{div}_{\tau}(\mathbf{v}):=\mu$. 
\end{definition}

Before we proceed to the Gauss--Green theorem, we make a remark that ultimately turns out 
to be one of the main reasons for having singled out Theorem \ref{thm:tracemain1}\ref{item:trace2}.

\begin{rem}\label{rem:tangentialimportance}
The defining condition \eqref{eq:DMINFTYfieldsMANIFOLD} does not make sense for general
fields $\mathbf{v}\in\lebe^{\infty}(\Sigma;\R^{n})$, since in this case identity \eqref{eq:IBPmanifolds1} does not hold true and so consistency might be lost. This  is due to the curvature term embodied by the Weingarten map $\mathbf{H}_{\Sigma}$ in \eqref{eq:IBPmanifolds}, and only leads to \eqref{eq:IBPmanifolds1} provided that $\mathbf{v}(x)\in T_{\Sigma}(x)$ for $\mathscr{H}^{n-1}$-a.e. $x\in\Sigma$. 
\end{rem}

We do not aim for an extensive treatment of $\mathscr{DM}^{\infty}(\Sigma)$-fields here. However, for our future applications, we record the following elementary lemma. 

\begin{lem}\label{lem:dualchar}
In the situation of {\rm Definition \ref{def:DMinftymanifolds}}, let $\overline{\Sigma}$ also be compact. 
\begin{enumerate}
\item\label{item:chrisrea1} Let $\mathbf{v}\in\lebe^{1}(\Sigma;T_{\Sigma})$. If  
\begin{align}\label{eq:measfin}
m:=\sup\Big\{\int_{\Sigma}\nabla_{\tau}\varphi\cdot\mathbf{v}\dif\mathscr{H}^{n-1}\colon\;\varphi\in\hold_{\rm c}^{1}(\Sigma),\;\|\varphi\|_{\lebe^{\infty}(\Sigma)}\leq 1 \Big\}<\infty, 
\end{align}
then there exists $\mu\in\mathrm{RM}_{\mathrm{fin}}(\Sigma)$ such that 
\begin{align}\label{eq:dualchar1}
\int_{\Sigma}\varphi\dif\mu = - \int_{\Sigma}\nabla_{\tau}\varphi\cdot\mathbf{v}\dif\mathscr{H}^{n-1} \qquad\text{for all}\;\varphi\in\hold_{\rm c}^{1}(\Sigma) 
\end{align}
and 
$\mu=\mathrm{div}_{\tau}(\mathbf{v})$ together with $|\mathrm{div}_{\tau}(\mathbf{v})|(\Sigma)=m$.

\item\label{item:chrisrea2} If, moreover,  $\mathbf{v}\in\mathscr{DM}^{\infty}(\Sigma)$, 
then identity \eqref{eq:dualchar1} holds for any Lipschitz function 
$\varphi\in\mathrm{Lip}_{0}(\Sigma):=\{\psi\in\mathrm{Lip}(\Sigma)\colon\;\psi|_{\Gamma_{\Sigma}}=0\}$. 
\end{enumerate}
\end{lem}

\begin{proof}
For \ref{item:chrisrea1}, by \eqref{eq:measfin}, the functional 
\begin{align*}
G(\varphi):=-\int_{\Sigma}\nabla_{\tau}\varphi\cdot\mathbf{v}\dif\mathscr{H}^{n-1}\qquad 
\mbox{for any $\varphi\in\hold_{\rm c}^{1}(\Sigma)$}, 
\end{align*}
is a bounded linear functional on $(\hold_{\rm c}^{1}(\Sigma),\|\cdot\|_{\hold(\Sigma)})$. By routine localization and mollification, $\hold_{\rm c}^{1}(\Sigma)$ is seen to be dense in $\hold_{0}(\Sigma)$ with respect to $\|\cdot\|_{\hold(\Sigma)}$; note that, since $\Sigma$ is a $\hold^{2}$-manifold, 
the $\hold^{1}$-regularity of maps $\varphi$ is not destroyed 
during the localization process. Hence, $G$ has a continuous extension $\overline{G}\colon\mathrm{C}_{0}(\Sigma)\to \R$ with $\|\overline{G}\|_{\hold_{0}(\Sigma)'}=m$. By the Riesz representation theorem, there exists a uniquely determined  Radon measure $\mu\in\mathrm{RM}_{\mathrm{fin}}(\Sigma)$ such that 
\begin{align*}
\overline{G}(\varphi) = \int_{\Sigma}\varphi\dif\mu\qquad\text{for all}\;\varphi\in\hold_{0}(\Sigma).  
\end{align*}
Specifying to $\varphi\in\hold_{\rm c}^{1}(\Sigma)$, we arrive at \eqref{eq:dualchar1}. 
The Riesz representation theorem in conjunction with the aforementioned density then yields $|\mathrm{div}_{\tau}(\mathbf{v})|(\Sigma)=m$ 
and thus completes the proof of the first assertion. 

\smallskip
For \ref{item:chrisrea2}, let $\varphi\in\mathrm{Lip}_{0}(\Sigma)$. As in the geometrically flat case, the present assumptions allow us to find a sequence $(\varphi_{j})\subset\hold_{\rm c}^{1}(\Sigma)$ such that $\varphi_{j}\to \varphi$ strongly in $\lebe^{\infty}(\Sigma)$ and $\nabla_{\tau}\varphi_{j}\stackrel{*}{\rightharpoonup}\nabla_{\tau}\varphi$ in $\lebe^{\infty}(\Sigma;T_{\Sigma})$. 
By \eqref{eq:dualchar1}, this immediately implies the underlying identity for $\mathrm{Lip}_{0}$-functions.
The proof is complete. 
\end{proof}

\subsection{Tangential traces of $\mathscr{CM}^{\infty}$-fields as $\mathscr{DM}^{\infty}$-fields}
We now introduce the Gauss--Green functional for $\mathscr{DM}^{\infty}(\Sigma)$-fields, allowing us to assign  distributional traces to $\mathscr{DM}^{\infty}$-fields along $\Gamma_{\Sigma}$. This leads 
to the announced variant of the Stokes theorem with respect to the normal or transversal variations. 
We begin with

\begin{definition}[Gauss--Green functional]\label{def:GG}
In the situation of {\rm Definition \ref{def:DMinftymanifolds}}, 
let $\mathbf{v}\in\mathscr{DM}^{\infty}(\Sigma)$, and let  $U\subset\Sigma$ be a Borel set with $\overline{U}\subset\Sigma$. We then define 
\begin{align}\label{eq:normaldefmanif}
\langle \mathbf{v}\cdot\nu,\varphi\rangle_{\Gamma_{U}}:= -\int_{U}\varphi\,\dif\,\mathrm{div}_{\tau}(\mathbf{v}) - \int_{U} \nabla_{\tau}\varphi\cdot\mathbf{v}\dif\mathscr{H}^{n-1}
\qquad\mbox{for any $\varphi\in\mathrm{C}_{\rm c}^{1}(\Sigma)$}. 
\end{align}
Moreover, if $U=\Sigma$, we put 
\begin{align}\label{eq:brady}
\langle \mathbf{v}\cdot\nu,\varphi\rangle_{\Gamma_{\Sigma}} := -\int_{\Sigma}\varphi\,\dif\,\di_{\tau}(\mathbf{v}) - \int_{\Sigma}\nabla_{\tau}\varphi\cdot\mathbf{v}\dif\mathscr{H}^{n-1}\qquad 
\mbox{for any $\varphi\in\mathrm{C}^{1}(\overline{\Sigma})$},
\end{align}
where $\hold^{1}(\overline{\Sigma})$ is the space of all functions $\varphi\in\hold^{1}(\Sigma)$ for which $\varphi$ and $\nabla_{\tau}\varphi$ can be continuously extended to $\overline{\Sigma}$. 
\end{definition}
It is trivial to see that these functionals extend to non-relabelled functionals
\begin{align}\label{eq:GGmanifoldLipdual}
\langle\mathbf{v}\cdot\nu,\cdot\rangle_{\Gamma_{U}},\langle\mathbf{v}\cdot\nu,\cdot\rangle_{\Gamma_{\Sigma}}\in \mathrm{Lip}(\Sigma)'. 
\end{align}
\begin{lem}\label{lem:orderonmanifolds}
In the situation of {\rm Definition \ref{def:DMinftymanifolds}}, the following hold{\rm :} 
\begin{enumerate}
\item\label{item:Lipbdry1A} Let $\mathbf{v}\in\mathscr{DM}^{\infty}(\Sigma)$. Viewing \eqref{eq:brady} as an element of $\mathscr{D}'(\Omega)$, $\langle\mathbf{v}\cdot\nu,\cdot\rangle_{\Gamma_{\Sigma}}$ is a well-defined distribution of order at most one with $\spt(\langle\mathbf{v}\cdot\nu,\cdot\rangle_{\Gamma_{\Sigma}})\subset\Gamma_{\Sigma}$.
\item\label{item:Lipbdry1B} Suppose that $\Omega'\subset\R^{n}$ be open and bounded with $\hold^{2}$-boundary, and let $\Sigma$ be a $\hold^{2}$-regular Lipschitz manifold relative to $\Omega'$. Moreover, let $\mathbf{v}\in\mathscr{DM}^{\infty}(\Sigma)$. Then there exists a function $(\mathbf{v}\cdot\nu)_{\Gamma_{\Sigma}}\in\lebe^{\infty}(\Gamma_{\Sigma})$ such that 
\begin{align}\label{eq:divmeasmanifreg}
\langle\mathbf{v}\cdot\nu,\,\varphi\rangle_{\Gamma_{\Sigma}}=\int_{\Gamma_{\Sigma}}(\mathbf{v}\cdot\nu)_{\Gamma_{\Sigma}}\,\varphi\dif\mathscr{H}^{n-2}\qquad 
\text{for all}\;\varphi\in\mathrm{Lip}(\Sigma). 
\end{align}
\end{enumerate}
 
\end{lem}
\begin{proof}
For \ref{item:Lipbdry1A}, let $\varphi\in\hold_{\rm c}^{1}(\Omega)$ be such that $\spt(\varphi)\cap \Gamma_{\Sigma}=\emptyset$. Since $\Gamma_{\Sigma}$ is closed, $\mathrm{dist}(\spt(\varphi),\Gamma_{\Sigma})>0$. In particular, $\varphi|_{\Sigma}\in\hold_{\rm c}^{1}(\Sigma)$, and thus the claim is immediate from \eqref{eq:DMINFTYfieldsMANIFOLD}.

\smallskip
For \ref{item:Lipbdry1B}, this follows in the same way as the corresponding statement established by 
Chen-Frid \cite[Theorem 2.2]{ChenFrid1999} for divergence-measure fields 
on open and bounded sets $\Omega\subset\R^{n}$ with Lipschitz deformable boundary. 
We do not repeat the proof here, but comment on three key points required for the argument of \cite{ChenFrid1999} to work: Since $\Sigma$ is a $\hold^{2}$-regular Lipschitz manifold relative to $\Omega'$, it automatically has Lipschitz deformable boundary $\Gamma_{\Sigma}$ in $\partial\Omega'$, the Lipschitz deformations being given by the collar maps $\Psi_{\Sigma}\colon (0,1)\times\Gamma_{\Sigma}\to\mathcal{O}\subset\Sigma$. 
Secondly, the requisite integration-by-parts formula is available since $\mathbf{v}(x)\in T_{\Sigma}(x)$ 
for $\mathscr{H}^{n-1}$-a.e. $x\in \Sigma$; see also Remark \ref{rem:tangentialimportance}. 
Finally, the adaptation of the proof of \cite[Thm. 2.2]{ChenFrid1999} gives us the validity of \eqref{eq:divmeasmanifreg} for all $\varphi\in\hold^{1}(\partial\Omega')$. For general $\varphi\in\mathrm{Lip}(\Sigma)$, we denote by $\overline{\varphi}\in\mathrm{Lip}(\partial\Omega')$ an arbitrary but fixed Lipschitz extension of $\varphi$ to $\partial\Omega'$. By localization and smooth approximation, we then find a sequence $(\overline{\varphi}_{j})\subset\hold^{1}(\Omega)$ such that $\overline{\varphi}_{j}\to\overline{\varphi}$ locally uniformly on $\partial\Omega'$ and $\nabla_{\tau}\overline{\varphi}_{j}\stackrel{*}{\rightharpoonup}\nabla_{\tau}\overline{\varphi}$ in $\lebe^{\infty}(\partial\Omega';T_{\partial\Omega'})$. Since $\overline{\varphi}|_{\Gamma_{\Sigma}}$ is independent of the extension, this allows us to pass to the limit $j\to\infty$ in \eqref{eq:divmeasmanifreg} for $\overline{\varphi}_{j}$ and to obtain \eqref{eq:divmeasmanifreg} also for general $\varphi\in\mathrm{Lip}(\Sigma)$.
This completes the proof. 
\end{proof}

We now establish that {$\mathscr{L}^{1}$-a.e.  boundary manifold $\Sigma$ in $\Omega\subset\R^{3}$} satisfies the Stokes theorem for $\mathscr{CM}^{\infty}(\Omega)$-fields.  

\begin{theorem}[Stokes Theorem for the Transversal Variations]\label{thm:stokes1st}
Let $\Omega\subset\R^{3}$ be open and bounded, let $\Omega'\Subset\Omega$ be open with boundary of class $\hold^{2}$, 
and let $\Sigma$ be a $\hold^{2}$-regular Lipschitz boundary manifold relative to $\Omega'$. 
Recalling the collar map $\Phi_{\partial\Omega'}$ from {\rm Lemma \ref{lem:collar1}}, define 
\begin{align*}
(\partial\Omega')^{\Phi,t}:=\Phi_{\partial\Omega'}(\{t\}\times\partial\Omega')\;\;\;\text{and recall that}\;\;\;\Sigma_{\Omega'}^{\Phi,t}:=\Phi_{\partial\Omega'}(\{t\}\times\Sigma),\qquad |t|<\frac{1}{2}. 
\end{align*}
Then, for $\FF\in\mathscr{CM}^{\infty}(\Omega)$, the following hold{\rm :}
\begin{enumerate}
    \item\label{item:StokesNormal1} \emph{Divergence-measure field property:} 
    There exists a set $\widetilde{\mathscr{I}}\subset I := (-\frac{1}{2},\frac{1}{2})$ with $\mathscr{L}^{1}(I\setminus\widetilde{\mathscr{I}})=0$ such that the interior tangential traces from 
    {\rm Theorem \ref{thm:tracemain1}} along $(\partial\Omega')^{\Phi,t}$ satisfy 
\begin{align}\label{eq:tracco1}
(\FF\times \nu_{(\partial\Omega')^{\Phi,t}})_{(\partial\Omega')^{\Phi,t}}^{\mathrm{int}}|_{\Sigma_{\Omega'}^{\Phi,t}}\in\mathscr{DM}^{\infty}(\Sigma_{\Omega'}^{\Phi,t})\qquad\text{for all}\;t\in\widetilde{\mathscr{I}},
\end{align}
together with the bound{\rm :}
\begin{align}\label{eq:StokesNormalBound}
|\mathrm{div}_{\tau}((\FF\times \nu_{(\partial\Omega')^{\Phi,t}})_{(\partial\Omega')^{\Phi,t}}^{\mathrm{int}}|_{\Sigma_{\Omega'}^{\Phi,t}})|(\Sigma_{\Omega'}^{\Phi,t}) \leq c\,\mathcal{M}_{\partial\Omega',\partial\Omega'}^{\Phi,+}(\curl\FF)(t) \qquad\mbox{for all $t\in\widetilde{\mathscr{I}}$},
\end{align}
where constant $c>0$ is independent of $\FF$ and $t$. The same holds true for the exterior tangential traces 
from {\rm Theorem \ref{thm:tracemain1}} 
with a potentially different set $\widetilde{\mathscr{I}}$, if we replace the right-hand side of \eqref{eq:StokesNormalBound} by $c\,\mathcal{M}_{\Sigma,\Omega'}^{\Phi,-}(\curl\FF)(t)$. 
\item\label{item:StokesNormal2} \emph{Stokes property for the transversal variations:} 
Based on {\rm Definition \ref{def:stokes}}, define 
\begin{align*}
\mathfrak{S}_{\Sigma_{\Omega'}^{\Phi,t}}^{\mathrm{int}}:=\mathfrak{S}_{(\Sigma_{\Omega'}^{\Phi,t})^{\tau,0}}^{\mathrm{int}},\qquad\;\mathfrak{S}_{\Sigma_{\Omega'}^{\Phi,t}}^{\mathrm{ext}}:=\mathfrak{S}_{(\Sigma_{\Omega'}^{\Phi,t})^{\tau,0}}^{\mathrm{ext}}. 
\end{align*} 
If $t\in\widetilde{\mathscr{I}}$ and adopting the notation from 
{\rm Lemma \ref{lem:orderonmanifolds}\ref{item:Lipbdry1B}}, then the function 
\begin{align*}
f_{t,\Sigma,\Omega'}:= -((\FF\times\nu_{(\partial{\Omega}')^{\Phi,t}})_{(\partial{\Omega}')^{\Phi,t}}^{\mathrm{int}}|_{\Sigma_{\Omega'}^{\Phi,t}}\cdot\nu)_{\Gamma_{\Sigma_{\Omega'}^{\Phi,t}}}\in\lebe^{\infty}({\Gamma_{\Sigma_{\Omega'}^{\Phi,t}}})
\end{align*}
satisfies 
\begin{align}\label{eq:walter}
\mathfrak{S}_{\Sigma_{\Omega'}^{\Phi,t}}^{\mathrm{int}}(\varphi) =  \int_{\Gamma_{\Sigma_{\Omega'}^{\Phi,t}}}f_{t,\Sigma,\Omega'}\varphi\dif\mathscr{H}^{1}\qquad\text{for all}\;\varphi\in\hold_{\rm c}^{1}(\Omega). 
\end{align}
In particular, $\mathfrak{S}_{\Sigma_{\Omega'}^{\Phi,t}}^{\mathrm{int}}$ is a \emph{well-defined distribution of order zero}. 
\end{enumerate}
The same holds true with the obvious modifications for the exterior Stokes functional.
\end{theorem}

\begin{proof}
Throughout, we focus on the interior Stokes functionals. We define 
\begin{align*}
\widetilde{\mathscr{I}}
:=\big\{t\in I\colon\;\mathcal{M}_{\partial\Omega',\partial\Omega'}^{\Phi,+}(\curl\FF)(t)<\infty\big\}. 
\end{align*}
Letting $t\in\widetilde{\mathscr{I}}$, we now divide the proof of \ref{item:StokesNormal1} and \ref{item:StokesNormal2}  into three steps. 

\smallskip
\emph{Step 1. Set-up.} 
For notational brevity, we put  $\widetilde{\Sigma}:=\Sigma_{\Omega'}^{\Phi,t}$. Based on Lemma \ref{lem:collar1}\ref{item:collD1}, $\widetilde{\Sigma}$ still is a $\hold^{2}$-regular Lipschitz boundary manifold relative to the open and bounded set $\widetilde{\Omega}':=\Phi_{\partial\Omega'}((t,t_{0})\times\partial\Omega')$ for an arbitrary but fixed $t<t_{0}<1$. By Lemma \ref{lem:collar1}, $\partial\widetilde{\Omega}'$ is of class $\hold^{\infty}$. In view of \ref{item:StokesNormal1} and \ref{item:StokesNormal2}, we aim to apply Lemma \ref{lem:dualchar}; hence, we let $\varphi\in\hold_{\rm c}^{1}(\widetilde{\Sigma})$ be arbitrary with $\|\varphi\|_{\lebe^{\infty}(\widetilde{\Sigma})}\leq 1$. We define $\overline{\varphi}\in\hold_{\rm c}^{1}(\partial\widetilde{\Omega}')$ to be the trivial extension of $\varphi$ to $\partial \widetilde{\Omega}'$. 
Let $\delta_{0}>0$ be so small such that 
\begin{align}\label{eq:crysler}
\mathrm{dist}(\Phi_{\partial\Omega'}(\{t_{0}\}\times\partial\Omega'),\{x\in\widetilde{\Omega}'\colon\dista(x,\Phi_{\partial\Omega'}(\{t\}\times\partial\Omega')))<\delta_{0}\})>0.
\end{align}
Now let  $0<\delta<\delta_{0}$. As a consequence of Lemma \ref{lem:GoodLip}\ref{item:Lipextend1}, there exists a function $\varphi_{\delta}\in\hold^{1}(\Phi_{\partial\Omega'}([t,t_{0}]\times\partial{\Omega}'))$ such that 
\begin{align}\label{eq:psideltabds}
\begin{split}
&\varphi_{\delta}=\overline{\varphi}\quad\text{on}\;\Phi_{\partial\Omega'}(\{t\}\times\partial\Omega'), \\&\big\{x\in\widetilde{\Omega}'\colon\;\varphi_{\delta}(x)\neq 0\big\}
\Subset \big\{x\in\widetilde{\Omega}'\colon\;
\dista(x,\Phi_{\partial{\Omega}'}(\{t\}\times\partial\Omega'))<{\delta}\big\}=:\mathcal{U}_{\delta}, \\ 
&|\nabla \varphi_{\delta}|\leq c\big(\|\nabla_{\tau}\varphi\|_{\lebe^{\infty}(\widetilde{\Sigma})} + \frac{1}{\delta}\|\varphi\|_{\lebe^{\infty}(\widetilde{\Sigma})}\big)\quad\;\text{in $\widetilde{\Omega}'$},
\end{split}
\end{align}
where constant $c>0$ neither depends on $\varphi$ nor $\delta$; for the right-hand side of $\eqref{eq:psideltabds}_{3}$, note that $\varphi$ vanishes outside $\widetilde{\Sigma}$. From \eqref{eq:crysler} and $\eqref{eq:psideltabds}_{2}$, we conclude that $\varphi_{\delta}$ vanishes in an open neighborhood of $\Phi_{\partial\Omega'}(\{t_{0}\}\times\partial\Omega')$. 

\smallskip
\emph{Step 2. Proof of \ref{item:StokesNormal1}.} We aim to apply 
Theorem \ref{thm:tracemain1}\ref{item:trace2}--\ref{item:trace3} to $\widetilde{\Omega}'$ as underlying set. In particular, whenever $\psi\in\hold^{2}(\Phi_{\partial\Omega'}([t,t_{0}]\times\partial\Omega'))$ vanishes 
in a neighborhood of $\Phi(\{t_{0}\}\times\partial\Omega)$, we may employ the relation: $\curl(\nabla\psi)=0$ to find 
\begin{align}\label{eq:carlsen}
\int_{\Phi_{\partial\Omega'}(\{t\}\times\partial\Omega')} (\FF\times\nu_{\partial\widetilde{\Omega}'})_{\partial\widetilde{\Omega}'}^{\mathrm{int}}\cdot\nabla_{\tau}\psi\dif\mathscr{H}^{2} \stackrel{\eqref{eq:scaltovec}}{=}  \int_{\Phi_{\partial\Omega'}((t,t_{0})\times\partial\Omega')}\nabla \psi\cdot\dif\,(\curl\FF).
\end{align}
Now let $\psi\in\hold^{1}(\Phi_{\partial\Omega'}([t,t_{0}]\times\partial\Omega'))$ be such that $\psi$ vanishes in a neighborhood of $\Phi_{\partial\Omega'}(\{t_{0}\}\times\partial\Omega')$. In order to apply \eqref{eq:carlsen} to $\varphi_{\delta}$, let $\overline{\psi}\in\hold_{\rm c}^{1}(\Omega)$ be an extension of $\psi$; such a function exists by the very definition of  $\hold^{1}(\Phi_{\partial\Omega'}([t,t_{0}]\times\partial\Omega'))$ and $\Phi_{\partial\Omega'}((t,t_{0})\times\partial\Omega')$ being relatively compact in $\Omega$.  Mollifying $\overline{\psi}$ yields 
a sequence $(\overline{\psi}_{\varepsilon})\subset\hold_{\rm c}^{\infty}(\Omega)$ such that $\overline{\psi}_{\varepsilon}\to \overline{\psi}$ and $\nabla\overline{\psi}_{\varepsilon}\to\nabla\overline{\psi}$ uniformly in $\Omega$ as $\varepsilon\searrow 0$. Applying \eqref{eq:carlsen} to $\overline{\psi}_{\varepsilon}$, we recall $(\FF\times\nu_{\partial\widetilde{\Omega}'})_{\partial\widetilde{\Omega}'}^{\mathrm{int}}\in\lebe^{\infty}(\partial\widetilde{\Omega}';\R^{3})$ and thus may pass to the limit $\varepsilon\searrow 0$ by the dominated convergence theorem. In consequence, we may apply \eqref{eq:carlsen} to $\psi=\varphi_{\delta}$. 

Since $\varphi$ vanishes outside $\widetilde{\Sigma}$, we thus obtain 
\begin{align}
&\bigg\vert \int_{\widetilde{\Sigma}} (\FF\times\nu_{\partial\widetilde{\Omega}'})_{\partial\widetilde{\Omega}'}^{\mathrm{int}}\cdot\nabla_{\tau}\varphi\dif\mathscr{H}^{2}\bigg\vert\label{eq:kasparov} \\
&\, \stackrel{\varphi_{\delta}|_{\partial\widetilde{\Omega}'}=\varphi}{=} \bigg\vert \int_{\Phi_{\partial\Omega'}(\{t\}\times\partial\Omega')} (\FF\times\nu_{\partial\widetilde{\Omega}'})_{\partial\widetilde{\Omega}'}^{\mathrm{int}}\cdot\nabla_{\tau}\varphi_{\delta}\dif\mathscr{H}^{2}\bigg\vert \notag \\ 
&\,\stackrel{\eqref{eq:carlsen}}{\leq} \bigg\vert \int_{\Phi_{\partial\Omega'}((t,t_{0})\times\partial\Omega')}\nabla\varphi_{\delta}\cdot\dif\,(\curl\FF)\bigg\vert \notag \\ 
&\,\stackrel{\eqref{eq:psideltabds}}{\leq} c\Big( \|\nabla_{\tau}\varphi\|_{\lebe^{\infty}(\widetilde{\Sigma})}|\curl\FF|(\mathcal{U}_{\delta}) + \frac{1}{\delta}\|\varphi\|_{\lebe^{\infty}(\widetilde{\Sigma})}|\curl\FF|(\mathcal{U}_{\delta})\Big), \notag
\end{align}
where $\mathcal{U}_{\delta}$ is as in $\eqref{eq:psideltabds}_{2}$.
Estimate \eqref{eq:kasparov} holds for all sufficiently small $\delta>0$. 
Since $\mathcal{U}_{\delta}\to\emptyset$ as $\delta\searrow 0$, we may send $\delta\searrow 0$ in \eqref{eq:kasparov} to find by use of $\|\varphi\|_{\lebe^{\infty}(\widetilde{\Sigma})}\leq 1$ that 
\begin{align}\label{eq:kasparov1}
\left\vert \int_{\widetilde{\Sigma}} (\FF\times\nu_{\partial\widetilde{\Omega}'})_{\partial\widetilde{\Omega}'}^{\mathrm{int}}\cdot\nabla_{\tau}\varphi\dif\mathscr{H}^{2}\right\vert \leq c\,\limsup_{\delta\searrow 0} \frac{|\curl\FF|(\mathcal{U}_{\delta})}{\delta}. 
\end{align}
By Lemma \ref{lem:collar1} and Remark \ref{rem:aeintro}, we may assume that there exists a constant $\theta\geq 1$ such that 
$\mathcal{U}_{\delta}\subset \Phi_{\partial\Omega'}((t,t+\theta\delta)\times\partial\Omega')$ holds for all sufficiently small $\delta>0$. Then
\begin{align}\label{eq:rook}
\begin{split}
|\diver_{\tau}((\FF\times\nu_{\partial\widetilde{\Omega}'})_{\partial\widetilde{\Omega}'}^{\mathrm{int}})|(\widetilde{\Sigma}) & \stackrel{\text{Lem. \ref{lem:dualchar}}}{=}  \sup_{\substack{\varphi\in\hold_{\rm c}^{1}(\widetilde{\Sigma}) \\ |\varphi|\leq 1}} \left\vert \int_{\widetilde{\Sigma}} (\FF\times\nu_{\partial\widetilde{\Omega}'})_{\partial\widetilde{\Omega}'}^{\mathrm{int}}\cdot\nabla_{\tau}\varphi\dif\mathscr{H}^{2}\right\vert \\ 
& \;\;\stackrel{\eqref{eq:kasparov1}}{\leq} c\,\limsup_{\delta\searrow 0} \frac{|\curl\FF|(\mathcal{U}_{\delta}))}{\delta} \\ 
& \;\;\;\; \leq c\limsup_{\delta\searrow 0} \frac{|\curl\FF|(\Phi_{\partial\Omega'}((t,t+\theta\delta)\times\partial\Omega')}{\delta} \\ 
& \;\;\;\;\leq c\,\mathcal{M}_{\partial\Omega',\partial\Omega'}^{\Phi,+}(\curl\FF)(t).
\end{split}
\end{align}
Since $t\in\widetilde{\mathscr{I}}$, the ultimate expression is finite and $\mathscr{L}^{1}(I\setminus\widetilde{\mathscr{I}})=0$ by Proposition \ref{prop:HLWnormal}. 
In particular, $\mathrm{div}_{\tau}((\FF\times \nu_{\partial\widetilde{\Omega}'})_{\partial\widetilde{\Omega}'}^{\mathrm{int}}|_{\widetilde{\Sigma}})\in\mathrm{RM}_{\mathrm{fin}}(\widetilde{\Sigma})$ and,  since $(\FF\times \nu_{\partial\widetilde{\Omega}'})_{\partial\widetilde{\Omega}'}^{\mathrm{int}}|_{\widetilde{\Sigma}}\in\lebe^{\infty}(\widetilde{\Sigma};T_{\widetilde{\Sigma}})$ by Theorem \ref{thm:tracemain1}, $(\FF\times \nu_{\partial\widetilde{\Omega}'})_{\partial\widetilde{\Omega}'}^{\mathrm{int}}|_{\widetilde{\Sigma}}\in\mathscr{DM}^{\infty}(\widetilde{\Sigma})$. This yields \eqref{eq:tracco1}, 
while \eqref{eq:StokesNormalBound} is a direct consequence of \eqref{eq:rook}. 
Hence, \ref{item:StokesNormal1} follows.  

\smallskip
\emph{Step 3. Proof of \ref{item:StokesNormal2}.} 
We first note that the right-hand side \eqref{eq:walter} has a clear meaning: Indeed, since $\mathcal{M}_{\partial\Omega',\partial\Omega'}^{\Phi,+}(\curl\FF)(t)<\infty$, the boundary regularity criterion \eqref{eq:measregcrit} from Lemma \ref{lem:localisabilitynormaltrace} is satisfied, so that
$\langle (\curl\FF)\cdot\nu,\cdot\rangle_{\partial\widetilde{\Omega}'}$ can be represented 
by a finite Radon measure $\mu_{(\curl\FF)\cdot\nu_{\partial\widetilde{\Omega}'}}$ 
on $\partial\widetilde{\Omega}'$. Now let $\psi_{0,\delta,\widetilde{\Sigma}}$ be as in \eqref{eq:psiddefmain}. Based on \eqref{eq:alternativemain}, we conclude that 
\begin{align}\label{eq:superstrat1}
\begin{split}
\mathfrak{S}_{\widetilde{\Sigma}}^{\mathrm{int}}(\varphi) 
&= - \lim_{\delta\searrow 0}\int_{\partial\widetilde{\Omega}'} \varphi(\FF\times\nu_{\partial\widetilde{\Omega}'})_{\partial\widetilde{\Omega}'}^{\interior}\cdot\nabla_{\tau}\psi_{0,\delta,\widetilde{\Sigma}}\dif\mathscr{H}^{2} \\ 
& = \lim_{\delta\searrow 0}\Big(\int_{\widetilde{\Sigma}}\psi_{0,\delta,\widetilde{\Sigma}}(\FF\times\nu_{\partial\widetilde{\Omega}'})_{\partial\widetilde{\Omega}'}^{\mathrm{int}}\cdot\nabla_{\tau}\varphi\dif\mathscr{H}^{2}\Big. \\ 
& \Big. \qquad\qquad
-\int_{\widetilde{\Sigma}} \nabla_{\tau}(\psi_{0,\delta,\widetilde{\Sigma}}\varphi)\cdot(\FF\times\nu_{\partial\widetilde{\Omega}'})_{\partial\widetilde{\Omega}'}^{\mathrm{int}} \dif\mathscr{H}^{2}\Big) \\ 
& =  \int_{\widetilde{\Sigma}}(\FF\times\nu_{\partial\widetilde{\Omega}'})_{\partial\widetilde{\Omega}'}^{\mathrm{int}}\cdot\nabla_{\tau}\varphi\dif\mathscr{H}^{2} + \lim_{\delta\searrow 0} \int_{\widetilde{\Sigma}} \psi_{0,\delta,\widetilde{\Sigma}}\varphi\dif\,(\mathrm{div}_{\tau}(\FF\times\nu_{\partial\widetilde{\Omega}'})_{\partial\widetilde{\Omega}'}^{\mathrm{int}})\\
&=: \mathrm{II}. 
\end{split}
\end{align}
In the ultimate step, we have used the dominated convergence theorem for the first term, 
whereas the second one is dealt with  by $(\psi_{0,\delta,\widetilde{\Sigma}}\varphi)\in\mathrm{Lip}_{\rm c}(\widetilde{\Sigma})$ and $(\FF\times\nu_{\partial\widetilde{\Omega}'})_{\partial\widetilde{\Omega}'}^{\mathrm{int}}\in\mathscr{DM}^{\infty}(\widetilde{\Sigma})$; note that the underlying integration 
by parts is justified by Lemma \ref{lem:orderonmanifolds}\ref{item:chrisrea2}. Since $\mathrm{div}_{\tau}((\FF\times\nu_{\partial\widetilde{\Omega}'})_{\partial\widetilde{\Omega}'}^{\mathrm{int}})\in\mathrm{RM}_{\mathrm{fin}}(\widetilde{\Sigma})$ and $\widetilde{\Sigma}$ is relatively open, 
then the dominated convergence theorem yields
\begin{align}\label{eq:superstrat2}
\begin{split}
\mathrm{II} & =  \int_{\widetilde{\Sigma}}(\FF\times\nu_{\partial\widetilde{\Omega}'})_{\partial\widetilde{\Omega}'}^{\mathrm{int}}\cdot\nabla_{\tau}\varphi\dif\mathscr{H}^{2} + \int_{\widetilde{\Sigma}}\varphi\dif\,(\mathrm{div}_{\tau}(\FF\times\nu_{\partial\widetilde{\Omega}'})_{\partial\widetilde{\Omega}'}^{\mathrm{int}}) \\ 
& =  - \langle (\FF\times\nu_{\partial\widetilde{\Omega}'})_{\partial\widetilde{\Omega}'}^{\mathrm{int}}|_{\widetilde{\Sigma}}\cdot\nu,\varphi\rangle_{\Gamma_{\widetilde{\Sigma}}},
\end{split}
\end{align}
where we have used the very definition of the Gauss--Green functional from Definition \ref{def:GG}. By \ref{item:StokesNormal1} and Lemma \ref{lem:orderonmanifolds}\ref{item:Lipbdry1B}, we obtain \eqref{eq:walter} and so the full statement of \ref{item:StokesNormal2}. 
This completes the proof. 
\end{proof}

\begin{rem}[Further Localizations]
With slightly more refined estimates that take into the account the geometry of $\Phi_{\partial\Omega'}((t,t_{0})\times\partial\Omega')$, it is possible to improve \eqref{eq:StokesNormalBound} to only have the term $c\,\mathcal{M}_{\Sigma,\partial\Omega'}^{\Phi,+}(\curl\FF)(t)$, 
instead of $c\,\mathcal{M}_{\partial\Omega',\partial\Omega'}^{\Phi,+}(\curl\FF)(t)$,  
on its right-hand side. However, we do not need this improvement in the sequel. 
\end{rem}
\begin{rem}
In stating the above results, we stick to the notation $\Sigma_{\Omega'}^{\Phi,t}$. 
This is so because we partially work on the full boundaries  $(\partial\Omega')^{\Phi,t}$, 
which arise from $\partial\Omega'$ by transversal shifting and depend on $\Omega'$. 
On the contrary, in \S\ref{sec:stokes}, we only work on $\Sigma$ and with $0\leq t<\frac{1}{2}$. For such $t$, we have $\Sigma_{\Omega'}^{\tau,t}\subset\Sigma$, whereby it is convenient to simply write $\Sigma^{\tau,t}$ instead of $\Sigma_{\Omega'}^{\tau,t}$ in the context of \S \ref{sec:stokes}. 
Moreover, note that, if the two-sided maximal operator satisfies $\mathcal{M}_{\partial\Omega',\partial\Omega'}^{n}(\curl\FF)(t)<\infty$, Lemma \ref{lem:curlvanishmax} yields $|\curl\FF|((\partial\Omega')^{\Phi,t})=0$. This is still the case on a subset of full $\mathscr{L}^{1}$-measure in $I$. In particular, for such $t$, the Stokes theorem on the respective transversally shifted boundary manifolds is available 
for both the interior and exterior Stokes functionals. 
\end{rem}

As a consequence of the proof of Theorem \ref{thm:stokes1st}, we do not have to vary tangentially when $t$ is such that 
$\mathcal{M}_{\partial\Omega',\partial\Omega'}^{\Phi,+}(\curl\FF)(t)<\infty$. In view of our discussion 
in \S\ref{sec:consistency} (see \eqref{eq:permanentstate}\emph{ff.}),  
there is no need to work with localizers adapted to $\Sigma$ and the specific collar map $\Psi_{\Sigma}$ when aiming to introduce Stokes functionals as in Definition \ref{def:stokes}. On transversally shifted manifolds $\Sigma_{\Omega'}^{\Phi,t}$ with $\mathcal{M}_{\partial\Omega',\partial\Omega'}^{\Phi,+}(\curl\FF)(t)<\infty$, we may therefore directly state a Stokes theorem for relatively open subsets as follows.

\begin{corollary}[Low Regularity]\label{cor:lowreg}
Let $\Omega\subset\R^{3}$ be open and bounded, and let $\Omega'\Subset\Omega$ be open with boundary of class $\hold^{2}$. Moreover, let $\FF\in\mathscr{CM}^{\infty}(\Omega)$, and let $t\in (-\frac{1}{2},\frac{1}{2})$ be such that 
\begin{align*}
\emph{(i)}\;\;\;\mathcal{M}_{\partial\Omega',\partial\Omega'}^{\Phi,+}(\curl\FF)(t)<\infty\;\;\;\text{or}\;\;\;\emph{(ii)}\;\mathcal{M}_{\partial\Omega',\partial\Omega'}^{\Phi,-}(\curl\FF)(t)<\infty.
\end{align*}
Let $E\subset(\partial\Omega')^{\Phi,t}$ be relatively open. 
Finally, for all sufficiently small $\delta>0$, 
suppose that  $\psi_{\delta}\colon (\partial\Omega')^{\Phi,t}\to[0,1]$ be a Lipschitz function such that $\psi_{\delta}=0$ on $(\partial\Omega')^{\Phi,t}\setminus E$ and $\psi_{\delta}\to 1$ pointwisely on $E$ as $\delta\searrow 0$. We define  \emph{interior} and \emph{exterior Stokes functionals} on $E$ for $\varphi\in\hold_{\rm c}^{1}(\Omega)$ by 
    \begin{align*}
        \mathfrak{S}_{E}^{\mathrm{int}}(\varphi):= 
        \lim_{\delta\searrow 0}\langle(\curl\FF)\cdot\nu,\,\psi_{\delta}\varphi\rangle_{(\partial\Omega')^{\Phi,t}}
    +  \int_{E}(\FF\times\nu_{\partial\Omega'})_{\partial\Omega'}^{\interior}\cdot\nabla_{\tau}\varphi\dif\mathscr{H}^{2}
    \end{align*}
    in case \emph{(i)}, and 
    \begin{align*}
    \mathfrak{S}_{E}^{\mathrm{ext}}(\varphi):= 
        \lim_{\delta\searrow 0}\langle(\curl\FF)\cdot\nu,\psi_{\delta}\varphi\rangle_{(\partial\Omega')^{\Phi,t}}
    +  \int_{E}(\FF\times\nu_{\partial\Omega'})_{\partial\Omega'}^{\exterior}\cdot\nabla_{\tau}\varphi\dif\mathscr{H}^{2}
    \end{align*}
in case \emph{(ii)}. 
Then, in each of cases \emph{(i)}--\emph{(ii)},  
\begin{enumerate}
\item\label{item:StokesSelect1} the functionals $\mathfrak{S}_{E}^{\mathrm{int}}$ and $\mathfrak{S}_{E}^{\mathrm{ext}}$ 
is independent of the specific choice of the sequence $(\psi_{\delta})$ with the above properties.
\item\label{item:StokesSelect2} the functionals $\mathfrak{S}_{E}^{\mathrm{int}}$ and $\mathfrak{S}_{E}^{\mathrm{ext}}$ are \emph{distributions of order at most one} supported on $\Gamma_{E}:=\overline{E}\setminus\mathrm{int}(E)$. More precisely, we have the representations{\rm :} 
\begin{align}\label{eq:guthriegovan1}
\begin{split}
    \mathfrak{S}_{E}^{\mathrm{int}}(\varphi)=  - \langle (\FF\times\nu_{\partial\widetilde{\Omega}'})_{\partial\widetilde{\Omega}'}^{\mathrm{int}}|_{E}\cdot\nu,\varphi\rangle_{\Gamma_{E}},\\ 
    \mathfrak{S}_{E}^{\mathrm{ext}}(\varphi) =  - \langle (\FF\times\nu_{\partial\widetilde{\Omega}'})_{\partial\widetilde{\Omega}'}^{\mathrm{ext}}|_{E}\cdot\nu,\varphi\rangle_{\Gamma_{E}}
    \end{split}
\end{align}
for all $\varphi\in\hold_{\rm c}^{1}(\Omega)$. 
\end{enumerate}

\end{corollary}
\begin{proof}
We focus on $\mathfrak{S}_{E}^{\mathrm{int}}$, the proof for $\mathfrak{S}_{E}^{\mathrm{ext}}$ being analogous. Let $\varphi\in\hold_{\rm c}^{1}(\Omega)$. As in the derivation of \eqref{eq:alternativemain}, we have 
\begin{align*}
\mathfrak{S}_{E}^{\mathrm{int}}(\varphi)=-\lim_{\delta\searrow 0}\int_{\partial\widetilde{\Omega}'}\varphi(\FF\times\nu_{\partial\widetilde{\Omega}'})_{\partial\widetilde{\Omega}'}^{\mathrm{int}}\cdot\nabla\psi_{\delta}\dif\mathscr{H}^{2}. 
\end{align*}
Based on this observation, we can follow the proof of Theorem \ref{thm:stokes1st}. More precisely, in \eqref{eq:superstrat1}--\eqref{eq:superstrat2}, we replace $\widetilde{\Sigma}$ by $E$ and $\psi_{0,\delta,\widetilde{\Sigma}}$ by $\psi_{\delta}$. Based on Definition \ref{def:GG} and \eqref{eq:GGmanifoldLipdual}, this yields \eqref{eq:guthriegovan1}, from where both \ref{item:StokesSelect1} and \ref{item:StokesSelect2} follow at once. The proof is complete. 
\end{proof}
\begin{rem}
While Theorem \ref{thm:stokes1st} and Corollary \ref{cor:lowreg} have a parallel structure, the outcomes differ in the regularity of the Stokes functionals. The Lipschitz regularity of the boundaries in the context of Theorem \ref{thm:stokes1st} yields the $\lebe^{\infty}$-representability of the Stokes functionals. In the situation of general relatively open sets $E$ as considered in Corollary \ref{cor:lowreg}, the validity of the Stokes theorem happens at the cost of the Stokes functionals potentially being non-regular distributions on $\Gamma_{E}$. 
\end{rem}

\subsection{Examples} 
We conclude the overall section with two examples.
\begin{example}[Example \ref{ex:nonex}, continued]\label{ex:nonexctd} 
We compute the tangential distributional divergence of the function $(\FF\times\nu_{\partial\Omega'})_{\partial\Omega'}^{\mathrm{int}}$ from \eqref{eq:gdef}, which we directly view as a function on $\R^{2}$: 
\begin{align*}
\di_{\tau}((\FF\times\nu_{\partial\Omega'})_{\partial\Omega'}^{\mathrm{int}}) & = 2\sum_{k=1}^{\infty} (-1)^{k+1}\mathscr{H}^{1}\mres\partial\!\ball_{1-2^{-k}}\qquad\text{in}\;\mathscr{D}'(\R^{2}), 
\end{align*}
and so $\mathrm{tr}(\FF)\times\nu\notin\mathscr{DM}^{\infty}(\R^{2})$ follows from 
\begin{align*}
|\mathrm{div}_{\tau}((\FF\times\nu_{\partial\Omega'})_{\partial\Omega'}^{\mathrm{int}})|(\R^{2})= 4\pi\sum_{k=1}^{\infty}(1-2^{-k}) =\infty. 
\end{align*} 
In particular, Theorem \ref{thm:stokes1st} is not applicable here, whereas Theorem \ref{thm:stokes} is indeed. 
\end{example}
\begin{example}
The boundary regularity criterion $\mathcal{M}_{\partial\Omega',\partial\Omega'}^{\Phi,t}(\curl\FF)<\infty$ is automatically satisfied for any choice of $t$, provided that \emph{e.g.}  $\curl\FF\in\lebe^{\infty}(\Omega;\R^{3})$. In this sense, Theorem \ref{thm:stokes1st} and Corollary \ref{cor:lowreg} provide unconditional Stokes theorems without manifold selection for $\sobo^{\curl,\infty}$-fields.
\end{example}

\section{The Stokes Theorem for $\mathscr{CM}^{p}$-Fields, $1\leq p <\infty$}\label{sec:Stokesgeneral}
As a hybrid of the methods developed in the previous two sections, we now establish a (tangential) Stokes theorem in the spirit of Theorem \ref{thm:stokes} for a large subclass of $\mathscr{CM}^{p}$-fields, $1\leq p<\infty$. 
This requires localizing certain distributions, namely divergences of $\lebe_{\locc}^{1}$-fields, 
to certain closed sets; see \S \ref{sec:locclosed}. 
In the intermediate subsection \S\ref{sec:bdrypairing}, 
we introduce a boundary pairing and prove, as the $\mathscr{CM}^{p}$-substitute of Theorem \ref{thm:tracemain1}\ref{item:trace2}, in which sense the distributional tangential traces are tangential. This leads to a general version of the Stokes theorem (see \S\ref{sec:Stokestheorem}), 
and eventually fills the remaining gaps from Table \ref{table:1}; see \S \ref{sec:StokesMostGeneral}.

\subsection{Localization elements of $\mathrm{div}_{\tau}(\lebe_{\locc}^{1})$}\label{sec:locclosed}
Let $\Omega\subset\R^{n}$ be open and bounded,  and let $\Omega'\Subset\Omega$ be open with $\hold^{2}$-boundary. Moreover, let $\Sigma\Subset\partial\Omega'$ be a $\hold^{2}$-regular Lipschitz boundary manifold relative to $\Omega'$, with  boundary $\Gamma_{\Sigma}\subset\partial\Omega'$; in particular, $\Sigma$ is relatively open in $\partial\Omega'$. Temporarily assuming that $\partial\Omega'$ is even smooth and  $S\in\mathscr{D}'(\partial\Omega')$, $S$ can be localized to $\Sigma$ by restricting it to an element $S|_{\Sigma}\in\mathscr{D}'(\Sigma)$ via 
\begin{align}\label{eq:restrictiontoopen}
\langle S|_{\Sigma},\varphi\rangle_{\mathscr{D}'(\Sigma)\times\mathscr{D}(\Sigma)} := \langle S,\varphi\rangle_{\mathscr{D}'(\partial\Omega')\times\mathscr{D}(\partial\Omega')}\qquad 
\mbox{for $\varphi\in\mathscr{D}(\Sigma)$}. 
\end{align}
For our future purposes, we require a restriction of $S$ to an element of $\hold^{\infty}(\overline{\Sigma})'$, and this is what we mean by \emph{localizing to closed sets}. 

In the case of $\hold^{2}$-regular Lipschitz boundary manifolds, $\mathscr{D}(\partial\Omega')$ is not well-defined, but it is clear that the task of localizing elements of $\mathrm{Lip}(\partial\Omega')'$ comes with similar issues. To this end, we consider the subspace
\begin{align}\label{eq:divL1loc}
\mathrm{div}_{\tau}(\lebe_{\locc}^{1}(\partial\Omega';T_{\partial\Omega'}))
:=\big\{\mathrm{div}_{\tau}(\GG)\colon\;\GG\in\lebe_{\locc}^{1}(\partial\Omega';T_{\partial\Omega'})\big\}, 
\end{align}
where $\mathrm{div}_{\tau}$ denotes the tangential distributional divergence. 
Since $\partial\Omega'$ is compact, we could equivalently admit $\GG\in\lebe^{1}(\partial\Omega';T_{\partial\Omega'})$ in \eqref{eq:divL1loc} and end up with the same space. We explicitly record that the duality pairing is given by 
\begin{align}\label{eq:representationLip}
\langle\mathrm{div}_{\tau}(\GG),\vartheta\rangle_{\Sigma} = - \int_{\Sigma}\GG\cdot\nabla_{\tau}\vartheta\dif x
\end{align}
whenever $\GG\in\lebe_{\locc}^{1}(\partial\Omega';T_{\partial\Omega'})$ and $\vartheta\in\mathrm{Lip}_{\rm c}(\Sigma)$. In this context, the $\hold^{2}$-regularity of $\partial\Omega'$ implies the  consistency of \eqref{eq:representationLip} with the classical integration by parts-formula from Theorem \ref{thm:IBPsmooth}\emph{ff.} for sufficiently smooth maps. Moreover, note that \eqref{eq:representationLip} extends to the case where $\Sigma=\partial\Omega'$ and $\vartheta\in\mathrm{Lip}(\partial\Omega')$, in which case $\Gamma_{\partial\Omega'}=\emptyset$.

We now recall the collar map $\Psi_{\Sigma}\colon(-1,1)\times\Gamma_{\Sigma}\to\mathcal{O}$ from Lemma \ref{lem:collar2}. 
For $0<\delta<\frac{1}{4}$ and $0\leq  t<\frac{1}{2}$, we define interior height functions as in \eqref{eq:psiddefmain}, and proceed to display the definition of the boundary pairing as required for 
our Stokes theorem, Theorem \ref{thm:StokesDist}, below.

\begin{definition}[Boundary Pairing $\lla\cdot,\cdot\rra_{\Gamma_{\Sigma}}$]\label{def:locclosed} In the situation described above, let $\GG\in\lebe^{1}(\partial\Omega';T_{\partial\Omega'})$. 
For $0\leq t < \frac{1}{2}$, define
\begin{align}\label{eq:distlocaliser}
\begin{split}
& \lla \GG\cdot\nu_{\Gamma_{\Sigma}^{t}},\psi\rra_{\Gamma_{\Sigma}^{t}}:= \lim_{\delta\searrow 0}\int_{\Sigma} \psi\,\GG\cdot\nabla_{\tau}\psi_{t,\delta,\Sigma}\dif\mathscr{H}^{n-1}\qquad\mbox{for $\psi\in\mathrm{Lip}(\Sigma)$}, 
\end{split}
\end{align}
provided that this limit exists \emph{for all} $\psi\in\mathrm{Lip}(\Sigma)$. In this case, we say that $\lla \GG\cdot\nu_{\Gamma_{\Sigma}^{t}},\cdot\rra_{\Gamma_{\Sigma}^{t}}$ \emph{exists}. 
\end{definition}
If $\mathbf{G}$ and $\psi$ are  sufficiently smooth, then Lemma \ref{lem:bdryintegralcont} and its proof imply that 
\begin{align*}
 \lla \GG\cdot\nu_{\Gamma_{\Sigma}^{t}},\psi\rra_{\Gamma_{\Sigma}^{t}} = \int_{\Gamma_{\Sigma^{\tau,t}}}\psi\,\mathbf{G}\cdot\nu_{\Gamma_{\Sigma^{\tau,t}}}\dif\mathscr{H}^{n-2}.
\end{align*}
Hence, Definition \ref{def:locclosed} gives a generalization of the ultimate boundary integral to non-smooth scenarios. We may then introduce the localization of certain elements of $\mathrm{Lip}(\partial\Omega')'$  as follows:

\begin{definition}[Localization]\label{def:locdivL1}
In the above situation, suppose that $S\in\mathrm{Lip}(\partial\Omega')'$ satisfies $S=\mathrm{div}_{\tau}(\GG)$ for some $\GG\in\lebe^{1}(\partial\Omega';T_{\partial\Omega'})$. Moreover, let $0\leq t<\frac{1}{2}$ be such that $\lla\mathbf{G}\cdot\nu_{\Gamma_{\Sigma}^{t}},\cdot\rra_{\Gamma_{\Sigma}^{t}}$ exists.
We then define 
\begin{align}\label{eq:observe}
\langle S,\psi\rangle_{{\Sigma^{\tau,t}}}  :=  - \int_{\Sigma^{\tau,t}}\GG\cdot\nabla_{\tau}\psi\dif\mathscr{H}^{n-1} -  \lla \mathbf{G}\cdot\nu_{\Gamma_{\Sigma}^{t}},\psi\rra_{\Gamma_{\Sigma}^{t}}
\qquad\mbox{for $\psi\in\mathrm{Lip}(\Sigma^{\tau,t})$}, 
\end{align}
and the \emph{mass} of $S$ on $\Sigma^{\tau,t}$ by 
\begin{align}\label{eq:mass}
[S]_{\Sigma^{\tau,t}} := \langle S,\mathbbm{1}_{\Sigma^{\tau,t}}\rangle_{{\Sigma^{\tau,t}}}:=-\lim_{\delta\searrow 0}\int_{\Sigma}\mathbf{G}\cdot\nabla_{\tau}\psi_{t,\delta,\Sigma}\dif\mathscr{H}^{n-1}. 
\end{align}
\end{definition}
We now confirm that \eqref{eq:observe} gives rise to a well-defined functional indeed. 

\begin{lem}\label{lem:independence}
In the above situation, let $\GG,\widetilde{\GG}\in\lebe^{1}(\partial\Omega';T_{\partial\Omega'})$ be such that 
\begin{align}\label{eq:divequality}
\mathrm{div}_{\tau}(\GG)=\mathrm{div}_{\tau}(\widetilde{\GG})\qquad\text{as an identity in $\mathrm{Lip}(\partial\Omega')'$}.
\end{align}
If $0\leq t<\frac{1}{2}$ is such that both $\lla\mathbf{G}\cdot\nu_{\Gamma_{\Sigma}^{t}},\cdot\rra_{\Gamma_{\Sigma}^{t}}$ and $\lla\widetilde{\mathbf{G}}\cdot\nu_{\Gamma_{\Sigma}^{t}},\cdot\rra_{\Gamma_{\Sigma}^{t}}$ exist, then 
\begin{align}\label{eq:unam}
\langle\mathrm{div}_{\tau}(\GG),\cdot\rangle_{{\Sigma^{\tau,t}}}=\langle\mathrm{div}_{\tau}(\widetilde{\GG}),\cdot\rangle_{{\Sigma^{\tau,t}}}\qquad\text{on}\;\;\mathrm{Lip}(\Sigma^{\tau,t})
\end{align}
and, in particular, 
\begin{align}\label{eq:massesequal}
[\mathrm{div}_{\tau}(\GG)]_{\Sigma^{\tau,t}}=[\mathrm{div}_{\tau}(\widetilde{\GG})]_{\Sigma^{\tau,t}}. 
\end{align}
\end{lem}

\begin{proof}
Let $\psi\in\mathrm{Lip}({\Sigma^{\tau,t}})$ and denote by  $\overline{\psi}\in\mathrm{Lip}(\partial\Omega')$ an arbitrary Lipschitz extension to $\partial\Omega'$. Then $\psi_{t,\delta,\Sigma}\overline{\psi}\in\mathrm{Lip}(\partial\Omega')$. Directly using Definition \ref{def:locdivL1} and dominated convergence in the first step, we find that 
\begin{align*}
\langle\mathrm{div}_{\tau}(\mathbf{G}),\psi\rangle_{{\Sigma^{\tau,t}}} & \stackrel{\text{Def}}{=} - \lim_{\delta\searrow 0}\Big(\int_{\Sigma^{\tau,t}}\psi_{t,\delta,\Sigma}\GG\cdot\nabla_{\tau}\psi\dif\mathscr{H}^{n-1} +  \int_{\Sigma^{\tau,t}} \psi\,\mathbf{G}\cdot\nabla_{\tau}\psi_{t,\delta,\Sigma}\dif\mathscr{H}^{n-1}\Big) \\ 
& \;= - \lim_{\delta\searrow 0}\int_{\Sigma^{\tau,t}}\mathbf{G}\cdot\nabla_{\tau}(\psi_{t,\delta,\Sigma}\overline{\psi})\dif\mathscr{H}^{n-1} \\ & \; = -\lim_{\delta\searrow 0} \int_{\partial\Omega'}\mathbf{G}\cdot\nabla_{\tau}(\psi_{t,\delta,\Sigma}\overline{\psi})\dif\mathscr{H}^{n-1} \\ 
& \!\!\!\!\!\!\!\stackrel{\eqref{eq:representationLip}, \eqref{eq:divequality}}{=} - \lim_{\delta\searrow 0}\int_{\partial\Omega'}\widetilde{\mathbf{G}}\cdot\nabla_{\tau}(\psi_{t,\delta,\Sigma}\overline{\psi})\dif\mathscr{H}^{n-1} = \langle\mathrm{div}_{\tau}(\widetilde{\mathbf{G}}),\psi\rangle_{{\Sigma^{\tau,t}}}, 
\end{align*}
where the last equality can be seen by reading the first three lines backwards with $\mathbf{G}$ being replaced by $\widetilde{\mathbf{G}}$. This is \eqref{eq:unam}, from where \eqref{eq:massesequal} also follows.
The proof is complete. 
\end{proof}

\begin{rem}\label{rem:independentchoicerem}
In \eqref{eq:unam}, we do not assert that $\lla\GG\cdot\nu_{\Gamma_{\Sigma}^{t}},\cdot\rra_{\Gamma_{\Sigma}^{t}}=\lla\widetilde{\GG}\cdot\nu_{\Gamma_{\Sigma}^{t}},\cdot\rra_{\Gamma_{\Sigma}^{t}}$, which can easily seen 
to be wrong by putting $\widetilde{\mathbf{G}}=\mathbf{G}+\mathbf{g}$, 
where $\mathrm{div}_{\tau}(\mathbf{g})=0$. However, the above proof yields that, if $\mathbf{G},\widetilde{\mathbf{G}}\in\lebe^{1}(\partial\Omega';T_{\partial\Omega'})$ are such that $\mathrm{div}_{\tau}(\mathbf{G})=\mathrm{div}_{\tau}(\widetilde{\mathbf{G}})$ in $\mathrm{Lip}(\partial\Omega')'$ 
and $0\leq t<\frac{1}{2}$, then 
$\lla \mathbf{G}\cdot\nu_{\Gamma_{\Sigma}^{t}},\cdot\rra_{\Gamma_{\Sigma}^{t}}$ exists if and only if $\lla \widetilde{\mathbf{G}}\cdot\nu_{\Gamma_{\Sigma}^{t}},\cdot\rra_{\Gamma_{\Sigma}^{t}}$ exists. 
\end{rem}

\begin{rem}[Masses and Admissibility]\label{rem:massesadmissibility}
In Definition \ref{def:locdivL1}, we have declared the mass of an element $S\in\mathrm{Lip}(\partial\Omega')'$ which satisfies $S=\mathrm{div}_{\tau}(\GG)$ for some $\GG\in\lebe^{1}(\partial\Omega';T_{\partial\Omega'})$
in the sense that  
\begin{align}\label{eq:gutealteZeiten}
\langle S,\psi\rangle = - \langle \mathbf{G},\nabla_{\tau}\psi\rangle= - \int_{\partial\Omega'}\mathbf{G}\cdot\nabla_{\tau}\psi\dif\mathscr{H}^{2}\qquad\,\,\text{for all}\;\psi\in\mathrm{Lip}(\partial\Omega'). 
\end{align}
Now, if $S\in\hold^{2}(\partial\Omega')'$ satisfies, for some $\GG\in\lebe^{1}(\partial\Omega';T_{\partial\Omega'})$, the identity 
\begin{align}\label{eq:idreduced}
\langle S,\psi\rangle = - \int_{\partial\Omega'}\GG\cdot\nabla_{\tau}\psi\dif\mathscr{H}^{2}\qquad\text{for all}\;\psi\in\hold^{2}(\partial\Omega'), 
\end{align}
then $S$ automatically gives rise to an element of $\widetilde{S}\in\mathrm{Lip}(\partial\Omega')'$: Given $\psi\in\mathrm{Lip}(\partial\Omega')$, the  $\hold^{2}$-regularity of $\partial\Omega'$ allows us to choose a sequence $(\psi_{j})\subset\hold^{2}(\partial\Omega')$ such that $\psi_{j}\to\psi$ strongly in $(\hold(\partial\Omega');\|\cdot\|_{\sup})$ and $\nabla_{\tau}\psi_{j}\stackrel{*}{\rightharpoonup}\nabla_{\tau}\psi$ in $\lebe^{\infty}(\partial\Omega';T_{\partial\Omega'})$. We then define 
\begin{align*}
\langle\widetilde{S},\psi\rangle := \lim_{j\to\infty}\langle S,\psi_{j}\rangle \stackrel{\eqref{eq:idreduced}}{=} - \lim_{j\to\infty}\int_{\partial\Omega'}\GG\cdot\nabla_{\tau}\psi_{j}\dif\mathscr{H}^{2} \stackrel{\GG\in\lebe^{1}}{=} - \int_{\partial\Omega'}\GG\cdot\nabla_{\tau}\psi\dif\mathscr{H}^{2}. 
\end{align*}
From here, it follows at once that $\widetilde{S}$ is well-defined, meaning that it does not depend on
the specific choice of sequence $(\psi_{j})$. In the situation of Definition \ref{def:locdivL1}, 
we may also introduce  
its mass on $\Sigma^{\tau,t}$ by 
$[S]_{\Sigma^{\tau,t}}:=[\widetilde{S}]_{\Sigma^{\tau,t}}$ for 
$S\in\hold^{2}(\partial\Omega')'$ with \eqref{eq:idreduced}.
\end{rem}

We now have
\begin{prop}[of Gauss--Green type]\label{prop:StokesPreparatory}
Let $\Omega\subset\R^{n}$ be open and bounded, and let $\Omega'\Subset\Omega$ be open and bounded with $\hold^{2}$-boundary. Moreover, let $\Sigma$ be a $\hold^{2}$-regular Lipschitz boundary manifold relative to $\Omega'$, and let $\mathbf{G}\in\lebe^{1}(\partial\Omega';T_{\partial\Omega'})$.  Then we have: 
\begin{enumerate}
    \item\label{item:welldef1} There exists a subset $\mathscr{I}\subset I:=[0,\frac{1}{2})$ with $\mathscr{L}^{1}(I\setminus\mathscr{I})=0$ such that, for each $t\in\mathscr{I}$, $\lla\mathbf{G}\cdot\nu_{\Gamma_{\Sigma}^{t}},\cdot\rra_{\Gamma_{\Sigma}^{t}}$ exists. In particular, for each $t\in\mathscr{I}$, $\langle\mathrm{div}_{\tau}(\mathbf{G}),\cdot\rangle_{{\Sigma^{\tau,t}}}$ is well-defined as an element of $\mathrm{Lip}(\Sigma^{\tau,t})'$. \item\label{item:GGdist} For $t\in\mathscr{I}$ as in \ref{item:welldef1},  the \emph{Gauss--Green formula} holds{\rm :}
\begin{align}\label{eq:GGdistMAIN}
\lla\mathbf{G}\cdot\nu_{\Gamma_{\Sigma}^{t}},\psi\rra_{\Gamma_{\Sigma}^{t}}=-\langle\mathrm{div}_{\tau}(\GG),\psi\rangle_{{\Sigma^{\tau,t}}} -  \int_{\Sigma^{\tau,t}}\GG\cdot\nabla_{\tau}\psi\dif\mathscr{H}^{n-1}
\end{align}
for all $\psi\in\mathrm{Lip}({\Sigma^{\tau,t}})$. 
\end{enumerate}
\end{prop}
\begin{proof}
From the proofs of Lemma \ref{lem:bdryintegralcont} and hence Proposition \ref{prop:diff}, it is clear that both statements hold for Lipschitz functions $\varphi$ (or $\psi=\varphi|_{\Sigma}$, respectively) 
if we systematically replace $\|\cdot\|_{\hold^{1}(\Sigma)}$ by $\|\cdot\|_{\mathrm{Lip}(\Sigma)}$. 
Based on this modification, assertion \ref{item:welldef1} follows from the 
obvious $n$-dimensional variant of Proposition \ref{prop:diff} with $\mathbf{v}=\mathbf{G}$ and Definition \ref{def:locdivL1}. 
For such $t\in\mathscr{I}$, assertion \ref{item:GGdist} then is an immediate 
consequence of Definition \ref{def:locdivL1}. This completes the proof. 
\end{proof}

For our future applications, we explicitly record the following corollary.
\begin{corollary}\label{cor:contstok}
In the situation of {\rm Proposition \ref{prop:StokesPreparatory}}, suppose that $\GG\in\lebe^{1}(\partial\Omega';T_{\partial\Omega'})$ has an $\lebe_{\locc}^{1}$-representative which is continuous in a relatively open neighborhood of $\Gamma_{\Sigma}$ in $\partial\Omega'$. Then there exists $0<t_{0}<\frac{1}{4}$ such that 
conclusions \ref{item:welldef1}--\ref{item:GGdist} from {\rm Proposition \ref{prop:StokesPreparatory}} 
hold on the \emph{entire} interval $\mathscr{I}':=[0,t_{0})$.   
\end{corollary}
\begin{proof}
As we noted in the preceding proof, Lemma \ref{lem:bdryintegralcont} holds
for Lipschitz functions $\varphi$ (or $\psi=\varphi|_{\Sigma}$, respectively). However, the statement still requires the underlying vector field $\mathbf{v}$ to be of class $\hold^{1}$. We now prove that \eqref{eq:lim0a} holds
for all $\mathbf{v}\in\hold(\partial\Omega';T_{\partial\Omega'})$ 
and all $\varphi\in\mathrm{Lip}(\partial\Omega')$; yet, note that 
estimate \eqref{eq:lim00a} does not extend to merely continuous fields $\mathbf{v}$. Once this is achieved, we multiply $\mathbf{G}$ as in the present corollary with a Lipschitz localizer (which in turn determines $t_{0}$) so that $\mathbf{G}$ can indeed be assumed to be globally continuous on $\partial\Omega'$. We then conclude by setting $\mathbf{v}=\mathbf{G}$.

Hence, let $\mathbf{v}\in\hold(\partial\Omega';T_{\partial\Omega'})$ and $\varphi\in\mathrm{Lip}(\partial\Omega')$. Let $\varepsilon>0$ be arbitrary. By the obvious $n$-dimensional variant of Lemma \ref{lem:Lipbounds} and  \eqref{eq:steinalgen}, there exists a constant $c>0$ such that 
\begin{align}\label{eq:fixyou}
    \sup_{0<\delta<\frac{1}{4}}\int_{\Sigma^{\tau,t}}|\nabla_{\tau}\psi_{t,\delta,\Sigma}|\dif\mathscr{H}^{2} \leq c.
    \end{align}
By density (see Proposition \ref{prop:smoothapp} in Appendix \hyperref[sec:AppendixA2]{A.2}), 
we find $\mathbf{v}_{\varepsilon}\in\hold^{1}(\partial\Omega';T_{\partial\Omega'})$ 
such that $\|\mathbf{v}-\mathbf{v}_{\varepsilon}\|_{\lebe^{\infty}(\partial\Omega')}<\varepsilon$. 
Then 
\begin{align*}
 \mathrm{I} & := \left\vert  \int_{\Gamma_{\Sigma}^{t}}\varphi\mathbf{v}\cdot\nu_{\Gamma_{\Sigma}^{t}}\dif\mathscr{H}^{1} - \int_{\Sigma^{\tau,t}}\varphi\mathbf{v}\cdot\nabla_{\tau}\psi_{t,\delta,\Sigma}\dif\mathscr{H}^{2}\right\vert \\ 
 &  \leq  \int_{\Gamma_{\Sigma}^{t}}|\varphi|\,|\mathbf{v}-\mathbf{v}_{\varepsilon}| \dif\mathscr{H}^{1} + \int_{\Sigma^{\tau,t}}|\varphi|\,|\mathbf{v}-\mathbf{v}_{\varepsilon}|\,|\nabla_{\tau}\psi_{t,\delta,\Sigma}|\dif\mathscr{H}^{2} \\ 
 &\quad\; + \left\vert  \int_{\Gamma_{\Sigma}^{t}}\varphi\mathbf{v}_{\varepsilon}\cdot\nu_{\Gamma_{\Sigma}^{t}}\dif\mathscr{H}^{1} - \int_{\Sigma^{\tau,t}}\varphi\mathbf{v}_{\varepsilon}\cdot\nabla_{\tau}\psi_{t,\delta,\Sigma}\dif\mathscr{H}^{2}\right\vert \\ 
    & \leq \varepsilon\,\|\varphi\|_{\lebe^{\infty}(\partial\Omega')}
       \big(\mathscr{H}^{1}(\Gamma_{\Sigma}^{t}) + c\big) 
       + \left\vert  \int_{\Gamma_{\Sigma}^{t}}\varphi\mathbf{v}_{\varepsilon}\cdot\nu_{\Gamma_{\Sigma}^{t}}\dif\mathscr{H}^{1} - \int_{\Sigma^{\tau,t}}\varphi\mathbf{v}_{\varepsilon}\cdot\nabla_{\tau}\psi_{t,\delta,\Sigma}\dif\mathscr{H}^{2}\right\vert. 
    \end{align*}
    We then use the obvious $n$-dimensional variant of Lemma \ref{lem:bdryintegralcont} with Lipschitz functions $\varphi$ to send $\delta\searrow 0$ in the overall inequality, whereby the ultimate term vanishes. By arbitrariness of $\varepsilon>0$, we may send $\varepsilon\searrow 0$ to deduce that $\mathrm{I}=0$. This completes the proof.  
\end{proof}

\subsection{The boundary pairing $\mathcal{T}$ and tangentiality}\label{sec:bdrypairing}
For the general Stokes theorem, Theorem \ref{thm:StokesDist}, 
it is useful to give an intrinsic pairing on the boundaries of subsets of $\Omega$. 
This allows us to clarify in which sense the distributions $\langle(\FF\times\nu)_{\Omega'}^{\mathrm{int}},\cdot\rangle_{\partial\Omega'}$ and $\langle(\FF\times\nu)_{\Omega'}^{\mathrm{ext}},\cdot\rangle_{\partial\Omega'}$ are tangential to $\partial\Omega'$; see Corollary \ref{cor:tangentiality}. We begin with 

\begin{lem}[Boundary Pairing]\label{lem:bdrypairingSec7}
Let $\Omega\subset\R^{3}$ be open and bounded, and let $\Omega'\Subset\Omega$ be open with $\hold^{2}$-boundary. Moreover, let $1\leq p<\infty$ and $\FF\in\mathscr{CM}^{p}(\Omega)$. Define
\begin{align}\label{eq:bdrypairing1A}
    \mathcal{T}_{(\FF\times\nu)_{\Omega'}^{\mathrm{int}}}(\bphi) := \langle(\FF\times\nu)_{\Omega'}^{\mathrm{int}},\overline{\bphi}\rangle_{\partial \Omega'}\qquad \mbox{for $\bphi\in\mathrm{Lip}(\partial\Omega';\R^{3})$},  
    \end{align}
    where $\overline{\bphi}\in\mathrm{Lip}_{\rm c}(\Omega;\R^{3})$ is  arbitrary with $\overline{\bphi}|_{\partial\Omega'}=\bphi$ on $\partial \Omega'$. Then $  \mathcal{T}_{(\FF\times\nu)_{\Omega'}^{\mathrm{int}}}$ is a \emph{well-defined bounded linear operator} from $\mathrm{Lip}(\partial\Omega';\R^{3})$ to $\R$. In particular, it is independent of the specific choice of $\overline{\bphi}$ with the above properties. The same holds true  with the natural modifications  for $\mathcal{T}_{(\FF\times\nu)_{\Omega'}^{\mathrm{ext}}}$.
\end{lem}
\begin{proof}
Let $\overline{\bphi}_{1},\overline{\bphi}_{2}\in\mathrm{Lip}_{\rm c}(\Omega;\R^{3})$ be such that $\overline{\bphi}_{1}=\overline{\bphi}_{2}$ on $\partial\Omega'$. Then $\bm{\psi}:=(\overline{\bphi}_{1}-\overline{\bphi}_{2})|_{\Omega'}\in\mathrm{Lip}_{0}(\Omega';\R^{3})$. Since $\hold_{\rm c}^{\infty}(\Omega';\R^{3})$ is weak*-dense in $\mathrm{Lip}_{0}(\Omega';\R^{3})$, we find a sequence $(\bm{\psi}_{j})\subset\hold_{\rm c}^{\infty}(\Omega';\R^{3})$ such that $\bm{\psi}_{j}\to\bm{\psi}$ strongly in $\hold_{0}(\Omega';\R^{3})$ and $\nabla\bm{\psi}_{j}\stackrel{*}{\rightharpoonup}\nabla\bm{\psi}$ in $\lebe^{\infty}(\Omega';\R^{3\times 3})$, respectively. By Lemma \ref{lem:smoothapprox}\ref{item:SMAP1}, we may  choose a sequence $(\FF_{k})\subset\hold_{\rm c}^{\infty}(\Omega;\R^{3})$ such that $\FF_{k}\to\FF$ strongly in $\lebe^{1}(\omega;\R^{3})$ and $(\curl\FF_{k})\mathscr{L}^{3}\stackrel{*}{\rightharpoonup}\curl\FF$ in $\mathrm{RM}_{\mathrm{fin}}(\omega;\R^{3})$ for any open $\omega\Subset\Omega$. 
We then obtain  
\begin{align*}
\langle(\FF\times\nu)_{\Omega'}^{\mathrm{int}},{\bm{\psi}}\rangle_{\partial\Omega'} & \stackrel{\eqref{eq:vectorialtrace}}{=} - \int_{\Omega'}\FF\cdot\curl\bm{\psi}\dif x + \int_{\Omega'}\bm{\psi}\cdot\dif\,(\curl\FF) \\ 
& \;= \lim_{j\to\infty}\Big(- \int_{\Omega'}\FF\cdot\curl\bm{\psi}_{j}\dif x + \int_{\Omega'}\bm{\psi}_{j}\cdot\dif\,(\curl\FF)\Big) \\ 
& \;= \lim_{j\to\infty}\lim_{k\to\infty}\Big(- \int_{\Omega'}\FF_{k}\cdot\curl\bm{\psi}_{j}\dif x + \int_{\Omega'}\bm{\psi}_{j}\cdot\dif\,(\curl\FF_{k})\Big) \\ 
& = \lim_{j\to\infty}\lim_{k\to\infty} \int_{\partial\Omega'}(\FF_{k}\times\nu_{\partial\Omega'})\cdot\bm{\psi}_{j}\dif x = 0, 
\end{align*}
where we have used the classical integration-by-parts formula for curl in the penultimate step; see formula \eqref{eq:calculus7} in Appendix \hyperref[sec:AppendixD]{D}. From here, the claim is immediate. 
\end{proof}

\begin{rem}\label{rem:onesided}
Note that the right-hand side of \eqref{eq:bdrypairing1A}, in turn defined via \eqref{eq:vectorialtrace}, only incorporates the values of $\overline{\bphi}\in\mathrm{Lip}_{\rm c}(\Omega;\R^{3})$ on $\Omega'$. Since any $\bphi\in\mathrm{Lip}(\Omega';\R^{3})=\mathrm{Lip}(\overline{\Omega'};\R^{3})$ can be extended to an element of $\mathrm{Lip}_{\rm c}(\Omega;\R^{3})$, we have 
\begin{align}\label{eq:tabealove}
    \mathcal{T}_{(\FF\times\nu)_{\Omega'}^{\mathrm{int}}}(\bphi) = \langle(\FF\times\nu)_{\Omega'}^{\mathrm{int}},\overline{\bphi}\rangle_{\partial \Omega'} = \int_{\Omega'}\overline{\bphi}\cdot\dif\,(\curl\FF) - \int_{\Omega'}\FF\cdot\curl\overline{\bphi}\dif x
\end{align}
for any $\bphi\in\mathrm{Lip}(\partial\Omega';\R^{3})$, where $\overline{\bphi}\in\mathrm{Lip}(\Omega';\R^{3})$ with $\overline{\bphi}|_{\partial\Omega}=\bphi$ is arbitrary. 
\end{rem}

In an intermediate step, we give an explicit description of the distributional traces of $\mathscr{CM}^{p}$-fields in terms of integrals. To this end, it is useful to work with the particular collar map $\Phi_{\partial\Omega'}\colon (-1,1)\times\partial\Omega'\to\mathcal{O}\subset\Omega$ described in Remark \ref{rem:aeintro}; here, $\partial\Omega'$ is as in Lemma \ref{lem:bdrypairingSec7}. We then define
\begin{align}\label{eq:solidheight}
\psi_{\varepsilon,\Omega'}(x):=\begin{cases}
    0&\;\text{if}\;x\in\Omega\setminus\Omega', \\ 
    \frac{s}{\varepsilon} &\;\text{if}\;x\in\Phi_{\partial\Omega'}(\{s\}\times\partial\Omega')\;\text{for}\;0<s<\varepsilon, \\ 
    1 &\;\text{if}\;x\in\Omega'\setminus\Phi_{\partial\Omega'}((0,\varepsilon)\times\partial\Omega'),   
\end{cases}
\end{align}
and these height functions are the solid substitutes of those on manifolds from \eqref{eq:psiddefmain}. 

\begin{lem}[Representation of Tangential Traces]\label{lem:reptanglimit}
Let $\Omega\subset\R^{3}$ be open and bounded, and let $\Omega'\Subset\Omega$ be open with $\hold^{2}$-boundary.
Then, for any $\FF\in\mathscr{CM}^{p}(\Omega)$ with $1\leq p <\infty$ and any $\bm{\psi}\in\mathrm{Lip}(\Omega';\R^{3})$, 
\begin{align}\label{eq:distreppo1}
\langle(\FF\times\nu)_{\Omega'}^{\mathrm{int}},\bm{\psi}\rangle_{\partial\Omega'} = \lim_{\varepsilon\searrow 0}\int_{\Phi_{\partial\Omega'}((0,\varepsilon)\times\partial\Omega')} (\FF\times\nabla \psi_{\varepsilon,\Omega'})\cdot\bm{\psi}\dif x, 
\end{align}
where  $\psi_{\varepsilon,\Omega'}$ are the height functions from \eqref{eq:solidheight}.
Thus, for all $\bphi\in\mathrm{Lip}(\partial\Omega';\R^{3})$ 
and any $\overline{\bphi}\in\mathrm{Lip}(\Omega';\R^{3})$ with $\overline{\bphi}|_{\partial\Omega'}=\bphi$,
\begin{align}\label{eq:distreppo2}
\mathcal{T}_{(\FF\times\nu)_{\Omega'}^{\mathrm{int}}}(\bphi)= \lim_{\varepsilon\searrow 0}\int_{\Phi_{\partial\Omega'}((0,\varepsilon)\times\partial\Omega')} (\FF\times\nabla \psi_{\varepsilon,\Omega'})\cdot\overline{\bphi}\dif x. 
\end{align}
The same holds true with the natural modifications for $\mathcal{T}_{(\FF\times\nu)_{\Omega'}^{\mathrm{ext}}}$.
\end{lem}
\begin{proof}
We begin with \eqref{eq:distreppo1} and may conclude as in Lemma \ref{lem:Lipbounds} 
that $\psi_{\varepsilon,\Omega'}$ is Lipschitz.  
Based on Lemma \ref{lem:smoothapprox}\ref{item:SMAP1}, we first choose a sequence $(\FF_{j})\subset\hold_{\rm c}^{\infty}(\Omega;\R^{3})$ such that 
\begin{align}\label{eq:Fjapproxprops}
\FF_{j}\to\FF\;\text{strongly in}\;\lebe^{1}(\Omega;\R^{3}),\qquad (\curl\FF_{j})\mathscr{L}^{3}\stackrel{*}{\rightharpoonup}\curl\FF\;\;\text{in}\;\mathrm{RM}_{\mathrm{fin}}(\Omega;\R^{3}). 
\end{align}
Now let $\bm{\psi}\in\mathrm{Lip}(\Omega';\R^{3})$ and let $\varepsilon>0$ be sufficiently small.  For $s>0$, we abbreviate $\Omega'_{s\varepsilon}:=\Omega'\setminus\Phi_{\partial\Omega'}((0,s\varepsilon]\times\partial\Omega')$ 
and define 
\begin{align*}
A_{\varepsilon}:=\big\{s\in (0,1)\,\colon\;|\curl(\FF)|(\Phi_{\partial\Omega'}(\{s\varepsilon\}\times\partial\Omega'))=0\big\},  
\end{align*}
whereby $\mathscr{L}^{1}((0,1)\setminus A_{\varepsilon})=0$. Setting 
\begin{align*}
    \mathbf{G}_{j}(x) := \begin{cases}\Big(\FF_{j}(x)\times\tfrac{\nabla\psi_{\varepsilon,\Omega'}(x)}{|\nabla\psi_{\varepsilon,\Omega'}(x)|}\Big)\cdot\bm{\psi}(x)&\;\text{if}\;\nabla\psi_{\varepsilon,\Omega'}(x)\neq 0, \\ 
    0&\;\text{otherwise},
    \end{cases}
\end{align*}
we employ the usual coarea formula (see \emph{e.g.}  \cite[\S 3.4.2, Theorem 1]{EvansGariepy}) to obtain 
\begin{align*}
&\int_{\Phi_{\partial\Omega'}((0,\varepsilon)\times\partial\Omega')}(\FF\times\nabla\psi_{\varepsilon,\Omega'})\cdot\bm{\psi}\dif x \\  
&\stackrel{\eqref{eq:Fjapproxprops}}{=}  \lim_{j\to\infty} 
\int_{\Phi_{\partial\Omega'}((0,\varepsilon)\times\partial\Omega')}(\FF_{j}\times\nabla\psi_{\varepsilon,\Omega'})\cdot\bm{\psi}\dif x \\ &   = \lim_{j\to\infty}\int_{\Phi_{\partial\Omega'}((0,\varepsilon)\times\partial\Omega')}\mathbf{G}_{j}|\nabla\psi_{\varepsilon,\Omega'}|\dif x \\ 
& = \lim_{j\to\infty} \int_{0}^{1}\int_{\Phi_{\partial\Omega'}(\{s\varepsilon\}\times\partial\Omega')}\mathbf{G}_{j}\dif\mathscr{H}^{2}\dif s \\ 
& = \lim_{j\to\infty} \int_{0}^{1}\int_{\Phi_{\partial\Omega'}(\{s\varepsilon\}\times\partial\Omega')}(\FF_{j}\times\nu_{\partial\Omega'_{s\varepsilon}})\cdot\bm{\psi}\dif\mathscr{H}^{2}\dif s \\ 
& \!\!\! \stackrel{\eqref{eq:calculus7}}{=}  \lim_{j\to\infty} \int_{(0,1)\cap A_{\varepsilon}}\int_{\Omega'_{s\varepsilon}} (\curl\FF_{j})\cdot\bm{\psi} - \FF_{j}\cdot\curl\bm{\psi}\dif x\dif s =: \mathrm{I}.
\end{align*}
By \eqref{eq:Fjapproxprops}, the inner integral can be bounded, independently of $j$ and $\varepsilon$. 
By the definition of $A_{\varepsilon}$, the dominated convergence theorem thus yields 
\begin{align*}
\mathrm{I} & =  \int_{(0,1)\cap A_{\varepsilon}}\int_{\Omega'_{s\varepsilon}} \bm{\psi}\cdot\dif\,(\curl\FF)\dif s - \int_{(0,1)\cap A_{\varepsilon}}\int_{\Omega'_{s\varepsilon}}\FF\cdot\curl\bm{\psi}\dif x\dif s \\ 
& \!\!\! \stackrel{\varepsilon\searrow 0}{\longrightarrow }  \int_{\Omega'} \bm{\psi}\cdot\dif\,(\curl\FF) - \int_{\Omega'} \FF\cdot\curl\bm{\psi}\dif x  \stackrel{\eqref{eq:vectorialtrace}}{=}  \langle(\FF\times\nu)_{\Omega'}^{\mathrm{int}},\bm{\psi}\rangle_{\partial\Omega'}.
\end{align*}
This completes the proof. 
\end{proof}

We now come to the final main result of this subsection.

\begin{theorem}[Distributional Tangentiality]\label{cor:tangentiality} 
In the situation of {\rm Lemma \ref{lem:bdrypairingSec7}}, we have  
\begin{align}\label{eq:tangentialityDist1}
 \mathcal{T}_{(\FF\times\nu)_{\Omega'}^{\mathrm{int}}}(\bphi)= \mathcal{T}_{(\FF\times\nu)_{\Omega'}^{\mathrm{int}}}(\bphi_{\tau})\qquad\text{for all}\;\bphi\in\mathrm{Lip}(\partial\Omega';\R^{3}),
\end{align}
where $\bphi_{\tau}$ denotes the tangential component of $\bphi$ as usual. In particular, we have 
\begin{align}\label{eq:tangentialityDist2}
 \mathcal{T}_{(\FF\times\nu)_{\Omega'}^{\mathrm{int}}}(\nabla\varphi)= \mathcal{T}_{(\FF\times\nu)_{\Omega'}^{\mathrm{int}}}(\nabla_{\tau}\varphi)\qquad\text{for all}\;\varphi\in\hold^{2}(\partial\Omega').
\end{align}
The same holds true with the natural modifications for $\mathcal{T}_{(\FF\times\nu)_{\Omega'}^{\mathrm{ext}}}$.
\end{theorem}

\begin{proof}
For  \eqref{eq:tangentialityDist1}, we may directly assume that $\bphi\in\mathrm{Lip}_{\rm c}(\Omega;\R^{3})$. Let $0<\varepsilon<\varepsilon_{0}<1$ be sufficiently small. For notational convenience, we  abbreviate $S_{\varepsilon}:=\Phi_{\partial\Omega'}((0,\varepsilon)\times\partial\Omega')$ and, for $0<t<\varepsilon_{0}$, put $\Omega_{t}:=\Omega'\setminus\Phi_{\partial\Omega'}((0,t)\times\partial\Omega')$, so that 
$\partial\Omega_{t}:=\Phi_{\partial\Omega'}(\{t\}\times\partial\Omega')$. 
In what follows, we let $\Pi_{2}\colon (0,1)\times\partial\Omega'\to\partial\Omega'$ be the projection on the second component and define $T:=\Pi_{2}\circ \Phi_{\partial\Omega'}^{-1}$. We define 
\begin{align}\label{eq:vogel}
\widetilde{\bphi}(x):=\rho(x)\bphi_{\tau}(T(x)) =\rho(x)\big(\bphi(T(x))-(\bphi(T(x))\cdot\nu_{\partial\Omega'}(T(x)))\nu_{\partial\Omega'}(T(x)) \big)
\end{align}
for $x\in\Omega'$,
where $\rho\in\hold^{\infty}(\overline{\Omega'};[0,1])$ is such that $\rho=1$ in $S_{\frac{1}{2}}$ and $\rho=0$ in $\Omega'\setminus S_{\frac{3}{4}}$. 

We firstly claim that $\widetilde{\bphi}\colon\Omega'\to\R^{3}$ is Lipschitz. 
To this end, we note that $\Phi_{\partial\Omega'}^{-1}\colon S_{\frac{3}{4}}\to (0,1)\times\partial\Omega'$ 
is Lipschitz and, since $\Pi_{2}$ is the projection onto second component, $\Pi_{2}$ is also Lipschitz.
Thus, $T$ is Lipschitz, whose Lipschitz constant is denoted by $L_{T}$.
Analogously, we denote the Lipschitz constants of $\bphi$ and $\rho$ by $L_{\bphi}$ and $L_{\rho}$. 
Finally, since $\partial\Omega'$ is compact and of class $\hold^{2}$, the inner unit normal field $\nu_{\partial\Omega'}\colon\partial\Omega'\to\mathbb{S}^{2}$ is also Lipschitz, whose Lipschitz constant
is denoted by $L_{\nu}$. Now let $x,y\in\Omega'\setminus S_{\frac{3}{4}}$. We estimate 
\begin{align*}
&|\widetilde{\bphi}(x)-\widetilde{\bphi}(y)| \\
& \leq \rho(x)|\bphi(T(x))-\bphi(T(y))|
+ |\rho(x)-\rho(y)|\,|\bphi(T(y))| \\
&\quad+ \rho(x)|(\bphi(T(x))\cdot\nu_{\partial\Omega'}(T(x)))\nu_{\partial\Omega'}(T(x))- (\bphi(T(y))\cdot\nu_{\partial\Omega'}(T(x)))\nu_{\partial\Omega'}(T(x))|\\ 
&\quad + \rho(x)|(\bphi(T(y))\cdot\nu_{\partial\Omega'}(T(x)))\nu_{\partial\Omega'}(T(x))- (\bphi(T(y))\cdot\nu_{\partial\Omega'}(T(y)))\nu_{\partial\Omega'}(T(x))|\\ 
&\quad+ |\rho(x)-\rho(y)|\,| (\bphi(T(y))\cdot\nu_{\partial\Omega'}(T(y)))\nu_{\partial\Omega'}(T(x))| \\ 
&\quad + \rho(y)|(\bphi(T(y))\cdot\nu_{\partial\Omega'}(T(y)))\nu_{\partial\Omega'}(T(x))- (\bphi(T(y))\cdot\nu_{\partial\Omega'}(T(y)))\nu_{\partial\Omega'}(T(y))| \\ 
&\leq L_{\bphi}L_{T}|x-y| + L_{\rho}\|\bphi\|_{\lebe^{\infty}(\partial\Omega')}|x-y| + L_{\bphi}L_{T}|x-y|+ \|\bphi\|_{\lebe^{\infty}(\partial\Omega')}L_{\nu}L_{T}|x-y| \\ 
& \quad + L_{\rho}\|\bm{\varphi}\|_{\lebe^{\infty}(\partial\Omega')}|x-y|+L_{\nu}\|\bm{\varphi}\|_{\lebe^{\infty}(\partial\Omega')}L_{T}|x-y| \\ 
&=: C(\rho,\bphi,T,\partial\Omega')|x-y|. 
\end{align*}
Whenever $x\in\partial\Omega'$ and $(x_{j})\subset\Omega'$ are such that $x_{j}\to x$, then
\begin{align*}
|\bphi_{\tau}(T(x_{j}))-\bphi_{\tau}(x)|\leq |\bphi(T(x_{j}))-\bphi(x)|\leq L_{\bphi}|T(x_{j})-x| \to 0,\qquad j\to\infty, 
\end{align*}
so that
\begin{align*}
\widetilde{\bphi}(x_{j}) & = \bphi_{\tau}(T(x_{j})) \to \bphi_{\tau}(x),\qquad j\to\infty. 
\end{align*}
It follows that 
\begin{align}\label{eq:ware}
\widetilde{\bphi}(x):=\begin{cases}
    \widetilde{\bphi}(x)\;\text{as defined by \eqref{eq:vogel}}&\;\text{if}\;x\in\Omega',\\ 
    \bphi_{\tau}(x)&\;\text{if}\;x\in\partial\Omega' 
\end{cases}
\end{align}
is a Lipschitz extension of $\bphi_{\tau}\colon\partial\Omega'\to T_{\partial\Omega'}$. In particular, it is admissible in \eqref{eq:distreppo1}. We now  choose a sequence $(\FF_{j})\subset\hold_{\rm c}^{\infty}(\Omega;\R^{3})$ such that $\FF_{j}\to\FF$ strongly in $\lebe^{1}(\Omega';\R^{3})$. In consequence,
\begin{align}\label{eq:gosling}
 \begin{split}
&\int_{S_{\varepsilon}}\bphi\cdot(\FF\times\nabla \psi_{\varepsilon,\Omega'})\dif x\\
& = \int_{S_{\varepsilon}}\widetilde{\bphi}\cdot(\FF\times\nabla \psi_{\varepsilon,\Omega'})\dif x 
+ \lim_{j\to\infty}\int_{S_{\varepsilon}}\bphi\cdot(\FF_{j}\times\nabla\psi_{\varepsilon,\Omega'})-\widetilde{\bphi}\cdot(\FF_{j}\times\nabla\psi_{\varepsilon,\Omega'})\dif x\\
&=: \mathrm{I}_{\varepsilon} + \mathrm{II}_{\varepsilon}. 
    \end{split}
    \end{align}
Since $\varepsilon>0$ is assumed small, we may assume that $\rho=1$ on $S_{\varepsilon}$. For $x\in \partial\Omega_{t}=\Phi_{\partial\Omega'}(\{t\}\times\partial\Omega')$ with $0<t\leq\varepsilon$, we let $\nu_{t}(x)$ be the inner unit normal to $\partial\Omega_{t}$ at $x$ and 
\begin{align*}
& \bphi_{\tau,t}(x):=\bphi(x)-(\bphi(x)\cdot\nu_{t}(x))\nu_{t}(x). 
\end{align*}
Setting $\mathrm{III}(x):=|\bphi_{\tau,t}(x)-\widetilde{\bphi}(x)|$, we estimate
\begin{align*}
\mathrm{III}(x) & = |\bphi(x)-(\bphi(x)\cdot\nu_{t}(x))\nu_{t}(x)- (\bphi(T(x))-(\bphi(T(x))\cdot\nu_{\partial\Omega'}(T(x)))\nu_{\partial\Omega'}(T(x)))| \\ 
& \leq |\bphi(x)-\bphi(T(x))| + |(\bphi(x)\cdot\nu_{t}(x))\nu_{t}(x)-(\bphi(x)\cdot\nu_{\partial\Omega'}(T(x)))\nu_{t}(x)| \\ 
& \quad +|(\bphi(x)\cdot\nu_{\partial\Omega'}(T(x)))\nu_{t}(x)-(\bphi(T(x))\cdot\nu_{\partial\Omega'}(T(x)))\nu_{t}(x)| \\ 
&\quad + |(\bphi(T(x))\cdot\nu_{\partial\Omega'}(T(x)))\nu_{t}(x)- (\bphi(T(x))\cdot\nu_{\partial\Omega'}(T(x)))\nu_{\partial\Omega'}(T(x))| \\ 
& \leq L_{\bphi}|x-T(x)|+ \|\bphi\|_{\lebe^{\infty}(\partial\Omega')}|\nu_{t}(x)-\nu_{\partial\Omega'}(T(x))| +L_{\bphi}|x-T(x)| \\ &\quad  + \|\bphi\|_{\lebe^{\infty}(\partial\Omega')}|\nu_{t}(x)-\nu_{\partial\Omega'}(T(x))|. 
\end{align*}
As we mentioned prior to Lemma \ref{lem:reptanglimit}, we work with collar maps $\Phi_{\partial\Omega'}$ which satisfy the properties listed in Remark \ref{rem:aeintro}. In particular, 
by Remark \ref{rem:aeintro}\ref{item:collario4}--\ref{item:collario5}, we have 
\begin{align}\label{eq:dexterdebra}
\begin{split}
& \sup_{y\in\partial\Omega'}|y-\Phi_{\partial\Omega'}(t,y)|\leq ct\qquad\text{for all}\;0<t<t_{0}, \\ 
& \sup_{y\in\partial\Omega'}|\nu_{\partial\Omega'}(y)-\nu_{t}(\Phi_{\partial\Omega'}(t,y))|\leq Ct\qquad \text{for all}\;0<t<t_{0}, 
\end{split}
\end{align}
where $t_{0}>0$ is suitably small and $c,C>0$ are constants independent of $0<t<t_{0}$. For fixed $0<t<t_{0}$, 
$\Phi_{\partial\Omega'}(t,\cdot)\colon \partial\Omega'\to \partial\Omega_{t}$ is bijective. 
Hence, there exists a unique $y\in\partial\Omega'$ such that $x=\Phi_{\partial\Omega'}(t,y)$ and so $y=\Pi_{2}\circ\Phi_{\partial\Omega'}^{-1}(x)=T(x)$. In consequence, \eqref{eq:dexterdebra} gives us both  
\begin{align*}
|x-T(x)|\leq ct\;\;\;\text{and}\;\;\;
|\nu_{t}(x)-\nu_{\partial\Omega'}(T(x))| \leq Ct. 
\end{align*}
Going back to the estimation of $\mathrm{III}$, we thus find that there exists a constant $C=C(\bphi,T,\Phi_{\partial\Omega'})>0$ such that 
\begin{align}\label{eq:harryscodex}
& \mathrm{III}(x) \leq Ct\qquad\text{for all}\;0<t<\min\{\varepsilon,t_{0}\}. 
\end{align}
By the definition of $\psi_{\varepsilon,\Omega'}$, there exists a constant $C>0$ such that $|\nabla\psi_{\varepsilon,\Omega'}|\leq\frac{C}{\varepsilon}$ for all sufficiently small $\varepsilon>0$; see \eqref{eq:solidheight}.
Now note that, for all $0<t<\varepsilon$, 
\begin{align*}
\int_{\partial\Omega_{t}}\bphi\cdot(\FF_{j}\times\nabla\psi_{\varepsilon,\Omega'})\dif\mathscr{H}^{2} = \int_{\partial\Omega_{t}}{\bphi}_{\tau,t}\cdot(\FF_{j}\times\nu_{\partial\Omega_{t}})|\nabla\psi_{\varepsilon,\Omega'}|\dif\mathscr{H}^{2},
\end{align*}
since  $(\FF_{j}\times\nu_{\partial\Omega_{t}})\bot\,\nu_{\partial\Omega_{t}}$ holds $\mathscr{H}^{2}$-a.e. on $\partial\Omega_{t}$ in the smooth context. 
The usual coarea formula implies that, for all $0<\varepsilon<\min\{t_{0},\frac{1}{4}\}$, 
\begin{align}\label{eq:batista}
\begin{split}
|\mathrm{II}_{\varepsilon}| & \leq \lim_{j\to\infty} \int_{0}^{1}\int_{\Phi_{\partial\Omega'}(\{t\varepsilon\}\times\partial\Omega')\cap\{|\nabla\psi_{\varepsilon,\Omega'}|\neq 0\}} \mathrm{III}(x)\,\left\vert \FF_{j}(x)\times\frac{\nabla\psi_{\varepsilon,\Omega'}(x)}{|\nabla\psi_{\varepsilon,\Omega'}(x)|}
\right\vert\dif\mathscr{H}^{2}(x)\dif t \\ 
& \leq \lim_{j\to\infty} \frac{1}{\varepsilon} \int_{0}^{\varepsilon}\int_{\Phi_{\partial\Omega'}(\{t\}\times\partial\Omega')} \mathrm{III}(x)\,|\FF_{j}(x)|\dif\mathscr{H}^{2}(x)\dif t \\ 
& \!\!\!\!\stackrel{\eqref{eq:harryscodex}}{\leq} C\lim_{j\to\infty}\int_{S_{\varepsilon}}|\FF_{j}|\dif x \leq \int_{S_{\varepsilon}}|\FF|\dif x \stackrel{\varepsilon\searrow 0}{\to} 0. 
\end{split}
\end{align}
We thus conclude by Lemma \ref{lem:reptanglimit} that 
\begin{align*}
\langle(\FF\times\nu)_{\Omega'}^{\mathrm{int}},\bphi\rangle_{\partial\Omega'} & = \lim_{\varepsilon\searrow 0}\int_{\Phi_{\partial\Omega'}((0,\varepsilon)\times\partial\Omega')}(\FF\times\nabla\psi_{\varepsilon,\Omega'})\cdot\bphi\dif x \\ 
& \!\!\!\!\!\!\!\!\!\!\!  \stackrel{\eqref{eq:gosling},\,\eqref{eq:batista}}{=} \lim_{\varepsilon\searrow 0}\int_{\Phi_{\partial\Omega'}((0,\varepsilon)\times\partial\Omega')} (\FF\times\nabla\psi_{\varepsilon,\Omega'})\cdot\widetilde{\bphi}\dif x = \langle(\FF\times\nu)_{\Omega'}^{\mathrm{int}},\widetilde{\bphi}\rangle_{\partial\Omega'}. 
\end{align*}
In view of Lemma \ref{lem:bdrypairingSec7} and \eqref{eq:ware}\emph{ff.}, we conclude the validity of \eqref{eq:tangentialityDist1}. Since \eqref{eq:tangentialityDist2} is a direct consequence of \eqref{eq:tangentialityDist1}, this completes the proof.  
\end{proof}

\subsection{The general Stokes theorem}\label{sec:Stokestheorem}
We now come to the Stokes theorem for a large subclass of $\mathscr{CM}^{p}$-fields to be defined next, and so let $\Omega'\Subset\Omega$ be open with $\hold^{2}$-boundary. We denote by $\mathrm{Lip}_{2}(\partial\Omega')$ the space of all Lipschitz continuous functions $\varphi\colon\partial\Omega'\to\R$ for which their weak tangential gradients $\nabla_{\tau}\varphi$ belong to $\mathrm{Lip}(\partial\Omega';T_{\partial\Omega'})$. For $\Sbold\in\mathrm{Lip}(\partial\Omega';T_{\partial\Omega'})'$ and {$\varphi\in\mathrm{Lip}_{2}(\partial\Omega')$},  $\mathrm{div}_{\tau}(\Sbold)$ is defined by 
\begin{align}\label{eq:definitiondivtandist}
\langle\mathrm{div}_{\tau}(\Sbold),\varphi\rangle = -\langle \Sbold,\nabla_{\tau}\varphi\rangle. 
\end{align}
Note that we require $\varphi\in\mathrm{Lip}_{2}(\partial\Omega')$ here, as otherwise the right-hand side of \eqref{eq:definitiondivtandist} is ill-defined. Moreover, if $\GG\in\lebe^{1}(\partial\Omega';T_{\partial\Omega'})$, then $\mathrm{div}_{\tau}(\GG)$ belongs to $\mathrm{Lip}(\partial\Omega')'$, and we understand the statement $\mathrm{div}_{\tau}(\Sbold)\in\mathrm{div}_{\tau}(\lebe^{1}(\partial\Omega';T_{\partial\Omega'}))$ as the equality{\rm :} 
\begin{align}\label{eq:precisemeaningcontainment}
\langle\mathrm{div}_{\tau}(\Sbold),\varphi\rangle =\langle\mathrm{div}_{\tau}(\GG),\varphi\rangle\qquad\text{for all}\;\varphi\in\mathrm{Lip}_{2}(\partial\Omega')
\end{align}
with a suitable $\GG\in\lebe^{1}(\partial\Omega';T_{\partial\Omega'})$. If $\partial\Omega'$ is of class $\hold^{\infty}$, then one can give a uniform approach to the regularity of test maps 
in \eqref{eq:definitiondivtandist}--\eqref{eq:precisemeaningcontainment}; 
see Remark \ref{rem:uniformiseregularity} below. 

\begin{definition}[The Classes $\mathscr{CM}_{1,\partial\Omega'}^{p,\mathrm{int}}$ 
and $\mathscr{CM}_{1,\partial\Omega'}^{p,\mathrm{ext}}$]\label{def:Cm1}
Let $\Omega\subset\R^{3}$ be open and bounded, let $\Omega'\Subset\Omega$ be open with $\hold^{2}$-boundary, and let $\Sigma$ be a $\hold^{2}$-regular Lipschitz boundary manifold relative to $\Omega'$. 
For $1\leq p \leq \infty$, define 
\begin{align}\label{eq:definingcondCm1p}
\mathscr{CM}_{1,\partial\Omega'}^{p,\mathrm{int}}(\Omega) 
:=\Big\{\FF\in\mathscr{CM}^{p}(\Omega)\,\colon\;  \mathrm{div}_{\tau}(\mathcal{T}_{(\FF\times\nu)_{\Omega'}^{\mathrm{int}}})\in\mathrm{div}_{\tau}(\lebe^{1}(\partial\Omega';T_{\partial\Omega'})) \Big\}, 
\end{align}
where $\mathcal{T}_{(\FF\times\nu)_{\Omega'}^{\mathrm{int}}}$ is as in {\rm Definition \ref{lem:bdrypairingSec7}}, 
and the defining condition in \eqref{eq:definingcondCm1p} is understood in the sense of \eqref{eq:precisemeaningcontainment}. The space $\mathscr{CM}_{1,\partial\Omega'}^{p,\mathrm{ext}}(\Omega)$ is defined analogously:
\begin{align}\label{eq:definingcondCm1pext}
\mathscr{CM}_{1,\partial\Omega'}^{p,\mathrm{ext}}(\Omega) 
:=\Big\{\FF\in\mathscr{CM}^{p}(\Omega)\,\colon\;  
\mathrm{div}_{\tau}
(\mathcal{T}_{(\FF\times\nu)_{\Omega'}^{\mathrm{ext}}})\in\mathrm{div}_{\tau}(\lebe^{1}(\partial\Omega';T_{\partial\Omega'}))
\Big\}, 
\end{align}
Finally, define 
\begin{align}
\mathscr{CM}_{1}^{p,\mathrm{int}}(\Omega):=\bigcap_{\Omega'\Subset\Omega\;\text{\emph{open with $\hold^{2}$-boundary}}}\mathscr{CM}_{1,\partial\Omega'}^{p,\mathrm{int}}(\Omega). 
\end{align}
The space $\mathscr{CM}_{1}^{p,\mathrm{ext}}(\Omega)$ is defined by analogous means.
\end{definition}
In other words, $\FF\in\mathscr{CM}^{p}(\Omega)$ belongs to $\mathscr{CM}_{1,\partial\Omega'}^{p,\mathrm{int}}(\Omega)$ if and only if there exist $\mathbf{G}\in\lebe^{1}(\partial\Omega';T_{\partial\Omega'})$ and $\mathbf{H}\in\mathrm{Lip}(\partial\Omega')'$ such that $\mathcal{T}_{(\FF\times\nu)_{\Omega'}^{\mathrm{int}}}=\mathbf{G}+\mathbf{H}$ and $\mathrm{div}_{\tau}(\mathbf{H})=0$. Based on this consideration, it is convenient to introduce an equivalence relation '$\sim$' on $\lebe^{1}(\partial\Omega';T_{\partial\Omega'})$ via 
\begin{align}\label{eq:equivrelation}
\mathbf{G}_{1}\sim \mathbf{G}_{2} \,\Longleftrightarrow\, \mathrm{div}_{\tau}(\mathbf{G}_{1}-\mathbf{G}_{2})=0, 
\end{align}
and to write $\mathbf{G}/_{\sim}$ for the corresponding equivalence class. As we will see in Theorems \ref{thm:StokesDist} and \ref{thm:StokesWcurl} below, the flexibility in switching between different representatives turn out important, since the Stokes theorem in the vorticity flux formulation only hinges on equivalence classes rather than the specific representatives. This comes to the following effect: Even if $\mathcal{T}_{(\FF\times\nu)_{\Omega'}^{\mathrm{int}}}$ is not representable by an $\lebe^{1}(\partial\Omega';T_{\partial\Omega'})$-field, whereby we cannot directly introduce the Stokes functionals as in \eqref{eq:Stokesfunctionaldef}, it still might happen that  
\begin{align}\label{eq:goodrepvortflux}
\mathrm{div}_{\tau}(\mathcal{T}_{(\FF\times\nu)_{\Omega'}^{\mathrm{int}}}) = \mathrm{div}_{\tau}(\mathbf{G})\qquad\text{for  some}\;\mathbf{G}\in\lebe^{1}(\partial\Omega';T_{\partial\Omega'}). 
\end{align}
In this case, we may perform a localization procedure analogous to that of \S\ref{sec:stokestanvarfirst}, 
now relying on Definition \ref{def:locdivL1}\emph{ff.}. Even though the fully analogous Stokes functionals as in \eqref{eq:Stokesfunctionaldef}  \emph{do} depend on the specific choice of $\mathbf{G}$, this is \emph{not} so for the masses of the distributions from \eqref{eq:goodrepvortflux}. It is this observation that lets  us obtain unconditional Stokes theorems in the vorticity flux formulation for $\sobo^{\curl,p}$-fields in Theorem \ref{thm:StokesWcurl}; see also Remark \ref{rem:vortfluxStokesfctl}  below. For our future purposes, we record the following observation:

\begin{rem}[$p=\infty$ versus $1\leq p<\infty$]
In the situation of Definition \ref{def:Cm1}, if $p=\infty$, then $\mathscr{CM}_{1}^{\infty}(\Omega)=\mathscr{CM}^{\infty}(\Omega)$ by Theorem \ref{thm:tracemain1}. This observation leads to the consistency 
of the results of this section with those in \S \ref{sec:stokes}. 
\end{rem}

We now have
\begin{theorem}[Stokes Theorem $\mathscr{CM}_{1}^{p}$-Fields with respect to the Tangential Variations]\label{thm:StokesDist}
Let $\Omega\subset\R^{3}$ be open and bounded, and let $\Omega'\Subset\Omega$ be open and bounded with $\hold^{2}$-boundary. Moreover, let $\Sigma$ be a $\hold^{2}$-regular Lipschitz boundary manifold relative to $\Omega'$. For any $\FF\in\mathscr{CM}_{1,\partial\Omega'}^{p,\mathrm{int}}(\Omega)$ and any $\mathbf{G}\in\lebe^{1}(\partial\Omega';T_{\partial\Omega'})$ with 
\begin{align}\label{eq:lucid}
\mathrm{div}_{\tau}\big(\mathcal{T}_{(\FF\times\nu)_{\Omega'}^{\mathrm{int}}}\big) = \mathrm{div}_{\tau}(\GG)\qquad\text{as an identity in}\;\hold^{2}(\partial\Omega')', 
\end{align}
the following hold{\rm :} 
\begin{enumerate}
    \item\label{item:StokesMG1} $\mathcal{N}_{(\curl\FF)\cdot \nu,\Omega'} = \mathrm{div}_{\tau}(\GG)$ as an identity in $\hold^{2}(\partial\Omega')'$, where $\mathcal{N}_{(\curl\FF)\cdot\nu,\Omega'}$ is as in Definition \ref{def:intrinsicnormaltrace}.
    
    \item\label{item:StokesMG2} Adopting the setting and terminology of 
    {\rm Proposition \ref{prop:StokesPreparatory}}, there exists a subset $\mathscr{I}\subset I:=[0,\frac{1}{2})$ with $\mathscr{L}^{1}(I\setminus\mathscr{I})=0$ and \emph{independent of the specific choice of the $\sim$-representative $\mathbf{G}$} such that, for each $t\in\mathscr{I}$, the \emph{Stokes theorem in 
    the vorticity flux formulation} holds{\rm :}
    \begin{align}\label{eq:vortfluxdist}
        [\mathcal{N}_{\curl(\FF)\cdot\nu,\Omega'}]_{\Sigma^{\tau,t}}=-\lla\mathbf{G}\cdot\nu_{\Gamma_{\Sigma}^{t}},\mathbbm{1}_{\Sigma^{\tau,t}}\rra_{\Gamma_{\Sigma}^{t}}. 
    \end{align}
    \item\label{item:StokesMG3} Suppose moreover that a $\sim$-representative  $\GG$ has an $\lebe_{\locc}^{1}$-representative 
    that is continuous in a relatively open neighborhood of $\Gamma_{\Sigma}$ in $\partial\Omega'$. 
    Then there exists $0<t_{0}<\frac{1}{4}$ such that the conclusion of \ref{item:StokesMG2} holds on the entire interval $\mathscr{I}'=[0,t_{0})$. 
\end{enumerate}
\end{theorem}

\begin{proof}
For \ref{item:StokesMG1}, let $\varphi\in\hold^{2}(\partial\Omega')$ be arbitrary, and let $\overline{\varphi}\in\hold^{2}(\overline{\Omega'})$ be such that $\overline{\varphi}|_{\partial\Omega'}=\varphi$; such an extension can be obtained as in the proof of Lemma \ref{lem:GoodLip}. In particular, $\nabla\overline{\varphi}|_{\Omega'}\in\hold^{1}(\overline{\Omega'};\R^{3})$. For the following, we note that \eqref{eq:vectorialtrace} clearly extends to $\bphi\in\hold^{1}(\overline{\Omega'};\R^{3})$ in the form: 
\begin{align}\label{eq:tomturbo}
\langle(\FF\times\nu)_{\Omega'}^{\mathrm{int}},\bphi\rangle_{\partial\Omega'} = \int_{\Omega'}\bphi\cdot\dif\,(\curl\FF)-\int_{\Omega'}\FF\cdot\curl\bphi\dif x.  
\end{align}
Hence, 
\begin{align}\label{eq:goldenbrown1}
\begin{split}
\mathcal{N}_{(\curl\FF)\cdot\nu,\Omega'}(\varphi) & \stackrel{\eqref{eq:normaltraceextdivmeasintrinsic}}{=}\langle(\curl\FF)\cdot\nu,\overline{\varphi}\rangle_{\partial\Omega'} \stackrel{\eqref{eq:easynormaltraceextdivmeas1}}{=}  -\int_{\Omega'}\nabla\overline{\varphi}\cdot\dif\,(\curl\FF) \\ 
& \stackrel{\eqref{eq:tomturbo}}{=} -\langle(\FF\times\nu)_{\Omega'}^{\mathrm{int}},\nabla\overline{\varphi}\rangle_{\partial\Omega'}. 
\end{split}
\end{align}
Moreover, $\nabla\overline{\varphi}|_{\Omega'}$ in particular has a continuous extension  to $\overline{\Omega'}$, and the restriction $\bm{\psi}$ of the latter to $\partial\Omega$ belongs to $\mathrm{Lip}(\partial\Omega';\R^{3})$ and has tangential component  $\bm{\psi}_{\tau}=\nabla_{\tau}\varphi$. Recalling Remark \ref{rem:onesided} and Theorem \ref{cor:tangentiality}, we thus have 
\begin{align}
-\langle(\FF\times\nu)_{\Omega'}^{\mathrm{int}},\nabla\overline{\varphi}\rangle_{\partial\Omega'} & \stackrel{\eqref{eq:tabealove}}{=} - \mathcal{T}_{(\FF\times\nu)_{\Omega'}^{\mathrm{int}}}(\bm{\psi}) & \notag\\ & \stackrel{\eqref{eq:tangentialityDist1}}{=} - \mathcal{T}_{(\FF\times\nu)_{\Omega'}^{\mathrm{int}}}(\bm{\psi}_{\tau})
 &\text{(since $\bm{\psi}\in\mathrm{Lip}(\partial\Omega';\R^{3})$)} \notag \\ 
 & \;\; = - \mathcal{T}_{(\FF\times\nu)_{\Omega'}^{\mathrm{int}}}(\nabla_{\tau}\varphi) & \text{(since $\bm{\psi}_{\tau}=\nabla_{\tau}\varphi$ on $\partial\Omega'$)} \label{eq:almostalmost} \\ &\;\stackrel{\text{Def}}{=} \langle \mathrm{div}_{\tau}(\mathcal{T}_{(\FF\times\nu)_{\Omega'}^{\mathrm{int}}}),\varphi\rangle & \notag \\ 
& \! \stackrel{\eqref{eq:lucid}}{=} \langle \mathrm{div}_{\tau}(\mathbf{G}),\varphi\rangle.&\notag
\end{align}
Combining \eqref{eq:goldenbrown1} with \eqref{eq:almostalmost}, \ref{item:StokesMG1} follows. 

\smallskip
For \ref{item:StokesMG2}, we first note that, by Remark \ref{rem:massesadmissibility}, the identity from \ref{item:StokesMG1} extends to $\mathrm{Lip}(\partial\Omega')$. Let $t\in[0,\frac{1}{2})$. If $\mathbf{G},\widetilde{\mathbf{G}}\in\lebe^{1}(\partial\Omega';T_{\partial\Omega'})$ are such that they both satisfy \eqref{eq:lucid}, then Remark \ref{rem:independentchoicerem} implies that $\lla\mathbf{G}\cdot\nu_{\Gamma_{\Sigma}^{t}},\cdot\rra_{\Gamma_{\Sigma}^{t}}$ exists if and only if $\lla\widetilde{\mathbf{G}}\cdot\nu_{\Gamma_{\Sigma}^{t}},\cdot\rra_{\Gamma_{\Sigma}^{t}}$ exists. In particular, by Proposition \ref{prop:StokesPreparatory}, 
\begin{align*}
\mathscr{I}:=\big\{t\in I:=[0,\tfrac{1}{2})\,\colon\; \lla\mathbf{G}\cdot\nu_{\Gamma_{\Sigma}^{t}},\cdot\rra_{\Gamma_{\Sigma}^{t}}\;\text{exists}\big\}
\end{align*}
has full $\mathscr{L}^{1}$-measure in $I$, and is independent of the specific choice of $\mathbf{G}$ with \eqref{eq:lucid}. Hence, in light of Definition \ref{def:locdivL1} and Remark \ref{rem:massesadmissibility}, 
we conclude that, for  $t\in\mathscr{I}$,   
\begin{align*}
[\mathcal{N}_{\curl(\FF)\cdot\nu,\Omega'}]_{\Sigma^{\tau,t}} \stackrel{\ref{item:StokesMG1}}{=} [\mathrm{div}_{\tau}(\mathbf{G})]_{\Sigma^{\tau,t}} \stackrel{\eqref{eq:GGdistMAIN}}{=} - \lla \mathbf{G}\cdot\nu_{\Gamma_{\Sigma}^{t}},\mathbbm{1}_{\Sigma^{\tau,t}}\rra_{\Gamma_{\Sigma}^{t}}. 
\end{align*}
This is \ref{item:StokesMG2}. 

\smallskip
For \ref{item:StokesMG3}, this is a direct consequence of \ref{item:StokesMG1}, \ref{item:StokesMG2}, and Corollary \ref{cor:contstok}. The proof is complete. 
\end{proof}

\begin{rem}[$\lebe^{1}$-Tangential Traces and Full Stokes Functionals]\label{rem:vortfluxStokesfctl} 
Let $\FF\in\mathscr{CM}_{1,\partial\Omega'}^{p,\mathrm{int}}(\Omega)$. 
In analogy with Definition \ref{def:stokes} (see \eqref{eq:Stokesfunctionaldef}), 
one might be inclined to introduce a (interior) Stokes functional via
\begin{align}\label{eq:wrongstokes}
\widetilde{\mathfrak{S}}_{\Sigma^{\tau,t}}^{\mathrm{int}}(\psi):= \langle \mathcal{N}_{\curl(\FF)\cdot\nu,\Omega'},\psi\rangle_{{\Sigma^{\tau,t}}} +   \int_{\Sigma^{\tau,t}}\mathbf{G}\cdot\nabla_{\tau}\psi\dif\mathscr{H}^{2}\qquad \mbox{for $\psi\in\hold^{1}(\partial\Omega')$}. 
\end{align}
This functional is not well-defined in the sense that it is not independent of the specific choice of $\mathbf{G}$, the reason being the second term on the right-hand side of \eqref{eq:wrongstokes}. 
However, the mass of $\mathcal{N}_{\curl(\FF)\cdot\nu,\Omega'}$ is indeed well-defined, 
thereby leading to Theorem \ref{thm:StokesDist} through the vorticity flux formulation. 
On the other hand, if, in addition, $\mathcal{T}_{(\FF\times\nu)_{\Omega'}^{\mathrm{int}}}$ is representable by an $\lebe^{1}$-field $\mathbf{G}_{\FF}$ on $\partial\Omega'$ (meaning that it is a regular distribution), 
then we may define 
\begin{align}\label{eq:wrongstokes1}
\widetilde{\mathfrak{S}}_{\Sigma^{\tau,t}}^{\mathrm{int}}(\psi):= \langle \mathcal{N}_{\curl(\FF)\cdot\nu,\Omega'},\psi\rangle_{\overline{\Sigma^{\tau,t}}} +  \int_{\Sigma^{\tau,t}}\mathbf{G}_{\FF}\cdot\nabla_{\tau}\psi\dif\mathscr{H}^{2}\qquad \mbox{for  $\psi\in\hold^{1}(\partial\Omega')$}. 
\end{align}
In this case, we may prove an analogous result as Theorem \ref{thm:stokes}. We refrain from stating such a theorem explicitly, since a characterization of those $\FF\in\mathscr{CM}^{p}(\Omega)$ such that  $\mathcal{T}_{(\FF\times\nu)_{\Omega'}^{\mathrm{int}}}$ is a regular distribution is unclear to us. 
\end{rem}

\begin{rem}[Regularity of $\partial\Omega'$]\label{rem:uniformiseregularity}
In \eqref{eq:definitiondivtandist}--\eqref{eq:precisemeaningcontainment}, 
we require $\varphi\in\mathrm{Lip}_{2}(\partial\Omega')$ for both expressions to be well-defined. It would be natural to state both for test functions $\varphi\in\hold_{\rm c}^{\infty}(\partial\Omega')$, but this requires $\partial\Omega'$ to be of class $\hold^{\infty}$. If $\partial\Omega'$ satisfies this condition, one can indeed read $\mathrm{div}_{\tau}$ as the classical distributional tangential divergence. 
\end{rem}

Next, we confirm the consistency of the vorticity-flux formulation of the Stokes theorem for $\mathscr{CM}^{\infty}$-fields with that of Corollary \ref{cor:stokesvortflux}:  

\begin{lem}[$p=\infty$ and Consistency with Theorem \ref{cor:stokesvortflux}] 
Let $\Omega\subset\R^{3}$ be open and bounded, let $\Omega'\Subset\Omega$ be open with $\hold^{2}$-boundary,
and let $\Sigma$ be a $\hold^{2}$-regular Lipschitz boundary manifold relative to $\Omega'$. 
Then, if $\FF\in\mathscr{CM}^{\infty}(\Omega)$, the conclusions of {\rm Corollary \ref{cor:stokesvortflux}} 
and  {\rm Theorem \ref{thm:StokesDist}\ref{item:StokesMG2}} are equivalent. 
In particular, there exists a subset $\mathscr{I}\subset I:=[0,\frac{1}{2})$ such that, 
for every $t\in\mathscr{I}$, 
\begin{align}\label{eq:consistencydist}
[\mathcal{N}_{\curl(\FF)\cdot\nu,\Omega'}]_{\Sigma^{\tau,t}}=
\Big[\overline{\overline{\curl(\FF)\cdot\nu_{\Sigma^{\tau,t}}}}\Big]{_{\Sigma^{\tau,t}}^{\mathrm{int}}} = -\lla(\FF\times\nu_{\partial\Omega'})_{\partial\Omega'}^{\mathrm{int}}\cdot\nu_{\Gamma_{\Sigma}^{t}},\mathbbm{1}_{\Sigma^{\tau,t}}\rra_{\Gamma_{\Sigma}^{t}}. 
\end{align}
\end{lem}
\begin{proof}
If $\FF\in\mathscr{CM}^{\infty}(\Omega)$, then we may choose $\mathbf{G}=(\FF\times\nu_{\partial\Omega'})_{\partial\Omega'}^{\mathrm{int}}\in\lebe^{\infty}(\partial\Omega';T_{\partial\Omega'})$  by Theorem \ref{thm:tracemain1}. By Lemma \ref{lem:independence},  the masses are independent of the specific $\lebe_{\locc}^{1}$-representative, whereby we may directly work with this particular choice of $\GG$ in the sequel;
see \eqref{eq:massesequal}. 
Let $\varphi\in\hold_{\rm c}^{1}(\Omega)$ be such that that $\varphi=1$ in an open neighborhood of $\Sigma$. Then 
\begin{align*}
[\mathcal{N}_{(\curl\FF)\cdot\nu,\Omega'}]_{\Sigma^{\tau,t}} & \stackrel{\eqref{eq:vortfluxdist}}{=} - \lla\mathbf{G}\cdot\nu_{\Gamma_{\Sigma}^{t}},\mathbbm{1}_{\Sigma^{\tau,t}}\rra_{\Gamma_{\Sigma}^{t}} \stackrel{\eqref{eq:distlocaliser}}{=}  - \lim_{\delta\searrow 0}\int_{\Sigma}\varphi\,\mathbf{G}\cdot\nabla_{\tau}\psi_{t,\delta,\Sigma}\dif\mathscr{H}^{2} \\ 
& \;\; = - \lim_{\delta\searrow 0} \int_{\partial\Omega'}\varphi\,(\FF\times\nu_{\partial\Omega'})_{\partial\Omega'}^{\mathrm{int}}\cdot\nabla_{\tau}\psi_{t,\delta,\Sigma}\dif\mathscr{H}^{2} \stackrel{\eqref{eq:alternativemain}}{=} \mathfrak{S}_{\Sigma^{\tau,t}}^{\mathrm{int}}(\varphi) \\ & \;\stackrel{\text{Def}}{=} \Big[\overline{\overline{(\curl \FF)\cdot\nu_{\Sigma^{\tau,t}}}} \Big]{_{\Sigma^{\tau,t}}^{\mathrm{int}}}. 
\end{align*}
This completes the proof. 
\end{proof}

The aforementioned desired flexibility in choosing $\lebe^{1}$-representatives is visible most clearly 
in the following Stokes theorem for $\sobo^{\curl,p}$-fields. 
To the best of our knowledge, this is the first result in this direction, 
even in the context of $\mathrm{H}^{\curl}$-fields. 

\begin{theorem}[Stokes Theorem for $\sobo^{\curl,p}$-Fields with respect to the Tangential Variations]\label{thm:StokesWcurl}
Let $\Omega\subset\R^{3}$ be open and bounded, let $\Omega'\Subset\Omega$ be open with $\hold^{2}$-boundary,
and let $\Sigma$ be a $\hold^{2}$-regular Lipschitz boundary manifold relative to $\Omega'$. 
Then the following hold{\rm :} 
\begin{enumerate}
    \item\label{item:StokesMG4} If $\frac{3}{2}< p\leq 3$, then $\sobo^{\curl,p}(\Omega)\subset\mathscr{CM}_{1,\partial\Omega'}^{p,\mathrm{int}}(\Omega)$. In particular, the conclusions of {\rm Theorem \ref{thm:StokesDist}\ref{item:StokesMG1}--\ref{item:StokesMG2}} hold for any $\FF\in\sobo^{\curl,p}(\Omega)$. 
    \item\label{item:StokesMG5} If $3<p\leq\infty$, then $\sobo^{\curl,p}(\Omega)\subset\mathscr{CM}_{1,\partial\Omega'}^{p,\mathrm{int}}(\Omega)$. Moreover, for every $\FF\in\sobo^{\curl,p}(\Omega)$, there exists a continuous representative $\GG\in\hold(\partial\Omega')$ with \eqref{eq:lucid}. 
    As a consequence, the conclusions of {\rm  Theorem \ref{thm:StokesDist}\ref{item:StokesMG1}\&\ref{item:StokesMG3}} hold for any $\FF\in\sobo^{\curl,p}(\Omega)$. In particular, in this case, the Stokes theorem is available \emph{without} having to vary $\Sigma$ tangentially.
\end{enumerate}
\end{theorem}
In Theorem \ref{thm:StokesWcurl}, it does not matter whether we consider the interior or tangential traces along $\partial\Omega'$. This is due to the following elementary observation: 
\begin{lem}\label{lem:leftrightsame}
In the situation of {\rm Theorem \ref{thm:StokesWcurl}}, 
let $\FF\in\sobo^{\curl,p}(\Omega)$ with $1\leq p\leq\infty$. 
With the tangential trace functionals for $\mathscr{CM}^{p}$-fields from \eqref{eq:tangentialtrace},
we have  
\begin{align}\label{eq:pequality}
(\FF\times\nu)_{\Omega'}^{\mathrm{int}}=(\FF\times\nu)_{\Omega'}^{\mathrm{ext}}. 
\end{align}
\end{lem}
\begin{proof}
Let $\varphi\in\hold_{\rm c}^{1}(\Omega)$. Since $\mathrm{dist}(\spt(\varphi),\partial\Omega)>0$, we find an open and bounded set $\widetilde{\Omega}$ with $\hold^{\infty}$-boundary such that $\spt(\varphi)\subset\widetilde{\Omega}\Subset\Omega$. Since $\curl\FF\in\lebe^{p}(\Omega;\R^{3})$, $|\curl\FF|(\partial\Omega')=0$. By routine smooth approximation, we find a sequence $(\FF_{j})\subset\hold^{\infty}(\Omega;\R^{3})\cap\sobo^{\curl,p}(\Omega)$ such that $\FF_{j}\to\FF$ strongly in $\lebe^{p}(\Omega;\R^{3})$ and $\curl\FF_{j}\to\curl\FF$ strongly in $\lebe^{p}(\Omega;\R^{3})$. 
Working directly from \eqref{eq:tangentialtrace}, we then obtain 
\begin{align*}
\langle (\FF \times \nu)_{\Omega'}^{\mathrm{int}}, \varphi \rangle_{\partial \Omega'} - \langle (\FF \times \nu)_{\Omega'}^{\mathrm{ext}}, \varphi \rangle_{\partial \Omega'}  & = \int_{{\Omega'}} \varphi \, \d ( \curl \FF) - \int_{{\Omega'}}   \FF\times\nabla \varphi  \,\dif x, \\  
& = \lim_{j\to\infty}\int_{{\Omega'}}\varphi\,\curl\FF_{j}\dif x - \int_{{\Omega'}}\FF_{j}\times\nabla\varphi\dif x \\ 
& = \lim_{j\to\infty} \int_{\Omega'}\curl(\varphi\FF_{j})\dif x =  \lim_{j\to\infty} \int_{\widetilde{\Omega}}\curl(\varphi\FF_{j})\dif x \\ 
& \!\!\! \stackrel{\eqref{eq:calculus2}}{=} \lim_{j\to\infty} \int_{\partial\widetilde{\Omega}}\varphi(\FF_{j}\times\nu_{\partial\widetilde{\Omega}})\dif\mathscr{H}^{2} = 0,  
  \end{align*}
where we used the classical integration-by-parts formula \eqref{eq:calculus2} from the Appendix \hyperref[sec:AppendixD]{D}. 
This is \eqref{eq:pequality}, and the proof is complete. 
\end{proof}

We now come to Theorem \ref{thm:StokesWcurl}.
\begin{proof}[Proof of Theorem \ref{thm:StokesWcurl}]
We begin with some preliminary remarks. By Lemma \ref{lem:leftrightsame}, we have  $ (\FF\times\nu)_{\Omega'}^{\mathrm{int}}=(\FF\times\nu)_{\Omega'}^{\mathrm{ext}}$. Moreover, we note that $\sobo^{\curl,p}(\Omega)\hookrightarrow\mathscr{CM}_{1}^{p}(\Omega)$, which can be seen as follows: By the results gathered in Appendix \hyperref[sec:AppendixC]{C}, the tangential trace operator on $\sobo^{\curl,p}$ maps $\mathrm{tr}_{\tau}^{p}\colon\sobo^{\curl,p}(\Omega')\to\mathcal{X}_{\partial\Omega'}^{p} (\subset \sobo^{1/p,p'}(\partial\Omega';\R^{3})')$ so that, in particular, $\mathrm{tr}_{\tau}^{p}(\FF)\in\sobo^{-1/p,p}(\partial\Omega';T_{\partial\Omega'})$ and $f:=\mathrm{div}_{\tau}(\mathrm{tr}_{\tau}^{p}(\FF))\in\sobo^{-1/p,p}(\partial\Omega')$; clearly, $\mathrm{tr}_{\tau}^{p}(\FF)=(\FF\times\nu)_{\Omega'}^{\mathrm{int}}$. We now consider the auxiliary equation
\begin{align}\label{eq:auxsolve}
\begin{cases}
-\Delta_{\tau}u = \mathrm{div}_{\tau}(\mathrm{tr}_{\tau}^{p}(\FF)) &\;\text{in}\;\partial\Omega',\\[1mm] 
\displaystyle\int_{\partial\Omega'}u\dif\mathscr{H}^{2}=0& 
\end{cases}
\end{align}
with the Laplace-Beltrami operator $\Delta_{\tau}$. 
Because of $p> \frac{3}{2}$, we have  $f\in\sobo^{-1,2}(\partial\Omega')$, which can be seen directly via Proposition \ref{prop:embeddingsmanifolds}\ref{item:embmanifold1}: 
\begin{align*}
\sobo^{1,2}(\partial\Omega')\stackrel{\text{Prop. \ref{prop:embeddingsmanifolds}\ref{item:embmanifold1}}}{\hookrightarrow} \sobo^{1/p,p'}(\partial\Omega')\Longrightarrow \sobo^{-1/p,p}(\partial\Omega')\cong \sobo^{1/p,p'}(\partial\Omega')'\hookrightarrow \sobo^{-1,2}(\partial\Omega'). 
\end{align*}
Now, since $\mathrm{tr}_{\tau}^{p}(\FF)$ is tangential to $\partial\Omega'$ (see the definition of $\mathcal{X}_{\partial\Omega'}^{p}$) and the manifold $\partial\Omega'$ has no boundary, $f$ moreover satisfies condition \eqref{eq:fixing} from the Appendix \hyperref[sec:AppendixA3]{A.3}: 
\begin{align*}
\langle f,\mathbbm{1}_{\partial\Omega'}\rangle_{\partial\Omega'} = -\langle \mathrm{tr}_{\tau}^{p}(\FF),\nabla_{\tau}\mathbbm{1}_{\partial\Omega'}\rangle =0. 
\end{align*}
Hence, by Remark \ref{rem:existencevar} from the Appendix \hyperref[sec:AppendixA3]{A.3}, \eqref{eq:auxsolve} has a solution $u\in\sobo^{1,2}(\partial\Omega')$. 
We then define 
$\mathbf{G}:=-\nabla_{\tau}u$, so that $\mathrm{div}_{\tau}(\GG)=\mathrm{div}_{\tau}(\mathrm{tr}_{\tau}^{p}(\FF))$ by \eqref{eq:auxsolve}. In particular, 
\begin{align*}
\mathrm{div}_{\tau}\big(\mathcal{T}_{(\FF\times\nu)_{\Omega'}^{\mathrm{int}}}\big) = \mathrm{div}_{\tau}(\GG)\qquad\text{as an identity in}\;\hold^{2}(\partial\Omega')'. 
\end{align*}
For \ref{item:StokesMG4}, if $\frac{3}{2}< p \leq 3$, then $u\in\sobo^{1,2}(\partial\Omega')$ and so $\mathbf{G}=-\nabla_{\tau}u\in\lebe^{2}(\partial\Omega';T_{\partial\Omega'})$. The conclusion of \ref{item:StokesMG4} thus follows from Theorem \ref{thm:StokesDist}\ref{item:StokesMG1}--\ref{item:StokesMG2}.\\ 

\noindent
For \ref{item:StokesMG5}, if $3<p\leq\infty$,
then the elliptic regularity result from Proposition \ref{prop:ellreg} from the Appendix \hyperref[sec:AppendixA3]{A.3} gives us 
\begin{align*}
\mathrm{div}_{\tau}(\mathrm{tr}_{\tau}^{p}(f))\in\sobo^{-1/p,p}(\partial\Omega') \Longrightarrow u\in\sobo^{2-1/p,p}(\partial\Omega') \Longrightarrow \mathbf{G}=-\nabla_{\tau}u\in \sobo^{1-1/p,p}(\partial\Omega'). 
\end{align*}
By Proposition \ref{prop:embeddingsmanifolds}\ref{item:embdmanifold2}  from the Appendix \hyperref[sec:AppendixA3]{A.3}, $\sobo^{1-1/p,p}(\partial\Omega')\hookrightarrow\hold(\partial\Omega')$ provided that $p>3$. Hence, $\mathbf{G}\in\hold(\partial\Omega';T_{\partial\Omega'})$, and so the conclusion of \ref{item:StokesMG5} follows from Theorem \ref{thm:StokesDist}\ref{item:StokesMG1}\&\ref{item:StokesMG3}. The proof is complete. 
\end{proof}
We conclude the subsection with the following remark. 
\begin{rem}

In the preceding proof, the key point is that the $\mathrm{curl}$ comes with more structure (or compatibility conditions) than the divergence. This is reflected by the fact that, for $\FF\in\sobo^{\curl,p}(\Omega')$, the distributional tangential divergence $\mathrm{div}_{\tau}(\mathrm{tr}_{\tau}^{p}(\FF))$ belongs to $\sobo^{-1/p,p}(\partial\Omega')$; this can be viewed as a very weak, yet crucial regularity result on $\mathrm{tr}_{\tau}(\FF)$. 
By way of comparison, in the case of $\sobo^{\mathrm{div},p}$-fields, we can only assert that the normal trace $\mathrm{tr}_{\nu}^{p}(\FF)$ belongs to $\sobo^{-1/p,p}(\partial\Omega')$ but, in general, satisfies no additional condition.

We note that, if $p>3$, then $\sobo^{\curl,p}(\Omega)\not\hookrightarrow\hold(\Omega;\R^{3})$, different from $\sobo^{1,p}(\Omega;\R^{3})\hookrightarrow\hold_{b}(\Omega;\R^{3})$ (if, \emph{e.g.}, $\partial\Omega$ is Lipschitz). In the latter case, the Stokes theorem would follow by easier means. Yet, we wish to stress that despite the lack of continuity of $\sobo^{\curl,p}$-fields for $p>3$, \emph{no} tangential variations of the underlying manifold $\Sigma$ are required. 
\end{rem}
\subsection{Examples and special cases}\label{sec:StokesMostGeneral}
We now discuss the results of the previous sections in view of Theorem \ref{thm:StokesDist}, and start with illustrating the results by expanding on Examples \ref{ex:limitationsp}--\ref{ex:1leqpleqinfty}. 
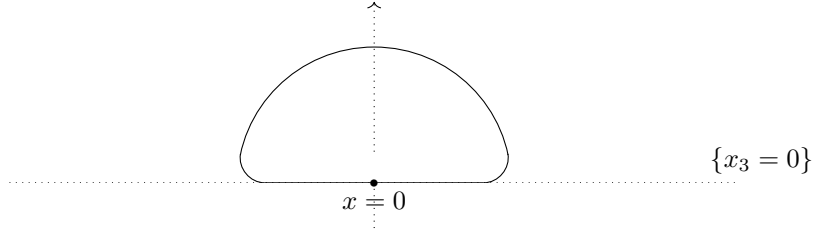
\begin{figure}[t]
\begin{tikzpicture}[scale=1.2]
\begin{scope}
 \node at (0,0) {\tiny\textbullet};
 \node[below] at (0,0) {$x=0$};
 \draw[dotted,->] (0,-0.5) -- (0,2);
    \clip (-6,0) rectangle (6,1.6);
    \draw (0,0) circle(1.5);
    \draw[-] (-1.5,0) -- (1.5,0);
    \draw (-1.5,0) -- (1.5,0);
    \draw[white,fill=white] (-1.7,-0.2) -- (1.7,-0.2) -- (1.7,0.3) -- (-1.7,0.3) -- (-1.7,0);
     \draw[-] (-1.2,0) -- (1.2,0);
     \draw[-,black] (-1.466,0.32) [out = -100, in = 180] to (-1.2,0);
     \draw[-,black] (1.466,0.32) [out = 280, in = 0] to (1.2,0);
     \draw[dotted,-] (-4,0) -- (4,0); 
     \node at (4.25,0.25) {$\{x_{3}=0\}$}; 
      \node at (0,0) {\tiny\textbullet};
\end{scope}
\end{tikzpicture}
\caption{Cross section of the domain from Example \ref{ex:limitationsctd}.}\label{fig:crosssection}
\end{figure}

\begin{example}[Example \ref{ex:limitationsp}, continued]\label{ex:limitationsctd}
In this example, $\FF(x)=-\frac{1}{4\pi}\frac{x}{|x|^{3}}$ and $\Omega=\ball_{2}(0)$
so that $\curl\FF=0$ on $\Omega$. We let $\Omega'\Subset\Omega$ be the open and bounded domain with $\hold^{\infty}$-boundary which emerges by rotating the cross section indicated in Figure \ref{fig:crosssection} around the $x_{3}$-axis; compared to the upper half-circle, we require this modification in order to work with sufficiently regular domains. We note that, for every $0<\varepsilon<1$, $(\FF\times\nu_{\partial\Omega'})$ is a $\hold^{2}$-function on $\partial\Omega'\setminus\overline{\ball}{_{\varepsilon}^{(2)}}(0)$. By the same computation as in Example \ref{ex:limitationsp}, $\mathcal{T}_{(\FF\times\nu)_{\Omega'}^{\mathrm{int}}}$ is a distribution of order one. Moreover, testing \eqref{eq:maxlaim} with tangential gradients, we equally find that $\mathrm{div}_{\tau}(\mathcal{T}_{(\FF\times\nu)_{\Omega'}^{\mathrm{int}}})$ can be represented as $\mathrm{div}_{\tau}(\mathbf{G})$, where $\mathbf{G}$ is a tangential $\hold^{2}$-field such that $\mathrm{div}(\mathbf{G})=0$ on  $V=\partial\Omega'\cap\{x_{3}=0\}$. In particular,
\begin{align*}
\langle\mathrm{div}_{\tau}(\mathcal{T}_{(\FF\times\nu)_{\Omega'}^{\mathrm{int}}}),\varphi\rangle = - \langle \mathcal{T}_{(\FF\times\nu)_{\Omega'}^{\mathrm{int}}},\nabla_{\tau}\varphi\rangle =0\qquad\,\,
\text{for all}\;\varphi\in\hold_{\rm c}^{2}(V),
\end{align*}
and we may fix this choice of $\GG$ in Theorem \ref{thm:StokesDist}. Yet, we emphasize that, by Example \ref{ex:limitationsp}, $\mathcal{T}_{(\FF\times\nu)_{\Omega'}^{\mathrm{int}}}$ cannot be represented as a Radon measure on $\partial\Omega'$.

Since $\FF\notin\lebe_{\locc}^{\infty}(\Omega;\R^{3})$, it is not possible to apply Theorems \ref{thm:stokes} or \ref{thm:stokes1st}. In particular, 
\begin{itemize}
    \item[(i)] $(\FF\times\nu)_{\Omega'}^{\mathrm{int}}$ is not even an extended divergence measure field. This especially implies that a potential analogue of Theorem \ref{thm:stokes1st}, where one uses a finiteness condition on the transversal maximal operator in order to reduce to the Gauss--Green theorem for $\mathscr{DM}^{\mathrm{ext}}$-fields on $\partial\Omega'$, is impossible for $1\leq p<\infty$ in general.
    \item[(ii)] the Stokes theorem now is a consequence of Theorem \ref{thm:StokesDist}: If $\Sigma\subset\partial\Omega'$ is a $\hold^{2}$-regular Lipschitz boundary manifold relative to $\Omega'$ with $0\in\mathrm{int}(\Sigma)$ and $\Sigma\Subset V$, then we may apply Theorem \ref{thm:StokesDist}\ref{item:StokesMG3}. The vector field $\GG$ is continuous and vanishes in a relative neighborhood of $\Gamma_{\Sigma}$, whereby 
    \begin{align*}
    [\mathcal{N}_{\curl(\FF)\cdot\nu,\Omega'}]_{\Sigma} = - \langle\GG\cdot\nu_{\Gamma_{\Sigma}},\mathbbm{1}_{\Sigma}\rangle_{\Gamma_{\Sigma}}=0. 
    \end{align*}
\end{itemize}

\end{example}
\begin{example}[Example \ref{ex:singCMP}, continued] With $\Omega,\Omega'\subset\R^{3}$ as in the previous example, we consider the field \eqref{eq:curldef}. Its distributional tangential trace belongs to $\lebe^{q}(\partial\Omega';T_{\partial\Omega'})$ for any $1\leq q<2$; along $\partial\Omega'\cap\{x_{3}=0\}$, it is given by 
\begin{align*}
(\FF\times\nu_{\partial\Omega'})(x_{1},x_{2},x_{3}) 
= \frac{1}{2\pi}\Big(\frac{x_{1}}{x_{1}^{2}+x_{2}^{2}},\frac{x_{2}}{x_{1}^{2}+x_{2}^{2}},0\Big),
\end{align*}
and its tangential divergence is given by 
\begin{align*}
\mathrm{div}_{\tau}(\FF\times\nu_{\partial\Omega'})=\delta_{(0,0,0)} +\mathrm{div}_{\tau}(\mathbf{H}),  
\end{align*}
where $\mathbf{H}$ is supported on $\overline{\partial\Omega'\cap\{x_{3}>0\}}$ and of class $\hold^{1}$. Based on the representation \eqref{eq:distclaim}, we have with any natural choice of the transversal collar map $\Phi$ that 
\begin{align}\label{eq:goodfiniteness}
\mathcal{M}_{\partial\Omega',\partial\Omega'}^{\Phi}(\curl\FF)(0) \leq  \sup_{0<\varepsilon<\frac{1}{2}}\frac{|\curl\FF|(\{x\in\Omega'\colon\;\mathrm{dist}(x,\partial\Omega')<\varepsilon\})}{\varepsilon} \leq c<\infty. 
\end{align}
Here, Theorem \ref{thm:stokes1st} is not applicable, since the distributional tangential trace does not belong to $\lebe^{\infty}(\partial U;\R^{3})$. Moreover,  Example \ref{ex:limitationsctd} shows that there cannot be a general theorem on the reducibility of tangential traces to those of $\mathscr{DM}^{p}$-fields even in the presence of a condition such as \eqref{eq:goodfiniteness}. Yet, $(\FF\times\nu_{\partial\Omega'})$ is continuous away from zero, and so Theorem \ref{thm:StokesDist}\ref{item:StokesMG3} gives us for balls $\Sigma_{r}:=\ball_{r}^{(2)}(0)$ with sufficiently small $0<r<1$ that 
\begin{align*}
-\int_{\partial\Sigma_{r}}(\FF\times\nu_{\Sigma_{r}})\cdot\nu_{\Gamma_{\Sigma_{r}}^{0}}\dif\mathscr{H}^{1} & =\frac{1}{2\pi r}\mathscr{H}^{1}(\partial\Sigma_{r}) = 1 = [\mathrm{div}_{\tau}(\mathcal{T}_{(\FF\times\nu)_{\Omega'}^{\mathrm{int}}})]_{\Sigma_{r}}.  
\end{align*}
\end{example}

{\begin{example}[High Regularity Scenarios]\label{ex:highreg}
Expanding on Example \ref{ex:tanjumps1A}, the main assumptions of Theorem \ref{thm:StokesDist} are always satisfied if, \emph{e.g.}, $\FF\in\sobo^{2,1}(\Omega;\R^{3})$; in this case, the full (and hence tangential) traces belong to $\sobo^{1,1}(\partial\Omega';T_{\partial\Omega'})$. Classical descriptions of the traces (see, \emph{e.g.} \cite[\S 5.3, Theorem 2]{EvansGariepy}) then can be employed to directly yield the conclusion of Theorem \ref{thm:StokesDist}, but not so here. Going below this regularity regime, yet staying \emph{above} the regularity considered here, analogous and more direct arguments would require results as  \emph{e.g.} \cite[\S 5.3, Theorem 2]{EvansGariepy} for Besov functions; to the best of our knowledge, such results are not available at present; see \cite{Schneider} for an overview. 
\end{example}}
\subsection{Remarks on generalizations}\label{sec:remsgen} In stating our results, we have not opted for the highest possible generality but rather deal with the cases which we view most relevant for the first paper on the topic. 
We thus conclude the overall section with some comments on potential generalizations; in this regard, we highlight three points:
\begin{rem}[On Traces]
In the trace theory as discussed in \S \ref{sec:tracetheorem}, we primarily focused on sufficiently regular sets $\Omega'$ which are compactly contained in $\Omega$. Since we particularly require the results on the (distributional) tangential traces in view of Stokes' theorem, we did not aim for traces on irregular sets (see, \emph{e.g.} \cite{ChenLiTorres}); some aspects of such more irregular scenarios shall be covered in the follow-up paper \cite{ChenGmeinederStephanTorresYeh1} based on the methods developed here. Yet, even in the case of regular boundaries, a characterisation of the trace space of $\mathscr{CM}^{p}$-fields in the spirit of Irving's recent work \cite{Irving} seems difficult; this is essentially due to the fact that the curl is self-adjoint. Lastly, besides traces, associated extension theorems for curl-measure fields can be obtained by employing the recent work \cite{GmeinederSchiffer} on divergence-free extensions.
\end{rem}
\begin{rem}[On Stokes Theorems]\label{rem:generalizations}
  First, in Cartan's formulation, the Stokes theorem can be stated via differential forms and, more generally, certain differential operators. Keeping in mind physical applications, the curl as considered here is the primary object of interest, which is why we leave such generalizations to future work. Secondly, we often work with $\hold^{1}$- or even $\hold^{2}$-manifolds in the main part. Whereas statements in the low regularity context (\emph{e.g.} Lipschitz manifolds) are certainly desirable, some tools such as integration by parts and weakly differentiable functions on Lipschitz manifolds come with numerous technical difficulties. In particular, this concerns scenarios where curvature terms are involved, and would be in contrast to focusing on the conceptual points we aim to present in this paper; our findings are novel even in the case of $\hold^{\infty}$-manifolds. 
\end{rem}
\begin{rem}[Matrix fields]
In problems from elasticity in three-dimensional space, one often considers the curls of matrix fields; see, \emph{e.g.} \cite{GmeinederLewintanNeff,GmeinederLewintanNeff1,GmeinederSpector}, and the references therein. In this case, the curl is being applied row-wise, and the results obtained in the present paper apply to such situations as well with their natural modifications. 
\end{rem}
\section{$\mathscr{CM}$-Fields, Vorticity Fluxes, and Vortex Sheets}
\label{sec:applications}
In this final section, we showcase the utility of curl-measure fields and the Stokes theorems from the previous sections in view of physical problems describing vortex phenomena. By the ubiquity of the latter, we confine ourselves to a selection of applications in fluid mechanics, electromagnetism, and magnetohydrodynamics. Here, we focus on rigorous derivations and justifications of various types of consistency results in the low-regularity context. This illustrates the application of the above results, rather than devising existence theories for individual cases, a task we aim to address in future work.

\subsection{Differential form of Maxwell's equations in the low regularity context}\label{sec:differentialformMaxwell}

The connection between magnetic and electric fields as well as electric charges and currents is given by \emph{Maxwell's equations}. Given an open and bounded set $\Omega\subset\R^{3}$ with Lipschitz boundary $\partial\Omega$ and a final time $T>0$, 
we set $\Omega_{T}:=(0,T)\times\Omega$. 
For a charge density $\rho\colon\Omega_{T}\to\R$, 
let $\mathbf{E}\colon\Omega_{T}\to\R^{3}$ be an electric field that 
causes an electric displacement 
field $\mathbf{D}\colon\Omega_{T}\to\R^{3}$ with current density 
$\mathbf{J}\colon\Omega_{T}\to\R^{3}$. 
On the other hand, $\mathbf{E}$ leads to a magnetic field with 
magnetic flux density $\mathbf{{ H}}\colon\Omega_{T}\to\R^{3}$. 
These quantities are coupled via the following Maxwell equations: 
\begin{align}\label{eq:maxwell}
\begin{cases} 
\mathrm{div \,} \mathbf{{E}}= \rho &\text{(Gauss' electric law)}, \\ 
\mathrm{div \,}\mathbf{H}= 0 & \text{(Gauss' magnetic law),}\\ 
\partial_{t}\mathbf{E} = \curl \mathbf{H}-\mathbf{J} & \text{(Amp\'{e}re's law),}\\ 
\partial_{t}\mathbf{H} = - \curl \mathbf{E} & \text{(Faraday's law of induction)}. 
\end{cases}
\end{align}
The derivation of Amp\'{e}re's law is based on the following empirical law: {\it The circulation of the magnetic field is equal to the net current through $S$}. Similarly, the derivation of Faraday's law starts with the following empirical law: {\it The voltage around a curve $C$ is equal to the negative rate of change of the magnetic flux through surface $S$ with boundary $C$}. These principles are discussed in detail, \emph{e.g.}, in Feynman \cite{Feynman}. 

In the following, we aim to derive the equations in $\eqref{eq:maxwell}_{4}$ 
based on these physical principles in the low regularity context. 
To keep our exposition at a reasonable length, we assume for simplicity 
that the electric field and magnetic flux densities satisfy  
\begin{align}\label{eq:functionspacesetup}
\mathbf{E}\in\sobo^{1,1}(0,T;X(\Omega)),\qquad\;
\mathbf{H}\in\sobo^{1,1}(0,T;Y(\Omega)), 
\end{align}
where 
\begin{align*}
X(\Omega) & := \{\FF\in\lebe^{\infty}(\Omega;\R^{3})\colon\;\curl\FF\in\lebe^{1}(\Omega;\R^{3})\;\text{and}\;\mathrm{div}\FF \in\mathrm{RM}_{\mathrm{fin}}(\Omega)\}, \\ 
Y(\Omega) & := \{\FF\in\lebe^{1}(\Omega;\R^{3})\colon\;\curl\FF\in\lebe^{\infty}(\Omega;\R^{3})\}.
\end{align*}
Both assumptions are natural from a physical perspective  and in view of \eqref{eq:maxwell}, of which we aim to derive $\eqref{eq:maxwell}_{4}$. We wish to stress that, for the purpose of the following derivation, the regularity in time can be lowered. However, primarily aiming to address the spatial low regularity context, we stick to this assumption. \\

\begin{figure}[t]
\begin{center}
\begin{tikzpicture}
\draw[black!20!white,fill=black!20!white,opacity=0.4] (-2.5,-1.05) -- (-1.5,-1.05) -- (-1.5,1.05) -- (-2.5,1.05) -- (-2.5,-1.05);
\node[black!40!white] at (-3.5,0.5) {$\mathcal{B}_{R_{1},R_{2},\delta}(z_{0})$};
\draw[black!40!white] (-2.75,0.25) -- (-2,0);
\draw[thick, blue, rounded corners = 10pt] (-1.15,-1.5) -- (1.5,-1.5) -- (1.5,1.5) -- (-1.5,1.5) -- (-1.5,-1.5) -- (1.15,-1.5);
\node[blue] at (1,0.5) {$\partial\mathrm{Q}_{R_{1}}$};
\draw[dotted] (-1.6,-1.6) -- (1.6,-1.6) -- (1.6,1.6) -- (-1.6,1.6) -- (-1.6,1.6) -- (-1.6,-1.6);
\draw[dotted] (-1.7,-1.7) -- (1.7,-1.7) -- (1.7,1.7) -- (-1.7,1.7) -- (-1.7,1.7) -- (-1.7,-1.7);
\draw[dotted] (-1.8,-1.8) -- (1.8,-1.8) -- (1.8,1.8) -- (-1.8,1.8) -- (-1.8,1.8) -- (-1.8,-1.8);
\draw[dotted] (-1.9,-1.9) -- (1.9,-1.9) -- (1.9,1.9) -- (-1.9,1.9) -- (-1.9,1.9) -- (-1.9,-1.9);
\draw[dotted] (-2,-2) -- (2,-2) -- (2,2) -- (-2,2) -- (-2,2) -- (-2,-2);
\draw[dotted] (-2.1,-2.1) -- (2.1,-2.1) -- (2.1,2.1) -- (-2.1,2.1) -- (-2.1,2.1) -- (-2.1,-2.1);
\draw[dotted] (-2.2,-2.2) -- (2.2,-2.2) -- (2.2,2.2) -- (-2.2,2.2) -- (-2.2,2.2) -- (-2.2,-2.2);
\draw[dotted] (-2.3,-2.3) -- (2.3,-2.3) -- (2.3,2.3) -- (-2.3,2.3) -- (-2.3,2.3) -- (-2.3,-2.3);
\draw[dotted] (-2.4,-2.4) -- (2.4,-2.4) -- (2.4,2.4) -- (-2.4,2.4) -- (-2.4,2.4) -- (-2.4,-2.4);
\draw[dotted] (-2.5,-2.5) -- (2.5,-2.5) -- (2.5,2.5) -- (-2.5,2.5) -- (-2.5,2.5) -- (-2.5,-2.5);
\node at (0,0) {\tiny\textbullet};
\node[right] at (0,-0.2) {$z_{0}$};
\draw[->] (-1.5,0.75) -- (-0.75,0.75);
\node[above] at (-0.75,0.75) {$\nu=e_{2}$};
\draw[dotted,blue,thick] (-2.5,2.8) -- (-2.5,-4);
\draw[dotted,blue,thick] (-2.4,2.8) -- (-2.4,-4);
\draw[dotted,blue,thick] (-2.3,2.8) -- (-2.3,-4);
\draw[dotted,blue,thick] (-2.2,2.8) -- (-2.2,-4);
\draw[dotted,blue,thick] (-2.1,2.8) -- (-2.1,-4);
\draw[dotted,blue,thick] (-2.0,2.8) -- (-2.0,-4);
\draw[dotted,blue,thick] (-1.9,2.8) -- (-1.9,-4);
\draw[dotted,blue,thick] (-1.8,2.8) -- (-1.8,-4);
\draw[dotted,blue,thick] (-1.7,2.8) -- (-1.7,-4);
\draw[dotted,blue,thick] (-1.6,2.8) -- (-1.6,-4);
\draw[dotted,blue,thick] (-1.5,2.8) -- (-1.5,-4);
\draw[-,blue] (-2.5,-4) -- (-1.5,-4); 
\node[below,blue] at (-2.5,-4) {$R_{2}$}; 
\node[below,blue] at (-1.5,-4) {$R_{1}$}; 
\draw[-] (-5,-1) [out = -20, in = 160] to (-1,-3.25) [out = -20, in =200] to (3,-3.25) [out = 20, in = 200] to (6,-1); 
\node at (4.5,-0.5) {\LARGE $\Omega$}; 
\node at (5,-2.75) {\LARGE $\R^{3}\setminus\overline{\Omega}$}; 
\end{tikzpicture}
\end{center}
\caption{The construction from \S \ref{sec:differentialformMaxwell}.}\label{fig:constructionMaxwell}
\end{figure}
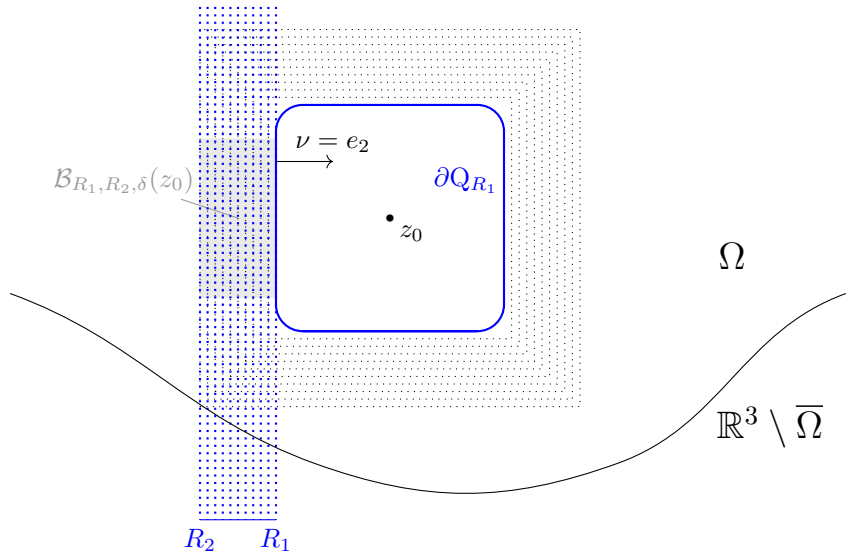

\emph{Step 1.} For $z_{0}\in\Omega$, a suitable interval $I:=I_{R_{1},R_{2}}:=[R_{1},R_{2}]$ and $R\in I$, we consider 
sets $\mathrm{Q}_{R}$ of the form as indicated in Figure \ref{fig:constructionMaxwell};  
these are nothing but cubes $\{x\colon |x-z_{0}|_{\infty}<R\}$ with the edges being rounded off. In particular, each $\mathrm{Q}_{R}$ has $\hold^{\infty}$-boundary. For $R\in [R_{1},R_{2}]$,  $\partial\mathrm{Q}_{R}$ consists of eight mutually disjoint straight faces and eight mutually disjoint curved parts, each of the mutually disjoint straight faces being parallel to some coordinate plane. Out of those straight faces, 
we let $\mathcal{F}_{R}$ be the one that 
is parallel to the $x_{1}$-$x_{3}$ {plane}
and whose elements have smaller $x_{2}$-coordinate. We let $0<\delta<1$ be arbitrary but sufficiently small; then the set 
$\mathcal{F}_{R,\delta}:=\{x\in\mathcal{F}_{R}\,\colon\;\mathrm{dist}(x,\Gamma_{\mathcal{F}_{R}})>\delta\}$ is relatively compact in $\mathcal{F}_{R}$.

Based on this geometric set-up, it is not difficult to see that there exists a sufficiently smooth collar map $\Phi$ on $ (-1,1)\times\partial\mathrm{Q}_{R_{1}}$ such that, for $t\in(-\frac{1}{2},0]$, 
\begin{itemize}
    \item[(i)] $\Phi(t,\partial\mathrm{Q}_{R_{1}})=\partial\mathrm{Q}_{R_{1}-2t(R_{2}-R_{1})}$, 
    \item[(ii)] $\Phi(t,A)=2t(R_{2}-R_{1})e_{2}+ A $ for any $A\subset\mathcal{F}_{R_{1},\delta}$.
\end{itemize}
We then define the box 
\begin{align}\label{eq:box}
\mathcal{B}_{R_{1},R_{2},\delta}(z_{0}):=\bigcup_{-\frac{1}{2}<t<0}\Phi(t,\mathcal{F}_{R_{1},\delta}),  
\end{align}
see Figure \ref{fig:constructionMaxwell}.
This box has positive $\mathscr{L}^{3}$-measure. Moreover, we record that 
\begin{align}\label{eq:goodunitnormal}
\text{Along any face $\Phi(t,\mathcal{F}_{R_{1},\delta})$, the inner unit normal to the face equals $e_{2}$.}
\end{align}
Finally, notice that we can cover the entire $\Omega$ by boxes of 
form \eqref{eq:box} by suitably adjusting $z_{0}$, $R_{1},R_{2}$, and $\delta$. 

\emph{Step 2.} Let $\mathbf{G}\in\lebe^{1}(\Omega;\R^{3})$. Then, based on the geometric set-up given in Step 1, there exists a set $I_{\mathbf{G}}\subset I$ with $\mathscr{L}^{1}(I\setminus I_{\mathbf{G}})=0$ such that, for every $R\in I_{\mathbf{G}}$, $\mathscr{H}^{2}$-a.e. $x\in\partial\mathrm{Q}_{R}$ is a Lebesgue point of $\mathbf{G}$; see Lemma \ref{lem:Lebesgue}. In particular, for each $R\in I_{\mathbf{G}}$, there exists $\mathcal{N}_{R,\mathbf{G}}\subset\partial\mathrm{Q}_{R}$ such that $\mathscr{H}^{2}(\mathcal{N}_{R,\mathbf{G}})=0$, 
and for every $x\in\partial\mathrm{Q}_{R}\setminus \mathcal{N}_{R,\mathbf{G}}$,  
\begin{align*}
    \mathbf{G}^{*}(x):=\lim_{r\searrow 0}\dashint_{\ball_{r}(x)}\mathbf{G}(y)\dif y 
\end{align*}
exists and is finite. Thus, defining $\widetilde{\mathbf{G}}\colon \partial\mathrm{Q}_{R}\to\R^{3}$ for $R\in I_{\mathbf{G}}$ by
\begin{align*}
\widetilde{\mathbf{G}}(x):=\begin{cases} 
\mathbf{G}^{*}(x)&\;\text{if}\;x\in\partial\mathrm{Q}_{R}\setminus \mathcal{N}_{R,\mathbf{G}}, \\ 
0&\;\text{if}\;x\in \mathcal{N}_{R,\mathbf{G}}, 
\end{cases}
\end{align*}
we have $\widetilde{\mathbf{G}}\in\lebe^{1}(\partial\mathrm{Q}_{R};\R^{3})$. By the Lebesgue differentiation theorem, there exists a set $\mathcal{L}'_{R,\mathbf{G}}\subset\partial\mathrm{Q}_{R}$ with $\mathscr{H}^{2}(\partial\mathrm{Q}_{R}\setminus \mathcal{L}'_{R,\mathbf{G}})=0$ such that 
\begin{align*}
\widetilde{\mathbf{G}}^{*}(x):=\lim_{r\searrow 0}\dashint_{\ball_{r}(x)\cap\partial\mathrm{Q}_{R}}\widetilde{\mathbf{G}}(y)\dif\mathscr{H}^{2}(y) \qquad\text{exists and is finite for all $x\in \mathcal{L}'_{R,\mathbf{G}}$}. 
\end{align*}
Specifically, we note that 
\begin{align*}
\overline{\mathbf{G}}(x):=\begin{cases}
\widetilde{\mathbf{G}}^{*}(x) &\;\text{if}\;x\in \mathcal{L}'_{R,\mathbf{G}}, \\ 
0&\;\text{otherwise}, 
\end{cases}
\end{align*}
is an $\lebe^{1}$-representative of $\widetilde{\mathbf{G}}$ and so $\mathscr{H}^{2}(\{x\in\partial\mathrm{Q}_{R}\colon\;\widetilde{\mathbf{G}}(x)\neq\overline{\mathbf{G}}(x)\})=0$. In particular, this means that there exists a set $\mathcal{L}_{R,\mathbf{G}}\subset\partial\mathrm{Q}_{R}$ with $\mathscr{H}^{2}(\partial\mathrm{Q}_{R}\setminus\mathcal{L}_{R,\mathbf{G}})=0$ and 
\begin{align}\label{eq:oh!darling}
\lim_{r\searrow 0}\dashint_{\ball_{r}(x)}\mathbf{G}(y)\dif y = \lim_{r\searrow 0}\dashint_{\ball_{r}(x)\cap\partial\mathrm{Q}_{R}}\widetilde{\mathbf{G}}(y)\dif\mathscr{H}^{2}(y)\qquad\text{for every  $x\in\mathcal{L}_{R,\mathbf{G}}$}.
\end{align}
Now suppose moreover that $\mathbf{G}\in\lebe^{1}(\Omega;\R^{3})$ is such that $\mathrm{div}(\mathbf{G})\in\lebe^{1}(\Omega)$. We then use the fact that, 
\begin{align}\label{eq:crucialID}
\langle \mathbf{G}\cdot\nu,\varphi\rangle_{\partial\mathrm{Q}_{R}}=\int_{\partial\mathrm{Q}_{R}}\varphi\widetilde{\mathbf{G}}\cdot\nu_{\partial\mathrm{Q}_{R}}\dif\mathscr{H}^{2}\qquad\text{for all $R\in I_{\mathbf{G}}$ and  $\varphi\in\mathrm{Lip}_{0}(\Omega)$},
\end{align}
as can even be shown in the larger context of $\mathscr{DM}^{1}$-fields. 
As a consequence of the preceding identity, we record that, 
for each $R\in I_{\mathbf{G}}$, $\langle\mathbf{G}\cdot\nu,\cdot\rangle_{\partial\mathrm{Q}_{R}}$ can be represented by 
the finite Radon measure $\widetilde{\mathbf{G}}\cdot\nu_{\partial\mathrm{Q}_{R}}\mathscr{H}^{2}\mres\partial\mathrm{Q}_{R}$. 

\smallskip
\emph{Step 3.} By our assumptions, we see that  $\mathbf{E}(t,\cdot)\in\lebe^{\infty}(\Omega;\R^{3})$, $\curl\mathbf{E}(t,\cdot)\in\lebe^{1}(\Omega;\R^{3})$ and $\partial_{t}\mathbf{H}(t,\cdot)\in\lebe^{1}(\Omega;\R^{3})$, $\mathrm{div}(\partial_{t}\mathbf{H}(t,\cdot))=0(\in\lebe^{1}(\Omega))$ for $\mathscr{L}^{1}$-a.e. $0<t<T$. In particular, $\FF_{1}:=\curl\mathbf{E}(t,\cdot)$ satisfies $\FF_{1}\in\lebe^{1}(\Omega;\R^{3})$ together with $\mathrm{div}(\FF_{1})=0(\in\lebe^{1}(\Omega))$. Moreover, we put $\FF_{2}:=\partial_{t}\mathbf{H}(t,\cdot)$. Now, as a consequence of \textcolor{black}{Theorem \ref{thm:stokes1st}}, there exists a set $\mathscr{I}\subset I:=(R_{1},R_{2})$ such that $\mathscr{L}^{1}(I\setminus\mathscr{I})=0$ and the following holds: 
\begin{itemize}
    \item For every $R\in \mathscr{I}_{1}$, $\langle\FF_{1}\cdot\nu,\cdot\rangle_{\partial\mathrm{Q}_{R}}$ can be represented by an $\R^{3}$-valued Radon measure $\mu_{\FF_{1}\cdot\nu}^{R}$ on $\partial\mathrm{Q}_{R}$, and by (ii), the Stokes theorem is available for $\mathbf{E}(t,\cdot)$ on $\mathrm{Q}_{R}$ in the following form: For each sufficiently small $r>0$, there exists $g_{x_{0},r,R}\in\lebe^{\infty}(\partial\!\ball_{r}(x_{0})\cap\partial\mathrm{Q}_{R})$ such that 
    \begin{align*}
    \mathfrak{S}_{\ball_{r}(x_{0})\cap\partial\mathrm{Q}_{R},\mathbf{E}(t,\cdot)}(\varphi) = \int_{\partial\!\ball_{r}(x_{0})\cap\partial\mathrm{Q}_{R}}g_{x_{0},R,r}\varphi\dif\mathscr{H}^{1}. 
    \end{align*}
    holds for all $\varphi\in\hold^{2}(\partial\mathrm{Q}_{R})$. Here, we use the subscript $\mathbf{E}(t,\cdot)$ to express that the Stokes functional is applied to the field $\mathbf{E}(t,\cdot)$.
\end{itemize}
We then define 
\begin{align*}
\widetilde{I}:= I_{\FF_{1}}\cap I_{\FF_{2}}\cap \mathscr{I}
\end{align*}
and record that $\mathscr{L}^{1}(I\setminus\widetilde{I})=0$. Clearly, if $R\in \widetilde{I}$, then we find by \eqref{eq:crucialID} that 
\begin{align*}
\mu_{\FF_{1}\cdot\nu}^{R}=\widetilde{\FF}_{1}\cdot\nu_{\partial\mathrm{Q}_{R}}\mathscr{H}^{2}\mres\partial\mathrm{Q}_{R}
\end{align*}
and similarly for $\widetilde{\FF}_{2}$. Moreover, the set 
\begin{align*}
\mathcal{L}_{R,\FF_{1}}\cap\mathcal{L}_{R,\FF_{2}}\qquad\text{has full $\mathscr{H}^{2}$-measure in $\partial\mathrm{Q}_{R}$}. 
\end{align*}
As a consequence of Fubini's theorem, $\mathscr{L}^{1}(I\setminus\widetilde{I})=0$, 
and the preceding statements imply that, with $\mathcal{B}_{R_{1},R_{2},\delta}(z_{0})$ as in \eqref{eq:box}, 
\begin{align}\label{eq:boxmeasure}
\bigcup_{R\in\widetilde{I}}(\mathcal{L}_{R,\FF_{1}}\cap\mathcal{L}_{R,\FF_{2}}\cap\mathcal{B}_{R_{1},R_{2},\delta}(z_{0}))\qquad\text{has full $\mathscr{L}^{3}$-measure in}\;\mathcal{B}_{R_{1},R_{2},\delta}(z_{0}). 
\end{align}
\emph{Step 4.} For $R\in\widetilde{I}$, we now let $x_{0}\in\mathcal{L}_{R,\FF_{1}}\cap\mathcal{L}_{R,\FF_{2}}\cap\mathcal{B}_{R_{1},R_{2},\delta}(z_{0})$ be arbitrary. For all sufficiently small $r>0$, (ii) implies that we may write $\ball_{r}(x_{0})\cap\partial\mathrm{Q}_{R}$ as $\Phi(t,\Sigma)$, where $\Sigma$ is a ball contained in $\mathcal{F}_{R_{1},\delta}$, and therefore a smooth boundary manifold relative to $\partial\mathrm{Q}_{R_{1}}$. By our definition of $\FF_{1},\FF_{2}$ and $\widetilde{I}$, the Stokes theorem now allows us express the empirical form of Faraday's law as 
\begin{align}\label{eq:FaraFara}
\int_{\partial\!\ball_{r}(x_{0})\cap\partial\mathrm{Q}_{R}}g_{x_{0},R,r}\varphi\dif\mathscr{H}^{1} =  - \int_{\ball_{r}(x_{0})\cap\partial\mathrm{Q}_{R}}\varphi\,\widetilde{\partial_{t}\mathbf{H}}\cdot\nu_{\partial\mathrm{Q}_{R}}\dif\mathscr{H}^{2}
\end{align}
whenever $\varphi=1$ in an open neighborhood of $\ball_{r}(x_{0})\cap\partial\mathrm{Q}_{R}$. On the other hand, since $\langle\FF_{1}\cdot\nu,\cdot\rangle_{\partial\mathrm{Q}_{R}}$ can be represented by the finite Radon measure $\widetilde{\FF}_{1}\cdot\nu_{\partial\mathrm{Q}_{R}}\mathscr{H}^{2}\mres\partial\mathrm{Q}_{R}$,
Lemma \ref{lem:goodrepmeas} yields that, for $0<r\ll 1$, 
\begin{align*}
\int_{\partial\mathrm{Q}_{R}\cap\ball_{r}(x_{0})}\widetilde{\FF}_{1}\cdot\nu_{\partial\mathrm{Q}_{R}}\dif\mathscr{H}^{2} & = \mathfrak{S}_{\ball_{r}(x_{0})\cap\partial\mathrm{Q}_{R},\mathbf{E}(t,\cdot)}(\varphi) \\ & =  \int_{\partial\!\ball_{r}(x_{0})\cap\partial\mathrm{Q}_{R}}g_{x_{0},R,r}\varphi\dif\mathscr{H}^{1} \\ & \!\!\!\stackrel{\eqref{eq:FaraFara}}{=} - \int_{\ball_{r}(x_{0})\cap\partial\mathrm{Q}_{R}}\widetilde{\partial_{t}\mathbf{H}}\cdot\nu_{\partial\mathrm{Q}_{R}}\dif\mathscr{H}^{2}, 
\end{align*}
where $\varphi=1$ is an open neighbourhood of $\ball_{r}(x_{0})\cap\partial\mathrm{Q}_{R}$. 
Since $\nu_{\partial\mathrm{Q}_{R}}=e_{2}$ is constant along $\partial\mathrm{Q}_{R}\cap\ball_{r}(x_{0})$, we may divide the overall identity by $r$ to obtain 
\begin{align*}
\Big(\dashint_{\partial\mathrm{Q}_{R}\cap\ball_{r}(x_{0})}(\widetilde{\FF}_{1}+\widetilde{\partial_{t}\mathbf{H}})\dif\mathscr{H}^{2}\Big)\cdot e_{2} = 0
\end{align*}
for all sufficiently small $0<r\ll1$. Sending $r\searrow 0$ and recalling \eqref{eq:oh!darling}, our choice of $x_{0}$ gives us 
\begin{align}\label{eq:voomvoomvanillacamera}
\lim_{r\searrow 0}\Big(\dashint_{\ball_{r}(x_{0})}(\FF_{1}(y)+\partial_{t}\mathbf{H}(y))\dif y\Big)\cdot e_{2} = 0.
\end{align}
Because $R\in\widetilde{I}$ and $x_{0}\in\mathcal{L}_{R,\FF_{1}}\cap\mathcal{L}_{R,\FF_{2}}\cap\mathcal{B}_{R_{1},R_{2},\delta}(z_{0})$ are arbitrary,  \eqref{eq:boxmeasure}--\eqref{eq:voomvoomvanillacamera} yield that 
\begin{align}\label{eq:boxcomponent}
(\curl\mathbf{E}(t,\cdot))_{2}+\partial_{t}\mathbf{H}_{2}(t,\cdot) = 0\qquad\text{$\mathscr{L}^{3}$-a.e. in $\mathcal{B}_{R_{1},R_{2},\delta}(z_{0})$}
\end{align}
for the second components of the underlying fields. 
As mentioned at the end of Step 2, we may cover $\Omega$ by boxes of form \eqref{eq:box}. 
Thus, \eqref{eq:boxcomponent} yields that $(\curl\mathbf{E}(t,\cdot))_{2}+\partial_{t}\mathbf{H}_{2}(t,\cdot)=0$ holds $\mathscr{L}^{3}$-a.e. in $\Omega$. 

\emph{Step 5.} With the natural modifications, the above reasoning also applies to boxes not only on the left- and right-hand sides, but also to boxes $\mathcal{B}'_{R_{1},R_{2},\delta}(z_{0})$ above and below $\mathrm{Q}_{R_{1}}$. An analogous reasoning as in Steps 1--4 then gives us  $(\curl\mathbf{E}(t,\cdot))_{3}+\partial_{t}\mathbf{H}_{3}(t,\cdot)=0$ $\mathscr{L}^{3}$-a.e. in $\Omega$; similarly, we conclude $(\curl\mathbf{E}(t,\cdot))_{1}+\partial_{t}\mathbf{H}_{1}(t,\cdot)=0$ holds  $\mathscr{L}^{3}$-a.e. in $\Omega$. Hence, $\curl\mathbf{E}(t,\cdot))+\partial_{t}\mathbf{H}_(t,\cdot)=0$ holds $\mathscr{L}^{3}$-a.e. in $\Omega$, which is precisely the claimed differential form $\eqref{eq:maxwell}_{4}$ of Faraday's law. \\ 

In the preceding derivation, it is also possible to admit proper curl-measure fields (meaning that $\curl\mathbb{E}(t,\cdot)$ is a measure), but then the conclusion only persists for the absolutely continuous part. We conclude with several remarks: 
\begin{rem}
Many problems from superconductivity (see, \emph{e.g.},   \cite{Mirandaetal,WanLaforest}) or Bean's critical state model \cite{BarrettPrigozhin} often involve the $\curl$-based $p$-Laplacian operator $\curl(|\curl\mathbf{u}|^{p-2}\curl\mathbf{u})$ on the Sobolev-type space $\sobo^{\curl,p}(\Omega)$ for large values of $p$; in particular, $p>3$. For such problems, the Stokes theorem
(Theorem \ref{thm:StokesWcurl}) asserts that no proper manifold selection is required; specifically, this gives scenarios where the above derivation can be simplified.
\end{rem}

\begin{rem}[Alfv\'{e}n's theorem]
By a similar approach, one can give a derivation of Alfv\'{e}n's theorem \cite{Alfven1,Alfven2} from ideal magnetohydrodynamics in the low regularity context. Alfv\'{e}n's theorem asserts that, in the situation of high magnetic Reynolds number, electrically conducting fluids and the embedded  magnetic fields move together. On a more formal level, denote by $\Phi_{\mathbf{H}}$ the magnetic flux of an electrically conducting fluid through a $\hold^{1}$-regular Lipschitz boundary manifold $\Sigma$, which is advected by the velocity field $\mathbf{u}$. Denoting the associated material or advective derivative by $\frac{\mathrm{D}}{\mathrm{D}t}=\partial_{t}+(\mathbf{u}\cdot\nabla)$, Alfv\'{e}n's theorem can be expressed as $\frac{\mathrm{D}}{\mathrm{D}t}\Phi_{\mathbf{H}}=0$. 
\end{rem}

\subsection{$\mathscr{CM}^{p}$-fields and measure vorticities} We now briefly sketch how curl-measure fields can be used in fluid mechanical problems; the precise implementation shall be the objective of future work. To this end, consider the Cauchy problem of the $n$-dimensional incompressible Euler equations ($n=2,3$):
\begin{equation}\label{8.1}
\begin{cases}
\uu_t +\nabla\cdot (\uu\otimes \uu)+\nabla p=0,\qquad x\in \mathbb{R}^n, t>0,\\
{\rm div}\,\uu=0,\\
\uu|_{t=0}=\uu_0(x),
\end{cases}
\end{equation}
where $\uu=(u_1, \cdots,u_n)^\top$ is the fluid velocity, $p$ is the scalar pressure,
\begin{align*}
\uu\otimes\uu=(u_iu_j)_{i,j=1,\cdots,n}
\end{align*}
and $\uu_0=\uu_{0}(x)$ is an initial incompressible velocity field, {\it i.e.},
${\rm div}\uu_0=0$. The fluid vorticity is $\omega={\rm curl}\, \uu$; in particular, for $n=2$,
$\omega={\rm curl}\, \uu=u_{2x_1}-u_{1x_2}$.

One of the fundamental problems is to understand the structure of global weak solutions whose vorticity 
is a Radon measure, that is,
\begin{equation}\label{8.1c}
|\omega(t,\cdot)|(\R^{n})<\infty\qquad\,\,
\mbox{for any $t>0$}.
\end{equation}
To achieve this, it is essential to analyze the solution behaviors across any discontinuities such as
vortex sheets of the solutions of the above  equations.

\medskip
To this end, consider a discontinuity 
$$
\Sigma=\{(t,x)\,:\, \Phi(t,x)=0\}
$$
with the normal $N=(\partial_t\Phi, \nabla_x\Phi)=(\nu_0, \nu)=(\nu_0, \nu_1,\cdots, \nu_n)\ne 0$
in the solution $\uu(t,x)$:
\begin{align*}
\uu(t,x)=\begin{cases}
\uu^-(t,x) \qquad\mbox{when $\Phi(t,x)<0$},\\
\uu^+(t,x) \qquad\mbox{when $\Phi(t,x)>0$},
\end{cases}
\end{align*}
where $\uu^\pm(t,x)$ are weak solutions in $\{\pm \Phi(t,x)>0\}$, respectively, 
with suitable traces on $\Sigma$, and the vorticity $\omega(t,x)$ 
satisfies \eqref{8.1c}.
Then the Gauss-Green formula for divergence-measure fields in $\lebe^2$ indicates that
the necessary and sufficient conditions for $\uu(t,x)$ to be a weak solution
are the Rankine-Hugoniot conditions for the traces of $\uu$:
\begin{align}
&[\uu \nu_0 +(\uu\otimes\uu)\nu+p\nu]=0,\label{8.1a}\\
&[\uu\cdot\nu]=0,\label{8.2a}
\end{align}
where $[v]:=v|_{\Sigma+}-v|_{\Sigma-}$ denotes the jump of $v$ 
across the discontinuity $\Sigma$. Equally, the Trace Theorems  \ref{thm:tangtrace} and \ref{thm:tracemain1} for curl-measure fields implies that
\begin{align}
[\uu\times \nu] =\omega|_{\Sigma}\ne 0, \label{8.3a}
\end{align} 
where, for any fixed $t>0$, $\uu\times \nu$ is tangential to $\Sigma$ in the $x$-variables; here, we suppress the trace operators for ease of notation.

Denote $\uu_{\nu}=\uu\cdot \nu$ and $\uu_\tau=\mathbf{u}\times \nu$. Then \eqref{8.2a} becomes
\begin{align}
\uu_\nu^+= \uu_\nu^-, \label{8.2b}
\end{align}
which means that the normal trace of the velocity must be continuous 
across the discontinuity $\Sigma$, while
\eqref{8.1a} becomes
\begin{align}
\partial_t\Phi = -\uu_\nu^+= -\uu_\nu^-. \label{8.1b}
\end{align}
This is a condition on the normal traces from the left- and the right-hand sides of the sheet, whereas \eqref{8.3a} becomes
\begin{align}
\uu_{\tau}^+ - \uu_{\tau}^-=\omega|_{\Sigma}\ne 0  \label{8.3b}
\end{align}
as a condition on the tangential traces from the left- and the right-hand sides of the sheet. In particular, this indicates that the strength of the vorticity $\omega$ on $\Sigma$
determines the strength of the jump of the tangential velocity fields along $\Sigma$.

\medskip
For $n=2$, Delort in \cite{Delort} first established the following existence theorem:
Let $\omega_0={\rm curl}\,\uu_0\in \mathrm{RM}^+_{\mathrm{fin},\rm comp}\cap \mathrm{H}^{-1}(\mathbb{R}^{2};\R^{2})$.
Then there exist
\begin{align}\label{eq:Delort}
\uu\in \lebe_{\rm loc}^\infty(\mathbb{R}; \lebe_{\rm loc}^2(\mathbb{R}^2; \mathbb{R}^2)),
\qquad p\in \lebe^\infty_{\rm loc}(\mathbb{R}; \mathcal{D}'(\mathbb{R}^2))
\end{align}
such that $(u,p)$ is a global weak solution of the Cauchy problem \eqref{8.1} 
satisfying 
\begin{align}\label{eq:Delort1}
\omega(t,x)\ge 0, \quad 
|\omega(t,\cdot)|(\R^{2})\le |\omega_0|(\R^{2})\qquad\,\,
\mbox{for any $t>0$}.
\end{align}
See also  Evans-M\"{u}ller \cite{Evans} and Vecchi-Wu \cite{Wu}.
For the two-dimensional case, both $\uu_{\nu}=\uu\cdot \nu$ and $\uu_\tau=\uu\times \nu$
are scalar and satisfy
\begin{align}\label{8.2d}
\uu_\nu^+= \uu_\nu^-, \qquad 
\uu_{\tau}^+ - \uu_{\tau}^-=\omega|_{\Sigma}\ge 0. \
\end{align}
which means that the normal trace of the velocity must be continuous 
across the discontinuity $\Sigma$, while the positive strength of the vorticity $\omega$ on $\Sigma$
determines the positive strength of the jump of the tangential velocity fields along $\Sigma$.

In the two-dimensional case, the spatial local $\lebe^{2}$-integrability asserted in \eqref{eq:Delort} implies that \eqref{8.2d} has a precise meaning in the sense of normal or tangential traces of $\mathscr{DM}^{2}$- or  $\mathscr{CM}^{2}$-fields on $\R^{2}$; note, however, that in the two-dimensional case, the curl 
is a rotated divergence, and so the theory of $\mathscr{CM}^{2}$-fields can essentially be reduced to that of $\mathscr{DM}^{2}$-fields. In the three-dimensional case, however, this is not possible. 
Based on the tangential trace theorems from \S \ref{sec:tracetheorem}, a suitable version of 
\eqref{8.2d} persists \emph{if} one assumes that the solution $\mathbf{u}$ belongs to $\lebe_{\locc}^{\infty}(\R;\lebe_{\locc}^{p}(\R^{3};\R^{3}))$. However, at present, criteria for the latter are not clear. In this sense, a corresponding variant of \eqref{8.2d} is conditional to the existence of a sufficiently locally integrable solution. Subject to such an assumption, the Stokes theorem from \S \ref{sec:StokesMostGeneral} immediately allows for the description of the voriticity flux through the sheet.

The analysis of vortex sheets has faced numerous contributions over the past decades; see, \emph{e.g.}, \cite{Caflisch1988,Caflisch1989,Caflisch1992,ChenWang,Wu2002} and Wu \cite{Wu2006} for the detailed treatment of the two-dimensional case.
 In this regard, one is primarily interested in the evolution of the sheets, and the latter is governed by the so-called \emph{Birkhoff-Rott equation}. This is an evolution equation driven by a singular integral operator on manifolds. To keep the paper at a reasonable length, we do not aim to contribute to the solution theory of this equation here, a task that we shall pursue elsewhere. Instead, we argue that the non-local Birkhoff-Rott operator is indeed well-defined on suitable subclasses of $\mathscr{CM}^{\infty}$-fields, which also allows us to consider the equation in the low regularity context. 

To this end, we briefly revisit the derivation of the Birkhoff-Rott equation as given, for instance, by  Caflisch et al. \cite{Caflisch1992}; see also \cite{AgishteinMigdal,Birkhoff,Caflisch1989,Dahm,Kaneda}. 
For simplicity, we consider a fluid in full space $\R^{3}$ which splits $\R^{3}$ into two open subsets $\Omega^-$ and $\Omega^+$ divided by a two-dimensional smooth or, 
at least, Lipschitz submanifold $\Sigma$ oriented by 
$\nu:=\nu_{\partial\Omega^+}\colon\Sigma\to\mathbb{S}^{2}$; as above, its velocity is denoted by $\mathbf{u}$. 
To view $\Sigma$ as a vortex sheet, we 
suppose that the fluid satisfies is solenoidal and rotation-free away from $\Sigma$, meaning that 
\begin{align}\label{eq:solcurl}
\mathrm{div}\,\mathbf{u}^\pm = 0 \quad \;
\curl\mathbf{u}^\pm=0\;\qquad\text{in $\,\Omega^\pm$} 
\end{align}
and that it satisfies Euler's equations away from $\Sigma$ (see also \eqref{8.1}): 
\begin{align}
\partial_{t}\mathbf{u}^\pm 
+\nabla\cdot(\mathbf{u}^\pm\otimes\mathbf{u}^\pm)+\nabla p^\pm = 0\qquad\,\text{in $\,\Omega^\pm$}.
\end{align}
Here, one assumes the pressure functions $p^\pm\colon\Omega^\pm\to\R$ 
to be sufficiently regular. Moreover, in order to call $\Sigma$ a \emph{vortex sheet}, one typically requires continuity both of the pressure and the normal components 
along $\partial\Omega$: 
\begin{align}\label{eq:euleracrossboundary} 
\mathbf{u}_{\nu}^{+} = \mathbf{u}_{\nu}^{-}\qquad\text{along $\;\Sigma$}. 
\end{align}
In view of  $\eqref{eq:solcurl}_{1}$, this amounts to $\mathrm{div}\mathbf{u}=0$ globally in the sense of distributions on $\R^{3}$. 
For a vortex sheet, the vorticity $\bm{\omega}:=\curl\mathbf{u}$ is a measure on $\Sigma$ which is absolutely continuous with respect to $\mathscr{H}^{2}$ on $\Sigma$; more precisely, we have 
\begin{align}\label{eq:vortexsheets1}
\bm{\omega} = (\mathbf{u}_{\tau}^{+}-\mathbf{u}_{\tau}^{-})\mathscr{H}^{2}\mres\Sigma.
\end{align}
In order to describe the evolution of $\Sigma$ subject to \eqref{eq:solcurl}--\eqref{eq:euleracrossboundary}, it is customary to choose a parametrization of $\Sigma_{t}$ by a map $\mathbf{X}\colon \R^{2}\times [0,t_{0})\to\R^{3}$ so  that $\Sigma_{t}=\{\mathbf{X}(\xi_{1},\xi_{2},t)\colon\;(\xi_{1},\xi_{2})\in\R^{2}\}$ for $0\leq t<t_{0}$. The evolution of $\Sigma$ is primarily influenced by the normal component of $\partial_{t}\mathbf{X}$, as changes in its tangential component can be treated by re-parametrizations. As a modelling hypothesis, one assumes that the motion of $\Sigma_{t}$ is driven by the mean velocity $\overline{\mathbf{u}}:=\frac{1}{2}(\mathbf{u}^-|_{\Sigma}+\mathbf{u}^+|_{\Sigma})$ along the sheet. Taking into account the parametrization, this corresponds to the partial differential equation 
\begin{align}\label{eq:weatherreport}
\partial_{t}\mathbf{X}(\xi_{1},\xi_{2},t)=\overline{\mathbf{u}}(\xi_{1},\xi_{2},t). 
\end{align}
On the other hand, the global solenoidality hypothesis on $\mathbf{u}$ gives us the representation 
\begin{align*}
 \curl\bm{\omega} = \curl\curl\mathbf{u} = \nabla\mathrm{div}\mathbf{u} -\Delta\mathbf{u} = -\Delta\mathbf{u}
\end{align*}
based on the Biot-Savart law; see, \emph{e.g.}, \cite{Lopez} for a rigorous discussion of this point. Supposing a natural decay assumption on $\mathbf{u}$ at infinity, convolving with the fundamental solution of $(-\Delta)$ and a subsequent integration by parts imply that 
\begin{align*}
\mathbf{u}(x) & = - \frac{1}{4\pi}\int_{\R^{3}}\frac{x-y}{|x-y|^{3}}\times\dif\bm{\omega}(y) \\ 
& = \frac{1}{4\pi}\int_{\Sigma}\frac{\dif\bm{\omega}}{\dif\mathscr{H}^{2}}(y)\times \frac{x-y}{|x-y|^{3}}\dif\mathscr{H}^{2}(y)
\end{align*}
whenever $x\notin\Sigma$. Now let $x_{0}\in\Sigma$. 
Approaching a sheet point $x_{0}$ from $\Omega^{+}$ or $\Omega^{-}$, the classical three dimensional Plemelj formulas (see also \cite[Appendix A]{Caflisch1992}) then yield 
\begin{align*}
\mathbf{u}^\pm(x_{0}) = \frac{1}{4\pi}\mathrm{p.v.}\int_{\Sigma}\frac{\dif\bm{\omega}}{\dif\mathscr{H}^{2}}(y)\times \frac{x_{0}-y}{|x_{0}-y|^{3}}\dif\mathscr{H}^{2}(y) \pm \frac{1}{2}\frac{\dif\bm{\omega}}{\dif\mathscr{H}^{2}}(x_{0})\times\nu(x_{0}).
\end{align*}
Based, \emph{e.g.}, on Theorem \ref{thm:tracemain1} and our choice of orientations, we may express the density as the difference of tangential traces via $\frac{\dif\bm{\omega}}{\dif\mathscr{H}^{2}}=-(\mathbf{u}_{\tau}^{+}-\mathbf{u}_{\tau}^{-})$ $\mathscr{H}^{2}$-a.e. on $\Sigma$. Hence, it follows for $\overline{\mathbf{u}}$ that
\begin{align*}
\overline{\mathbf{u}}(x)& =-\frac{1}{4\pi}\mathrm{p.v.}\int_{\Sigma}(\mathbf{u}_{\tau}^{+}(y)-\mathbf{u}_{\tau}^{-}(y))\times \frac{x-y}{|x-y|^{3}}\dif\mathscr{H}^{2}(y), 
\end{align*}
and the non-local operator on the right-hand side is also called the \emph{Birkhoff-Rott operator}. Based on the parametrization of $\Sigma$ by $\mathbf{X}$, \eqref{eq:weatherreport} becomes 
\begin{align}\label{eq:birkhoffrott}
\partial_{t}\mathbf{X} = -\frac{1}{4\pi}\mathrm{p.v.}\int_{\Sigma}(\mathbf{u}_{\tau}^{+}(X')-\mathbf{u}_{\tau}^{-}(X'))\times \frac{\mathbf{X}-X'}{|\mathbf{X}-X'|^{3}}\dif\mathscr{H}^{2}(X'), 
\end{align}
which is one possible form of the Birkhoff-Rott equation. Note, however, that this derivation works subject to strong assumptions on the boundary behaviour of $\mathbf{u}^{\pm}$ along $\Sigma$.

If $\Sigma$ is sufficiently regular (\emph{e.g.} of class $\hold^{1}$), then the Birkhoff-Rott operator corresponds to a singular integral of convolution type; by classical results (see, \emph{e.g.} \cite[Theorem 6.6]{Duoandikoetxea}), the latter are well-defined on $\lebe^{\infty}(\Sigma;\R^{3})$. By Theorem \ref{thm:tracemain1}, the Birkhoff-Rott operator is well-defined for spatial $\mathscr{CM}^{\infty}$-fields, and in this sense the driving non-local term in the equation is even well-defined for $\mathscr{CM}^{\infty}$-fields. Even subject to harmonicity of $\mathbf{u}$ away from $\Sigma$ as assumed here, the boundary behaviour close to $\Sigma$ is crucially governed by the properties of the tangential traces of $\mathscr{CM}^{\infty}$-fields; the precise investigation of the $\lebe^{p}$-attainment of traces subject to the natural $\mathscr{CM}^{\infty}$-hypothesis will be the objective of future work.

Finally, it is well-known that even in the flat case, singular integrals of convolution type do not map $\lebe^{\infty}$ to itself; instead, the target is typically $\bmo$, the functions of bounded mean oscillation. This indicates that, in order to obtain a low smoothness space where the Birkhoff-Rott operator is stable, the above derivation indicates that a suitable space might be given by the vector fields  $\FF\in\bmo(\R^{3};\R^{3})$ for which $\curl\FF\in\mathrm{RM}_{\mathrm{fin}}(\R^{3};\R^{3})$. 

\section*{Appendix}
\subsection*{Appendix A: Sobolev and Lipschitz spaces on manifolds }\label{sec:AppendixA}
For the reader's convenience, we collect here some background definitions and results on Sobolev spaces on manifolds which have entered the main part.
\subsection*{Appendix A.1:  Weakly differentiable functions on manifolds}\label{sec:AppendixA1} 
 
Weakly differentiable functions and Sobolev spaces on manifolds are usually introduced by means of completions of spaces of smooth functions (see  \emph{e.g.}, Hebey \cite{Hebey}). 

Let $\Sigma\subset\R^{n}$ be a closed, $(n-1)$-dimensional $\hold^{k}$-manifold, and let $\Sigma=\bigcup_{j=1}^{N}U_{j}$ be a cover of $\Sigma$ by relatively open sets and with corresponding charts $\kappa_{j}\in\hold^{k}(U_{j};\R^{n-1})$. Letting $\varphi_{j}\in\hold^{k}(\Sigma)$ be such that $\spt(\varphi_{j})\subset U_{j}$, $j=1,\cdots,N$, with $\sum_{j=1}^{N}\varphi_{j}=1$ on $\Sigma$, 
we define
\begin{align}\label{eq:manifoldSob}
\sobo^{s,p}(\Sigma):=\Big\{u\colon\;\Sigma\to\R\colon\;\|u\|_{\sobo^{s,p}(\Sigma)}^{p}:=\sum_{|\alpha|\leq m}\|(\varphi_{j}u)\circ\kappa_{j}^{-1}\|_{\sobo^{s,p}(\R^{n-1})}^{p}<\infty\Big\},
\end{align}
for $0<s\leq k$ and $1\leq p<\infty$.
We point out that the restriction $s\leq k$ is due to the fact that $\sobo^{s,p}(\Sigma)$ might trivialize otherwise. If, in \eqref{eq:manifoldSob},  $s\in\mathbb{N}$, we understand $\sobo^{s,p}(\R^{n-1})$ as the usual integer order Sobolev spaces, whereas $\sobo^{s,p}(\R^{n-1})$ are the fractional or Sobolev-Slobodecki\u{\i} spaces otherwise. If $s=0$, we put $\sobo^{0,p}(\Sigma):=\lebe^{p}(\Sigma)$, and if $s<0$, we define $\sobo^{s,p'}(\Sigma):=(\sobo^{-s,p}(\Sigma))'$. The corresponding Bessel potential spaces $\mathscr{L}^{s,p}(\Sigma)$ are defined analogously by reduction to $\R^{n-1}$. 

\begin{prop}[Embeddings]\label{prop:embeddingsmanifolds}
In the above situation, let $s\leq k$. Then 
\begin{enumerate}
\item\label{item:embmanifold1} If $\frac{2n}{n+1}<p\leq\infty$, then $\sobo^{1,2}(\Sigma)\hookrightarrow\sobo^{1/p,p'}(\Sigma)$ and so $\sobo^{-1/p,p}(\Sigma)\hookrightarrow \sobo^{-1,2}(\Sigma)$. 
\item\label{item:embdmanifold2} If $p>n$, then $\sobo^{1-1/p,p}(\Sigma)\hookrightarrow\hold_{b}(\Sigma)$.
\end{enumerate}
\end{prop}
\begin{proof} First let $1<p<\infty$. We directly work on flat space; the assumed regularity allows to localize. 
Following Triebel \cite[\S 3.3.1]{Triebel}, $\sobo^{1,2}(\R^{n-1})\simeq\mathrm{B}_{2,2}^{1}(\R^{n-1})\hookrightarrow\besov_{p',p'}^{1/p}(\R^{n-1})$ if and only if 
\begin{align*}
1-\frac{n-1}{2}>\frac{1}{p}-\frac{n-1}{p'}\Longleftrightarrow p'<\frac{2n}{n-1} \Longleftrightarrow p>\frac{2n}{n+1}. 
\end{align*} 
If $p=\infty$, then we interpret $\sobo^{1/p,p'}(\R^{n-1}):=\lebe^{1}(\R^{n-1})$, and so we obtain the continuity of the embedding $\sobo^{1,2}(\R^{n-1})\hookrightarrow\sobo^{1/p,p'}(\R^{n-1})$ in both cases. From here, \ref{item:embmanifold1} follows. On the other hand, the H\"{o}lder spaces can be characterized as $\hold^{0,t}(\R^{n-1})\simeq\besov_{\infty,\infty}^{t}(\R^{n-1})$, provided that $t\in\R_{>0}\setminus\mathbb{N}$. The embedding $\sobo^{1-1/p,p}(\R^{n-1})\simeq\besov_{p,p}^{1-1/p}(\R^{n-1})\hookrightarrow\besov_{\infty,\infty}^{t}(\R^{n-1})$ then holds provided that 
\begin{align*}
1-\frac{1}{p}-\frac{n-1}{p}>t. 
\end{align*}
In particular, the exponent on the left-hand side is larger than $0$ if $p>n$.
In this case, we therefore find $t>0$ such that $\sobo^{1-1/p,p}(\R^{n-1})\hookrightarrow \hold^{0,t}(\R^{n-1})\hookrightarrow\hold_{b}(\R^{n-1})$. This implies \ref{item:embdmanifold2}, and the proof is complete. 
\end{proof}

We now provide an equivalent intrinsic characterization of the integer-order Sobolev spaces. In order to switch between the different characterizations, we recall  the notion of pseudo-inverses from \cite{Moore,Penrose}: For $S\in\R^{n\times (n-1)}$,  define 
\begin{align*}
S^{\dagger} :=(S^{\top}S)^{-1}S^{\top}\in \R^{(n-1)\times n}, 
\end{align*}
which immediately yields $S^{\dagger}S=\mathrm{Id}_{\R^{n}}$; on the other hand, $SS^{\dagger}$ is the orthogonal projector onto the range of $S$. The following lemma is routine, an can be established by the methods employed, \emph{e.g.}, in \cite{Gurtin}.
\begin{lem}
Let $\Sigma\subset\R^{n}$ be a closed, $(n-1)$-dimensional $\hold^{1}$-manifold, and let $1\leq p \leq \infty$. Then the following are equivalent for $u\in\lebe^{p}(\Sigma)${\rm :}
    \begin{enumerate}
        \item $u\in\sobo^{1,p}(\Sigma)$. 
        \item\label{item:weakoweako1} With $\kappa_{j}$ as in \eqref{eq:manifoldSob}, for each $j\in\{1, \cdots,N\}$, $v:=(\varphi_{j}u)\circ\kappa_{j}^{-1}\in\sobo_{\locc}^{1,p}(\kappa_{j}(U_{j}))$. 
        \item $u$ has tangential weak gradients in $\lebe^{p}(\Sigma;T_{\Sigma})$; with the notation from \ref{item:weakoweako1}, they can locally be computed via
        \begin{align*}
            \nabla_{\tau}{u}(x_{0})=(\nabla{v}(z_{0}))(\nabla\kappa_{j}^{-1}(z_{0}))^{\dagger}
        \end{align*}
        for $\mathscr{H}^{n-1}$-a.e. $x_{0}\in\Sigma$ and $z_{0}:=\kappa_{j}(x_{0})$. 
    \end{enumerate}
\end{lem}
Based on the preceding lemma, one then deduces the following assertion which is clear in the flat case.
\begin{lem}\label{lem:LipCharManifold}
Let $\Omega'\subset\R^{n}$ be open and bounded with boundary of class $\hold^{1}$. If a function $f\colon \partial\Omega'\to\R$ is Lipschitz, then $f\in\sobo^{1,\infty}(\Sigma)$, and there exists a constant $c=c(\Omega')>0$ such that $\|\nabla_{\tau}f\|_{\lebe^{\infty}(\partial\Omega')}\leq c\,\mathrm{Lip}_{\partial\Omega'}(f)$. 
\end{lem}

\subsection*{Appendix A.2:  Smooth approximation}\label{sec:AppendixA2} 

The following result is clearly well-known to the experts, but we have not found a precise reference and thus supply the quick proof.
\begin{prop}\label{prop:smoothapp}
Let $\Omega'\subset\R^{n}$ be open and bounded with $\hold^{1}$-boundary $\partial\Omega'$. Moreover, let $\mathbf{u}\in\hold(\partial\Omega';T_{\partial\Omega'})$. Then there exists a sequence $(\mathbf{u}_{j})\subset\hold^{1}(\partial\Omega';T_{\partial\Omega'})$ such that $\|\mathbf{u}-\mathbf{u}_{j}\|_{\lebe^{\infty}(\partial\Omega')}\to 0$ as $j\to\infty$. 
\end{prop}
\begin{proof}
We emphasize that the actual point of the proof is that the resulting mollified fields take values in the tangent space $T_{\partial\Omega'}$. Throughout, we cover $\partial\Omega'$ by finitely many open balls $\ball^{1}, \cdots,\ball^{N}\subset\R^{n}$ such that, for each $k\in\{1,\cdots,N\}$, $\partial\Omega'\cap \ball^{k}$ can be translated and rotated such that $\partial\Omega'\cap\ball^{k}$ can be written as $\mathrm{graph}(g_{k})$, where $g_{k}\colon U_{k}\to\R$ is a $\hold^{2}$-function on an open set $U_{k}\subset\R^{2}$. 

\smallskip
We write, with $(\frac{\partial}{\partial x_{i}}\rvert_{x})_{i=1,\cdots,n-1}$ 
denoting a basis of $T_{\mathrm{graph}(g_{k})}(x)$, 
\begin{align*}
\mathbf{u}(x)=\sum_{i=1}^{n-1}u_{i}(x)\frac{\partial}{\partial x_{i}}\bigg\rvert_{x}
\end{align*}
and define $v_{i}\colon U_{k}\ni x'\mapsto u_{i}(x',g_{k}(x'))\in\R^{n}$. 
For a rescaled standard mollifier $\rho_{\varepsilon}\in\hold_{\rm c}^{\infty}(\R^{2})$, 
we denote $v_{i}^{\varepsilon}:=\rho_{\varepsilon}*v_{i}$ so that $v_{i}^{\varepsilon}\to v_{i}$ uniformly on every relatively compact subset of $U_{i}$ as $\varepsilon\searrow 0$. For $x=(x',g_{k}(x'))$, 
we define $u_{i}^{\varepsilon}(x):=v_{i}^{\varepsilon}(x')$ and 
\begin{align*}
\mathbf{u}_{\varepsilon}(x) := \sum_{i=1}^{n-1}u_{i}^{\varepsilon}(x)\frac{\partial}{\partial x_{i}}\bigg\rvert_{x},
\end{align*}
whereby we conclude that, for every $x\in U_{k}$,
\begin{align*}
|\mathbf{u}(x)-\mathbf{u}_{\varepsilon}(x)| \leq \sum_{i=1}^{n-1}|u_{i}(x)-u_{i}^{\varepsilon}(x)| = \sum_{i=1}^{n-1}|v_{i}(x')-v_{i}^{\varepsilon}(x')|. 
\end{align*}
In consequence, $\mathbf{u}_{\varepsilon}\to\mathbf{u}$ uniformly on every relatively compact subset $V\subset\partial\Omega'\cap\ball^{k}$. By use of a partition of unity, this gives rise to the requisite sequence. This completes the proof. 
\end{proof}

\subsection*{Appendix A.3:  Elliptic regularity for 
the Laplace-Beltrami operator}\label{sec:AppendixA3} 
In this subsection, we collect some background material on the existence and regularity of elliptic partial differential equations on manifolds. To this end, we recall that a manifold is closed provided that it is compact and without boundary; we denote by $g=(g^{jk})_{j,k=1}^{n}$ its metric as usual. For a sufficiently regular function $u\colon\Sigma\to\R$, its  \emph{Laplace-Beltrami operator}  is given in local coordinates by
\begin{align}\label{eq:laplacebeltrami}
\Delta_{\tau}u:=\frac{1}{\sqrt{|g|}}\sum_{j,k=1}^{n}\partial_{j}(\sqrt{|g|}g^{jk}\partial_{k}u). 
\end{align}
\begin{prop}[Elliptic regularity]\label{prop:ellreg}
Let $\Sigma\subset\R^{n}$ be a closed $(n-1)$-dimensional $\hold^{2}$-manifold. For each $1<p<\infty$ and every $s<2$, $u\in\mathscr{L}^{s-1,p}(\Sigma)$ and $\Delta_{\tau}u\in\mathscr{L}^{s,p}(\Sigma)$ imply that $u\in\mathscr{L}^{s,p}(\Sigma)$. Moreover, there exists a constant $c=c(\Sigma,s,p)>0$ such that 
\begin{align}\label{eq:ellipticestimate}
\|u\|_{\mathscr{L}^{s,p}(\Sigma)}\leq c\big(\|\Delta_{\tau}u\|_{\mathscr{L}^{s-2,p}(\Sigma)}+\|u\|_{\mathscr{L}^{s-1,p}(\Sigma)} \big).
\end{align}
\end{prop}
The proof of Proposition \ref{prop:ellreg} is classical when $\Sigma$ is a $\hold^{\infty}$-manifold and $p=2$; see, for instance, 
Agranovich \cite[Theorem 2.2.5]{Agranovich}, Kumano-go \cite[Theorem 6.8]{Kumanogo}, 
and Taylor \cite[Chapter XI, Proposition 2.4]{Taylor} for a treatment of general elliptic pseudodifferential operators on manifolds, and \cite[Example 2.2.2]{Agranovich}\emph{ff.} for an explicit treatment of the Laplace-Beltrami operator. We briefly sketch how Proposition \ref{prop:ellreg} can be obtained: The referenced results hinge on the passage to the Laplace-Beltrami operator in local coordinates 
(see \eqref{eq:laplacebeltrami}), from where the corresponding estimate \eqref{eq:ellipticestimate} can be inferred from classical elliptic estimates on $\R^{n-1}$. The latter also holds for $1<p<\infty$. 
In the present context, the $\hold^{2}$-regularity allows to perform an analogous localization procedure \emph{as long as} the smoothness satisfies $s<2$. We briefly comment on the existence of solution for the associated Poisson problem:
\begin{rem}[Existence]\label{rem:existencevar}
Let $\Omega'\subset\R^{n}$ be open and bounded with $\hold^{2}$-boundary. 
We consider the Dirichlet energy 
\begin{align*}
\mathscr{E}[{u}]:=\frac{1}{2}\int_{\partial\Omega'}|\nabla_{\tau}{u}|^{2}\dif x - \langle f,{u}\rangle_{\sobo^{-1,2}(\partial\Omega')\times\sobo^{1,2}(\partial\Omega')}\qquad\text{on}\;\sobo_{0}^{1,2}(\partial\Omega'), 
\end{align*}
where now 
\begin{align*}
\sobo_{0}^{1,2}(\partial\Omega'):=\Big\{v\in\sobo^{1,2}(\partial\Omega')\colon\;\int_{\partial\Omega'}v\dif\mathscr{H}^{2}=0\Big\}. 
\end{align*}
Moreover, we assume that $f\in\sobo^{-1,2}(\partial\Omega')$ satisfies 
\begin{align}\label{eq:fixing}
\langle f,\mathbbm{1}_{\partial\Omega}\rangle_{\partial\Omega} =0
\end{align}
in the sense of the usual pairing between $f\in\sobo^{-1,2}(\partial\Omega)$ and $\sobo^{1,2}(\partial\Omega)$-functions. 
Then there exists a unique minimizer of $\mathscr{E}$ in $\sobo_{0}^{1,2}(\partial\Omega')$. This is a straightforward consequence of the direct method: By Young's and Poincar\'{e}'s inequalities, $\mathscr{E}$ is bounded below on $\sobo_{0}^{1,2}(\partial\Omega')$. We may thus pick a minimizing sequence $(u_{j})$, meaning that $\mathscr{E}[u_{j}]\to m := \inf_{\sobo_{0}^{1,2}(\partial\Omega')}\mathscr{E}$.
Again, by Poincar\'{e}'s inequality, $(u_{j})$ is bounded in $\sobo^{1,2}(\partial\Omega')$, and so there exists a weakly convergent subsequence: $u_{j_{i}}\rightharpoonup u\in\sobo_{0}^{1,2}(\partial\Omega')$. Clearly, $\mathscr{E}$ is lower semicontinuous with respect to weak convergence in $\sobo^{1,2}(\partial\Omega')$, whereby $u$ is a minimiser of $\mathscr{E}$. By virtue of \eqref{eq:fixing}, this implies that 
\begin{align}\label{eq:weakLaplacian}
\int_{\partial\Omega'}\nabla_{\tau}u\cdot\nabla_{\tau}\varphi\dif\mathscr{H}^{2}=\langle f,\varphi\rangle_{\sobo^{-1,2}(\partial\Omega')\times\sobo^{1,2}(\partial\Omega')}\qquad\text{for all}\;\varphi\in\sobo^{1,2}(\partial\Omega').
\end{align}
It is clear that \eqref{eq:fixing} is necessary, which can be seen the easiest for $f\in\lebe^{2}(\partial\Omega')$; indeed, in this case, \eqref{eq:weakLaplacian} is the weak formulation of $-\Delta_{\tau}u=f$ on $\partial\Omega'$, and integrating 
by parts yields 
\begin{align*}
\int_{\partial\Omega'}f\dif\mathscr{H}^{2}= -\int_{\partial\Omega'}\Delta_{\tau}u\dif\mathscr{H}^{2} = \int_{\Gamma_{\partial\Omega'}}u(\nabla_{\tau}u)\cdot\nu_{\Gamma_{\partial\Omega'}}\dif\mathscr{H}^{2}=0
\end{align*}
because $\partial\Omega'$ has empty boundary: $\Gamma_{\partial\Omega'}=\emptyset$. 
\end{rem}

\subsection*{Appendix B: Proof of Lemma \ref{lem:GoodLip}}\label{sec:AppendixB}
We now establish Lemma \ref{lem:GoodLip}. The statement is certainly clear to the experts, but we have not found a precise reference and thus include the proof for the sake of completeness.

\begin{proof}[Proof of Lemma \ref{lem:GoodLip}\ref{item:Lipextend1}]
There are many ways to come up with such a function $f_{\delta}$, and we address one possibility: Up to localising by use of a smooth partition of unity, rotating and translating, we may assume $\partial\Omega'$ to be a finite union of graphs of $\hold^{1}$-functions. Flattening the single pieces, we may then reduce to the half-space situation. Hence, let $f\in\hold_{\rm c}^{1}(\R^{n-1}\times\{0\})$ be given and consider a standard mollifier $\rho\colon \R^{n-1}\to\R$. We then consider 
\begin{align*}
(Ef)(x',x_{n}):=\begin{cases}\displaystyle \int_{\R^{n-1}}\rho_{x_{n}}(x'-y')f(y',0)\dif y' = \int_{\R^{n-1}}\rho(z')f(x'+x_{n}z',0)\dif z'&\!\!\!\!\text{if}\;x_{n}>0, \\[1mm] f(x',0)&\!\!\!\!\text{if}\;x_{n}=0
\end{cases}
\end{align*}
for $(x',x_{n})\in\R^{n-1}\times (0,\infty)$.
We choose a monotonously decreasing  $\hold^{\infty}$-function $\theta\colon\R_{\geq 0}\to [0,1]$ with $\theta(0)=1$ for $0\leq t\leq \frac{1}{2}$, $\theta(t)=0$ for $t\geq 1$, and $|\theta'|\leq 4$, and set
\begin{align}\label{eq:psideltadef1}\tag{B.1}
\overline{f}_{\delta}(x):=\theta(\frac{x_{n}}{\delta})(Ef)(x),\qquad x=(x',x_{n})\in\R^{n-1}\times [0,\infty). 
\end{align} 
Since $f$ is compactly supported, $\overline{f}_{\delta}$ is supported in the cylinder $(\spt(f)+\ball_{2\delta}^{(n-1)}(0))\times[0,\delta]$ and is clearly continuous. Now, whenever $0<x_{n},y_{n}<\delta$, 
\begin{align}\label{eq:starter}\tag{B.2}
\begin{split}
|(Ef)(x',x_{n})- (Ef)(y',y_{n})| & \leq \sup_{|z'|\leq 1}|f(x'+x_{n}z',0)-f(y'+y_{n}z',0)| \\ 
& \leq \|\nabla_{\tau}f\|_{\lebe^{\infty}(\R^{n-1}\times\{0\})}
\big(|x'-y'|+|x_{n}-y_{n}|\big). 
\end{split}
\end{align}
In particular, $Ef$,  and so $\overline{f}_{\delta}$, are Lipschitz. 
On the other hand, if $k\in\{1,\cdots,n-1\}$, we have 
\begin{align*}
|\partial_{x_{k}}(Ef)(x',x_{n})-\partial_{x_{k}}(Ef)(x',0)| & \leq \int_{\R^{n-1}}\rho(z')|(\partial_{x_{k}}f)(x'+x_{n}z',0))-(\partial_{x_{k}}f)(x',0)|\dif z' \\ 
& \leq \sup_{|z'|\leq 1}|(\partial_{x_{k}}f)(x'+x_{n}z',0)-(\partial_{x_{k}}f)(x',0)|. 
\end{align*}
Moreover, since $f\in\hold_{\rm c}^{1}(\R^{n-1}\times\{0\})$, the function: 
\begin{align*}
x'\mapsto \int_{\R^{n-1}}\rho(z')\nabla_{\tau}f(x'+z')\cdot z'\dif z'
\end{align*}
is uniformly continuous, and 
\begin{align*}
|\partial_{x_{n}}(Ef)(x',x_{n})-\partial_{x_{n}}(Ef)(x',0)| & \leq \int_{\R^{n-1}}\rho(z')|(\nabla_{\tau}f)(x'+x_{n}z',0)-(\nabla_{\tau}f)(x',0)|\dif z' \\ 
& \leq \sup_{|z'|\leq 1}|(\nabla_{\tau}f)(x'+x_{n}z')-(\nabla_{\tau}f)(x')|
\end{align*}
uniformly tends to zero as $x_{n}\searrow 0$. Summarizing, the functions $Ef$ and $\overline{f}_{\delta}$ belong to $\hold^{1}(\R^{n-1}\times[0,\infty))$. Moreover, as in \eqref{eq:starter}\emph{ff.}, we arrive at the estimates 
\begin{align*}
& \|Ef\|_{\lebe^{\infty}(\R^{n-1}\times[0,\infty))}\leq \|f\|_{\lebe^{\infty}(\R^{n-1}\times\{0\})},\\[1mm] 
& \|\nabla Ef\|_{\lebe^{\infty}(\R^{n-1}\times[0,\infty))}\leq \|\nabla_{\tau}f\|_{\lebe^{\infty}(\R^{n-1}\times\{0\})}.
\end{align*}
Together with the Leibniz rule, we therefore end up with 
\begin{align}\label{eq:carlsen3}\tag{B.3}
\|\nabla\overline{f}_{\delta}(x)\|_{\lebe^{\infty}(\R^{n-1}\times(0,\infty))} \leq c(\theta)\Big(\frac{1}{\delta}\|f\|_{\lebe^{\infty}(\R^{n-1}\times\{0\})}+\|\nabla_{\tau}f\|_{\lebe^{\infty}(\R^{n-1}\times\{0\})}\Big).
\end{align}
Recalling that $\partial\Omega'$ is of class $\hold^{1}$, we may undo the flattening, rotating and translating. Patching together the single localised pieces by use of a smooth partition of unity yields the requisite function $f_{\delta}\in\hold^{1}(\overline{\Omega'})$; note that, based on  \eqref{eq:carlsen3} and employing the Leibniz rule for the localised pieces, we may regroup the emerging terms to arrive at the estimate from Lemma \ref{lem:GoodLip}\ref{item:Lipextend1}(iii). This completes the proof. 
\end{proof}

\begin{proof}[Proof of Lemma \ref{lem:GoodLip}\ref{item:Lipextend2}] Throughout, note that we do not assert that $g\in\hold_{b}^{1}(\R^{n})$ (which would be impossible due to (i) and $f\in\mathrm{Lip}(\partial\Omega')$) but only $g|_{\R^{n}\setminus\overline{\Omega'}}\in\hold_{b}^{1}(\R^{n}\setminus\overline{\Omega'})$ and $g|_{\Omega'}\in\hold_{b}^{1}(\Omega')$. We adopt the same setting as in the proof of Lemma \ref{lem:GoodLip}\ref{item:Lipextend1} and firstly assume that $f\in\mathrm{Lip}_{\rm c}(\R^{n-1}\times\{0\})$. In the present situation, we let $\widetilde{\rho}\in\hold_{\rm c}^{\infty}(\R^{n-1})$ be a standard mollifier on $\R^{n-1}$ and consider 
\begin{align*}
(\widetilde{E}f)(x',x_{n})
:=\begin{cases}\displaystyle \int_{\R^{n-1}}\widetilde{\rho}_{|x_{n}|}(x'-y')f(y',0)\dif y'&\;\text{if}\;x_{n}\neq 0, \\[2mm]
f(x',0)&\;\text{if}\;x_{n}=0.
\end{cases}
\end{align*}
To achieve compact support, we may multiply $\widetilde{E}f$ with a rescaled cut-off as in the proof of Lemma \ref{lem:GoodLip}\ref{item:Lipextend1}; we do not repeat this step here. As in the proof of Lemma \ref{lem:GoodLip}\ref{item:Lipextend1}, $\widetilde{E}f$ is continuous and bounded in $\R^{n}$, and is differentiable with bounded derivatives in $\R^{n}\setminus(\R^{n-1}\times\{0\})$. Hence, it suffices to confirm (iii).  Let $\rho\in\hold_{\rm c}^{\infty}(\R^{n})$ be a standard mollifier {on} $\R^{n}$ and let $\delta>0$. For any $k\in\{1,\cdots,n-1\}$, we then have the following identity for the weak derivatives in the $k$-th direction:
\begin{align*}
&\int_{\R^{n-1}} |\partial_{x_{k}}(\rho_{\delta}*(\widetilde{E}f))(x',0) - \partial_{x_{k}}f(x',0)|\dif x' \\ 
& \! = \int_{\R^{n-1}}\left\vert\int_{\R^{n}}\rho(\xi)(\partial_{x_{k}}\widetilde{E}f)((x',0)+\delta \xi)\dif\xi - \partial_{x_{k}}f(x',0)\right\vert\dif x' \\ 
& \!=  \int_{\R^{n-1}}\left\vert\int_{\R^{n}}\int_{\R^{n-1}}\rho(\xi)\widetilde{\rho}(z')((\partial_{x_{k}}f)(x'+\delta\xi'+|\delta\xi_{n}|z',0)-(\partial_{x_{k}}f(x',0)))\dif z'\dif\xi\right\vert\dif x' \\ 
& \! =: \mathrm{I}. 
\end{align*}
We abbreviate $f_{k}:=\partial_{x_{k}}f(\cdot,0)$. Given $i\in\mathbb{N}$, we choose 
$g_{k}^{i}\in\hold_{\rm c}^{\infty}(\R^{n-1}\times\{0\})$ with $\|f_{k}-g_{k}^{i}\|_{\lebe^{1}(\R^{n-1}\times\{0\})}<2^{-i-2}$, and then fix some $\delta_{i}$ with 
\begin{align}\label{eq:deltaAppChoose}\tag{B.4}
0<\delta_{i}<\frac{2^{-i-2}}{1+\|\nabla g_{k}^{i}\|_{\lebe^{1}(\R^{n-1})}}.
\end{align}
We estimate $\mathrm{I} \leq \mathrm{II}+\mathrm{III}+\mathrm{IV}$, where now $\delta=\delta_{i}$, and the single terms are as follows:
\begin{align*}
\mathrm{II} & := \int_{\R^{n-1}}\int_{\R^{n}}\int_{\R^{n-1}}\rho(\xi)\widetilde{\rho}(z')|(f_{k}-g_{k}^{i})(x'+\delta_{i}\xi'+\delta_{i}|\xi_{n}|z',0)|\dif z'\dif\xi\dif x' \\ 
& \leq \int_{\R^{n-1}}\int_{\R^{n}}\rho(\xi)\widetilde{\rho}(z')\int_{\R^{n-1}}|(f_{k}-g_{k}^{i})(x'+\delta_{i}\xi'+\delta_{i}|\xi_{n}|z',0)|\dif x'\dif \xi\dif z' \\ 
& = \int_{\R^{n-1}}\int_{\R^{n}}\rho(\xi)\widetilde{\rho}(z')\int_{\R^{n-1}}|(f_{k}-g_{k}^{i})(x',0)|\dif x'\dif \xi\dif z' < 2^{-i-2}. 
\end{align*}
Next, we recall that both $\rho$ and $\widetilde{\rho}$ are supported in the $n$- or $(n-1)$-dimensional closed unit balls. Therefore, we have
\begin{align*}
\mathrm{III} & := \int_{\R^{n-1}}\int_{\R^{n}}\int_{\R^{n-1}} \rho(\xi)\widetilde{\rho}(z')|g_{k}^{i}(x'+\delta_{i}\xi'+\delta_{i}|\xi_{n}|z',0)-g_{k}^{i}(x',0)|\dif z'\dif\xi\dif x' \\ 
& \leq 2\delta_{i}\int_{\R^{n-1}}\int_{\R^{n}}\int_{\R^{n-1}} \rho(\xi)\widetilde{\rho}(z') \int_{0}^{1}|(\nabla g_{k}^{i})(x'+t(\delta_{i}\xi'+\delta_{i}|\xi_{n}|z'),0)|\dif t\dif z'\dif\xi\dif x' \\ 
& \leq 2\delta_{i} \int_{\R^{n}}\int_{\R^{n-1}}\rho(\xi)\widetilde{\rho}(z')\int_{0}^{1}\int_{\R^{n-1}}|(\nabla g_{k}^{i})(x'+t\delta_{i}(\xi'+|\xi_{n}|z'),0)|\dif x'\dif t \dif z'\dif\xi \\ 
& = 2\delta_{i}\int_{\R^{n}}\int_{\R^{n-1}}\rho(\xi)\widetilde{\rho}(z')\int_{0}^{1}\int_{\R^{n-1}}|(\nabla g_{k}^{i})(x',0)|\dif x'\dif t \dif z'\dif\xi \stackrel{\eqref{eq:deltaAppChoose}}{<} 2^{-i-1}.  
\end{align*}
Imitating the estimate of $\mathrm{II}$, we find  
\begin{align*}
\mathrm{IV} := \int_{\R^{n-1}}\int_{\R^{n}}\int_{\R^{n-1}}\rho(\xi)\widetilde{\rho}(z')|(f_{k}-g_{k}^{i})(x',0)|\dif z'\dif\xi \dif x' < 2^{-i-2}. 
\end{align*}
Gathering the above estimates, we have
\begin{align*}
\int_{\R^{n-1}} & |\partial_{x_{k}}(\rho_{\delta_{i}}*(\widetilde{E}f))(x',0) - \partial_{x_{k}}f(x',0)|\dif x' < 2^{-i}. 
\end{align*}
Passing to a (non-relabelled) subsequence of $(\delta_{i})$, we thus obtain
$$
\partial_{x_{k}}(\rho_{\delta_{i}}*(\widetilde{E}f))(x',0) \to \partial_{x_{k}}f(x',0)
\qquad\mbox{for $\mathscr{L}^{n-1}$-a.e. $x\in\R^{n-1}$}. 
$$
Since $\mathscr{L}^{n-1}$-a.e. $x\in\R^{n-1}$ is a differentiability point of $f$, the conclusion follows in the flat case. Localising and flattening as in the proof of Lemma \ref{lem:GoodLip}\ref{item:Lipextend1}, we may also conclude the proof in the case of open and bounded sets $\Omega'\subset\R^{n}$ with $\hold^{1}$-boundary. 
\end{proof}

\subsection*{Appendix C: Trace Theory for $\sobo^{\curl,p}$-Fields}\label{sec:AppendixC}
For completeness, we briefly return to the spaces $\sobo^{\curl,p}(\Omega)$ as introduced in \eqref{eq:sobocurl}, and revisit previous results (see \cite{Alonso,BuffaCiarlet1,BuffaCiarlet2,Sheen,Tartar1997}), some of which we recall for the reader's convenience. To this end, we recall from \cite{Alonso} the space 
\begin{align*} 
\mathcal{X}_{\partial\Omega}^{p} 
:= \big\{\Phi\in\sobo^{-1/p,p}(\partial\Omega;\R^{3})\colon\; (\Phi\cdot\nu)|_{\partial\Omega}=0\;\text{and}\;\mathrm{div}_{\tau}(\Phi)\in\sobo^{-1/p,p}(\partial\Omega)\big\}, 
\end{align*}
where $\nu=\nu_{\partial\Omega}$ represents the inner unit normal as usual. We briefly comment on the regularity of the tangential divergences based on the tangentiality $(\Phi\cdot\nu)|_{\partial\Omega}=0$. Note that, if $\FF\in\sobo^{\curl,p}(\Omega)$, 
then $\mathbf{G}:=\curl\FF\in\lebe^{p}(\Omega;\R^{3})$ 
satisfies $\di\mathbf{G}=0$, whereby $\mathbf{G}\in\sobo^{\di,p}(\Omega)$. Since there exists a (surjective) normal trace operator $\mathrm{tr}_{\nu}^{p}\colon\sobo^{\di,p}(\Omega)\to\sobo^{-1/p,p}(\partial\Omega)$, 
we have $\mathbf{G}\cdot\nu :=\mathrm{tr}_{\nu}^{p}(\mathbf{G})\in\sobo^{-1/p,p}(\partial\Omega)$. In particular, by the very definition of $\mathrm{tr}_{\nu}^{p}$, 
we have 
\begin{align}\label{eq:dividentity}\tag{C.1}
\langle(\mathbf{G}\cdot\nu),\varphi\rangle_{\partial\Omega} 
= -\int_{\Omega} (\di\mathbf{G})\,\varphi\dif x - \int_{\Omega}\mathbf{G}\cdot\nabla\varphi\dif x = -\int_{\Omega}\mathbf{G}\cdot\nabla\varphi\dif x
\end{align}
for all $\varphi\in\sobo^{1,p'}(\Omega)$. On the other hand, $\mathrm{tr}_{\tau}^{p}\colon\sobo^{\curl,p}(\Omega)\to\sobo^{-1/p,p}(\partial\Omega;\R^{3})$ satisfies 
\begin{align}\label{eq:curlidentity}\tag{C.2}
\langle \mathrm{tr}_{\tau}(\FF),\bm{\varphi}\rangle_{\partial\Omega'} = \int_{\Omega}\curl\FF\cdot\bm{\varphi}\dif x - \int_{\Omega}\FF\cdot\curl\bm{\varphi}\dif x
\end{align}
for all $\bm{\varphi}\in\sobo^{1,p'}(\Omega;\R^{3})$. Now let $\varphi\in\sobo^{2,p'}(\Omega)$, so that $\nabla\varphi\in\sobo^{1,p'}(\Omega;\R^{3})$. Using the tangentiality of $\mathrm{tr}_{\tau}^{p}(\FF)$, one finds that  
\begin{align}\label{eq:divcurlidentity}\tag{C.3}
\begin{split}
\langle(\mathbf{G}\cdot\nu),\varphi\rangle_{\partial\Omega} & \stackrel{\eqref{eq:dividentity}}{=} -\int_{\Omega}\curl\FF\cdot\nabla\varphi\dif x \\ & \stackrel{\eqref{eq:curlidentity}}{=} - \langle \mathrm{tr}_{\tau}^{p}(\FF),\nabla\varphi\rangle_{\partial\Omega} = \langle\mathrm{div}_{\tau}(\mathrm{tr}_{\tau}^{p}(\FF)),\varphi\rangle_{\partial\Omega}
\end{split}
\end{align}
for all $\varphi\in\sobo^{2,p'}(\Omega)$. The left-hand side of \eqref{eq:divcurlidentity} extends to a bounded linear functional on $\sobo^{1,p'}(\Omega)$, and so does the right-hand side. The trace space of $\sobo^{1,p'}(\Omega)$ is $\sobo^{1-1/p',p'}(\partial\Omega)$, and its dual space is $\sobo^{-1/p,p}(\partial\Omega)$. Therefore, $\mathrm{tr}_{\tau}^{p}\colon\sobo^{\curl,p}(\Omega)\to\mathcal{X}_{\partial\Omega}^{p}$ boundedly. 
\subsection*{Appendix D: Integral formulas for smooth maps}\label{sec:AppendixD}
For the reader's convenience, we here concisely collect 
some basic identities that have entered the main part; by the anti-commutativity of the cross product, some care regarding signs needs to be taken. To this end, let $\Omega\subset\R^{3}$ be open and bounded with Lipschitz boundary, and let $\FF\in\hold(\overline{\Omega};\R^{3})\cap\hold_{b}^{1}(\Omega;\R^{3})$ and 
$\varphi\in\hold(\overline{\Omega})\cap\hold_{b}^{1}(\Omega)$,  where, $\hold_{b}^{1}$ stands for the $\hold^{1}$-maps with bounded gradients. Based on our convention of $\nu_{\partial\Omega}$ denoting the \emph{inner} unit normal to $\partial\Omega$, we have the Gauss--Green-type formula:
\begin{align}\label{eq:calculus2}\tag{D.1}
\int_{\Omega}\curl \FF\dif x =  \int_{\partial\Omega}\FF\times\nu_{\partial\Omega}\dif\mathscr{H}^{2}. 
\end{align}
Because of the identity $\curl(\varphi \FF)= \varphi\,\curl\FF - \FF \times \nabla\varphi $, \eqref{eq:calculus2} yields 
\begin{align}\label{eq:calculus2a}\tag{D.2}
\int_{\Omega}  (\varphi\,\curl \FF - \FF \times \nabla \varphi )\dif x= \int_{\partial\Omega}\varphi(\FF\times\nu_{\partial\Omega})\dif\mathscr{H}^{2}. 
\end{align}
Now, if $\mathbf{G}\in\hold(\overline{\Omega};\R^{3})\cap\hold_{b}^{1}(\Omega;\R^{3})$, then we have the following integration-by-parts rule:
\begin{align}\label{eq:calculus4}\tag{D.3}
\int_{\partial \Omega}(\FF\times \mathbf{G})\cdot\nu_{\partial\Omega}\dif\mathscr{H}^{2} = \int_{\Omega}\FF\cdot\curl \mathbf{G}\dif x -\int_{\Omega}\curl \FF\cdot \mathbf{G} \dif x. 
\end{align}
Based on the rule $\mathbf{a}\cdot(\mathbf{b}\times\mathbf{c})=\mathbf{c}\cdot(\mathbf{a}\times\mathbf{b})$ for $\mathbf{a},\mathbf{b},\mathbf{c}\in\R^{3}$, we have 
\begin{align*}
(\FF\times\GG)\cdot\nu_{\partial\Omega} = \GG\cdot(\nu_{\partial\Omega}\times\FF) = - \GG\cdot(\FF\times\nu_{\partial\Omega}), 
\end{align*}
and so \eqref{eq:calculus4} yields 
\begin{align}\label{eq:calculus7}\tag{D.4}
\int_{\partial\Omega}(\FF\times\nu_{\partial\Omega})\cdot\mathbf{G}\dif\mathscr{H}^{2} = \int_{\Omega}\curl \FF\cdot\mathbf{G}\dif x - \int_{\Omega}\FF\cdot\curl \GG\dif x. 
\end{align}
We wish to point out that, alternatively, \eqref{eq:calculus7} can be derived directly from \eqref{eq:calculus2a}.

\bigskip
{\small 
\subsection*{Acknowledgments}
The research of Gui-Qiang G. Chen was supported in part
by the UK Engineering and Physical Sciences Research Council Award
EP/L015811/1, EP/V008854, and EP/V051121/1. The research of 
Franz Gmeineder is supported by the Hector Foundation (Project Nr. FP 626/21). He wishes to thank his colleagues Robert Denk,  Markus  Kunze and Oliver Schn\"{u}rer for useful discussions around the theme of the paper.}

\end{document}